%% file: hcbgs.tex
\DeclareRobustCommand{\GenericInfo}[2]{} 
\RequirePackage[2022-06-01]{latexrelease} 
\documentclass{amsart}
\usepackage[margin=2cm]{geometry}

\makeatletter
\DeclareRobustCommand*\cal{\@fontswitch\relax\mathcal}
\DeclareRobustCommand*\mit{\@fontswitch\relax\mathnormal}
\DeclareRobustCommand*\frak{\@fontswitch\relax\mathfrak}
\makeatother

\usepackage[T1]{fontenc} 
\usepackage{newunicodechar}
\makeatletter
\def\newunicodechar#1#2{%
  \@tempswafalse
  \edef\nuc@tempa{\detokenize{#1}}%
  \if\relax\nuc@tempa\relax
    \nuc@emptyargerr
  \else
    \edef\@tempb{\expandafter\@car\nuc@tempa\@nil}%
    \nuc@check
    \if@tempswa
      \@namedef{u8:\nuc@tempa}{#2}%
    \fi
  \fi
}
\makeatother

\usepackage{amssymb} 
\usepackage[mathscr]{euscript} 

\usepackage[all,cmtip]{xy}
\SilentMatrices

\usepackage{tikz}
\usetikzlibrary{cd}
\usetikzlibrary{decorations.markings}
\usetikzlibrary{decorations.pathmorphing}
\usetikzlibrary{decorations.pathreplacing,calc,tikzmark}
\usetikzlibrary{positioning,arrows.meta,automata,shapes}
\input lectures.tikzstyles

\usepackage{pgfplots}
\pgfplotsset{width=7cm,compat=1.18}
\pgfdeclarelayer{edgelayer} 
\pgfdeclarelayer{nodelayer} 
\pgfsetlayers{edgelayer,nodelayer,main}

\usepackage{xr}[=v5] 

\usepackage{iftex} 
\ifpdf
\PassOptionsToPackage{pdfencoding=auto,psdextra}{hyperref}

\else
\PassOptionsToPackage{hypertex}{hyperref}
\fi
\usepackage[backref=section,linktoc=all]{hyperref} 


\usepackage[shortlabels]{enumitem} 

\usepackage[capitalize,nosort]{cleveref} 
\crefformat{equation}{(#2#1#3)} 
\crefmultiformat{equation}{(#2#1#3)}{ and~(#2#1#3)}{, (#2#1#3)}{, and~(#2#1#3)} 
\crefname{figure}{Figure}{Figures}
\Crefname{enumi}{}{}

\Crefname{subdefinition}{Definition}{Definitions}
\Crefname{subenumi}{Definition}{Definitions}
\AtBeginEnvironment{subdefinition}{%
    \crefalias{enumi}{subenumi}%
    \setlist[enumerate,1]{ref={\thedefinition.(\arabic*)}}%
}

\Crefname{conditionenumi}{Condition}{Conditions}
\newenvironment{conditions}{\crefalias{enumi}{conditionenumi}\enumerate}{\endenumerate\crefalias{enumi}{enumi}}
\Crefname{partenumi}{Part}{Parts}
\newenvironment{parts}{\crefalias{enumi}{partenumi}\enumerate}{\endenumerate\crefalias{enumi}{enumi}}

\def\axsymbol{{\sf A}} 
\def\frakaxsymbol{$\mathfrak A$}
\Crefname{axiomdefinition}{Definition}{Definitions}
\Crefname{axiomenumi}{Axiom}{Axioms}
\Crefformat{axiomenumi}{Axiom~#2#1#3}
\Crefrangeformat{axiomenumi}{Axioms #3#1#4--#5#2#6}
\Crefmultiformat{axiomenumi}{Axioms~#2#1#3}{ and~#2#1#3}{, #2#1#3}{, and~#2#1#3}
\AtBeginEnvironment{axiomdefinition}{%
    \crefalias{enumi}{axiomenumi}%
}

\def\doi#1{\href{https://doi.org/#1}{doi:#1}}
\def\arXiv#1{\href{https://arxiv.org/abs/#1}{arXiv:#1}}
\def\jstor#1{\href{https://www.jstor.org/stable/#1}{JSTOR:#1}}
\def\numdam#1{\href{http://www.numdam.org/item/?id=#1}{numdam:#1}}

\makeatletter
\let\over\@@over
\let\atop\@@atop
\let\above\@@above
\let\overwithdelims\@@overwithdelims
\let\atopwithdelims\@@atopwithdelims
\let\abovewithdelims\@@abovewithdelims
\def\eqlabel#1{\refstepcounter{equation}\label{#1}\ifmmode\ifinner\else\eqno\fi\fi\hbox{\@eqnnum}} 
\makeatother

\makeatletter
\@addtoreset{equation}{section}
\@addtoreset{equation}{subsection}
\let\proof@qed\displaymath@qed
\let\c@subsubsection\c@equation

\makeatother

\theoremstyle{definition}
\newtheorem{theorem}[equation]{Theorem}
\newtheorem{definition}[equation]{Definition}
\newtheorem{subdefinition}[equation]{Definition}
\newtheorem{axiomdefinition}[equation]{Definition}
\newtheorem{proposition}[equation]{Proposition}
\newtheorem{conjecture}[equation]{Conjecture}

\newtheorem{remark}[equation]{Remark}

\newtheorem{notation}[equation]{Notation}

\newtheorem{example}[equation]{Example}

\numberwithin{equation}{subsection}
\numberwithin{subsubsection}{subsection}

\newunicodechar{α}{\alpha} \newunicodechar{β}{\beta} \newunicodechar{γ}{\gamma} \newunicodechar{δ}{\delta} \newunicodechar{ϵ}{\epsilon} \newunicodechar{ζ}{\zeta} \newunicodechar{η}{\eta} \newunicodechar{θ}{\theta} \newunicodechar{ι}{\iota} \newunicodechar{κ}{\kappa} \newunicodechar{λ}{\lambda} \newunicodechar{μ}{\mu} \newunicodechar{ν}{\nu} \newunicodechar{ξ}{\xi} \newunicodechar{ο}{o} \newunicodechar{π}{\pi} \newunicodechar{ρ}{\rho} \newunicodechar{σ}{\sigma} \newunicodechar{τ}{\tau} \newunicodechar{υ}{\upsilon} \newunicodechar{ϕ}{\phi} \newunicodechar{χ}{\chi} \newunicodechar{ψ}{\psi} \newunicodechar{ω}{\omega} \newunicodechar{ε}{\varepsilon} \newunicodechar{ϑ}{\vartheta} \newunicodechar{ϖ}{\varpi} \newunicodechar{ϱ}{\varrho} \newunicodechar{ς}{\varsigma} \newunicodechar{φ}{\varphi} \newunicodechar{Γ}{\Gamma} \newunicodechar{Δ}{\Delta} \newunicodechar{Θ}{\Theta} \newunicodechar{Λ}{\Lambda} \newunicodechar{Ξ}{\Xi} \newunicodechar{Π}{\Pi} \newunicodechar{Σ}{\Sigma} \newunicodechar{Υ}{\Upsilon} \newunicodechar{Φ}{\Phi} \newunicodechar{Ψ}{\Psi} \newunicodechar{Ω}{\Omega} \newunicodechar{ℵ}{\aleph} \newunicodechar{ℏ}{\hbar} \newunicodechar{ℓ}{\ell} \newunicodechar{℘}{\wp} \newunicodechar{ℜ}{\Re} \newunicodechar{ℑ}{\Im} \newunicodechar{∂}{\partial} \newunicodechar{∞}{\infty} \newunicodechar{∅}{\emptyset} \newunicodechar{∇}{\nabla} \newunicodechar{√}{\surd} \newunicodechar{⊤}{\top} \newunicodechar{⊥}{\bot} \newunicodechar{∠}{\angle} \newunicodechar{△}{\triangle} \newunicodechar{∀}{\forall} \newunicodechar{∃}{\exists} \newunicodechar{♭}{\flat} \newunicodechar{♮}{\natural} \newunicodechar{♯}{\sharp} \newunicodechar{♣}{\clubsuit} \newunicodechar{♢}{\diamondsuit} \newunicodechar{♡}{\heartsuit} \newunicodechar{♠}{\spadesuit} \newunicodechar{∐}{\coprod} \newunicodechar{⋁}{\bigvee} \newunicodechar{⋀}{\bigwedge} \newunicodechar{⨄}{\biguplus} \newunicodechar{⋂}{\bigcap} \newunicodechar{⋃}{\bigcup} \newunicodechar{∫}{\int} \newunicodechar{∏}{\prod} \newunicodechar{∑}{\sum} \newunicodechar{⨂}{\bigotimes} \newunicodechar{⨁}{\bigoplus} \newunicodechar{⨀}{\bigodot} \newunicodechar{⨆}{\bigsqcup} \newunicodechar{◁}{\triangleleft} \newunicodechar{▷}{\triangleright} \newunicodechar{▽}{\bigtriangledown} \newunicodechar{∧}{\wedge} \newunicodechar{∨}{\vee} \newunicodechar{∩}{\cap} \newunicodechar{∪}{\cup} \newunicodechar{⊓}{\sqcap} \newunicodechar{⊔}{\sqcup} \newunicodechar{⊎}{\uplus} \newunicodechar{⨿}{\amalg} \newunicodechar{⋄}{\diamond} \newunicodechar{∙}{\bullet} \newunicodechar{≀}{\wr} \newunicodechar{⊙}{\odot} \newunicodechar{⊘}{\oslash} \newunicodechar{⊗}{\otimes} \newunicodechar{⊖}{\ominus} \newunicodechar{⊕}{\oplus} \newunicodechar{∓}{\mp} \newunicodechar{∘}{\circ} \newunicodechar{○}{\Orb} \newunicodechar{∖}{\setminus} \newunicodechar{⋅}{\cdot} \newunicodechar{∗}{\ast} \newunicodechar{⨯}{\times} \newunicodechar{⋆}{\star} \newunicodechar{∝}{\propto} \newunicodechar{⊑}{\sqsubseteq} \newunicodechar{⊒}{\sqsupseteq} \newunicodechar{∥}{\parallel} \newunicodechar{∣}{\divides} \newunicodechar{⊣}{\dashv} \newunicodechar{⊢}{\vdash} \newunicodechar{↗}{\nearrow} \newunicodechar{↘}{\searrow} \newunicodechar{↖}{\nwarrow} \newunicodechar{↙}{\swarrow} \newunicodechar{⇔}{\Leftrightarrow} \newunicodechar{⇐}{\Leftarrow} \newunicodechar{⇒}{\Rightarrow} \newunicodechar{≠}{\neq} \newunicodechar{≤}{\leq} \newunicodechar{≥}{\geq} \newunicodechar{≻}{\succ} \newunicodechar{≺}{\prec} \newunicodechar{≈}{\approx} \newunicodechar{≽}{\succeq} \newunicodechar{≼}{\preceq} \newunicodechar{⊃}{\supset} \newunicodechar{⊂}{\subset} \newunicodechar{⊇}{\supseteq} \newunicodechar{⊆}{\subseteq} \newunicodechar{∈}{\in} \newunicodechar{∋}{\ni} \newunicodechar{≫}{\gg} \newunicodechar{≪}{\ll} \newunicodechar{↔}{\leftrightarrow} \newunicodechar{↦}{\mapsto} \newunicodechar{∼}{\sim} \newunicodechar{≃}{\simeq} \newunicodechar{≡}{\equiv} \newunicodechar{≍}{\asymp} \newunicodechar{⌣}{\smile} \newunicodechar{⌢}{\frown} \newunicodechar{↼}{\leftharpoonup} \newunicodechar{↽}{\leftharpoondown} \newunicodechar{⇀}{\rightharpoonup} \newunicodechar{⇁}{\rightharpoondown} \newunicodechar{↪}{\hookrightarrow} \newunicodechar{↩}{\hookleftarrow} \newunicodechar{⋈}{\bowtie} \newunicodechar{⊨}{\models} \newunicodechar{⟹}{\Longrightarrow} \newunicodechar{⟶}{\longrightarrow} \newunicodechar{⟵}{\longleftarrow} \newunicodechar{⟸}{\Longleftarrow} \newunicodechar{⟼}{\longmapsto} \newunicodechar{⟷}{\longleftrightarrow} \newunicodechar{⟺}{\Longleftrightarrow} \newunicodechar{⋯}{\cdots} \newunicodechar{⋮}{\vdots} \newunicodechar{⋱}{\ddots} \newunicodechar{↕}{\updownarrow} \newunicodechar{⇑}{\Uparrow} \newunicodechar{⇓}{\Downarrow} \newunicodechar{⇕}{\Updownarrow} \newunicodechar{⌉}{\rceil} \newunicodechar{⌈}{\lceil} \newunicodechar{⌋}{\rfloor} \newunicodechar{⌊}{\lfloor} \newunicodechar{≅}{\cong} \newunicodechar{∉}{\notin} \newunicodechar{⇌}{\rightleftharpoons} \newunicodechar{≐}{\doteq} \newunicodechar{∄}{\not\exists} \newunicodechar{∌}{\not\ni} \newunicodechar{∔}{\dot+} \newunicodechar{∕}{/} \newunicodechar{∤}{\not|} \newunicodechar{∦}{\not\|} \newunicodechar{∬}{\int\!\!\!\int} \newunicodechar{∭}{\int\!\!\!\int\!\!\!\int} \newunicodechar{∮}{\oint} \newunicodechar{∸}{\dot-} \newunicodechar{≁}{\not\sim} \newunicodechar{≄}{\not\simeq} \newunicodechar{≆}{\not\cong} \newunicodechar{≇}{\not\cong} \newunicodechar{≉}{\not\approx} \newunicodechar{≔}{:=} \newunicodechar{≕}{=:} \newunicodechar{≢}{\not\equiv} \newunicodechar{≭}{\not\asump} \newunicodechar{≮}{\not<} \newunicodechar{≯}{\not<} \newunicodechar{≰}{\not\le} \newunicodechar{≱}{\not\ge} \newunicodechar{⊀}{\not\prec} \newunicodechar{⊁}{\not\succ} \newunicodechar{⊄}{\not\subset} \newunicodechar{⊅}{\not\supset} \newunicodechar{⊈}{\not\subseteq} \newunicodechar{⊉}{\not\supseteq} \newunicodechar{⊦}{\vdash} \newunicodechar{⊧}{\models} \newunicodechar{⊬}{\not\vdash} \newunicodechar{⊭}{\not\models} \newunicodechar{⊲}{\triangleleft} \newunicodechar{⊳}{\triangleright} \newunicodechar{⋠}{\not\preceq} \newunicodechar{⋡}{\not\succeq} \newunicodechar{⋤}{\not\sqsubseteq} \newunicodechar{⋥}{\not\sqsupseteq} \newunicodechar{⋪}{\not\triangleleft} \newunicodechar{⋫}{\not\triangleright} \newunicodechar{◻}{\square}

\newunicodechar{↑}{\uparrow}
\newunicodechar{↓}{\downarrow}
\newunicodechar{→}{\to}
\newunicodechar{←}{\gets}
\newunicodechar{⟨}{\langle}
\newunicodechar{⟩}{\rangle}
\newunicodechar{±}{\pm}
\newunicodechar{…}{\ldots}
\newunicodechar{∞}{\ifmmode\infty\else$\infty$\fi}
\newunicodechar{≔}{\colon=}
\let\oldO\O
\newunicodechar{Ø}{\oldO}
\newunicodechar{⊠}{\boxtimes}
\newunicodechar{◻}{\mathbin{\square}}
\newunicodechar{⇄}{\rightleftarrows}
\def\ppet{\mathbin{\bar\boxtimes}} 

\let\into\hookrightarrow

\def\nonisotopy#1{{\sf#1}}
\def\isotopy#1{{\frak#1}}
\def\nongeo#1{\rm#1}
\def\geo#1{G#1}
\def\smooth#1{C^\infty#1}

\def\smEmb{\nonisotopy{Emb}}
\def\fraksmEmb{\isotopy{Emb}}
\def\FEmb{\nonisotopy{\geo{Emb}}}
\def\frakFEmb{\isotopy{\geo{Emb}}}
\def\smFEmb{\nonisotopy{\smooth{Emb}}}
\def\fraksmFEmb{\isotopy{\smooth{Emb}}}
\def\SFEmb{\nonisotopy{\smooth{SEmb}}}
\def\frakSFEmb{\isotopy{\smooth{SEmb}}}
\def\cFEmb{\nonisotopy{HolEmb}}
\def\frakcFEmb{\isotopy{HolEmb}}

\def\Struct{\nonisotopy{Field}}
\def\frakStruct{\isotopy{Field}}
\def\frakI{{\frak I}} 

\def\Bord{\nonisotopy{Bord}} 
\def\frakBord{\isotopy{Bord}} 
\def\CatBord{\nonisotopy{CBord}} 
\def\frakCatBord{\isotopy{CBord}} 
\def\ncBord{\nonisotopy{PreBord}} 
\def\frakncBord{\mathfrak{PreBord}} 
\def\EBord{\nonisotopy{EBord}} 
\def\frakEBord{\isotopy{EBord}} 

\def\Cat#1{Cat^\otimes_{∞,\mit#1}}
\def\Catuple#1{Cat^{\otimes,\uple}_{∞,\mit#1}}
\def\cat#1{\nonisotopy{\nongeo{\Cat{#1}}}}
\def\catuple#1{\nonisotopy{\nongeo{\Catuple{#1}}}}
\def\smcat#1{\nonisotopy{\geo{\Cat{#1}}}}
\def\smcatuple#1{\nonisotopy{\geo{\Catuple{#1}}}}
\def\fraksmcat#1{\isotopy{\geo{\Cat{#1}}}}
\def\fraksmcatuple#1{\isotopy{\geo{\Catuple{#1}}}}

\def\ecm{\nonisotopy{e}} 
\def\frakecm{\isotopy{e}} 

\def\stcart{\nonisotopy{\geo{Cart}}}
\def\cart{\nonisotopy{\smooth{Cart}}} 
\def\scart{\nonisotopy{\smooth{SCart}}} 
\def\ccart{\nonisotopy{HolCart}} 

\def\man{\nonisotopy{\smooth{Man}}} 
\def\sman{\nonisotopy{\smooth{SMan}}} 
\def\cman{\nonisotopy{HolMan}} 

\def\cahiers{\nonisotopy{\smooth{Cahiers}}} 

\DeclareMathOperator{\red}{\mathtt{Re}} 
\def\base{\mathop{\flat}\nolimits}
\def\relgrobare{\smallint\nolimits}
\def\relgro{\relgrobare_{\base}}

\def\FFT{\nonisotopy{FFT}} 
\def\frakFFT{\isotopy{FFT}} 
\def\FRiem{\nonisotopy{Met}} 

\def\sm{{\rm C}^\infty}
\def\smpi{{\tilde\pi}_0}
\def\smset{\sm\set}
\def\smsset{\sm\sset}
\def\sPSh{\PSh_\Delta}
\def\smsPSh{{\sm\sPSh}}

\def\Inj{\mathop{\rm Inj}\nolimits}
\def\Surj{\mathop{\rm Surj}\nolimits}
\def\Aut{\mathop{\rm Aut}\nolimits}
\def\image{\mathop{\rm Im}\nolimits}

\def\ldf{\mathbb{L}} 
\def\dtp{\otimes^\ldf} 
\def\rdf{\mathbb{R}} 
\def\dmap{\rdf\map} 
\def\map{\mathop{\rm Map}\nolimits} 
\def\hom{\mathop{\rm Hom}\nolimits} 
\def\rmap{\mathop{\rm map}\nolimits} 
\def\Hom{\mathop{\rm Hom}\nolimits} 
\def\Homtheta{\Hom_\Theta} 
\def\Fun{\mathop{\rm Fun}\nolimits} 
\def\Funmonglob{\Fun^\otimes}
\def\Funmonuple{\Fun^\otimes_{\uple}}
\def\tglob{\otimes_{\glob}}

\def\Cut{\nonisotopy{Cut}}

\def\frakCut{\isotopy{Cut}}

\def\tCut{\Cut_\pitchfork^\uple}
\def\tCutglob{\Cut_\pitchfork}

\def\mtCut{\Cut_{⊗,\pitchfork}^\uple}
\def\mtCutglob{\Cut_{⊗,\pitchfork}}
\def\frakmtCut{\frakCut_{⊗,\pitchfork}^\uple}
\def\frakmtCutglob{\frakCut_{⊗,\pitchfork}}

\def\SO{{\rm SO}}
\def\Spin{{\rm Spin}}

\def\GL{{\rm GL}}
\def\lgO{{\rm O}}
\def\lgU{{\rm U}}
\def\Eucl{{\rm Eucl}}

\def\NN{{\bf N}}
\def\ZZ{{\bf Z}}
\def\RR{{\bf R}}
\def\EE{{\bf E}}
\def\CC{{\bf C}}
\def\Or{{\sf Or}} 

\def\T{{\sf T}} 
\def\trivialsite{{\bf1}} 

\def\munit{{\bf1}} 
\def\latch{{\cal L}} 

\def\proj{{\sf proj}}
\def\inj{{\sf inj}}
\def\Reedy{{\sf Reedy}}
\def\glob{{\sf glob}}
\def\local{{\sf local}}
\def\uple{{\sf uple}}
\def\Cech{{\sf \check Cech}} 

\def\crep{{\cal Q}} 
\def\frep{{\cal R}} 
\def\lbl{{\tt L}} 
\def\Fr{{\sf Fr}} 

\def\gs{{\cal F}} 
\def\site{{\cal S}} 

\def\cA{{\cal A}} 

\def\cB{{\cal B}} 

\def\smallcat{\mathscr{C}\mathrm{at}}
\def\set{\mathscr{S}\mathrm{et}}
\def\sset{\mathrm{s}\set}
\def\po{\mathscr{P}{\rm o}} 
\def\poset{\mathscr{P}{\rm o}\set}
\def\Sh{\mathscr{S}\mathrm{h}}
\def\PSh{\mathscr{PS}\mathrm{h}}

\def\hocolim{\mathop{\rm hocolim}} 
\def\holim{\mathop{\rm holim}} 
\def\colim{\mathop{\rm colim}} 
\def\sing{\mathop{\sf sing}} 
\def\ncore{\mathop{\rm core}\nolimits} 

\def\CSAlg{{\sf CSAlg}} 

\def\vcat{\mathscr{C}} 
\def\catC{\mathscr{C}}
\def\catD{\mathscr{D}}
\def\catV{{\sf V}}
\def\catW{{\sf W}}

\def\ev{\mathop{\rm ev}\nolimits}

\def\Yo#1{{\cal Y}_{#1}} 

\def\id{{\rm id}} 
\def\tr{{\rm tr}} 
\def\Vect{{\rm Vect}} 
\def\ad{{\rm ad}} 
\def\et{{\rm et}} 
\def\op{{\rm op}} 
\def\el{{\rm el}} 
\def\nerve{{\cal N}} 
\def\diag{\mathop{\sf diag}\nolimits} 
\def\deloop{{\bf B}} 
\def\delooptotal{{\bf E}} 
\def\tdeloop{{\rm B}} 
\def\tdelooptotal{{\rm E}} 
\def\hq{{/\!/}} 
\def\Bunconn{{\sf Bun}_\nabla} 

\def\Se{{\sf Se}} 
\def\Cpt{{\sf Cpt}} 
\def\thloc{\mathscr{S}} 

\def\Din{{\rm in}}
\def\Daux{{\rm aux}}
\def\Dout{{\rm out}}

\def\catObj{{\rm obj}}
\def\catMor{{\rm mor}}
\def\catIso{{\rm isot}}
\def\catVirt{{\rm vir}}
\def\catMult{{\rm mobj}}
\def\catMultIso{{\rm misot}}

\def\catMorIso{{\rm morisot}}
\def\catEmpty{\emptyset}
\def\catEmptyIso{\emptyset\catIso}

\def\gsim{\Delta_{e}} 
\def\gbou{\partial\Delta_{e}} 

\def\cC{{\sf C}} 
\def\cD{{\sf D}} 

\def\frakFRiem{\mathfrak{Met}}

\def\ltoarr#1{\mathop{\count0=#1 \loop\ifnum\count0>0 \smash-\mkern-7mu \advance\count0 -1 \repeat \mathord\rightarrow}\limits} 
\def\lto#1#2{\mathrel{\ltoarr{#1}^{#2}}} 
\def\lgetsarr#1{\mathop{\mathord\leftarrow \count0=#1 \loop\ifnum\count0>0 \mkern-7mu\smash-\advance\count0 -1 \repeat}\limits} 
\def\lgets#1#2{\mathrel{\lgetsarr{#1}\limits^{#2}}} 
\def\lmapsto{\mapstochar\lto}

\mathcode`:="603A 
\mathchardef\colon="303A 

\def\eprefix{EL-}
\def\gprefix{GL-}
\def\sprefix{SL-}
\def\eecref#1{\cref*{\eprefix#1}}
\def\ggcref#1{\cref*{\gprefix#1}}
\def\sscref#1{\cref*{\sprefix#1}}
\def\ecref#1{Grady–Pav\-lov \cite[\eecref{#1}]{GradyPavlov.Loc}}
\def\gcref#1{Grady–Pav\-lov \cite[\ggcref{#1}]{GradyPavlov.GCH}}
\def\scref#1{Grady–Pav\-lov \cite[\sscref{#1}]{GradyPavlov.Str}}

\makeatletter
\DeclareRobustCommand\and{\end{tabular}\hskip 1em plus.17fil \begin{tabular}[t]{c}}
\def\maketitle{\null\vskip 2em \begin{center}{\LARGE\@title\par}\vskip 1.5em {\large\lineskip .5em \begin{tabular}[t]{c}\authors\end{tabular}\par}\end{center}\par\vskip 1.5em \typesetabstract}
\def\typesetabstract{\begingroup
	\skip0=20pt \advance\skip0 -\lastskip \advance\skip0 -\baselineskip \vskip\skip0
	\box\abstractbox
	\prevdepth0pt
	\normalsize
	\dimen0=34pt \advance\dimen0 -\baselineskip \vskip\dimen0 \relax 
	\endgroup}
\def\subsubsection{\@startsection{subsubsection}{3}%
  \z@{.5\linespacing\@plus.7\linespacing}{-.5em}%
  {\normalfont\bfseries}}
\makeatother

\begin{document}

\author{Daniel Grady\\\href{https://www.gradydaniel.com/}{gradydaniel.com}
\\Wichita State University}
\author{Dmitri Pavlov\\\href{https://dmitripavlov.org/}{dmitripavlov.org}
\\Texas Tech University}


\title{Higher categories of bordisms with geometric structures}

\begin{abstract}
We introduce a system of axioms that uniquely defines an $(\infty,d)$-category of bordisms equipped with geometric data.
The underlying manifolds of these bordisms may be smooth, complex, super, or formal smooth manifolds,
as well as any class of manifolds satisfying conditions specified in this paper.
We develop a general notion of a field on a manifold,
encompassing structures such as Riemannian metrics, principal bundles with connection, conformal structures, and traditional tangential structures.
Using this framework, we construct a symmetric monoidal $(∞,d)$-category of bordisms with prescribed underlying manifolds and fields,
and prove that it satisfies our axioms.
\end{abstract}

\maketitle

\tableofcontents

\section{Introduction}

This paper, together with \cite{GradyPavlov.Loc,GradyPavlov.Str,GradyPavlov.GCH},
is the first of a series devoted to the study of nontopological extended functorial field theories (FFTs),
culminating in a proof of a geometric version of the cobordism hypothesis (Grady–Pavlov \cite{GradyPavlov.GCH}).

For a geometric symmetric monoidal $(\infty,d)$-category $\vcat$, we are interested in the moduli space of functorial field theories
$$\Funmonglob(\frakBord_d^\gs,\vcat).$$
A direct analysis of this object would be incredibly complicated, due to the large amount of structure imposed on both the source and target of the mapping object.
Following the topological cobordism hypothesis of Baez–Dolan \cite{BaezDolan}, and Lurie \cite{Lurie.TFT}, we seek to identify the above moduli object in simpler terms.
The end result (\gcref{intro.mainthm.geometric}) is a classification theorem that is similar to statement of the cobordism hypothesis,
as formulated by Lurie \cite[Theorem 2.4.18]{Lurie.TFT}. 

The framework of functorial quantum field theory was proposed by Segal \cite{Segal.CFT},
following Witten's axiomatic approach to path integrals
in the context of a two-dimensional conformal field theory.
Segal's definition can be summarized as follows:
a $d$-dimensional quantum field theory $Z$ is a monoidal functor
$$Z:{\rm Bord}_d\to \Vect_k,$$
where ${\rm Bord}_d$ is the $d$-dimensional bordism category (possibly equipped with additional structures),
with monoidal structure given by disjoint union of bordisms, and $\Vect_k$ is the category of vector spaces over a field~$k$.

Thus, an FFT assigns to an object in ${\rm Bord}_d$ (i.e., a closed $(d-1)$-dimensional smooth manifold $M_{d-1}$) a vector space $Z(M_{d-1})$.
This vector space is to be thought of as the vector space of quantum states (hence, emphasizing the Schrödinger picture).
Following Segal and Atiyah \cite{Atiyah}, Freed \cite{Freed.Extended,Freed.Ext} and Lawrence \cite{Lawrence}
proposed \emph{extended} topological quantum field theories,
which incorporate bordisms with corners of codimension~2 and higher.

In this paper, we introduce a version of an extended bordism category that is capable of handling nontopological fields on bordisms.
We also set the stage for the follow-up paper (Grady–Pavlov \cite{GradyPavlov.Loc}) on locality of extended field theories,
by introducing axioms for extended geometric bordism categories and proving that our definition satisfies these axioms.
The locality result in Grady–Pavlov \cite{GradyPavlov.Loc} holds for any extended bordism category satisfying the axioms,
which makes the proof independent of the construction of the bordism category presented here.

The first definition of a (nonextended) geometric bordism category is due to Segal \cite{Segal.CFT},
building on Milnor's definition \cite[Section 1]{Milnor} in the topological case.
A more elaborate bordism category capable of handling various geometries
was constructed by Stolz and Teichner \cite{StolzTeichner.Elliptic,StolzTeichner.SUSY} in their study of supersymmetric Euclidean field theories.
Despite these constructions, very little work has been done in the extended setting.
One reason for this is that a rigorous definition (in the topological case) of an extended bordism category
has not appeared until Lurie \cite[Section 2.2]{Lurie.TFT} and Calaque–Scheimbauer \cite[Section~5]{CalaqueScheimbauer}.
Moreover, to talk about the nontopological case, one needs to parameterize the bordism category over various flavors of manifolds (e.g., smooth, holomorphic, super),
and this adds another layer of complexity that has not yet been explored in the literature in the extended case.

To incorporate a variety of geometric fields on bordisms,
such as Riemannian metrics, orientations, spin and string structures,
as well as more complicated higher geometric fields such as principal bundles with connection or geometric string structures,
we introduce the notion of a \emph{field stack}.
Motivated by fields and their usage in physics, we make the following three observations.
\begin{enumerate}[(1)]
\item Given a field on a bordism, we should be able to restrict the field to an open subset of the bordism.
\item Given fields on elements of an open cover of a bordism that are compatible on intersections and higher intersections (possibly up to a gauge transformation),
we should be able to reconstruct a field on the entire bordism.
\item A functorial field theory should map smooth, holomorphic, or super families of bordisms and fields to corresponding families in the target category.
Such families can be encoded by varying both the bordism data and the field data with respect to some parameter space.
We expect that the values of the field theory should pull back along a map of parameter spaces.
\end{enumerate}

Thus, a class of fields should give rise to a sheaf (or a higher sheaf)
on the site of families of $d$-dimensional manifolds with fiberwise open embeddings as morphisms,
and it is natural to consider any sheaf on this site as a legitimate space of fields.
The definition of a field stack is given in \cref{geostr}.
The following list provides examples of geometric structures that can be encoded in our setting.
\begin{itemize}
\item Smooth maps to a fixed target manifold, i.e., sigma models.
\item Riemannian and pseudo-Riemannian metrics, possibly with restrictions on Ricci or sectional curvature.
\item Super-Euclidean structures of Stolz–Teichner \cite[Section~2.5]{StolzTeichner.SUSY}.
\item Topological tangential structures, such as orientation, spin structure, string structure, and framing.
\item Principal $G$-bundles with connection.
\item Conformal, complex, symplectic, contact, and Kähler structures.
\item Bundle $n$-gerbes with connection.
\item Geometric spin, string, and fivebrane structures (Fiorenza–Schreiber–Stasheff \cite[Definition 6.3.1]{FiorenzaSchreiberStasheff}).
\item Immersions into a fixed target manifold.
\item Foliations, possibly equipped with additional structures such as transversal metrics.
\end{itemize}

Continuing with observation (3) above, we will need to make sense of families of objects and morphisms in both the source and the target category for field theories.
This leads us to the notion of a \emph{geometric symmetric monoidal $(\infty,d)$-category} (\cref{globular.model.structure}).
In particular, for a field stack $\gs$, we get a geometric symmetric monoidal $(\infty,d)$-category of bordisms $\Bord_d^\gs$ (\cref{bord}).

Our extended bordism category is different from the definitions of Lurie \cite[Definition Sketch 2.4.17]{Lurie.TFT},
Calaque–Scheimbauer \cite[Definition~9.10]{CalaqueScheimbauer},
and Schommer-Pries \cite[Definition~5.8]{SchommerPries},
see \cref{comparison.conjecture}, though.
For experts, we list some of the most important differences.
\begin{itemize}
\item We allow nontopological fields on bordisms, such as Riemannian metrics or principal bundles with connections.
\item Bordisms come in geometric families, parameterized by a structured manifold (e.g., smooth, holomorphic, super),
and functorial field theories map families of bordisms to families in the target category.
\item
Since fields glue along open subsets, sources and targets of bordisms must be equipped with two-sided collars that themselves carry fields.
The actual codimension~1 manifold can then be extracted from such a collar as $h^{-1}(0)$ for a height function $h:M→\RR$ that has 0 as its regular value.
We do not keep track of such a height function~$h$.
Instead, we keep track of the \emph{cut} induced by~$h$, defined as the triple of sets $(h^{-1}(-∞,0),h^{-1}(0),h^{-1}(0,∞))$ (\cref{cut}).
\item We allow cuts in a composable sequence of bordisms to overlap (\cref{cut.figure}), but not to intersect transversally.
These features allow us to implement \emph{strict} globularity in our bordism categories (\cref{globular.cut.grid.example,globulargrid}).
This leads to simplifications compared to the usual approaches involving cylindrical neighborhoods, which produce weakly equivalent bordism categories.
\item We encode the homotopy type of the diffeomorphism group of a bordism as a \emph{smooth set}, i.e., a presheaf on $\cart$.
This presheaf-theoretic approach enables the use of stalks (\cref{stalk.eq}), leading to significantly more efficient proofs.
By encoding the moduli spaces of bordisms as smooth sets, we avoid using Whitney's $\sm$-topology and the resulting technicalities.
We do not embed bordisms into a contractible space like $\RR^∞$ or attempt to topologize the resulting set.
\item We do not require bordism categories to be fibrant in some model structure (\cref{not.fibrant,bordism.not.fibrant}).
All constructions in this series of papers respect weak equivalences and work equally well with fibrant or nonfibrant objects.
Furthermore, fibrancy is useful when mapping \emph{into} an object, whereas a functorial field theory maps \emph{out} of a bordism category.
The latter type of maps benefits from a cofibrancy condition, and all objects are cofibrant in the relevant model category.
Removing the fibrancy condition leads to further simplifications in the proofs.
For example, the embedded bordism category admits a simplified description (\cref{embcat.v1}) as a presheaf of sets,
which does not hold for a fibrant version because of \cref{deck.transformations}.
Working with a nonfibrant bordism category does not require us to add h-cobordisms to the moduli stack of objects, which would further complicate the proof.
\item We use \emph{embedded bordism categories} (\cref{embcat.v1})
as a fundamental building block of bordism categories with geometric structures.
These categories are encoded as presheaves of (smooth) \emph{sets} (as opposed to (smooth) simplicial sets)
on the category $\stcart⨯Γ⨯Δ^{⨯d}$.
This allows for further simplifications in the proof of locality (Grady–Pavlov \cite{GradyPavlov.Loc})
and the geometric cobordism hypothesis (Grady–Pavlov \cite{GradyPavlov.GCH}).
\end{itemize}

The above notions lead to a natural definition of a \emph{geometric} extended functorial field theory.

\begin{definition}
\label{def.fft}
Let $T$ be a fibrant geometric symmetric monoidal $(\infty,d)$-category (\cref{globular.model.structure}).
A \emph{$T$-valued extended field theory} $Z$ with field stack $\gs\in \Struct_d$
is a map
$$Z:\Bord_d^\gs\to T$$
in $\smcat{d}$.
\end{definition}

In the case of representable field stacks~$\gs$, i.e., representable objects in $\Struct_d=\sPSh(\FEmb_d)$,
it turns out that our bordism category simplifies to a very concrete and easy-to-understand category of \emph{embedded} bordisms (\cref{embcat.v1}).
For the proof of the locality property in Grady–Pavlov \cite{GradyPavlov.Loc}, we first reduce to the case of representable presheaves.
We then leverage the explicit description of this category of embedded bordisms to decompose bordisms into small pieces subordinate to an open cover.

The reduction to representable presheaves allows us to treat any bordism category that satisfies the following axioms.

\begin{axiomdefinition}
\label{axioms}
Fix $d\geq 0$.
Let $\gs\in \Struct_d$ be a field stack in the sense of \cref{geometric.structure}.
Let $\Bord_d^\gs$ be an object in $\smcat{d}$ (\cref{globular.model.structure}) that satisfies the following axioms.
\begin{enumerate}[label=(\axsymbol\arabic*),ref=(\axsymbol\arabic*)]
\item\label{axiom.functorial}
The object $\Bord_d^\gs$ is $\sset$-enriched functorial in the field stack $\gs$.
\item\label{axiom.cocontinuous}
The functor $\gs↦\Bord_d^\gs$ is cocontinuous as a functor on $\infty$-presheaves.
More precisely, it is homotopy cocontinuous as a functor on the injective model structure on simplicial presheaves.
\item\label{axiom.embedded}
Restricting to the subcategory of representable presheaves of objects $q\in \FEmb_d$ (\cref{fibered.geometric.site}),
the corresponding bordism category $\Bord_d^q$ is naturally (in $q$) weakly equivalent to the category of embedded bordisms $\EBord_d^q$ (\cref{embcat.v1}),
or its isomorphic variant in \cref{embcat.v3}.
\end{enumerate}
We call the functor $\gs\mapsto \Bord_d^\gs$ a \emph{geometric symmetric monoidal $(\infty,d)$-category of bordisms}.
\end{axiomdefinition}

See the discussion immediately preceding \cref{def.fft} for a motivation of \cref{axiom.embedded}.

\begin{remark}
\label{stronger.axiom.cocontinuous}
The main theorem of \cite{GradyPavlov.Loc} proves that \cref{axioms} implies the following much stronger variant of \cref{axiom.cocontinuous}:
\begin{enumerate}[start=2,label=(\axsymbol\arabic*$'$),ref=(\axsymbol\arabic*)]
\item\label{axiom.cocontinuous.sheaf}
The functor $\gs↦\Bord_d^\gs$ is cocontinuous as a functor on $\infty$-sheaves.
More precisely, it is homotopy cocontinuous as a functor on the Čech-local injective model structure on simplicial presheaves.
\end{enumerate}
The weaker \cref{axiom.cocontinuous} is much easier to verify in practice, as shown in \cref{bord.cocontinuous},
which motivates its appearance in \cref{axioms}.
\end{remark}

\begin{theorem}[\cref{unenriched.axioms.hold}]
\label{existencebord}
There exists a geometric symmetric monoidal $(\infty,d)$-category of bordisms, unique up to a contractible choice.
In particular, the category $\Bord_d$ of \cref{bord} satisfies \cref{axioms}.
\end{theorem}

If the field stack $\gs∈\Struct_d$ is a presheaf of sets,
the resulting geometric symmetric monoidal $(∞,d)$-category $\Bord_d^\gs$
is a geometric symmetric monoidal $(d+1,d)$-category,
with $(d+1)$-morphisms being structure-preserving diffeomorphisms of $d$-bordisms.
The truncation of this category to a $(d,d)$-category
has as $d$-morphisms structure-preserving diffeomorphism classes of $d$-bordisms,
mimicking the traditional constructions of bordism categories.

In the proof of the geometric cobordism hypothesis (Grady–Pavlov \cite{GradyPavlov.GCH}),
we will use another variant of the category of bordisms which incorporates the \emph{smooth set} of diffeomorphisms of bordisms.
As a consequence, this yields a genuine $(∞,d)$-category of bordisms,
with $(d+2)$-morphisms being isotopies of diffeomorphisms of $d$-bordisms, $(d+3)$-morphisms being isotopies of isotopies, and so on.
Including isotopies is crucial, since without them the category of bordisms does not have all duals (see \cref{d1bords,d1bordsisot}).
For additional motivation, see Lurie \cite[Section~1.4]{Lurie.TFT} for an exposition in the topological case.
We denote the geometric symmetric monoidal $(\infty,d)$-category of bordisms \emph{with isotopies} by $\frakBord_d^\gs$,
where $\gs$ is a field stack with isotopies in the sense of \cref{geometric.structure.isotopy}.
We remark that the proof of locality in Grady–Pavlov \cite{GradyPavlov.Loc} works for both bordism categories, with or without isotopies.

The bordism category with isotopies satisfies an analogous set of axioms as \cref{axioms}.

\begin{axiomdefinition}
\label{frakaxioms}
Fix $d\geq 0$.
Let $\gs\in \frakStruct_d$ be a field stack with isotopies in the sense of \cref{geometric.structure.isotopy}.
Let $\frakBord_d^\gs$ be an object in $\fraksmcat{d}$ (\cref{globular.model.structure.smooth}) that satisfies the following axioms.
\begin{enumerate}[label=(\frakaxsymbol\arabic*),ref=(\frakaxsymbol\arabic*)]
\item\label{frakaxiom.functorial}
The object $\frakBord_d^\gs$ is $\smsset$-enriched functorial in the geometric structure $\gs$.
\item\label{frakaxiom.cocontinuous}
The functor $\gs↦\frakBord_d^\gs$ is cocontinuous as a functor on $\infty$-presheaves.
More precisely, it is homotopy cocontinuous as a functor on the injective model structure on $\smsset$-valued presheaves.
\item\label{frakaxiom.embedded}
Restricting to the subcategory of representable presheaves of objects $q\in \frakFEmb_d$ (\cref{fibered.geometric.site.isotopy}),
the corresponding bordism category $\frakBord_d^q$ is naturally (in $q$) weakly equivalent
to the category of embedded bordisms with isotopies $\frakEBord_d^q$ (\cref{embcat.v1}),
or its isomorphic variant in \cref{embcat.v3}.
\end{enumerate}
We call the functor $\gs\mapsto \frakBord_d^\gs$ a \emph{geometric symmetric monoidal $(\infty,d)$-category of bordisms with isotopies}.
\end{axiomdefinition}

The discussion about \cref{axiom.cocontinuous} in \cref{stronger.axiom.cocontinuous} also applies to \cref{frakaxiom.cocontinuous}.
An additional motivation for \cref{frakaxiom.embedded} is provided by the \emph{geometric cobordism hypothesis} (\gcref{intro.mainthm}),
which proves that for every $U∈\cart$ the geometric symmetric monoidal $(∞,d)$-category $$\frakEBord_d^{\RR^d⨯U→U}$$
is the free geometric symmetric monoidal $(∞,d)$-category with duals on a $U$-family of objects.

\begin{theorem}[\cref{enriched.axioms.hold}]
\label{existencebord2}
There exists a geometric symmetric monoidal $(\infty,d)$-category of bordisms with isotopies, unique up to a contractible choice.
In particular, the category $\frakBord_d$ of \cref{frakbord} satisfies \cref{frakaxioms}.
\end{theorem}

\subsection{An overview of the series}

Below we provide a summary of how each paper in the series contributes to the problem of computing moduli spaces of functorial field theories.
In each paper, we use certain model categories that present corresponding $(\infty,1)$-categories or $(\infty,d)$-categories.
These specific models are chosen for technical reasons and allow us to simplify certain arguments.
To emphasize the key conceptual aspects, the summary below omits details of homotopy-theoretic constructions.

\subsubsection{Higher categories of bordisms (present work)}
In this first paper, we introduce the extended geometric bordism categories $\Bord_d^\gs$ and $\frakBord_d^\gs$
and define the moduli spaces of extended geometric functorial field theories $\FFT_{d,\vcat}^\gs$ and $\frakFFT_{d,\vcat}^\gs$.
The bordism categories introduced in this work are capable of encoding a large class of structures on bordisms,
both in the field content and in the structure of the bordism itself.
In particular, we allow for bordisms to be equipped with a supermanifold or holomorphic structure.
In addition, we allow for structured families of such bordisms, such as superfamilies or holomorphic families.
To encode this type of structured bordism, we introduce the notion of a \emph{fibered geometric site} $\FEmb_d$
and \emph{fibered geometric site with isotopies} $\frakFEmb_d$, where the $\sf G$ and $\frak G$ refer to the geometric structure. 

We also take a first step in describing this extended geometric bordism category in simpler terms.
For brevity, we work with the isotopy variants below, with the nonisotopy versions obtained in a similar manner.
\cref{proof.cocontinuous} asserts that the functor
$$\gs\mapsto \frakBord_d^{\gs}, \quad \gs\in \PSh_{(\infty,1)}(\frakFEmb_d)$$
is homotopy cocontinuous in its argument $\gs$,
which is an $\infty$-presheaf on a certain site $\frakFEmb_d$ of families of structured $d$-dimensional manifolds with fiberwise open embeddings as morphisms.
Every $∞$-presheaf can be written as a homotopy colimit over representables $\Yo{q}$, with $q\in \frakFEmb_d$.
Using formal properties of the mapping object, we get that 
$$\frakFFT_{d,\vcat}^\gs=\map(\frakBord_d^\gs,\vcat)≃\holim_{\Yo{q}→\gs}\map(\frakBord_d^{\Yo{q}},\vcat)≃\holim_{\Yo{q}→\gs}\map(\Yo{q},\frakFFT_{d,\vcat})≃\map(\gs,\frakFFT_{d,\vcat}),$$
where $\frakFFT_{d,\vcat}$ is the ∞-presheaf on $\frakFEmb_d$ given by $q\mapsto \frakFFT_{d,\vcat}^{\Yo{q}}$.

This is an interesting result, but not very useful in practice on its own, since the ∞-presheaf $\frakFFT_{d,\vcat}$ is still much too complicated to compute directly.
We emphasize that at this stage, the mapping object $\map(\gs,\frakFFT_{d,\vcat})$ computes maps of ∞-presheaves, not ∞-sheaves.

\subsubsection{Locality for extended field theories}
Next, the main theorem of \cite{GradyPavlov.Loc} proves that the copresheaf 
$$q\mapsto \frakBord_d^{\Yo{q}}, \quad q\in \frakFEmb_d$$ 
satisfies the homotopy codescent condition, i.e., is an ∞-cosheaf.
As a consequence, the presheaf $\frakFFT_{d,\vcat}$ satisfies the homotopy descent condition, i.e., is an ∞-sheaf.
Thus, the mapping object of ∞-presheaves $\map(\gs,\frakFFT_{d,\vcat})$ can be computed as the mapping object of ∞-sheaves:
$$\map_{\PSh}(\gs,\frakFFT_{d,\vcat}) ≃ \map_{\Sh}(\gs,\frakFFT_{d,\vcat}).$$

\subsubsection{Reduction to $\lgO(d)$-equivariant maps}
For simplicity, we consider the case $\frakFEmb_d=\fraksmFEmb_d$ and $\stcart=\cart$ below.
See \cite{GradyPavlov.Str} for the general case.
In \cite{GradyPavlov.Str} we establish an equivalence
$$\ev_d:\Sh_{(\infty,1)}(\fraksmFEmb_d)\lto5{\simeq} \Sh_{(\infty,1)}(\cart)^{\lgO(d)},\qquad \ev_d(\gs)(U)=\gs(p_U:\RR^d⨯U→U),$$
where the $\lgO(d)$-action on the evaluated presheaf is supplied by the canonical $\lgO(d)$-action on the fibers of the projection map $p_U:\RR^d⨯U→U$.
Consequently, we get an equivalence
$$\map_{\Sh}(\gs,\frakFFT_{d,\vcat}) ≃ \map^{\lgO(d)}(\ev_d\gs,\ev_d\frakFFT_{d,\vcat}).$$
Thus, \cite{GradyPavlov.Str} reduces the problem of computing the mapping object $\map(\gs,\frakFFT_{d,\vcat})$ of ∞-sheaves to the following three problems.
\begin{enumerate}
\item For each projection $p_U:\RR^d\times U\to U$, a calculation of $$(\ev_d\frakFFT_{d,\vcat})(U)=\frakFFT_{d,\vcat}^{\Yo{p_U}}.$$
\item An identification of the resulting $\lgO(d)$-action on $\frakFFT_{d,\vcat}^{\Yo{p_U}}$.
\item A computation of the $\lgO(d)$-equivariant mapping object $$\map^{\lgO(d)}(\ev_d\gs,\ev_d\frakFFT_{d,\vcat}).$$
\end{enumerate} 

\subsubsection{The geometric cobordism hypothesis}
In the final paper \cite{GradyPavlov.GCH}, we show that when the geometric symmetric monoidal $(\infty,d)$-category $\vcat$ has all duals,
then for all $U\in \stcart$, the evaluation map at a $U$-family of connected 0-bordisms
$$e:\frakFFT_{d,\vcat}^{\Yo{p_U}}\lto5{\simeq}\vcat^⨯(U)$$
is an equivalence, where $\vcat^⨯$ is $\vcat$ with noninvertible $k$-morphisms ($0<k≤d$) discarded.
Thus, $\vcat^⨯$ obtains an $\lgO(d)$-action induced by the $\lgO(d)$-action on~$p_U$.
The end result is an equivalence 
$$\map(\frakBord_d^\gs,\vcat)\simeq \map^{\lgO(d)}(\ev_d\gs,\vcat^⨯).$$

In many practical examples, the identification of the $\lgO(d)$-action on $\vcat^⨯$ can proceed as follows.
There is a simple guess as to what the action of $\lgO(d)$ can be for a particular choice of $\vcat$,
e.g., when $\vcat=\tdeloop^d \lgU(1)$ or $\vcat$ is a delooping of the category of Hilbert spaces.
To show that this guess~$\cD$ coincides with the $\lgO(d)$-action provided by the above construction,
we construct a map
$$ψ:\cD → \frakFFT_{d,\vcat}^{\Yo{p_U}},$$
which in practice amounts to writing down a particularly simple functorial field theory.
For example, in the case of $\vcat=\tdeloop^d \lgU(1)$ such a field theory
can be given by integrating a differential $d$-form over a $d$-bordism.
Once the map $ψ$ is set up, it remains to show that the composition $e ∘ ψ$ is an equivalence of ∞-sheaves on $\stcart$,
which in practice is often tautological given the construction of~$\cD$.

Therefore, we can compute the moduli space of functorial field theories as
$$\map^{\lgO(d)}(\ev_d\gs,\vcat^⨯)≃\map^{\lgO(d)}(\ev_d\gs,\cD).$$
The computation of the latter object is unique to the class of field theories at hand,
with some examples presented in Grady–Pavlov \cite{GradyPavlov.GCH}, Kenig–Pavlov \cite{KenigPavlov}, and further examples forthcoming.

\subsection{Applications}

The main theorems of \cite{GradyPavlov.Loc,GradyPavlov.GCH} rely on \cref{existencebord,existencebord2}.
We point out some additional applications of these results.

\begin{itemize}
\item
The codescent theorem \cite{GradyPavlov.Loc} enables a simple proof (forthcoming work in preparation)
of a geometric fully extended variant of the Galatius–Madsen–Tillmann–Weiss theorem \cite{GMTW}.
Another proof of the same result is presented in \gcref{invertible.tft},
using the geometric cobordism hypothesis.
\item
In \ecref{concordance.twisted}, we prove a conjecture originally posed by Stolz and Teichner \cite{StolzTeichner.SUSY}:
concordance classes of extended field theories admit a classifying space (in the sense of Brown's representability theorem).
This uses the codescent theorem of \cite{GradyPavlov.Loc}
together with the \emph{smooth Oka principle} established in Berwick-Evans–Boavida–Pavlov \cite{BEBdBP}.
\item
In \cref{2.1.Euclidean.bordisms} we give a complete definition of fully extended $2|1$-Euclidean field theories.
In a forthcoming paper, we will use the geometric cobordism hypothesis to classify such field theories in terms
of representation theory of the $2|1$-Euclidean group.
\end{itemize}

\subsection*{Prerequisites}

We assume familiarity with the language of simplicial homotopy theory and model categories, in particular, the following topics.
\begin{itemize}
\item
Simplicial homotopy theory, including simplicial sets, simplicial maps, simplicial weak equivalences, nerves of categories, Quillen's Theorem~A, simplicial categories.
See Goerss–Jardine \cite{GoerssJardine} and Richter \cite{Richter}.
\item
Model categories, including model structures, Quillen adjunctions, injective and projective model structures on presheaves,
monoidal and cartesian model categories, combinatorial model categories.
See Hovey \cite{Hovey}, Hirschhorn \cite{Hirschhorn}, Barwick \cite{Barwick.Model},
Lurie \cite[Appendix~A]{Lurie.HTT},
Cisinski \cite[Section~7.11]{Cisinski},
and the survey of Balchin \cite{Balchin}.
\item
Homotopy limits and colimits.
See Bousfield–Kan \cite[Chapters XI and XII]{BousfieldKan}, Hirschhorn \cite[Chapter~18]{Hirschhorn}, Shulman \cite{Shulman.hocolim}, Riehl \cite[Chapters 5–8]{Riehl}.
\item
Grothendieck fibrations in sets, groupoids, and categories.
See Borceux \cite[Chapter~8]{Borceux.2}, Vistoli \cite{Vistoli}, Johnstone \cite[Section B1.3]{Johnstone}.
\end{itemize}
We also make use of the following formalisms, for which we review the necessary background and provide appropriate references as needed.
\begin{itemize}
\item
Segal's $Γ$-objects.
See Segal \cite[Sections 1 and~2]{Segal.Gamma}, Bousfield–Friedlander \cite[Sections 3 and~5]{BousfieldFriedlander}, Schwede \cite[Section~1]{Schwede}, Richter \cite[Section~14.3]{Richter}.
\item
Complete Segal spaces (Rezk \cite{Rezk.CSS}) and $n$-fold complete Segal spaces (Barwick \cite[Section~2.3]{Barwick.CSS}).
\item
Simplicial presheaves and descent.
See Dugger–Hollander–Isaksen \cite{DHI}, Jardine \cite{Jardine}, Glass–\hskip0pt Minichiello \cite{GlassMinichiello}.
\end{itemize}

\subsection*{Acknowledgments}

We thank David Aretz for a careful reading of the entire manuscript and suggesting numerous improvements.
We thank Daniel Berwick-Evans, Arun Debray, Claudia Scheimbauer, and Luuk Stehouwer for discussions about bordism categories and functorial field theory.
We thank Daniel Brügmann, Jacek Kenig, Nino Scalbi, and Alexander Zahrer for questions and feedback on the first version of this article.
We thank André Henriques, Thomas Nikolaus, and David Reutter for discussions about this paper.
We thank Stephan Stolz and Peter Teichner for sharing their insights about bordism categories and field theories.
We thank Hisham Sati and Urs Schreiber for discussions on the physics applications.
We thank Nils Carqueville, Domenico Fiorenza, and Konrad Waldorf for organizing a workshop on this paper and \cite{GradyPavlov.Loc,GradyPavlov.GCH}, which led to many improvements in the exposition.

\clearpage
\section*{Table of notation}
\begin{center}
{\small
\begin{tabular}{c|c|l}
Notation & Definition & Description\\
\hline
\axsymbol$n$ & \cref{axioms} & Axioms for bordism categories\\
\frakaxsymbol$n$ & \cref{frakaxioms} & Axioms for bordism categories with isotopies\\
$d$ & \cref{symminfn} & Dimension of the bordisms and the category number\\
$\Delta$, $[m]$ & \cref{threecats} & The simplex category and objects in it\\
$\Delta^{\times d}$, ${\bf m}$ & \cref{threecats} & The $d$-fold product of the simplex category and objects in it\\
$\Gamma$, $⟨ℓ⟩$ & \cref{threecats} & Segal's Gamma category and objects in it\\
$\cart$ & \cref{def.man.cart} & Smooth cartesian spaces (e.g., $\RR^n$, for some $n\in \NN$)\\
$\man$ & \cref{def.man.cart} & Smooth manifolds\\
$\PSh(\catC,\catV)$ & \cref{presheaf.notation} & (Enriched) presheaves on a small category $\catC$ with values in $\catV$\\
$\Yo{}$ & \cref{enriched.yoneda} & The (enriched) Yoneda embedding\\
$\map$ & \cref{enriched.mapping.object} & The enriched mapping object in an enriched category\\
$⊗$, $\hom$ & \cref{day.convolution} & Day convolution and internal hom on presheaves\\
$\rdf$ & \cref{right.derived.functor} & Right derived functor\\
$\smset$ & \cref{def.smset} & Presheaves of sets on $\cart$\\
$\PSh(\catC)$ & \cref{presheaves.of.sets} & Presheaves of sets on a small category $\catC$\\
$\sPSh(\catC)$ & \cref{simplicial.presheaves} & Simplicial presheaves on a small category $\catC$\\
$|{-}|$, $\sing$ & \cref{smsset.injective} & Smooth realization and singular complex\\
$\smsset$ & \cref{smsset.injective} & Simplicial presheaves on $\cart$, with the cartesian $\RR$-local model structure\\
$\smsPSh(\catC)$ & \cref{injective.model.structure} & $\smsset$-valued presheaves on $\catC$ with the injective model structure\\
$⊠$& \cref{external.product} & External product of presheaves\\
$\lbl_S M$ & \cref{left.Bousfield.localization} & The (enriched) left Bousfield localization of~$M$ at~$S$\\
$\site$ & \cref{site.and.enrichment} & A small site (like $\stcart$)\\
$\catV$ & \cref{site.and.enrichment} & A combinatorial cartesian model category (like $\sset$ or $\smsset$)\\
$\smcatuple{d}$ & \cref{multiple.model.structure} & $\sPSh(\stcart\times \Gamma\times \Delta^{\times d})$, with the uple model structure\\
$\fraksmcatuple{d}$ & \cref{multiple.model.structure.smooth} &$\smsPSh(\stcart\times \Gamma\times \Delta^{\times d})$, with the uple model structure\\
$\catuple{d}$ & \cref{multiple.model.structures} & $\sPSh(\Gamma\times \Delta^{\times d})$, with the uple model structure\\
$\smcat{d}$ & \cref{globular.model.structure} &$\sPSh(\stcart\times \Gamma\times \Delta^{\times d})$, with the globular model structure\\
$\fraksmcat{d}$ & \cref{globular.model.structure.smooth} &$\smsPSh(\stcart\times \Gamma\times \Delta^{\times d})$, with the globular model structure\\
$\cat{d}$ & \cref{globular.model.structures} &$\sPSh(\Gamma\times \Delta^{\times d})$, with the globular model structure\\
$\Theta_d$ & \cref{joyal.cell.category} & Joyal's cell category\\
$\Funmonuple$& \cref{def.funmonuple} & The internal hom in $\smcatuple{d}$\\
$\tglob$, $\Funmonglob$& \cref{globular.hom} & Globular product and functor object in $\smcat{d}$\\
$\stcart$ & \cref{def.stcart} & Structured cartesian spaces (e.g., $\RR^n$, $\RR^{n\mid m}$, for some $n,m\in \NN$)\\
$\red=\red_b$ & \cref{def.stcart} & The reduction functor $\stcart→\cart$\\
$\smFEmb_d$ & \cref{def.smFEmb} & Submersions with $d$-dimensional fibers and fiberwise embeddings\\
$\FEmb_d$ & \cref{fibered.geometric.site} & A fibered geometric site\\
$\base$ & \cref{fibered.geometric.site} & The base space functor (on base objects)\\
$\red_t$ & \cref{fibered.geometric.site} & The reduction functor (on fibered objects)\\
$\fraksmFEmb_d$ & \cref{def.fraksmFEmb} & Submersions with $d$-dimensional fibers and fiberwise embeddings with isotopies\\
$\frakFEmb_d$ & \cref{fibered.geometric.site.isotopy} & A fibered geometric site with isotopies\\
$\Struct_d$ & \cref{geometric.structure} & $\sPSh(\FEmb_d)$, with the Čech-local injective model structure\\
$\frakStruct_d$ & \cref{geometric.structure.isotopy} & $\smsPSh(\frakFEmb_d,\smsset)$, with the Čech-local injective model structure\\
$\ncore(p,C)$ & \cref{nonembedded.core} & The core of a monoidal cut grid $(p,C)$\\
$\mtCutglob$ & \cref{globular.monoidal.cut.grid} & Globular monoidal cut grids\\
$\frakmtCutglob$ & \cref{globular.monoidal.cut.grid.smooth} & Globular monoidal cut grids with isotopies\\
$p:M→U$ & \cref{bord} & The smooth family of manifolds containing cut grids of bordisms\\
$\Bord_d^\gs$ & \cref{bordstr} & The globular $d$-category of bordisms with field stack~$\gs$\\
$\frakBord_d^\gs$ & \cref{enrichedbordstr} & The globular $d$-category of bordisms with field stack~$\gs$ and isotopies\\
$\FFT_{d,\vcat}^\gs$ & \cref{def.FFTspace} & The moduli stack of functorial field theories with structure~$\gs$ valued in $\vcat$\\
$\frakFFT_{d,\vcat}^\gs$ & \cref{def.FFTspace} & The moduli stack of functorial field theories with isotopies\\
$q:N→V$ & \cref{bord.cocontinuous} & The smooth family of manifolds into which bordisms are embedded\\
$\FFT_d$ & \cref{def.FFT} & The field stack of functorial field theories\\
$\frakFFT_d$ & \cref{def.frakFFT} & The field stack of functorial field theories with isotopies\\
$\EBord_d^q$ & \cref{embcat.v1} & Bordisms embedded fiberwise in $(q:N\to V)\in \FEmb_d$\\
$\frakEBord_d^q$ & \cref{embcat.v1} & Bordisms with isotopies embedded fiberwise in $(q:N\to V)\in \frakFEmb_d$\\
$\ecm$ & \cref{emapconstr.unenriched} & The comparison map for bordisms and embedded bordisms\\
$\frakecm$ & \cref{emapconstr.enriched} & The comparison map for bordisms and embedded bordisms with isotopies\\
\end{tabular}
}
\end{center}
\clearpage

\section{Homotopy theory and higher categories}
\label{symminfn}

As discussed in the introduction, the main goal of this paper is to establish a model for the extended bordism category that is capable of handling geometric structures.
Such an extended bordism category should have the following basic structure.
\begin{enumerate}[(1)]
\item It should be a higher category with $k$-morphisms for all $k>0$ such that for every $k>d$, all $k$-morphisms are invertible.
\item It should be symmetric monoidal, with coherences for the monoidal structure.
\item It should allow for geometric families for objects or $k$-morphisms, for example smooth or holomorphic families.
\end{enumerate}

We now have well-developed machinery that can encode (1) and (2).
The encoding of symmetric monoidal structures follows Segal \cite[Definition~1.2]{Segal.Gamma}.
The encoding of categorical composition follows the same idea of Segal, as developed by
Dwyer–Kan–Smith \cite[Section~7]{DKS.HCD} (Segal $(∞,1)$-categories),
Tamsamani \cite[Definition~1.3.2]{Tamsamani} (Segal $n$-categories),
Hirschowitz–Simpson \cite[Section~2]{HirschowitzSimpson} and Pellissier \cite{Pellissier} (Segal $(∞,n)$-categories),
Rezk \cite[Section~6]{Rezk.CSS} (complete Segal spaces),
Barwick \cite[Scholium 2.3.18]{Barwick.CSS} (complete $n$-fold Segal spaces),
and Lurie \cite[Section~1]{Lurie.TwoCat} (globular complete $n$-fold Segal spaces).

To encode (3), we will parametrize our higher monoidal categories by objects of a site $\stcart$ (\cref{def.stcart}).
We will want these parametrized higher categories to satisfy a homotopy coherent sheaf gluing property with respect to open covers so that families of objects and $k$-morphisms can be assembled from local data.
The encoding of geometric families was influenced by Stolz–Teichner \cite{StolzTeichner.Elliptic,StolzTeichner.SUSY}.
The use of the cartesian site in this context was proposed by Urs Schreiber (see, e.g., Fiorenza–Schreiber–Stasheff \cite[Appendix]{FiorenzaSchreiberStasheff}).
The site $\stcart$ is very general and allows for a variety of geometric structure on cartesian spaces.
In particular, we can encode smooth, super, and complex structures on cartesian spaces.

To encode invertible $k$-morphisms for $k>d$, we present the $∞$-groupoid of such morphisms
(with objects being $d$-morphisms, 1-morphisms being $(d+1)$-morphisms, etc.)\ as a simplicial set.
We also find it convenient to talk about \emph{isotopies} (i.e., $\RR$-parametrized smooth homotopies) of $k$-morphisms,
which leads to the structure of a \emph{smooth simplicial set} (\cref{def.smsset}).
From a theoretical viewpoint, smooth simplicial sets are Quillen equivalent to simplicial sets (\cref{smsset.injective}),
via an appropriate analog of the singular complex functor.
Thus, adding isotopies does not increase the level of generality of our theory.
From a practical viewpoint, it is convenient to treat isotopies in this manner,
since we can talk about \emph{stalks} of isotopies (see \cref{stalk.eq}).
This technique is exploited in the proof of the locality property in Grady–Pavlov \cite{GradyPavlov.Loc}.
Similar ideas were used by Galatius–Madsen–Tillmann–Weiss \cite[Sections 2.2 and 2.3]{GMTW}.

\subsection{Geometric, symmetric monoidal, and higher categorical structure}
\label{threecats}

Recall the simplex category~$\Delta$, whose objects are nonempty, totally ordered sets $[d]=\{0,\ldots,d\}$, and whose morphisms are order-preserving maps.
Recall also Segal's category~$Γ$ \cite[Definition~1.1]{Segal.Gamma} (see also Bousfield–Friedlander \cite[Section~3]{BousfieldFriedlander} and Schwede \cite[Section~1]{Schwede}),
whose opposite category has finite pointed sets $⟨ℓ⟩=\{\ast,1,\ldots,ℓ\}$ as objects and basepoint-preserving functions as morphisms.
Finally, recall the category $\cart$ (\cref{def.cart}, see also Fiorenza–Schreiber–Stasheff \cite[Definition~3.1.1]{FiorenzaSchreiberStasheff})
whose objects are cartesian spaces~$\RR^n$ and morphisms are smooth functions.
More generally, we can use any category $\stcart$ that satisfies certain properties making it similar to $\cart$ (\cref{def.stcart}),
see \cref{def.supercart,def.infinitesimal}.
The basic building blocks for geometric symmetric monoidal $(\infty,d)$-categories lie in the product category
\begin{equation}\label{product.no.isotopies}\stcart\times \Gamma\times \Delta^{\times d}\times \Delta.\end{equation}
Each category captures an independent aspect of the objects being constructed.
To include isotopies of higher morphisms, we also consider the product category
\begin{equation}\label{product.isotopies}\stcart\times \Gamma\times \Delta^{\times d}\times \cart\times \Delta,\end{equation}
with the factor $\cart$ playing a completely different role than the first factor $\stcart$.
We elaborate on this in \cref{smsset.vs.sset}.
We now offer conceptual interpretations for each component.

\subsubsection{The category $\Delta^{\times d}$}
Here, an object is a multisimplex ${\bf m}=([m_1],[m_2],\ldots, [m_d])\in \Delta^{\times d}$,
which can be thought of as indexing a composable grid of morphisms, with $m_i$ composed morphisms in the $i$th direction.
For example, for $d=2$, there are two composition directions.
If $m_1=m_2=4$, we obtain a grid
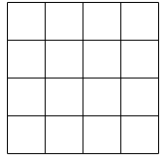
\begin{center}
\def\xMin{0}%
\def\xMax{4}%
\def\yMin{0}%
\def\yMax{4}%
\begin{tikzpicture}[scale=.5]
\foreach \i in {\xMin,...,\xMax}{\draw [very thin] (\i,\yMin) -- (\i,\yMax);}
\foreach \i in {\yMin,...,\yMax}{\draw [very thin] (\xMin,\i) -- (\xMax,\i);}
\end{tikzpicture}
\end{center}
A single square represents a 2-morphism, while edges represent 1-morphisms and vertices represent objects.
These 2-morphisms can be composed in two directions, thereby forming a grid.
Simplicial presheaves on $\Delta^{\times d}$ admit a model structure whose fibrant objects are $d$-fold Segal spaces.
Such a simplicial presheaf~$F$ encodes an $(∞,d)$-category~$C$ in the following manner.
Given a multisimplex ${\bf m}$, the value $F({\bf m})$ is an ∞-groupoid
whose objects are diagrams in~$C$ of the above shape,
morphisms are natural isomorphisms,
2-morphisms are natural 2-isomorphisms,
etc.
The structure of an $(∞,d)$-category can be reconstructed from the knowledge of such diagrams.
For examples, objects in~$C$ can be reconstructed by evaluating at ${\bf m}=([0],…,[0])$
and taking the set of objects (i.e., vertices) of the resulting ∞-groupoid.

For constructions that include symmetric monoidal and geometric structures, see \cref{multiple.model.structures,globular.model.structures}.
References on $d$-fold Segal spaces include
Barwick \cite[Scholium 2.3.18]{Barwick.CSS}, Lurie \cite[Section~1]{Lurie.TwoCat}, Lurie \cite[Section 2.1]{Lurie.TFT}, Barwick–Schommer-Pries \cite[Section 14]{BarwickSchommerPries}, Calaque–Scheimbauer \cite[Section~2]{CalaqueScheimbauer}.

\subsubsection{The category $\Gamma$}
This category captures the symmetric monoidal structure.
An object $⟨ℓ⟩\in\Gamma$ is a finite set $\{\ast,1,2,\ldots,ℓ\}$ and a morphism $A→B$ in~$Γ$ is a function $f:B→A$ satisfying $f(\ast)=\ast$.
This category encodes the symmetric monoidal structure as follows.
Consider a $Γ$-object $X:Γ^\op→\sset$.
Denote by $\phi_{⟨ℓ⟩}:⟨ℓ⟩ \to ⟨1⟩$ the function that sends $i$ to~$1$,
for all $i$ such that $1\leq i\leq ℓ$.
Then $\phi_{⟨ℓ⟩}$ yields, by functoriality, a map
$$X(\phi_{⟨ℓ⟩}):X⟨ℓ⟩→X⟨1⟩.$$
We also have maps $\delta_i:⟨ℓ⟩\to ⟨1⟩$ that send $i$ to $1$ and $j$ to $*$ for $j\neq i$.
They induce a map $$δ_{⟨ℓ⟩}=(X(δ_1),…,X(δ_ℓ)):X⟨ℓ⟩ \to X⟨1⟩^{\times ℓ}.$$
The multiplicative structure is given by the following zigzag, where the left leg is a weak equivalence, so can be (formally) inverted:
$$X⟨1⟩^{\times ℓ}\lgets5{\delta_{⟨ℓ⟩}} X⟨ℓ⟩ \xrightarrow{X(\phi_{⟨ℓ⟩})} X⟨1⟩.$$
The space $X⟨ℓ⟩$ can be thought of as the space of $ℓ$-tuples that can be multiplied.
The map $X(\phi_{⟨ℓ⟩})$ then performs the multiplication.
The map $δ_{⟨ℓ⟩}$ extracts the individual components of an $ℓ$-tuple.
Segal's definition of a $Γ$-object requires the map $δ_{⟨ℓ⟩}$ to be a weak equivalence,
which formalizes the idea that any $ℓ$-tuple can be deformed to an $ℓ$-tuple that can be multiplied,
and the space of such deformations is contractible.
The elegance of Segal's method is that it circumvents the need for explicitly keeping track of coherent homotopies in the symmetric monoidal structure,
although these are implicitly contained in the above equivalence.
For more information, see Segal \cite{Segal.Gamma}.
The specific case of symmetric monoidal $(∞,d)$-categories was examined by Toën \cite[Section~2.2]{Toen} and Barwick \cite[Section~3.1]{Barwick.CSS}.

\subsubsection{The category $\stcart$}
The categories $\cart$ and, more generally, $\stcart$ encode parametrizing families of bordisms.

\begin{definition}
\label{def.site}
(Johnstone \cite[Definition~C2.1.1]{Johnstone}, Minichiello \cite[Definition~2.6]{Minichiello}.)
A \emph{site} is a category equipped with a coverage.
An \emph{enriched site} is an enriched category together with a coverage on its underlying category.
\end{definition}

The following definition is designed so that for a fixed smooth manifold~$N$,
the class of manifolds~$M$ that admit a closed inclusion into~$N$ is a set.
Furthermore, the full subcategory of the slice category $\man/N$ on included submanifolds is a poset,
in particular, every isomorphism is identity.

\begin{definition}
\label{inclusion.manifold}
A \emph{closed inclusion of manifolds} is a closed smooth map $f:M→N$ that is an embedding of a submanifold
and whose underlying map of sets is an inclusion of sets.
\end{definition}

\begin{subdefinition}
\label{def.man.cart}
We define the following small sites.
\begin{enumerate}
\item\label{def.man}
The site $\man$ is defined as follows.
Objects are smooth manifolds.
Morphisms are smooth maps.
Covering families are open covers.
\item\label{def.cart}
The small site $\cart$ is defined as follows.
Objects are manifolds in $\man$ diffeomorphic to $\RR^n$ for some $n≥0$ and admit a closed inclusion of manifolds (\cref{inclusion.manifold})
into $\RR^m$ for some (necessarily unique) $m≥0$.
Morphisms are smooth maps.
Covering families are \emph{good open covers}, defined as open covers in which every finite intersection is empty or diffeomorphic to some $\RR^n$.
\end{enumerate}
\end{subdefinition}

\begin{remark}
\label{man.small}
The condition requiring objects of $\cart$ to admit a closed inclusion into $\RR^m$ for some $m≥0$
singles out a small subcategory of the large category of smooth manifolds.
Both $m$ and the inclusion map can be recovered from the object.
Therefore, we do not record them as part of the data of an object.
\end{remark}

The functor category $\Fun(\cart^\op,\catD)$ can be thought of as the category of smoothly parametrized objects of~$\catD$ over cartesian spaces.
As a basic example, a smooth manifold $M$ yields a functor $M:\cart^\op\to \set$,
which maps a cartesian space~$U$ to the set of smooth maps $\sm(U,M)$.
For more information about sheaves on this site, see Fiorenza–Schreiber–Stasheff \cite[Appendix]{FiorenzaSchreiberStasheff}
or Sati–Schreiber \cite[Section~2.1 (arXiv); 8.1 (book)]{SatiSchreiber.POC}.

For applications (including the Stolz–Teichner program), we find it convenient to generalize the site $\cart$ to accommodate manifolds with additional structures
such as supermanifolds, complex manifolds, or formal manifolds.
We will denote this generalized site by $\stcart$ (\cref{def.stcart}).
The site $\cart$ of cartesian spaces will serve as our primary example throughout the text.
The reader who is less familiar with the theory of supermanifolds may restrict to the case where $\stcart=\cart$ throughout.

\subsubsection{The categories encoding higher invertible morphisms}
\label{smsset.vs.sset}

In addition to the categories $\stcart$, $Γ$, and $Δ^{⨯d}$, which encode geometric, symmetric monoidal, and $d$-categorical structures, respectively,
we also need to encode the $(∞,d)$-categorical structure.
Traditionally, this is achieved by incorporating the copy of~$Δ$ appearing in \eqref{product.no.isotopies}, thereby enriching constructions over simplicial sets,
which encode higher homotopies between $d$-morphisms.
Alternatively, one could use other models for $\infty$-groupoids.
One convenient choice is the category of simplicial presheaves $\PSh_{\Delta}(\cart)$ with the injective $\RR$-local model structure (\cref{smsset.injective}).
This model allows us to encode higher invertible morphisms using the geometry of cartesian spaces, as opposed to the combinatorics of simplices.

More precisely, there are three ways in which $\infty$-groupoids enter the construction of the bordism category.
\begin{itemize}
\item Higher gauge transformations of fields form an $\infty$-groupoid, presented as a simplicial set (\cref{geometric.structure}).
\item Bordisms are composed by gluing open neighborhoods of the core (\cref{nonembedded.core}).
These open neighborhoods and their embeddings form a category.
Taking the nerve of this category corresponds to quotienting by the equivalence relation that identifies two bordisms if they coincide on a smaller open neighborhood;
that is, we pass to the germ of the core.
\item The isotopy space of cuts (\cref{monoidal.cut.grid.smooth}) combined with the isotopy space of diffeomorphisms of bordisms determines the $(∞,d)$-category structure (\cref{frakbord}).
\end{itemize}

The first type of $\infty$-groupoid is present throughout the entire paper.
In the case where the field stack is representable, the second type of $\infty$-groupoid is $0$-truncated and
is eliminated during the transition to the embedded bordism category in \cref{embcat.v1.iso}.
The third type of $\infty$-groupoid uses the model via presheaves of simplicial sets on the site $\cart$ to give a \emph{smooth space} of isotopies of cuts.

It is advantageous not to use simplicial sets in the third case.
This enables us to pass to \emph{stalks}
in the proof of locality of extended field theories (Grady–Pavlov \cite{GradyPavlov.Loc})
and the geometric cobordism hypothesis (Grady–Pavlov \cite{GradyPavlov.GCH}).
(See \cref{stalk.eq}.)
This technique would not be available to us if we were to use ordinary simplicial enrichments.
Instead of enriching in the model category $\sset$, we use a Quillen equivalent model category $\smsset$ (\cref{smsset.injective}) of \emph{smooth simplicial sets},
defined as presheaves of simplicial sets on the site $\cart$.

\subsection{Enriching categories}

We will make use of internal homs and enrichments of presheaves on monoidal categories,
such as $\stcart$ with its cartesian structure or $Γ$ with its smash product structure.
We introduce the following systematic notation for enriched presheaves.

\begin{notation}
Let $\catV$ be a monoidal category,
$\catC$ be a small category,
and $\catD$ be a $\catV$-enriched small category.
\begin{itemize}
\item\label{presheaf.notation}
The category of presheaves on~$\catC$ with values in $\catV$ will be denoted by
$$\PSh(\catC,\catV)≔\Fun(\catC^\op,\catV).$$

\item\label{enriched.presheaf.notation}
The category of $\catV$-enriched presheaves on~$\catD$ with values in $\catV$ will be denoted by
$$\PSh(\catD,\catV)≔\Fun(\catD^\op,\catV).$$

\item\label{enriched.mapping.object}
We denote the $\catV$-valued mapping object in a $\catV$-enriched category by
$$\map(X,Y).$$

\item\label{enriched.yoneda}
If $\catV$ is a cocomplete monoidal category,
we denote by $ι_\catV$ the cocontinuous functor
$$ι_\catV:\set→\catV$$ that maps $S↦∐_S \munit$,
where $\munit$ denotes the monoidal unit of~$\catV$.
The \emph{enriched Yoneda embedding}~$\Yo{}$ is the composition of functors
$$\catC → \PSh(\catC,\set) \lto9{\PSh(\catC,ι_\catV)} \PSh(\catC,\catV),\qquad \Yo{c}(d)=∐_{\catC(d,c)}\munit.$$

\item\label{day.convolution}
If $\catV$ is a complete and cocomplete closed symmetric monoidal category,
then $\PSh(\catC,\catV)$ is tensored, powered, and enriched over~$\catV$.
Furthermore, if $\catC$ is equipped with a symmetric monoidal structure,
then $\PSh(\catC,\catV)$ is equipped with a closed symmetric monoidal structure given by the Day convolution
(Day \cite[Chapter~3]{Day}, see also the $n$Lab \cite[Section~3]{nLab.Day}).
We denote the Day convolution product by
$$\otimes:\PSh(\catC,\catV)\times \PSh(\catC,\catV)\to \PSh(\catC,\catV)$$
and the internal hom by
$$\hom(-,-):\PSh(\catC,\catV)^\op\times \PSh(\catC,\catV)\to \PSh(\catC,\catV).$$

\item\label{right.derived.functor}
We use the notation $\rdf F$ to denote the right derived functor of a functor~$F$.
If $F$ is a right Quillen functor, then we can set $\rdf F=F∘\frep$, where $\frep$ is a fibrant replacement functor.
Likewise, $$\dmap(X,Y)=\map(\crep X,\frep Y)$$ denotes the right derived enriched mapping object functor.
\end{itemize}
\end{notation}

We also introduce the following notation for certain special cases of $\catC$ and $\catV$ in \cref{presheaf.notation}.
\begin{notation}
Recall the site $\cart$ (\cref{def.cart}).
\begin{itemize}
\item\label{def.smset}
\label{def.smsset}
Taking $\catC=\cart$ and $\catV=\set$ (or $\catV=\sset$),
 we use the shorthand notation
$$\smset=\PSh(\cart,\set), \qquad \smsset=\PSh(\cart,\sset).$$
We refer to these categories as the category of \emph{smooth sets} and \emph{smooth simplicial sets}, respectively.
Given $F∈\smset$, we write
$$F_L=F(L), \qquad L∈\cart.$$

\item\label{presheaves.of.sets}
\label{simplicial.presheaves}
\label{smooth.simplicial.presheaves}
We also write
$$\PSh(\catC)=\PSh(\catC,\set),\qquad \sPSh(\catC)=\PSh(\catC,\sset), \qquad \smsPSh(\catC)=\PSh(\catC,\smsset),$$
and refer to these categories as \emph{presheaves}, \emph{simplicial presheaves}, and \emph{smooth simplicial presheaves} on~$\catC$, respectively.

\item\label{enriched.family.notation}
Suppose $\catD$ is a $\smset$-enriched category and $L∈\cart$.
For objects $x,y\in \catD$, we call an element of $\catD(x,y)_L$ an \emph{$L$-family of morphisms}.
We call an $\RR^0$-family of morphisms simply a \emph{morphism}.
We denote by $\catD_L$ the category with the same objects as $\catD$
and morphisms given by $L$-families of morphisms, defined via evaluation at~$L∈\cart$:
$$\catD_L(X,Y)=\catD(X,Y)_L.$$
For $L=\RR^0$, we call $\catD_{\RR^0}$ the \emph{underlying category of the $\smset$-enriched category $\catD$}.

Given a $\smset$-enriched functor~$f:\catC→\catD$ and $L∈\cart$, we denote by $f_L:\catC_L→\catD_L$ the functor
given on objects by the same map as $f$ and on morphisms by the evaluation of the enriched hom object maps on~$L$.
\end{itemize}
\end{notation}

Next, we recall some basic properties of the category $\smsset$.
An important feature of this category is that it admits a model structure that makes it Quillen equivalent to the category of simplicial sets.
Thus, the category of smooth simplicial sets $\smsset$ can be used as another model for the $(\infty,1)$-category of spaces.
This model is used to encode isotopies (as explained in \cref{smsset.vs.sset}).
We start with some basic definitions.

\begin{definition}
Let $l\in \NN$.
We define the \emph{extended $l$-simplex} $\gsim^l$ as the smooth manifold
$$\gsim^l=\biggl\{t\in \RR^{l+1} \Bigm| \sum_{i}t_i=1\biggr\}$$
with its canonical smooth structure.
\end{definition}

In Pavlov \cite[Theorem~12.7]{Pavlov.Diffeo}, it was shown that the category of smooth simplicial sets admits a cartesian combinatorial model structure such that the resulting model category is Quillen equivalent to simplicial sets.
We recall the precise statement here.

\begin{proposition}
\label{smsets.projective}
(Pavlov \cite[Theorem~12.7]{Pavlov.Diffeo}.)
The category $$\smsset=\sPSh(\cart)$$
of smooth simplicial sets (\cref{def.smsset})
admits a left proper cartesian combinatorial model structure $\smsset_{\sing,\proj}$.
Its weak equivalences are created by the smooth singular complex functor
$$\sing:\smsset→\sset, \qquad F↦(l↦F(\gsim^l)_l).$$
Its generating cofibrations
are given by the maps $(∂Δ^m→Δ^m)\ppet(\gbou^n→\gsim^n)$,
where $\gbou^n=|∂Δ^n|$, with $|{-}|$ being the left adjoint of $\sing$,
and $\ppet$ denotes the pushout product associated to the external tensor product~$⊠$ (\cref{external.product}).
The resulting Quillen adjunction
$$\xymatrix{
\sset\ar@<.125cm>[r]^-{|-|} & \ar@<.125cm>[l]^-{\sing} \smsset_{\sing,\proj}
}$$
is a Quillen equivalence.
\end{proposition}

In this paper, almost all model categories used are left Bousfield localizations of \emph{injective} model structures on presheaves.
The injective model structure has the advantage that all objects are cofibrant, which implies that all left Quillen functors are automatically derived.
This motivates the following.

\begin{proposition}
\label{smsset.injective}
The category $$\smsset=\sPSh(\cart)$$
of smooth simplicial sets (\cref{def.smsset})
admits a cartesian combinatorial model structure $\smsset_{\sing,\inj}$,
henceforth denoted by $\smsset$.
Its weak equivalences coincide with the weak equivalences of \cref{smsets.projective},
and its cofibrations are monomorphisms.
The resulting Quillen adjunctions
$$\xymatrix{\sset\ar@<.125cm>[r]^-{|-|} &\ar@<.125cm>[l]^-{\sing}\smsset_{\sing,\proj} \ar@<.125cm>[r]^-{\id} &\ar@<.125cm>[l]^-{\id}\smsset_{\sing,\inj}}$$
are Quillen equivalences.
\end{proposition}

\begin{proof}
\cref{smsets.projective} constructs a cartesian combinatorial model structure on $\smsset$
with the same weak equivalences, but a smaller class of cofibrations.
Barwick \cite[Theorem~2.16]{Barwick.Model} shows that the class of monomorphisms is cofibrantly generated by a set of monomorphisms.
The Smith recognition theorem (Barwick \cite[Proposition~2.2]{Barwick.Model}) implies the existence of a model structure,
provided that the intersection of weak equivalences and monomorphisms is closed under cobase changes and transfinite compositions.
The latter condition follows from Pavlov \cite[Theorem~12.7]{Pavlov.Diffeo} by setting $\catV=\sset$ and proceeding as follows.
Monomorphisms in $\smsset$ are objectwise h-cofibrations, because h-cofibrations in $\sset$ coincide with monomorphisms.
Objectwise h-cofibrations in $\smsset$ are h-cofibrations.
A cobase change of an h-cofibration is a homotopy cobase change because $\smsset$ is left proper.
Homotopy cobase changes of weak equivalences are weak equivalences.
Transfinite compositions of weak equivalences are weak equivalences because they can be expressed as filtered colimits.

We now prove the cartesian property.
The terminal object is cofibrant because all objects are cofibrant.
Cofibrations are closed under the pushout product operation, as monomorphisms of sets are also closed under pushout products.
As previously shown, a cobase change of an acyclic cofibration is an acyclic cofibration.
The product of a weak equivalence and an identity map is a weak equivalence.
This is because weak equivalences in $\smsset$ are created by the singular functor to simplicial sets, which preserves products and weak equivalences.
Thus, by the 2-out-of-3 property for weak equivalences, the pushout product of a cofibration and an acyclic cofibration is a weak equivalence.
The Quillen adjunctions in the statement are Quillen equivalences,
which follows from \cref{smsets.projective} and the fact that the two model structures on $\smsset$ share the same weak equivalences.
\end{proof}

\begin{remark}
\label{stalk.eq}
By Pavlov \cite[Proposition~12.5]{Pavlov.Diffeo},
a stalkwise weak equivalence of simplicial presheaves is a weak equivalence
in the model structures $\smsset_{\sing,\proj}$ (\cref{smsets.projective}) and $\smsset_{\sing,\inj}$ (\cref{smsset.injective}).
The converse is false because the map $\RR^n→\RR^0$ is not a stalkwise weak equivalence; however, its singular complex is a weak equivalence of simplicial sets.
Thus, the model structures $\smsset_{\sing,\proj}$ and $\smsset_{\sing,\inj}$ are not Quillen equivalent to the Čech-local model structures (projective or injective)
$\smsset_{\Cech,\proj}$ and $\smsset_{\Cech,\inj}$.
\end{remark}

\subsection{Categories of presheaves and their left Bousfield localizations}
\label{geometric.multiple.categories}

In this section, we recall the theorem of Smith on the existence of left Bousfield localizations of model categories and enriched model categories.
This will set the stage for the construction of the model category of geometric symmetric monoidal $(\infty,d)$-categories in \cref{geocatssect},
obtained via a left Bousfield localization.

\begin{proposition}
\label{injective.model.structure}
(Smith; Barwick \cite[Theorem~2.16, Proposition~4.50]{Barwick.Model}, Lurie \cite[Proposition~A.2.8.2]{Lurie.HTT}.)
Suppose $\catV$ is a combinatorial model category and $\catC$ is a small category.
The category $\PSh(\catC,\catV)$ of $\catV$-valued presheaves on~$\catC$
admits a unique \emph{injective model structure}
such that a natural transformation $t:F→G$
is a cofibration (respectively weak equivalence)
if for every $c∈\catC$ the map $t_c:F(c)→G(c)$ is a cofibration (respectively weak equivalence) in~$\catV$.
The injective model structure is combinatorial.
If $\catV$ is a monoidal model category, then $\PSh(\catC,\catV)$ is a $\catV$-enriched model category.
If $\catV$ is a cartesian model category, then so is $\PSh(\catC,\catV)$.
If $\catV$ is left proper or tractable, then so is $\PSh(\catC,\catV)$.
\label{smspsh.injective}
In particular, there are combinatorial left proper injective model structures $$\sPSh(\catC)=\PSh(\catC,\sset)_\inj$$
and $$\smsPSh(\catC)=\PSh(\catC,\smsset_{\sing,\inj})_\inj$$
(\cref{smsset.injective}).
The model structure on $\sPSh(\catC)$ is $\sset$-enriched and the one on $\smsPSh(\catC)$ is $\smsset$-enriched.
\end{proposition}

\begin{proof}
Cofibrations and acyclic cofibrations in injective model structures are defined objectwise,
so the axioms for $\catV$-enrichments and cartesian model categories can also be verified objectwise.
\end{proof}

\begin{remark}
\label{injective.Gamma.nonmonoidal}
The injective model structure on $Γ$-spaces is not a monoidal model structure.
It suffices to exhibit a $Γ$-space~$X$ and a weak equivalence $f:A→B$ of $Γ$-spaces such that $f∧X$ is not a weak equivalence.
Fix $m>0$ and consider the representable $Γ$-space $\Yo{⟨m⟩}$ of $⟨m⟩$ with the canonical action of $Σ_m$ via automorphisms of~$⟨m⟩$.
Set $X=\Yo{⟨m⟩}/Σ_m$.
Take $B$ to be the terminal $Γ$-space and $A$ the $Γ$-space given by composing
the representable presheaf of~$⟨m⟩$ with the functor from sets to simplicial sets
that sends a set $S$ to the nerve of the contractible groupoid with $S$ as the set of objects.
Then $f∧\Yo{⟨m⟩}/Σ_m≅(f∧\Yo{⟨m⟩})/Σ_m≅f(-∧⟨m⟩)/Σ_m$.
Evaluating at $⟨1⟩$, we get the map $f(⟨m⟩)/Σ_m$,
whose domain is not contractible, but the codomain is.
Therefore, $f∧X$ is not a weak equivalence even though $f$ is.
\end{remark}

Because of \cref{injective.Gamma.nonmonoidal} we need a different model structure on $Γ$-objects.
A convenient choice is given by a Reedy-type model structure like in Reedy \cite{Reedy}, but replacing the indexing category~$Δ$ with~$Γ$.
Such a model structure was first constructed by Bousfield–Friedlander \cite[Theorem~3.5]{BousfieldFriedlander}.
Its monoidality was established by Lydakis \cite[Theorems 4.6 and 5.1]{Lydakis}.
Berger–Moerdijk \cite[Theorem~1.6]{BergerMoerdijk} and Shulman \cite[Theorem~8.9]{Shulman.Reedy}
replace $Γ$ with more general indexing categories.

\begin{proposition}
\label{Reedy.model.structure}
Suppose $\catW$ is a combinatorial model category.
The category $\PSh(Γ,\catW)$ of $\catW$-valued presheaves on~$Γ$
admits a unique \emph{generalized Reedy model structure} with the following properties:
\begin{enumerate}
\item a natural transformation $t:F→G$ is a weak equivalence if for every $c∈Γ$ the map $t_c:F(c)→G(c)$ is a weak equivalence in~$\catW$
and a cofibration if for every $c∈Γ$ the relative latching map
$$F(c)⊔_{\latch_c F}\latch_c Y→G(c)$$
is a projective cofibration in~$\catW^{\Aut(c)}$,
where $\latch_c F=\colim_{d→c}F(d)$ is the \emph{latching object} of~$F$ at~$c∈Γ$,
given by the colimit over all nonsurjective maps $d→c$ of finite pointed sets;
\item the generalized Reedy model structure is combinatorial;
\item if $\catW$ is left proper, then so is $\PSh(Γ,\catW)$;
\item if $\catW$ is a monoidal model category, then $\PSh(Γ,\catW)$ is a $\catW$-enriched monoidal model category.
\end{enumerate}

\label{Reedy.injective.model.structure}
If $\catW$ itself is the injective model structure on $\catV$-valued presheaves on a small category~$\catC$,
then we refer to the resulting model structure on
$$\PSh(Γ,\PSh(\catC,\catV))≅\PSh(Γ⨯\catC,\catV)$$
as the \emph{$Γ$-Reedy injective model structure}.
Here $\catV$ is a cartesian combinatorial model category.
\end{proposition}

\begin{proof}
The model structure and its monoidality is established by
Lydakis \cite[Theorems 4.6 and 5.1]{Lydakis} for $\catW=\sset$
and Berger–Moerdijk \cite[Theorem~7.6 and Example 7.7(b)]{BergerMoerdijk} for the general case.
The generating (acyclic) cofibrations of $\PSh(Γ,\catW)$
are given by pushout products $j◻∂_c$ of generating (acyclic) cofibrations~$j$ of $\catW$
and the boundary inclusion maps $∂_c:∂\Yo{c}→\Yo{c}$,
where $c∈Γ$ and $∂\Yo{c}$ is the subobject of~$\Yo{c}∈\PSh(Γ)$ such that for every $d∈Γ$ the set $∂\Yo{c}(d)$ is the set of noninjective pointed maps of sets $c→d$.

If $\catW$ is a monoidal model category, then $\PSh(Γ,\catW)$ is a $\catW$-enriched model category
because the pushout product of an (acyclic) cofibration~$i$ of~$\catW$
and a generating (acyclic) cofibration $j◻∂_c$ of $\PSh(Γ,\catW)$ is $(i◻j)◻∂_c$,
where $i◻j$ is an (acyclic) cofibration in~$\catW$ by the monoidality of~$\catW$.
By Hovey \cite[Corollary~4.2.5]{Hovey}, we deduce that the $\catW$-tensoring functor on $\PSh(Γ,\catW)$ is a left Quillen bifunctor.
The unit axiom for enrichments follows from the unit axiom for~$\catW$.
\end{proof}

\begin{proposition}
\label{Reedy.projective.injective.equivalent}
Suppose $\catW$ is a combinatorial model category.
The identity functors yield Quillen equivalences
$$\PSh(Γ,\catW)_\proj⇄\PSh(Γ,\catW)_\Reedy⇄\PSh(Γ,\catW)_\inj$$
of the projective, generalized Reedy (\cref{Reedy.model.structure}), and injective model structures on $Γ$-objects.
In particular, given a cofibrant object $X∈\catW$ and $⟨m⟩∈Γ$, the object $X⊗\Yo{⟨m⟩}$ is cofibrant in the generalized Reedy model structure.
\end{proposition}
\begin{proof}
It suffices to show
that generating projective cofibrations are generalized Reedy cofibrations
and generating generalized Reedy cofibrations are injective cofibrations.
Suppose $j:X→Y$ is a generating cofibration in~$\catW$ and $⟨m⟩∈Γ$.

Given a generating projective cofibration $j⊗\Yo{⟨m⟩}$, we have to show that it is a generalized Reedy cofibration,
i.e., for every $c∈Γ$ the latching map $j◻(\latch_c\Yo{⟨m⟩}→\Yo{⟨m⟩}(c))$
is a projective cofibration in $\catW^{\Aut(c)}$.
Following Lydakis \cite[Proposition~3.2]{Lydakis},
the map
$$\latch_c\Yo{⟨m⟩}→\Yo{⟨m⟩}(c)$$
splits as a coproduct injection map $$\latch_c\Yo{⟨m⟩}→\latch_c\Yo{⟨m⟩}⊔\Surj(⟨m⟩,c),$$
where $\Surj(⟨m⟩,c)$ is the set of surjective maps $⟨m⟩→c$ of pointed sets.
Therefore, the latching map is isomorphic to
$$\id_{Y⊗\latch_c\Yo{⟨m⟩}}⊔j⊗\Surj(⟨m⟩,c),$$
which is a cobase change of the map $j⊗\Surj(⟨m⟩,c)$.
The $\Aut(c)$-action on $\Surj(⟨m⟩,c)$ is free, so the map $j⊗\Surj(⟨m⟩,c)$ is a projective cofibration in~$\catW^{\Aut(c)}$.

Given a generating generalized Reedy cofibration $j◻∂_{⟨m⟩}$, we have to show that it is an injective cofibration,
i.e., after evaluating at an object $c∈Γ$ we get a cofibration in~$\catW$.
The map $∂_{⟨m⟩}(c)$ splits as
$$∂_{⟨m⟩}(c):∂\Yo{⟨m⟩}(c)→∂\Yo{⟨m⟩}(c)⊔\Inj(⟨m⟩,c),$$
where $\Inj(⟨m⟩,c)$ denotes the set of injective pointed maps $⟨m⟩→c$.
This, the map $(j◻∂_{⟨m⟩})(c)$ is isomorphic to
$$\id_{Y⊗∂\Yo{⟨m⟩}(c)}⊔j⊗\Inj(⟨m⟩,c),$$
which is a cobase change of the cofibration $j⊗\Inj(⟨m⟩,c)$.
\end{proof}

\begin{definition}
\label{external.product}
Let $\catC$ and $\catD$ be small categories and $\catV$ a symmetric monoidal category (e.g., $\catV=\sset$).
We define an \emph{external product functor}~$⊠$ as follows:
$$⊠:\PSh(\catC,\catV)⨯\PSh(\catD,\catV)→\PSh(\catC\times \catD,\catV), \qquad (F⊠G)(c,d)=F(c)⊗G(d).$$
\end{definition}

\begin{proposition}
\label{external.product.quillen}
Let $\catC$, $\catC'$, and $\catD$ be small categories and let $\catV$ be a combinatorial monoidal model category.
Equip the category of simplicial presheaves on $\catD$ with the injective model structure.
Suppose that one of the following holds:
\begin{conditions}[(1)]
\item\label{injective.case} the categories of $\catV$-valued presheaves on $\catC$ and $\catC⨯\catD$ are equipped with the injective model structure (\cref{injective.model.structure});
\item\label{gamma.reedy.case} $\catC=\Gamma\times \catC'$ and the categories of $\catV$-valued presheaves on $\catC$ and $\catC\times \catD$ are equipped with the $\Gamma$-Reedy injective model structure
(\cref{Reedy.injective.model.structure}).
\end{conditions}
Then the external product functor of \cref{external.product} is a $\catV$-enriched left Quillen bifunctor.
\end{proposition}

\begin{proof}
We first prove the claim for \cref{injective.case}.
We have $$(F⊠G)(c,d)=F(c)⊗G(d).$$
The monoidal product $F(c)⊗G(d)$ is separately $\catV$-enriched cocontinuous in $F$ and~$G$.
Injective cofibrations and acyclic cofibrations are defined objectwise.
Consequently, the pushout product axioms required for a left Quillen bifunctor follow from those of the monoidal model category~$\catV$.

For \cref{gamma.reedy.case}, denote by~$◻$ the pushout product associated with~$⊠$.
By \cref{Reedy.injective.model.structure},
generating (acyclic) cofibrations for the $Γ$-Reedy injective model structure on $\PSh(Γ⨯\catC',\catV)$ have the form $j◻∂_c$,
where $c∈Γ$ and $j$ is an (acyclic) injective cofibration in $\PSh(\catC',\catV)$.
The pushout product of $j◻∂_c$ and an (acyclic) injective cofibration~$k$ in $\PSh(\catD,\catV)$
is $(j◻∂_c)◻k≅(j◻k)◻∂_c$.
The map $j◻k$ is an injective cofibration in $\PSh(\catC'⨯\catD,\catV)$ by the first part.
Therefore, the map $(j◻k)◻∂_c$ is an (acyclic) cofibration in the $Γ$-Reedy injective model structure on $\PSh(Γ⨯\catC'⨯\catD,\catV)$.
\end{proof}

We will realize the $(\infty,1)$-category of sheaves of symmetric monoidal $(\infty,d)$-categories
via a particular model structure,
following Barwick \cite[Section~3.1]{Barwick.CSS} for $(∞,d)$-categories
and Toën \cite[Section~2.2]{Toen} for symmetric monoidal $(∞,d)$-categories.
We begin by taking the injective model structure on $\sPSh(\stcart\times \Gamma\times \Delta^{\times d})$.
To encode the Segal conditions for $Δ$ and~$Γ$, completeness and globularity conditions for~$Δ$, and descent conditions for $\stcart$,
we perform a left Bousfield localization of this model category.

\begin{definition}
\label{left.Bousfield.localization}
(Hirschhorn \cite[Definition~3.1.1.(1)]{Hirschhorn}, Barwick \cite[Definition~4.2]{Barwick.Model}, Lurie \cite[§A.3.7]{Lurie.HTT}.)
Suppose $M$ is a model category and $S$ is a set of morphisms in~$M$.
The \emph{left Bousfield localization} of~$M$ at~$S$ is a model category $\lbl_S M$ equipped with a left Quillen functor $F:M→\lbl_S M$
that satisfies the following universal property.
\begin{itemize}
\item For any model category~$N$, composition with~$F$ induces a bijection between left Quillen functors $\lbl_S M→N$
and left Quillen functors $G:M→N$ such that the left derived functor of~$G$ sends elements of $S$ to weak equivalences in~$N$.
\end{itemize}
An object $X∈M$ is \emph{$S$-local} if for every morphism $f:Y→Z$ in~$S$, the induced map $\rmap(Z,X)→\rmap(Y,X)$ is a weak equivalence.
A morphism $f:Y→Z$ is an \emph{$S$-local equivalence} if, for every $S$-local object~$X$, the induced map $\rmap(Z,X)→\rmap(Y,X)$ is a weak equivalence.
Here $\rmap(-,-)$ denotes the homotopy function complex in a model category, as defined in Hirschhorn \cite[Definition~17.4.1]{Hirschhorn}.
\end{definition}

\begin{proposition}
\label{left.Bousfield.localization.exists}
(Smith; Barwick \cite[Theorem~4.7]{Barwick.Model}.)
Suppose $M$ is a left proper combinatorial model category and $S$ is a set of morphisms in~$M$.
Then the left Bousfield localization (\cref{left.Bousfield.localization}) of~$M$ at~$S$ exists and is a left proper combinatorial model category.
Moreover, the localization has the following properties:
\begin{enumerate}
\item its underlying category is $M$;
\item its weak equivalences are precisely the $S$-local equivalences;
\item its cofibrations and acyclic fibrations coincide with those of~$M$;
\item its fibrant objects are precisely the $S$-local fibrant objects of~$M$.
\end{enumerate}
\end{proposition}

\begin{definition}
\label{enriched.left.Bousfield.localization}
(Barwick \cite[Definition~4.42]{Barwick.Model}.)
Suppose $\catV$ is a monoidal model category,
$M$ is a $\catV$-enriched model category and $S$ is a set of morphisms in~$M$.
The \emph{enriched left Bousfield localization} of~$M$ at~$S$ is a $\catV$-enriched model category $\lbl_S M$
equipped with a $\catV$-enriched left Quillen functor $F:M→\lbl_S M$
that satisfies the following universal property.
\begin{itemize}
\item For any $\catV$-enriched model category~$N$, composition with~$F$ induces a bijection between $\catV$-enriched left Quillen functors $\lbl_S M→N$
and $\catV$-enriched left Quillen functors $G:M→N$ such that the left derived functor of~$G$ sends elements of $S$ to weak equivalences in~$N$.
\end{itemize}
An object $X∈M$ is \emph{enriched $S$-local} if for every morphism $f:Y→Z$ in~$S$, the induced map $\dmap(Z,X)→\dmap(Y,X)$ is a weak equivalence.
A morphism $f:Y→Z$ is an \emph{enriched $S$-local equivalence} if, for every $S$-local object~$X$, the induced map $\dmap(Z,X)→\dmap(Y,X)$ is a weak equivalence.
Here $\dmap(-,-)$ denotes the derived enriched mapping object in a $\catV$-enriched model category.
\end{definition}

We will need the following characterization and existence theorem for enriched left Bousfield localizations.
Although all of the ideas already appear in the works of Barwick and Gorchinskiy–Guletskiĭ,
we provide a proof because this exact variant of the statement does not appear in either source.
(Our version does not assume $M$ to be tractable and gives a formula for~$S'$ in the enriched case.)

\begin{proposition}
\label{enriched.left.Bousfield.localization.exists}
(Barwick \cite[Theorem 4.46 and Proposition 4.47]{Barwick.Model},
Gorchinskiy–Guletskiĭ \cite[Lemma 31 and Remark~32 (journal), Lemma~28 and Remark~29 (arXiv v4)]{GorchinskiyGuletskii}.)
Suppose $\catV$ is a tractable monoidal model category,
$M$ is a left proper combinatorial $\catV$-enriched model category,
and $S$ is a set of morphisms in~$M$.
Then the $\catV$-enriched left Bousfield localization (\cref{enriched.left.Bousfield.localization}) of~$M$ at~$S$ exists
and is a left proper combinatorial $\catV$-enriched model category.
It can be computed as $\lbl_{S'}M$, where $S'=\{X\dtp f\mid X∈G, f∈S\}$
and $G$ is a set of \emph{homotopy generators} of~$\catV$,
that is, every object of~$\catV$ is weakly equivalent to the homotopy colimit of a small diagram of objects in~$G$.
If $\lbl_S M$ is a $\catV$-enriched model category, then $\lbl_S M=\lbl_{S'}M$.
If $M$ is tractable, then so is its localization.
Moreover, the localization has the following properties:
\begin{enumerate}
\item its underlying $\catV$-enriched category is $M$;
\item its weak equivalences are precisely the enriched $S$-local equivalences;
\item its cofibrations and acyclic fibrations coincide with those of~$M$;
\item its fibrant objects are precisely the enriched $S$-local fibrant objects of~$M$.
\end{enumerate}
\end{proposition}

\begin{proof}
Recall that every tractable model category~$\catV$ admits a set~$G$ of homotopy generators,
which we take (without loss of generality) to consist of cofibrant objects.
For example, we can take $G$ to comprise the domains and codomains of a set of generating cofibrations with cofibrant domains.
Indeed, every object in~$\catV$ admits a cofibrant replacement~$X$
such that $∅→X$ is a transfinite composition of cobase changes of generating cofibrations with cofibrant domains.
Since all objects in the diagram are cofibrant, all pushout squares are homotopy pushout squares.
Therefore, the transfinite composition is a homotopy transfinite composition.
Thus, $X$ is a homotopy colimit of a diagram in~$G$.

To show that $\lbl_{S'}M$ is the desired $\catV$-enriched left Bousfield localization,
we verify that $\lbl_{S'}M$ is a $\catV$-enriched model category,
i.e., the tensoring functor $\catV⨯M→M$ is a left Quillen bifunctor and the unit axiom is satisfied.
These properties follow from the same properties of~$M$, with the exception of the following property.
Suppose $f:X→Y$ is a generating cofibration with cofibrant domain of~$\catV$ and $s:A→B$ is an acyclic cofibration in~$\lbl_{S'}M$.
Then $f◻s$ is a weak equivalence in~$\lbl_{S'}M$.
To this end, consider the diagram
$$\xymatrix{
X⊗A \ar[rr]^{X⊗s} \ar[dd]_{f⊗A} && X⊗B \ar[dl] \ar[dd]^{f⊗B}\cr
& P \ar[dr]^{f◻s}\cr
Y⊗A \ar[rr]_{Y⊗s} \ar[ur] & & Y⊗B.\cr
}$$
Without loss of generality (performing cofibrant replacements on elements of~$S$ as necessary), assume that~$S$ (and therefore also $S'=\{X⊗f\}$)
consists of cofibrations with cofibrant domains.
For every $s∈S$ the functor $-⊗s:\catV→M^→$ is a left Quillen functor, where $M^→$ is equipped with the projective model structure.
Therefore, its left derived functor $-\dtp s$ preserves homotopy colimits.
By assumption, $-\dtp s$ sends elements of~$G$ to weak equivalences in~$\lbl_{S'}M$.
Since $G$ is a set of homotopy generators of~$\catV$, the functor $-\dtp s$ sends all objects of~$\catV$ to weak equivalences in~$\lbl_{S'}M$.
Thus, since $X$ and $Y$ are cofibrant, the maps $X⊗s$ and $Y⊗s$ are weak equivalences in~$\lbl_{S'}M$,
and so is the cobase change $Y⊗A→P$ of $X⊗s$ in the diagram.
By the 2-out-of-3 property, the map $f◻s$ is a weak equivalence in $\lbl_{S'}M$.

To establish the universal property of enriched left Bousfield localizations in~\cref{enriched.left.Bousfield.localization},
suppose $H:M→N$ is a $\catV$-enriched left Quillen functor into a $\catV$-enriched model category~$N$
such that the left derived functor of~$H$ sends elements of $S$ to weak equivalences in~$N$.
We want to show that $H∘F:\lbl_{S'}M→N$ is a $\catV$-enriched left Quillen functor.
Thanks to our cofibrancy assumptions on $G$, $S$, and $S'$, all functors and tensor products below are automatically derived.
By the universal property of left Bousfield localizations (\cref{left.Bousfield.localization}),
it suffices to show that $\ldf(H∘F)$ sends elements of~$S'$ to weak equivalences in~$N$.
Given $s∈S'$, i.e., $s=X⊗f$ for $X∈G$ and $f∈S$,
we compute $\ldf(H∘F)(s)=H(s)=H(X⊗f)≅X⊗H(f)$,
where the last isomorphism is induced by the $\catV$-enrichment of~$H$.
Since $f$ is a cofibration with cofibrant domain, so is $H(f)=\ldf H(f)$.
The latter map is a weak equivalence by assumption.
Since $X$ is cofibrant, $X⊗H(f)$ computes the derived tensor product, hence is a weak equivalence, as desired.

If $\lbl_S M$ is a $\catV$-enriched model category, then taking $N=\lbl_S M$
yields a left Quillen functor $\lbl_{S'}M→\lbl_S M$.
Since $S⊂S'$, this shows $\lbl_{S'}M=\lbl_S M$.

The characterizations of weak equivalences and fibrant objects follow from their nonenriched versions
using the fact that $G$ is a set of homotopy generators of~$\catV$.
\end{proof}

\begin{proposition}
\label{enriched.left.Bousfield.localization.criterion}
Suppose $\catV$ is a tractable monoidal model category,
$M$ is a left proper combinatorial $\catV$-enriched model category,
and $S$ is a set of morphisms in~$M$.
\label{localization.automatically.enriched}
We have the following.
\begin{parts}[(1)]
\item\label{unit.generator}
If $\{\munit_\catV\}$ is a set of homotopy generators of~$\catV$, then $\lbl_S M$ is the $\catV$-enriched localization of~$M$ at~$S$.
\item\label{localization.simplicial}
If $M$ is simplicial, then so is $\lbl_S M$.
\item\label{localization.smooth.simplicial}
If $M$ is $\smsset_{\sing,\inj}$-enriched (\cref{smsset.injective}), then so is $\lbl_S M$.
\item\label{presheaf.generators}
If $C$ is a small category and $G$ is a set of homotopy generators of~$\catV$, then $\{X⊗\Yo{c}\mid X∈G, c∈\catC\}$
is a set of homotopy generators for $M=\PSh(C,\catV)$
equipped with a model structure whose weak equivalences are defined objectwise.
\end{parts}
\end{proposition}

\begin{proof}
For \cref{unit.generator}, observe that for every $f∈S$ the map $\munit_\catV\dtp f=\crep\munit_\catV⊗f$
is weakly equivalent to~$f$ by the unit axiom for $\catV$-enriched model structures,
hence is a weak equivalence in $\lbl_S M$.
Thus, by \cref{enriched.left.Bousfield.localization.exists}, the model category $\lbl_S M$ is $\catV$-enriched.

Every simplicial set is a homotopy colimit of a constant point-valued diagram over its category of elements,
so the terminal simplicial set is a homotopy generator of $\sset$.
Thus, \cref{unit.generator} implies \cref{localization.simplicial}.

Since $\smsset$ is Quillen equivalent to $\sset$ (\cref{smsset.injective}), \cref{localization.simplicial} implies \cref{localization.smooth.simplicial}.

\cref{presheaf.generators} is implied by the enriched co-Yoneda lemma and the fact that homotopy colimits in categories of presheaves are computed objectwise.
\end{proof}

Leveraging \cref{enriched.left.Bousfield.localization.criterion},
we can construct a model for the $\infty$-category of geometric symmetric monoidal $(\infty,d)$-categories as a left Bousfield localization
of the injective model structure on a certain category of presheaves.
We need only specify the set of morphisms such that their respective local objects satisfy the Segal conditions for both $\Gamma$ and $\Delta^{\times d}$,
as well as the homotopy descent condition for a site~$\site$.
In preparation, we introduce the following definitions.

\begin{definition}
\label{bousloc}
Suppose $\catV$ is a cartesian model category.
Recall the functor
$$\PSh(C,\set)\lto9{\PSh(C,ι_\catV)}\PSh(C,\catV)$$
(\cref{enriched.yoneda}), where $C$ is a small category.
Using this functor, we define the following morphisms in $\PSh(\Delta,\catV)$, $\PSh(\Gamma,\catV)$, and $\PSh(\site,\catV)$, where $\site$ is a site.
These morphisms will be used later to construct left Bousfield localizations.
\begin{enumerate}[(i),series=segalmaps]
\item
(Segal's special $Δ$-condition.)
For $[a],[b]\in \Delta$, the maps
\begin{equation}\label{segal1}\phi^{a,b}:\Yo{[a]}\sqcup_{\Yo{[0]}}\Yo{[b]}\to \Yo{[a+b]},\end{equation}
where the pushout is taken along the inclusion $[0]→[a]$ of the maximal vertex of~$[a]$ (i.e., the vertex~$a$)
and the inclusion $[0]→[b]$ of the minimal vertex of $[b]$ (i.e., the vertex~$0$).
The local objects with respect to the maps $\phi^{a,b}$
are precisely Segal's special $Δ$-objects, for which the map $X_{a+b}→X_a⨯_{X_0}X_b$ is a weak equivalence, where $⨯$ denotes the homotopy pullback.
\item
(Completeness condition.)
The map
\begin{equation}\label{segal2}x:E\to \Yo{[0]}.\end{equation}
Here $E∈\PSh(\Delta)$ is the nerve of the groupoid with two objects $p$ and~$q$ and two nonidentity morphisms $p\to q$ and $q\to p$.
Alternatively, a weakly equivalent model for~$E$ is $E'=Δ^0⊔_{Δ^1}Δ^3⊔_{Δ^1}Δ^0$,
where the two maps $Δ^1→Δ^3$ are the inclusions of the edges $0→2$ and $1→3$, respectively.
Local objects with respect to the maps \cref{segal1,segal2} are Rezk's \emph{complete Segal objects},
characterized by Segal's special $Δ$-condition and the completeness condition.
The latter states that the map $X_0→X_1^⨯$ sending objects to their identity morphisms
is a weak equivalence, where $X_1^⨯$ is the subobject of invertible 1-morphisms in~$X_1$.
\item
(Segal's special $Γ$-condition.)
For $⟨\kappa⟩,⟨ℓ⟩\in\Gamma$, the maps
\begin{equation}\label{monoidal1}t^{\kappa,ℓ}:\Yo{⟨\kappa⟩}\sqcup\Yo{⟨ℓ⟩}\to \Yo{⟨\kappa+ℓ⟩}\end{equation}
and the map
$$τ:∅→\Yo{⟨0⟩}.\eqlabel{monoidal2}$$%
Local objects are precisely Segal's special $Γ$-objects,
for which the maps $X_{\kappa+ℓ}→X_\kappa⨯X_ℓ$ and $X_0→*$ are weak equivalences.
\item
(Sheaf condition.)
The maps
$$i^{{\cal U},V}:c_{\cal U}→\Yo{V},\eqlabel{cech}$$%
given by covering sieves of the Grothendieck topology on the site~$\site$.
The local objects for the maps $i^{{\cal U},V}$ are ∞-sheaves (also known as ∞-stacks).
For these objects, the restriction map
\begin{equation}\label{homotopy.descent}X(V)→\holim_{U∈(S/c_{\cal U})^\op} X(U)\end{equation}
is a weak equivalence.
For more information on Čech descent, see Dugger–Hollander–Isaksen \cite[Appendix~A]{DHI} and Glass–Minichiello \cite{GlassMinichiello}.
\end{enumerate}
\end{definition}

\begin{proposition}
\label{model.structure.sheaves}
The model structure $\sPSh(\site)_{\local}$ on the category of simplicial presheaves on a site~$\site$
given by the left Bousfield localization of the injective model structure on $\sPSh(\site)$ with respect to the morphisms \cref{cech}
is a left proper combinatorial simplicial cartesian model category.
\end{proposition}

\begin{proof}
The model structure exists and has the desired properties by \cref{enriched.left.Bousfield.localization.exists,enriched.left.Bousfield.localization.criterion}.
For the cartesian property,
observe that the Čech fibrant replacement functor, regarded as an endofunctor
on the (nonlocalized) injective model structure, preserves finite homotopy limits.
(See, for example, Lurie \cite[Proposition 6.2.2.7]{Lurie.HTT}.)
Furthermore, it sends Čech maps \cref{cech} to objectwise weak equivalences.
Therefore, the homotopy product
(in the nonlocalized injective model structure)
of any presheaf with a Čech map \cref{cech} of the form
$$i^{{\cal U},V}:c_{\cal U}→\Yo{V}$$
is mapped by the Čech fibrant replacement functor
to an objectwise weak equivalence.
This means the homotopy product is itself a Čech-local weak equivalence, as desired.
\end{proof}

\begin{proposition}
\label{gammasp}
(Lydakis \cite[Theorems 4.6 and 5.1]{Lydakis}, Berger–Moerdijk \cite[Theorem~1.6]{BergerMoerdijk}, Shulman \cite[Theorem~8.9]{Shulman.Reedy}.)
The model structure $\sPSh(\Gamma)_\local$ on $\Gamma$-spaces given by the left Bousfield localization of the generalized Reedy model structure on $\sPSh(\Gamma)$
(\cref{Reedy.model.structure})
with respect to the morphisms \cref{monoidal1} and \cref{monoidal2}
is a left proper combinatorial simplicial monoidal model category.
\end{proposition}

\begin{proof}
The model structure exists and has the desired properties by \cref{enriched.left.Bousfield.localization.exists,enriched.left.Bousfield.localization.criterion}.
For the monoidal property, we need to show that the derived smash product of an arbitrary presheaf and a localizing map is a local weak equivalence.
Observe that the smash product of $⟨r⟩∈Γ$ with a map \cref{monoidal1}
$$t^{\kappa,ℓ}:\Yo{⟨\kappa⟩}\sqcup\Yo{⟨ℓ⟩}\to \Yo{⟨\kappa+ℓ⟩}$$
is again a map of the same form, namely, $t^{⟨κ⟩∧⟨r⟩,⟨ℓ⟩∧⟨r⟩}$.
Likewise, the smash product of $τ:∅→\Yo{⟨0⟩}$ \cref{monoidal2} with any $⟨r⟩∈Γ$ is again~$τ$.
Both smash products are derived because coproducts of representable objects are projectively cofibrant, hence generalized Reedy cofibrant.
\end{proof}

\begin{remark}
Since the injective model structure is not monoidal (\cref{injective.Gamma.nonmonoidal}), neither is the localized injective structure.
The injective model structure on $\Gamma$-spaces has more cofibrations than the Bousfield–Friedlander strict model structure \cite[Theorem~3.5]{BousfieldFriedlander},
which coincides with the generalized Reedy model structure of \cref{Reedy.model.structure}.
The generalized Reedy model structure, in turn, has more cofibrations than Schwede's strict Q-model structure \cite[Lemma~1.4]{Schwede},
which coincides with the projective model structure.
All three model structures have the same weak equivalences.
The localized model structure on $\Gamma$-spaces of \cref{gammasp} has fewer weak equivalences
than the Bousfield–Friedlander stable model structure \cite[Theorem~5.2]{BousfieldFriedlander}
and Schwede's stable Q-model structure \cite[Theorem~1.5]{Schwede},
since we do not impose the group-like property on $Γ$-spaces.
\end{remark}

\begin{definition}
\label{rezkdelta}
The Rezk model structure \cite[Theorem~7.2]{Rezk.CSS} on $Δ$-spaces is
the left Bousfield localization of the Reedy model structure on $\sPSh(\Delta)$ (which coincides with the injective model structure)
with respect to the morphisms \cref{segal1} and \cref{segal2}.
We denote this model category by $\sPSh(\Delta)_{\local}$.
\end{definition}

\section{Geometric symmetric monoidal $(\infty,d)$-categories}

\begin{notation}
\label{site.and.enrichment}
Throughout this section, $\site$ denotes an arbitrary small site,
$\catV$ denotes an arbitrary combinatorial cartesian model category,
and $d$ is a nonnegative integer denoting the category number, i.e., $d$ in $(∞,d)$-categories.
In subsequent sections, $\site$ will be taken to be a structured cartesian site $\stcart$ of \cref{def.stcart}
and $\catV$ will be taken to be $\sset$ or $\smsset$ (\cref{smsset.injective}).
\end{notation}

\subsection{Geometric multiple and globular $(\infty,d)$-categories}
\label{geocatssect}

To combine symmetric monoidal, $d$-categorical, and geometric structures,
we lift the morphisms in \cref{bousloc} to the category of presheaves
on the product category $\site\times \Gamma\times \Delta^{\times d}$,
using the external product~$⊠$ of \cref{external.product}.
The resulting model category has fibrant objects that are local with respect to each morphism~$f$ in \cref{bousloc}
after evaluation on arbitrary objects in the factors not corresponding to~$f$.
For instance, evaluation at an object of $Γ⨯\Delta^{\times d}$ yields an ∞-sheaf of simplicial sets on $\site$,
which means the homotopy descent condition \cref{homotopy.descent} is satisfied.

\begin{notation}
\label{localmapsnot}
Assume \cref{site.and.enrichment}.
Recall \cref{bousloc}.
We consider the following sets of maps in $\PSh(\site\times\Gamma\times\Delta^{\times d},\catV)$.
\begin{itemize}
\item Let $S_\Delta=\{\Yo{c}⊠\phi^{a,b}, \Yo{c}⊠x\mid c\in\site\times\Gamma\times Δ^{\{1,\ldots,k-1,k+1,\ldots,d\}}, 1≤k≤d, a,b∈Δ\}$,
where $\phi^{a,b}$ are the morphisms \cref{segal1} and $x$ is the morphism \cref{segal2}.
\item Let $S_\Gamma=\{\Yo{c}⊠t^{\kappa,ℓ}, \Yo{c}⊠\tau\mid c\in\site\times\Delta^{\times d}, \kappa,ℓ∈Γ\}$,
where $t^{\kappa,ℓ}$ are the morphisms \cref{monoidal1} and $\tau$ is the morphism \cref{monoidal2}.
\item Let $S_\site=\{\Yo{c}⊠i^{{\cal U},V}\mid c\in\Gamma\times\Delta^{\times d}, V∈\site\}$,
where $i^{{\cal U},V}$ are the morphisms \cref{cech} for a covering family ${\cal U}$ of~$V$.
\end{itemize}
\end{notation}

With the powerful \cref{enriched.left.Bousfield.localization.exists,enriched.left.Bousfield.localization.criterion} in hand, we are ready to construct model categories whose objects are geometric symmetric monoidal $(\infty,d)$-categories.
We provide model structures for both the multiple and globular versions of $(∞,d)$-categories.
The following theorem is an immediate corollary of \cref{enriched.left.Bousfield.localization.exists}.

\begin{proposition}
\label{abstract.multiple.model.structure}
Assume \cref{site.and.enrichment}.
There is a combinatorial $\catV$-enriched left proper model category
$$\PSh(\site\times\Gamma\times\Delta^{\times d},\catV)_{\uple}$$
given by the $\catV$-enriched left Bousfield localization (\cref{enriched.left.Bousfield.localization}) of the injective model structure on $\PSh(\site\times \Gamma\times \Delta^{\times d},\catV)$
(\cref{injective.model.structure})
with respect to the set of morphisms $S_{\site}\cup S_{\Gamma}\cup S_{\Delta}$ as defined in \cref{localmapsnot}.
\end{proposition}

Specializing to the case where $\catV=\sset$ or $\catV=\smsset$, we make the following definition.

\begin{subdefinition}
\label{multiple.model.structures}
We define the following model categories.
\begin{enumerate}
\item\label{multiple.model.structure}
(Geometric symmetric monoidal $(\infty,d)$-uple categories.)
Setting $\catV=\sset$ in \cref{abstract.multiple.model.structure} yields a combinatorial simplicial left proper model category
$$\smcatuple{d}≔\sPSh(\site\times\Gamma\times\Delta^{\times d})_{\uple}.$$
\item\label{multiple.model.structure.smooth}
(Geometric symmetric monoidal $(\infty,d)$-uple categories with isotopies.)
Setting $\catV=\smsset$ in \cref{abstract.multiple.model.structure} yields a combinatorial $\smsset$-enriched left proper model category
$$\fraksmcatuple{d}≔\smsPSh(\site\times\Gamma\times\Delta^{\times d})_{\uple}.$$
\end{enumerate}

In the case where $\site$ is the terminal category and $\catV=\sset$, we obtain symmetric monoidal $(\infty,d)$-uple categories:
$$\catuple{d}≔\sPSh(\Gamma\times \Delta^{\times d})_\uple.$$
\end{subdefinition}

\cref{abstract.multiple.model.structure} provides a geometric variant of $d$-fold complete Segal spaces.
However, these only become a model for $(\infty,d)$-categories once we impose the globularity condition
(Lurie \cite[Definition~1.2.7.(2)]{Lurie.TwoCat}, see also Barwick–Schommer-Pries \cite[Notation~12.1]{BarwickSchommerPries}).
For $d=2$ this corresponds to the passage from double categories to bicategories.
Recall that double categories have two distinct notions of 1-morphisms: horizontal and vertical.
Both types of 1-morphisms can be composed, and 2-cells are squares involving two vertical and two horizontal morphisms.
Thus, our model so far describes the $d$-fold analog of double categories.
To eliminate the extra 1-morphisms, we further localize the functor category $\PSh(\site\times\Gamma\times\Delta^{\times d},\catV)_{\uple}$
at the morphisms formed by external tensoring of the map \cref{glob} with representable presheaves on $\site\times\Gamma$.

\begin{definition}
\label{globular.maps}
Assume \cref{site.and.enrichment}.
Recall the enriched Yoneda embedding~$\Yo{}$ (\cref{enriched.yoneda}).
We define the following maps, which are incorporated into our left Bousfield localization for the globular model structure.
\begin{enumerate}[resume*=segalmaps]
\item
(Globular maps.)
For an object ${\bf m}=([m_1],\ldots,[m_d])\in \Delta^{\times d}$, let $\hat{\bf m}$ be the object whose $j$th component is
$$[\hat{m}_j]=\begin{cases}
\left[0\right],& \text{if there is $i<j$ with $m_i=0$,}\cr
\left[m_j\right],& \text{otherwise.}\cr
\end{cases}$$
There is a canonical map from ${\bf m}$ to $\hat{\bf m}$ constructed using identities or unique maps to $[0]$ for each index~$j$.
The globular maps are defined as the following morphisms in $\PSh(Δ^{⨯d},\catV)$:
\begin{equation}\label{glob}\psi^{{\bf m}}:\Yo{{\bf m}}\to \Yo{\bf \hat{m}}.\end{equation}
\end{enumerate}
\end{definition}

The local objects are $d$-fold simplicial spaces~$X$ such that $X_0$
is homotopy constant and for every $k≥0$ the object $X_k$ is local within the category of $(d-1)$-fold simplicial spaces.
For $d=2$, the locality condition reduces to requiring the degeneration maps $X_{0,0}\to X_{0,b}$ to be equivalences.
This ensures that all vertical morphisms are homotopic to identities.

\begin{notation}
Assume \cref{site.and.enrichment}.
We consider the following set of morphisms in $\PSh(\site\times \Gamma\times \Delta^{\times d},\catV)$.
\begin{itemize}
\item Let $S_\glob=\{c⊠\psi^{\bf m}\mid c\in\site\times\Gamma, {\bf m}∈Δ^{⨯d}\}$,
where $\psi^{\bf m}$ are the morphisms \cref{glob}.
\end{itemize}
\end{notation}

We now introduce the globular analog of the model category in \cref{abstract.multiple.model.structure}.
Once again, the theorem follows immediately from \cref{enriched.left.Bousfield.localization.exists}.
\begin{proposition}
\label{abstract.globular.model.structure}
Assume \cref{site.and.enrichment}.
There is a combinatorial $\catV$-enriched left proper model category
$$\PSh(\site\times\Gamma\times\Delta^{\times d},\catV)_{\glob}$$
given by the $\catV$-enriched left Bousfield localization (\cref{enriched.left.Bousfield.localization}) of the injective model structure on $\PSh(\site\times \Gamma\times \Delta^{\times d},\catV)$
(\cref{injective.model.structure})
with respect to the set of morphisms $S_{\site}\cup S_{\Gamma}\cup S_{\Delta}\cup S_{\glob}$ as defined in \cref{localmapsnot}.
\end{proposition}

Specializing to the case where $\catV=\sset$ or $\catV=\smsset$, we make the following definition.
\begin{subdefinition}
\label{globular.model.structures}
We define the following model categories.
\begin{enumerate}
\item\label{globular.model.structure}
(Geometric symmetric monoidal $(\infty,d)$-categories.)
There is a combinatorial simplicial left proper model category
$$\smcat{d}≔\sPSh(\site\times\Gamma\times\Delta^{\times d})_{\glob}.$$
We define a \emph{geometric symmetric monoidal $(∞,d)$-category} to be an object of $\smcat{d}$.

\item\label{globular.model.structure.smooth}
(Geometric symmetric monoidal $(\infty,d)$-categories with isotopies.)
There is a combinatorial $\smsset$-enriched left proper model category
$$\fraksmcat{d}≔\smsPSh(\site\times\Gamma\times\Delta^{\times d})_{\glob}.$$
We define a \emph{geometric symmetric monoidal $(∞,d)$-category with isotopies} to be an object of $\fraksmcat{d}$.
\end{enumerate}

In the case where $\site$ is the terminal category and $\catV=\sset$, we obtain symmetric monoidal $(\infty,d)$-categories:
$$\cat{d}≔\sPSh(\Gamma\times \Delta^{\times d})_\glob.$$
\end{subdefinition}

The following proposition shows that the two model categories in \cref{multiple.model.structures,globular.model.structures} are Quillen equivalent.
\begin{proposition}
\label{simplicial.vs.smoothset.enrich}
The Quillen equivalence of \cref{smsset.injective} induces Quillen equivalences
$$\xymatrix{
\smcatuple{d}\ar@<.125cm>[r]^-{|-|} & \ar@<.125cm>[l]^-{\sing} \fraksmcatuple{d},
} \qquad
\xymatrix{
\smcat{d}\ar@<.125cm>[r]^-{|-|} & \ar@<.125cm>[l]^-{\sing} \fraksmcat{d},
}
$$
where the functors are applied objectwise.
\end{proposition}
\begin{proof}
The Quillen equivalence of \cref{smsset.injective} induces a Quillen equivalence of injective model structures.
By Hirschhorn \cite[Theorem~3.3.20]{Hirschhorn}, this Quillen equivalence descends to a Quillen equivalence of localized model structures, in both the uple and globular case.
\end{proof}

\begin{remark}
\label{not.fibrant}
We emphasize that \cref{globular.model.structure} does \emph{not} require geometric symmetric monoidal $(∞,d)$-categories to be fibrant in
the model category $\smcat{d}$, or any other model structure with the same weak equivalences.
Indeed, all the constructions on geometric symmetric monoidal $(∞,d)$-categories considered in the present work respect weak equivalences (for instance, by being derived).
Hence, performing a fibrant replacement of the bordism category will only alter the constructions up to weak equivalence.

A useful analogy is provided by the notion of a derived category in homological algebra.
An object of the derived category of a ring~$R$ is a chain complex of $R$-modules.
A morphism in the derived category can be presented as a composition of a chain map and a formal inverse of a quasi-isomorphism.
This formal inverse can be viewed as a (fibrant) resolution of the target.
Requiring all objects in the derived category to be injectively fibrant chain complexes
yields an equivalent derived category in which all morphisms can now be presented as chain maps.
However, this approach is not commonly adopted in homological algebra,
where chain complexes that are neither injectively fibrant nor projectively cofibrant
are considered valid objects in the derived category, on par with injectively fibrant and projectively cofibrant objects.

From a practical standpoint, requiring bordism categories to be fibrant does not result in tangible benefits,
as fibrant objects have good properties regarding morphisms \emph{into} them,
whereas a functorial field theory is a morphism originating \emph{from} a bordism category.
Owing to our choice of the injective model structure, all objects of $\smcat{d}$ are cofibrant.
In particular, all bordism categories introduced in this paper are cofibrant, which eliminates the need
for additional cofibrant replacements when discussing functorial field theories.
\end{remark}

\subsection{$(\infty,d)$-categories of functors.}
\label{functor.categories}

In this section, we define an object $\Funmonglob(X,Y)$ that encodes the geometric symmetric monoidal $(\infty,d)$-category of functors between geometric symmetric monoidal $(\infty,d)$-categories $X$ and~$Y$.
The model categories in \cref{abstract.multiple.model.structure,abstract.globular.model.structure} cannot be made into monoidal model categories in a straightforward way, so constructing this mapping object is somewhat subtle.
We remark that it is not necessary to understand the details of the construction in this section to use the abstract properties of the mapping object (\cref{derivedhom}).

\begin{definition}
\label{smcat.monoidal.structure}
Assume \cref{site.and.enrichment}.
We equip the category
$$\PSh(\site⨯Γ⨯Δ^{⨯d},\catV)=\PSh(Γ,\PSh(\site⨯Δ^{⨯d},\catV))$$
with a monoidal structure given by the Day convolution (\cref{day.convolution})
induced by the smash product on~$Γ$ and the categorical product on $\PSh(\site⨯Δ^{⨯d},\catV)$.
\end{definition}

The uple and globular model categories
$$\PSh(\site⨯Γ⨯Δ^{⨯d},\catV)_\uple,\qquad\PSh(\site⨯Γ⨯Δ^{⨯d},\catV)_\glob$$
(\cref{abstract.multiple.model.structure,abstract.globular.model.structure})
are not monoidal model categories (\cref{injective.Gamma.nonmonoidal,globular.nonmonoidal}).

In the uple case, the problem is caused by the fact that the Day convolution monoidal product
quotients by the group of automorphisms of objects in~$Γ$ (\cref{injective.Gamma.nonmonoidal}).
This problem is resolved by passing to the $Γ$-Reedy injective model structure (\cref{Reedy.injective.model.structure}).

In the globular case, there is one more problem: the globular weak equivalences are not closed under cartesian products
with objects of~$Δ^{⨯d}$ (\cref{globular.nonmonoidal}).
This problem is fixed by passing from $Δ^{⨯d}$ to Joyal's category~$Θ_d$ \cite{Joyal.Disks},
which yields a Quillen equivalent monoidal model structure (\cref{theta.model.equivalence}).
The monoidal product and internal hom can then be transferred from the resulting monoidal model category
via the Quillen equivalence (\cref{globular.hom}).

\begin{definition}
\label{joyal.cell.category}
(Joyal \cite{Joyal.Disks}, see also Rezk \cite[Section 3.2]{Rezk.Theta}.)
The category $Θ_d$ is defined inductively on~$d≥0$.
For $d=0$ we take $Θ_0$ to be the terminal category.
For $d>0$, objects of $Θ_d$ are tuples $([m],c_1,…,c_m)$,
where $[m]∈Δ$ and $c_i∈Θ_{d-1}$.
Morphisms $([m],c_1,…,c_m)→([n],d_1,…,d_n)$
are maps $δ:[m]→[n]$ in~$Δ$ together with morphisms $f_{i,j}:c_i→d_j$ in~$Θ_{d-1}$
for all $i$ and $j$ such that $0<i≤m$ and $δ(i-1)<j≤δ(i)$.
Composition is given by composing the morphisms in~$Δ$ and setting $h_{i,k}=g_{j,k}f_{i,j}$, where $j$ is uniquely determined.
\end{definition}

Recall (Rezk \cite[§8]{Rezk.Theta}) the set $\thloc_{Θ,d}$ (denoted there by $\mathscr{S}_\Theta$) of morphisms of presheaves of sets on $Θ_d$, defined inductively on~$d≥0$.
Set $\thloc_{Θ,0}=∅$.
For $d>0$, set
\begin{equation}\label{s.theta.d}\thloc_{Θ,d}=\Se_{Θ_{d-1}}∪\Cpt_{Θ_{d-1}}∪V[1](\thloc_{Θ,d-1}),\end{equation}
where $\Se_{Θ_{d-1}}$ (Rezk \cite[Section~5.1]{Rezk.Theta}) is the analog of Segal maps $ϕ^{a,b}$ \cref{segal1},
$\Cpt_{Θ_{d-1}}$ (Rezk \cite[Section~7.6]{Rezk.Theta}) is the analog of completion maps~$x$ \cref{segal2},
and $V[1]$ is the intertwining functor defined by Rezk \cite[Section~4.4]{Rezk.Theta},
which, when applied to morphisms in $\thloc_{Θ,d-1}$, plays the role analogous to the functors $\Yo{c}⊠-$ ($c∈Δ^{⨯(d-1)}$) used to define
the localizing morphisms for~$Δ^{⨯d}$.

\begin{notation}
\label{s.theta}
Assume \cref{site.and.enrichment}.
We consider the following set of morphisms in $\PSh(\site⨯Γ⨯Θ_d,\catV)$.
\begin{itemize}
\item Let $S_Θ=\{c⊠ψ \mid c\in\site⨯Γ, ψ∈\thloc_{Θ,d}\}$,
\end{itemize}
where $\thloc_{\Theta,d}$ is the set of morphisms \eqref{s.theta.d}.
\end{notation}

\begin{proposition}
\label{multiple.monoidal.model.structure}
\label{theta.model.structure}
Assume \cref{site.and.enrichment}.
There are combinatorial left proper $\catV$-enriched monoidal model structures
$$\PSh(\site⨯Γ⨯Δ^{⨯d},\catV)_{\Reedy,\inj,\uple},\qquad\PSh(\site⨯Γ⨯Θ_d,\catV)_{\Reedy,\inj,\local}$$
obtained by the $\catV$-enriched left Bousfield localization (\cref{enriched.left.Bousfield.localization})
of the $Γ$-Reedy injective model structure
(\cref{Reedy.injective.model.structure})
at the sets of morphisms $S_\site∪S_Γ∪S_Δ$ and $S_{\site}\cup S_Γ\cup S_Θ$, respectively.
Here $S_\site$, $S_Γ$, and $S_Δ$ are described in \cref{localmapsnot} (adapted by substituting $\Theta_d$ for $Δ^{⨯d}$ for the latter model structure)
and $S_Θ$ is defined in \cref{s.theta}.
\end{proposition}

\begin{proof}
The $\catV$-enriched left Bousfield localizations exist by \cref{enriched.left.Bousfield.localization.exists}.
In the following, we treat both model structures simultaneously,
setting $C_d=Δ^{⨯d}$ and $S_C=S_Δ$ for the former model structure
and $C_d=Θ_d$ and $S_C=S_Θ$ for the latter.
To establish the monoidality of both model structures,
by \cref{enriched.left.Bousfield.localization.exists} and \cref{presheaf.generators} of \cref{enriched.left.Bousfield.localization.criterion}
it suffices to show that for every morphism $g∈S_\site∪S_Γ∪S_C$, object $G∈\site⨯Γ⨯C_d$, and cofibrant object $X∈\catV$,
the derived monoidal product of $X⊗\Yo{G}$ and~$g$ is a local weak equivalence.
By \cref{Reedy.projective.injective.equivalent},
the object $X⊗\Yo{G}$ is cofibrant and the morphism $g$ has a cofibrant domain and codomain.
Thus, the derived monoidal product is computed by the monoidal product $X⊗\Yo{G}⊗g$.
Since the localization is $\catV$-enriched, the morphism $X⊗\Yo{G}⊗g$
is a local weak equivalence if $\Yo{G}⊗g$ is one.

Recall that morphisms $g∈S_\site∪S_Γ∪S_C$ have the form $g=\Yo{c}⊠f$,
where $f$ is a morphism of presheaves (as in \cref{bousloc} and \eqref{s.theta.d}) on a small category~$\catD$, like $\site$, $Γ$, $Δ$, or $Θ_d$,
and $c$ is an object in the product of the factors not corresponding to~$f$,
as explained in \cref{localmapsnot}.
If $\catD≠Γ$, then $f$ has cofibrant domain and codomain in the injective model structure, since all objects are injectively cofibrant.
If $\catD=Γ$, then by \cref{Reedy.projective.injective.equivalent}, the map $f$ has cofibrant domain and codomain in the $Γ$-Reedy model structure.
Therefore, by \cref{external.product.quillen},
the morphism $g=\Yo{c}⊠f$ has cofibrant domain and codomain in the $Γ$-Reedy injective model structure.
Given $G∈\site⨯Γ⨯C_d$, we have $$\Yo{G}⊗(\Yo{c}⊠f)≅(\Yo{c}⊗\Yo{H})⊠(f⊗\Yo{G_f}),$$
where $G_f$ is the projection of $G$ onto the factor corresponding to~$f$ and $H$ is the projection of $G$ onto the complementary factors.
The object $\Yo{c}⊗\Yo{H}$ is cofibrant
and by \cref{external.product.quillen} the functor $(\Yo{c}⊗\Yo{H})⊠-$ is a left Quillen functor, so it preserves weak equivalences between cofibrant objects.

Therefore, it remains to show that $f⊗\Yo{G_f}$ is a weak equivalence between cofibrant objects in $\catV$-valued presheaves on~$\catD$.
Cofibrancy follows once again from the monoidality of the $Γ$-Reedy injective model structure.
The object $\Yo{G_f}$ (\cref{enriched.yoneda}) and the morphism~$f$ (\cref{bousloc}) were defined by applying objectwise the functor
$$ι_\catV:\set→\catV, \qquad S↦∐_S \munit_\catV$$
to the corresponding object or morphism in the category of presheaves of sets.

Consider the left Quillen functor $κ_\catV:\sset→\catV$
induced from the functor $Δ→\catV$ given by a fixed Reedy cofibrant cosimplicial resolution of $\munit_\catV∈\catV$.
Since $\sset$ and $\catV$ are cartesian model categories,
derived monoidal products are separately homotopy cocontinuous in each argument.
Thus, the left derived functor of $κ_\catV$ preserves derived monoidal products
as long as it preserves derived monoidal products of homotopy generators of $\sset$,
i.e., preserves the terminal object, which it does by definition.

The functor $κ_\catV$ induces a left Quillen functor $κ_{\catV,\catD}$ from simplicial presheaves on~$\catD$ to $\catV$-valued presheaves on~$\catD$,
equipped with the injective model structures.
Since $κ_\catV$ extends $ι_\catV$, by Hirschhorn \cite[Theorem 3.3.20.(1)]{Hirschhorn}
the functor $κ_{\catV,\catD}$ is also a left Quillen functor between the localized injective model structures.
The left derived functor of $κ_{\catV,\catD}$ preserves derived monoidal products for the same reason as $κ_\catV$,
once we observe that a set of homotopy generators of simplicial presheaves on~$\catD$ is given by $Δ^0⊗\Yo{d}$,
which $κ_{\catV,\catD}$ maps to $\munit_\catV⊗\Yo{d}$.

Thus, the problem reduces to demonstrating that the left Bousfield localization of the category of simplicial presheaves on $\site$, $Γ$, $Δ$, or $Θ_d$,
at the corresponding morphisms from \cref{bousloc} and \eqref{s.theta.d}, is a monoidal model category.
For $\site$ this is shown in \cref{model.structure.sheaves}
and for $Γ$ in \cref{gammasp}.
For $Δ$ and $Θ_d$, the corresponding statements were proved by Rezk in \cite[Theorem~7.2]{Rezk.CSS} and \cite[Theorem~8.1]{Rezk.Theta}.
\end{proof}

For objects $X,Y∈\smcatuple{d}$, the functors $X→Y$ should themselves naturally form an object in $\smcatuple{d}$,
and likewise for $\smcat{d}$.
By \cref{multiple.monoidal.model.structure}, the multiple $Γ$-Reedy injective model structure on $\smcatuple{d}$ is a symmetric monoidal model category,
facilitating a convenient formalization of such a construction in the uple case.

\begin{definition}
\label{def.funmonuple}
Assume \cref{site.and.enrichment} and recall the Day internal hom functor from \cref{day.convolution}.
Let $X$ and~$Y$ be objects in $\PSh(\site⨯Γ⨯Δ^{⨯d},\catV)$.
We define the \emph{uple functor object} as the derived internal hom
in the monoidal model structure of \cref{multiple.monoidal.model.structure}:
$$\Funmonuple(X,Y)≔\Hom(\crep(X),\frep(Y))\in \smcatuple{d},$$
where $\crep$ and $\frep$ denote the cofibrant and fibrant replacement functor, respectively.
\end{definition}

Including the globular condition in \cref{def.funmonuple} yields a model structure that does not satisfy the pushout product axiom;
consequently, functor objects cannot be computed by deriving the internal hom.

\begin{example}
\label{globular.nonmonoidal}
Working in complete globular 2-fold Segal spaces (i.e., simplicial presheaves on $Δ⨯Δ$),
we consider the cartesian product of the cofibrant object $\Yo{[1],[0]}$ and the acyclic cofibration with cofibrant source $\Yo{[0],[0]}→\Yo{[0],[1]}$.
The resulting map $\Yo{[1],[0]}→\Yo{[1],[1]}$ is not a weak equivalence in the globular model structure.
This is demonstrated by mapping it to the fibrant local object~$T$ given by the multinerve of the free bicategory~$T_2$
on a single 2-morphism $t:f⇒g$, where $f,g:a→b$ are 1-morphisms.
The mapping object $\map(\Yo{[1],[0]},T)$ is a discrete simplicial set with four elements, corresponding to the four 1-morphisms $f$, $g$, $\id_a$, $\id_b$ in~$T_2$.
The mapping object $\map(\Yo{[1],[1]},T)$ is a discrete simplicial set with five elements, corresponding to the five 2-morphisms in~$T_2$, including the four identity 2-morphisms.
The restriction map $\map(\Yo{[1],[1]},T)→\map(\Yo{[1],[0]},T)$ is not a weak equivalence.
\end{example}

To define functor objects for globular $d$-fold Segal spaces,
we transfer the derived internal hom in Rezk's $\Theta_d$-spaces (which form a cartesian model category)
via the Quillen equivalence between globular complete $d$-fold Segal spaces and Rezk's $\Theta_d$-spaces.
We emphasize that the resulting functor object is not computed as the left or right derived functor of some Quillen bifunctor.

Recall from Rezk \cite[Section~1.2]{Rezk.Theta} the notion of a $\Theta_d$-space and from Bergner–Rezk \cite{BergnerRezk} the relationship between $\Theta_d$-spaces and $d$-fold Segal spaces.
There is a functor $g:\Delta^{\times d}\to \Theta_d$, defined as the composition
$$g:\Delta^{\times d}\lto3{g_1} \Delta^{\times d-1}\times \Theta_1 \lto3{g_2} \Delta^{\times d-2}\times \Theta_2 \lto3{g_3} \cdots \lto3{g_{d-1}} \Delta\times \Theta_{d-1}\lto3{g_d} \Theta_d,$$
where each $g_i$ is defined by
$$g_i([m_1],\ldots,[m_{d-i+1}],c)=([m_1],\ldots,[m_{d-i}],[m_{d-i+1}](c,\ldots, c )),$$
where $c\in \Theta_{i-1}$.
If $d=1$, then $g=g_1$ is the identity functor.
If $d=2$, the functor $g=g_2$ is given by
$g([m_1],[m_2])=[m_1]([m_2],[m_2],\ldots,[m_2]),$
where $[m_2]$ is repeated $m_1$~times.
The functor $g$ is described explicitly on morphisms in Bergner–Rezk \cite{BergnerRezk}.

By left and right Kan extension, we have an adjoint triple $g^{\#}\dashv g^*\dashv g_*$:
$$\xymatrix{
\sPSh(\Delta^{\times d})\ar@<.25cm>[r]^-{g^{\#}}\ar@<-.25cm>[r]_-{g_*} & \ar[l]|-{g^*} \sPSh(\Theta_d),
}$$
where the bottom adjunction $g^*\dashv g_*$ is a Quillen equivalence with respect to the globular model structure on $d$-fold Segal spaces and the Rezk model structure on $\Theta_d$-spaces (Bergner–Rezk \cite[Corollary 7.3]{BergnerRezk}).
By Rezk \cite[Theorem 8.1]{Rezk.Theta}, the category $\sPSh(\Theta_d)$ is a cartesian model category when equipped with the Rezk model structure.
The following proposition extends this result to the setting of geometric symmetric monoidal $(∞,d)$-categories.

\begin{proposition}
\label{theta.model.equivalence}
Assume \cref{site.and.enrichment}.
The functor
$$\tilde g=\id_{\site}\times \id_{\Gamma}\times g: \site⨯Γ⨯\Delta^{\times d} → \site⨯Γ⨯\Theta_d$$
induces a Quillen equivalence
$$\xymatrix{
\PSh(\site⨯Γ⨯\Theta_d,\catV)_{\Reedy,\inj,\local}
\ar@<.125cm>[r]^-{\tilde g^*} & \ar@<.125cm>[l]^-{\tilde g_*}
\PSh(\site⨯Γ⨯\Delta^{\times d},\catV)_\glob,
}$$
between the model structures of \cref{theta.model.structure,globular.model.structure}.
The left adjoint functor $\tilde g^*$ preserves weak equivalences.
\end{proposition}

\begin{proof}
The adjunction is a composite of the following two Quillen equivalences.
\begin{itemize}
\item The Quillen adjunction
$$\xymatrix{
\PSh(\site⨯Γ⨯\Theta_d,\catV)_{\Reedy,\inj,\local} \ar@<.125cm>[r]^-{\id} & \ar@<.125cm>[l]^-{\id} \PSh(\site⨯Γ⨯\Theta_d,\catV)_{\inj,\local}\cr
}$$
between the localized $Γ$-Reedy injective model structure (\cref{theta.model.structure}) and the localized injective model structure on $\PSh(\site⨯Γ⨯\Theta_d,\catV)$.
This is the left Bousfield localization of the Quillen adjunction between the $Γ$-Reedy injective model structure and the injective model structure
(\cref{Reedy.projective.injective.equivalent}).
This adjunction is a Quillen equivalence because both model structures have the same weak equivalences.
In particular, the left adjoint preserves weak equivalences.
\item The Quillen adjunction
$$\xymatrix{
\PSh(\site⨯Γ⨯\Theta_d,\catV)_{\inj,\local} \ar@<.125cm>[r]^-{\tilde g^*} & \ar@<.125cm>[l]^-{\tilde g_*} \PSh(\site⨯Γ⨯Δ^{⨯d},\catV)_\glob,\cr
}$$
whose underlying adjunction is given by the adjunction in the statement.
This adjunction is a Quillen equivalence by adapting the argument of Bergner–Rezk \cite[Corollary 7.3]{BergnerRezk},
adding a formal factor $\site⨯Γ$ to $Θ_d$ and $Δ^{⨯d}$, respectively.
The left adjoint preserves weak equivalences by Ken Brown's lemma, since all objects are cofibrant.
\qedhere
\end{itemize}
\end{proof}

Using \cref{theta.model.equivalence}, we transfer the derived internal hom from $\Theta_d$-spaces
to $d$-fold Segal spaces via the Quillen equivalence $\tilde g^*\dashv \tilde g_*$ as follows.

\begin{definition}
\label{globular.hom}
Assume \cref{site.and.enrichment}.
Let $\Homtheta(-,-)$ denote the internal hom and $\crep$ the cofibrant replacement functor
in the monoidal model category
$$\PSh(\site⨯Γ⨯Θ_d,\catV)_{\Reedy,\inj,\local}$$
of \cref{theta.model.structure}.
Denote by~$\frep$ the fibrant replacement functor for the model category $\PSh(\site⨯Γ⨯Δ^{⨯d},\catV)_\glob$ (\cref{abstract.globular.model.structure}).
Given $Y,Z∈\PSh(\site⨯Γ⨯Δ^{⨯d},\catV)_\glob$, we define the \emph{globular product}
$$(Y,Z)\mapsto Y \tglob Z ≔ \tilde g^* (\crep \tilde g_* \frep Y ⊗ \crep \tilde g_* \frep Z)$$
and the \emph{globular functor object}
$$(Y,Z)\mapsto \Funmonglob(Y,Z) ≔ \tilde g^* \crep \Homtheta(\crep \tilde g_* \frep Y,\tilde g_* \frep Z).$$
\end{definition}

\begin{proposition}
\label{derivedhom}
Assume \cref{site.and.enrichment} and the notation of \cref{globular.hom}.
For all $X,Y,Z∈\PSh(\site⨯Γ⨯Δ^{⨯d},\catV)_\glob$,
there is a natural weak equivalence of
derived mapping objects
$$\dmap(X \tglob Y,Z)\simeq \dmap(X,\Funmonglob(Y,Z)).$$
In particular, $\tglob$ is separately homotopy cocontinuous in each argument and $\Funmonglob$ is separately homotopy continuous in each argument.
\end{proposition}

\begin{proof}
We have a chain of natural isomorphisms and weak equivalences
\begin{align*}
\dmap(X \tglob Y,Z)
&=\map(X \tglob Y,\frep Z)
=\map(\tilde g^* (\crep \tilde g_* \frep X ⊗ \crep \tilde g_* \frep Y),\frep Z)\cr
&≅\map(\crep \tilde g_* \frep X ⊗ \crep \tilde g_* \frep Y,\tilde g_* \frep Z)
≅\map(\crep \tilde g_* \frep X, \Homtheta(\crep \tilde g_* \frep Y,\tilde g_* \frep Z))\cr
&\lgets0\sim\map(\crep \tilde g_* \frep X, \crep \Homtheta(\crep \tilde g_* \frep Y,\tilde g_* \frep Z))\cr
&\lto0\sim\map(\tilde g^* \crep \tilde g_* \frep X, \tilde g^* \crep \Homtheta(\crep \tilde g_* \frep Y,\tilde g_* \frep Z))
\lgets0\sim\map(\frep X, \tilde g^* \crep \Homtheta(\crep \tilde g_* \frep Y,\tilde g_* \frep Z))\cr
&=\map(\frep X, \Funmonglob(Y,Z))
\lto0\sim
\map(X, \frep \Funmonglob(Y,Z))
=\dmap(X, \Funmonglob(Y,Z)).\cr
\end{align*}
The first weak equivalence, indicated by~$\sim$, arises because
the mapping spaces on both sides are derived and the cofibrant replacement map is a weak equivalence.
The second weak equivalence is induced by the Quillen equivalence $\tilde g^*⊣\tilde g_*$.
Specifically, applying the (left derived) Quillen equivalence $\tilde g^*$ to the arguments of a (derived) mapping object
produces a weakly equivalent mapping object.
The third weak equivalence follows from the fact that the derived counit $\tilde g^* \crep \tilde g_* \frep X → \frep X$ is a weak equivalence.
Finally, the fourth weak equivalence holds because derived mapping objects preserve weak equivalences.
\end{proof}

\subsection{Categorical operations in geometric symmetric monoidal $(\infty,d)$-categories}

The main purpose of this section is to demonstrate the explicit construction of categorical operations such as composition of morphisms and monoidal products of objects in a fixed $(\infty,d)$-category $T$ (\cref{globular.model.structure.smooth}).
Readers uninterested in such explicit constructions may skip this section.

\begin{notation}
\label{target.and.family}
In this section, we maintain the conventions of \cref{site.and.enrichment}
and assume $T$ is a fibrant object in the model category $\fraksmcat{d}$ (\cref{globular.model.structure.smooth}).
All results presented in this section also apply equally if $T$ is a fibrant object in $\smcat{d}$ (\cref{globular.model.structure}),
provided that references to isotopies are disregarded.
Without loss of generality, we also impose the fibrancy condition of \cref{empty.tuple}.
Additionally, $U$ denotes an arbitrary object of the site~$\site$.
\end{notation}

\begin{remark}
\label{nonfibrant.target}
If $T$ is a nonfibrant object, the operations constructed in this section can be performed, provided the results are interpreted within the fibrant replacement.
Indeed, by the factorization axioms of a model category, the map $T\to \mathcal{R}T$ into the fibrant replacement is an acyclic cofibration, hence a monomorphism.
This implies that data in $T$ can be regarded as residing in the fibrant replacement.
Hence, the constructions below can be applied by replacing $T$ with its fibrant replacement $\mathcal{R}T$
and composing the input map $I\to T$ (from \cref{input}) with the map $T\to \mathcal{R}T$.
\end{remark}

The following proposition will be employed in the subsequent constructions.

\begin{proposition}
\label{abstract.operation}
Recall from \cref{enriched.mapping.object} the enriched mapping object functor
$$\map(-,-): (\fraksmcat{d})^\op ⨯ \fraksmcat{d} → \smsset.$$
Suppose $T$ is a fibrant object in the model category $\fraksmcat{d}$ (\cref{globular.model.structure.smooth})
and $w:I→A$ is an acyclic cofibration in $\fraksmcat{d}$.
Then the induced morphism
$$\map(w,T):\map(A,T)→\map(I,T)$$
is an acyclic fibration between fibrant objects.
Furthermore, the fiber of $\map(w,T)$ at every point of $\map(I,T)$ is a contractible fibrant object.
\end{proposition}

\begin{proof}
All objects in $\fraksmcat{d}$ are cofibrant, the object $T$ is fibrant by assumption, the map $w$ is an acyclic cofibration by assumption,
and the functor $\map(-,-)$ is a right Quillen functor.
Therefore, the morphism $\map(w,T)$ is an acyclic fibration between fibrant objects.
In particular, $\map$ computes the derived mapping objects.
The terminal object $1$ is fibrant in $\fraksmcat{d}$.
Therefore, for every map $x:1→\map(I,T)$, the pullback~$P$ of $\map(w,T)$ along~$x$
is an acyclic fibrant object, which implies $P$ is contractible.
\end{proof}

The typical usage of \cref{abstract.operation} is as follows.
We aim to perform a categorical operation on input data, which is encoded as a morphism
\begin{equation}\label{input}\Din:I→T.\end{equation}
The output of the operation is unique up to a contractible space of choices.
Moreover, there are typically coherences that demonstrate a specific output to be the result of applying an operation to the input data.
For example, the composition of 1-morphisms is represented by a 2-simplex.
We encode the input data, output data, and auxiliary data (tracking coherences) as a morphism
\begin{equation}\label{aux}\Daux:A→T.\end{equation}
There is an inclusion map $w:I→A$, which expresses the fact that $A$ contains information about the input data.
We require this inclusion to be a weak equivalence, hence an acyclic cofibration, which implies that the fiber~$P$ in the proof of \cref{abstract.operation} is contractible.
In other words, the space of ways to extend the input data to output data is a contractible space~$P$.
Finally, we extract the desired output by restricting along a map $O→A$, which is not necessarily an acyclic cofibration.
This produces a morphism
\begin{equation}\label{output}\Dout:O\to T.\end{equation}

Overall, the procedure produces a diagram
\begin{equation}\label{operationdiag}\xymatrix{
O\ar[d] \ar@/^1pc/[rdd]^-\Dout\cr
A\ar@{-->}[dr]^<<<<{\Daux}\cr
I\ar[r]_{\Din}\ar@{>->}[u]^-{\simeq} & T.\cr
}\end{equation}
Below we give several examples of such constructions.
We start by describing possible choices of~$I$ and~$O$.

Recall that the underlying category of the model category $\fraksmcat{d}$ is given by the category of presheaves
$$\fraksmcat{d}=\PSh(\site\times \Gamma\times \Delta^{\times d},\smsset)=\PSh(\site\times \Gamma\times \Delta^{\times d}\times \cart\times \Delta),$$
where, on the right, we have applied currying to incorporate the domain of the value category of presheaves $\smsset=\PSh(\cart\times \Delta)$ into the domain of the overall presheaf category.
Thus, an object
$T\in \fraksmcat{d}$ can be evaluated on a tuple $(U,⟨ℓ⟩,{\bf m},\RR^{ℓ},[n])\in \site\times \Gamma\times \Delta^{\times d}\times \cart\times \Delta$.
Below, we single out the last two factors because they encode the space of values for our Segal objects.
For such a tuple $(U,⟨ℓ⟩,{\bf m},\RR^{ℓ},[n])$, we let
$$\Yo{(U,⟨ℓ⟩,{\bf m}),\RR^{ℓ},[n]}\in \fraksmcat{d}$$
denote the corresponding representable presheaf given by the Yoneda embedding.

\begin{definition}
\label{cats.defs}
Assume \cref{target.and.family}.
We define the following categorical structures.

\begin{itemize}
 \item A \emph{$U$-family of objects} in~$T$ is a morphism in $\fraksmcat{d}$
$$x:\catObj_{U}≔\Yo{(U,⟨1⟩,[0],…,[0]),\RR^0,Δ^0}→T,$$
\item Let $1\leq k\leq d$.
A \emph{$U$-family of $k$-morphisms} in~$T$ is a morphism
$$f_k:\catMor_{k,U}≔\Yo{(U,⟨1⟩,[1],…,[1],[0],…,[0]),\RR^0,Δ^0}→T,$$
with $k$ copies of~$[1]$.
If $k=1$, we write $\catMor_U=\catMor_{1,U}$ and call $f=f_1$ a \emph{morphism}.
\item A \emph{$U$-family of virtual isomorphisms} in~$T$ is a morphism
$$v:\catVirt_{U}≔\Yo{(U,⟨1⟩,[0],…,[0]),\RR^0,Δ^1}→T.$$
\item An \emph{isotopy of $U$-families of objects} in~$T$
from an object~$x$ to an object~$y$
is a morphism
$$i:\catIso_{U}≔\Yo{(U,⟨1⟩,[0],…,[0]),\RR^1,Δ^0}→T,$$
whose restrictions along the maps $\catObj_{U}→\catIso_{U}$, which are induced by the inclusions $\RR^0≅\{s\}⊂\RR^1$ and $\RR^0≅\{t\}⊂\RR^1$ (in the $\cart$ factor),
coincide with $x$ and $y$, respectively.
Here $s,t∈\RR$ are distinct real numbers.
\item A \emph{composed pair of $U$-families of 1-morphisms} is a map
$$c:\Delta\text{-}\catMor_{U}≔\Yo{(U,⟨1⟩,[2],[0],\cdots,[0]),\RR^0,\Delta^0}→T.$$
\item A \emph{multiplied pair of $U$-families of objects} in~$T$ is a morphism
$$m:\catMult_U≔\Yo{(U,⟨2⟩,[0],…,[0]),\RR^0,Δ^0}→T.$$
\item In general, a multiplied pair of categorical structures is obtained by replacing $⟨1⟩$ with $⟨2⟩$.
For example, a \emph{multiplied pair of $U$-families of isotopies} is a morphism
$$m:\catMultIso_U≔\Yo{(U,⟨2⟩,[0],…,[0]),\RR^1,Δ^0}→T.$$
\item
An \emph{empty tuple of $U$-families of objects} is a map
$$\catEmpty_{U}≔\Yo{(U,⟨0⟩,[0],…,[0]),\RR^0,Δ^0}→T.$$
See \cref{empty.tuple} for an explanation of empty tuples.
\item
An \emph{empty tuple of isotopies of $U$-families of objects} in~$T$
is a morphism
$$\catEmptyIso_U≔\Yo{(U,⟨0⟩,[0],…,[0]),\RR^1,Δ^0}→T.$$
\end{itemize}
\end{definition}

\begin{remark}
\label{virtual.isomorphism}
The definitions of \emph{virtual isomorphism} and \emph{isotopy} require some explanation.
Recall that in the Segal space formalism, there are (a priori) two distinct notions of isomorphism:
a 1-simplex in the space of objects, or a vertex in the space of morphisms whose image in the homotopy category is an isomorphism.
We call an isomorphism of the first type a \emph{virtual isomorphism} and an isomorphism of the second type an \emph{isomorphism}.
For Segal complete objects, these two notions of isomorphism are equivalent:
the simplicial degeneration map from the moduli space of virtual isomorphisms to the moduli space of isomorphisms is a weak equivalence.

For geometric symmetric monoidal $(\infty,d)$-categories \emph{with isotopies}, the isotopies play a similar role to virtual isomorphisms.
It is possible to combine both virtual isomorphisms and isotopies into a single simplicial object, using the Quillen equivalence between $\smsset$ and $\sset$ (\cref{smsset.injective}).
However, we find that it is useful to keep track of this data separately in practice.
Virtual isomorphisms keep track of categorical data encoded by the fields (i.e., gauge transformations),
while isotopies keep track of smooth isotopies of such data.
\end{remark}

\begin{remark}
\label{empty.tuple}
To simplify the presentation below, we impose one more condition on the $(∞,d)$-category~$T$:
evaluating $T$ at an object whose $Γ$-component is $⟨0⟩∈Γ$ yields the terminal object in~$\catV$.
This condition is already satisfied for our variant of the bordism category.
In a generic setting, it does not result in a loss of generality:
for every~$T$ we can construct a weak equivalence $T→T'$, where $T'$ satisfies the additional property.
Indeed, take $T→T'$ to be the fibrant replacement map of~$T$ in the model structure given by enlarging the class of cofibrations in \cref{globular.model.structure}
to include maps given by tensoring the representable presheaf of an object $(U,⟨0⟩,{\bf m})∈\site⨯Γ⨯Δ^{⨯d}$ with an arbitrary map in~$\catV$.
The right lifting property with respect to the latter maps guarantees that evaluating $T$ on $⟨0⟩∈Γ$ yields the terminal object.

The advantage of this additional condition can be explained as follows.
Given a pair of entities like in \cref{cats.defs},
specified as maps out of representable presheaves on objects with $Γ$-component $⟨1⟩∈Γ$,
e.g., $$x,y: \catObj_U→T,$$
we would like to multiply them by lifting the map
$$[x,y]:\catObj_U⊔\catObj_U→T$$
along the map $$λ:\Yo{⟨1⟩}⊔\Yo{⟨1⟩}→\Yo{⟨2⟩}$$
and restricting along the map $\Yo{⟨1⟩}→\Yo{⟨2⟩}$, as explained in \cref{companion.monoidal.product}.
The problem with the first step is that the map~$λ$ is not a monomorphism:
the two elements corresponding to the maps of finite pointed sets $⟨1⟩→⟨0⟩$ in the first and second summands
both map to the same element in~$\Yo{⟨2⟩}$.

The underlying meaning of this fact is that restricting along the map $\Yo{⟨0⟩}→\Yo{⟨1⟩}$ amounts to discarding the only element in a 1-tuple.
After restriction, we get a map out of $\Yo{⟨0⟩}$, which encodes an (empty) 0-tuple.
The fibrancy conditions for $Γ$-spaces force the space of 0-tuples to be contractible.
However, there is no reason why two 0-tuples arising from two different objects should be equal.
Enforcing the additional condition on~$T$ makes all 0-tuples equal.
This makes the noninjective map $$ρ:\Yo{⟨1⟩}⊔\Yo{⟨1⟩}→\Yo{⟨1⟩}⊔_{\Yo{⟨0⟩}}\Yo{⟨1⟩}$$ satisfy the unique left lifting property against the object~$T$,
so any map out of $\Yo{⟨1⟩}⊔\Yo{⟨1⟩}$ can be canonically extended along the map~$ρ$,
e.g., the map $[x,y]$ lifts uniquely to a map
$$[x,y]_*:\catObj_U⊔_{\catEmpty_U}\catObj_U→T.$$
We will systematically use the notation $[x,y]_*$ for such lifts.
\end{remark}

Given the preceding definitions, we can explicitly construct the composition of morphisms and monoidal products of objects in~$T$.
We begin with the monoidal product.
\begin{example}
\label{companion.monoidal.product}
Assume \cref{target.and.family}.
The monoidal product of a pair of $U$-families of objects
$$x, y:\catObj_U\to T$$
is computed as follows.

The input data \cref{input} of the monoidal operation is given by the map
$$[x,y]:\catObj_U⊔\catObj_U\to T,$$
which lifts uniquely (\cref{empty.tuple}) to a map
$$[x,y]_*:\catObj_U⊔_{\catEmpty_U}\catObj_U\to T.$$
Consider the acyclic cofibration $$\catObj_U⊔_{\catEmpty_U}\catObj_U\to \catMult_U$$
induced by the maps of pointed finite sets $$ι_1,ι_2:\{*,1,2\}→\{*,1\},$$ where $$\iota_1:1↦1,2↦*, \qquad \iota_2:1↦*,2↦1,$$ respectively.
We take the object encoding the auxiliary data \cref{aux} to be $A=\catMult_U$.
Since $T$ is fibrant, there is an extension to a multiplied pair $m_{x,y}:\catMult_U\to T$ that makes the diagram
$$\xymatrix{
\catMult_U\ar@{-->}[dr]^-{m_{x,y}}\cr
\catObj_U⊔_{\catEmpty_U}\catObj_U\ar[r]_-{[x,y]_*}\ar@{>->}[u]^-{\simeq} & T\cr
}$$
commute.
The map $m_{x,y}$ is indicated by the dashed arrow $\Daux$ in \cref{aux}.
The monoidal product of $x$ and $y$ is unique up to a contractible space of choices.
The extension $m_{x,y}$ provides such a choice of object representing the monoidal product of $x$ and $y$,
together with a coherence that demonstrates this object to be the product.
From the extension, the product can be extracted by restriction along the map
$$μ:\catObj_U→\catMult_U$$
induced by the map of pointed finite sets $\{*,1,2\}→\{*,1\}$, where $1,2↦1$.
We obtain an object
$$x\otimes y=m_{x,y}\mu:\catObj_U\to T,$$
which is the output of the operation \cref{output}.
This map represents the monoidal product.
\end{example}

The next example constructs the braiding \emph{virtual} isomorphism for any pair of $U$-families of objects in $T$.

\begin{example}
\label{braidingiso}
Assume \cref{target.and.family}.
To compute the braiding virtual isomorphism for a pair of $U$-families of objects
$$x,y:\catObj_U\to T,$$
we form the commutative diagram
$$\xymatrix@R=.5cm{
\catObj_U⊔_{\catEmpty_U}\catObj_U\ar[dr]^-{[x,y]_*}\ar[dd]^-{f}\cr
& T,\cr
\catObj_U⊔_{\catEmpty_U}\catObj_U\ar[ru]_{[y,x]_*}\cr
}$$
where $f$ exchanges the two summands.
We also have a commutative diagram
$$\xymatrix{
\catMult_U\ar[r]^-{\sigma} & \catMult_U\cr
\catObj_U⊔_{\catEmpty_U}\catObj_U\ar[r]_{f}\ar@{>->}[u]^-{\simeq} & \catObj_U⊔_{\catEmpty_U}\catObj_U\ar@{>->}[u]^-{\simeq},\cr
}$$
where $\sigma$ is induced by the map $\{\ast,1,2\}\to \{\ast,1,2\}, 1\mapsto 2, 2\mapsto 1$ and the vertical maps are equal, defined in \cref{companion.monoidal.product}.
Lifting both $[x,y]_*$ and $[y,x]_*$ to multiplied pairs, as in \cref{companion.monoidal.product}, we get a diagram
$$\xymatrix@R=.5cm{
\catMult_U\ar[dd]^-{\sigma}\ar[dr]^-{m_{x,y}}\cr
& T.\cr
\catMult_U\ar[ru]_{m_{y,x}}\cr
}$$
This diagram need not commute strictly.
However, since both $m_{x,y}$ and $m_{y,x}\sigma$ are extensions of $[x,y]_*$ along an acyclic cofibration,
both define vertices in the (homotopy) fiber of the map $$\map(\catMult_U,T)\to \map(\catObj,T)$$
over the point $[x,y]_*$.
The resulting fiber is a contractible Kan complex by \cref{abstract.operation}.
Therefore, there is a 1-simplex in the (homotopy) fiber: $$\tilde b:m_{x,y}\to m_{y,x}\sigma.$$
Restricting $\tilde b$ along the map $$\mu\times \id:\catObj_U\times \Delta^1\to \catMult_U\times \Delta^1$$
(with $μ$ as in \cref{companion.monoidal.product})
yields a 1-simplex
$$b=\tilde b\circ (\mu\times \id):x\otimes y\to y\otimes x,$$
which is the braiding virtual isomorphism for the monoidal product.
\end{example}

We now extend \cref{companion.monoidal.product} to other types of data in $T$, such as 1-morphisms or isotopies.

\begin{example}
\label{companion.monoidal.product.isotopies}
Assume \cref{target.and.family}.
To compute the monoidal product of a pair of isotopies, we replace $\catObj_U$ with $\catIso_U$ and $\catMult_U$ by $\catMultIso_U$ in \cref{companion.monoidal.product}.
Let $x_1,y_1,x_2,y_2:\catObj_U\to T$ be $U$-families of objects.
Fix $s,t∈\RR$, $s\neq t$.
There are maps $s,t:\catObj_U\to \catIso_U$ induced by the two inclusions $\RR^0≅\{s\}⊂\RR$ and $\RR^0≅\{t\}⊂\RR$.
The input data \cref{input} of the operation is given by a pair of isotopies
$$[i_1,i_2]:\catIso_U⊔\catIso_U\to T$$
such that their restrictions along $s$ yield $x_1$ and $x_2$, respectively,
and their restrictions along $t$ yield $y_1$ and $y_2$, respectively.
In other words, $i_1$ is an isotopy connecting $x_1$ and $y_1$ and $i_2$ is an isotopy connecting $x_2$ and $y_2$.

Consider the diagram of pushouts
$$\xymatrix{
\catMult_U\ar[r] & s\text{-}\catMultIso_U \ar@{>->}[r]^-{\simeq} & \partial\text{-}\catMultIso_U\cr
\catObj_U⊔_{\catEmpty_U}\catObj_U \ar[r]^-{s⊔s}\ar@{>->}[u]^{\simeq} & \catIso_U⊔_{\catEmptyIso_U}\catIso_U\ar@{>->}[r]^-{\simeq}\ar@{>->}[u]_{\simeq} & t\text{-}\catMultIso_U \ar@{>->}[u]_{\simeq}\cr
& \catObj_U⊔_{\catEmpty_U}\catObj_U \ar[u]_-{t⊔t}\ar@{>->}[r]_{\simeq}& \catMult_U, \ar[u]\cr
}$$
where the decorations on the arrows denote acyclic cofibrations.
An extension along the acyclic cofibration $$\catIso_U⊔_{\catEmptyIso_U}\catIso_U\to \partial\text{-}\catMultIso_U$$
encodes a pair of isotopies together with the coherence data for multiplying the source and target.
The universal property for the pushout induces an acyclic cofibration $$\partial\text{-}\catMultIso_U\to \catMultIso_U.$$
Extending along this map further adds coherence data for multiplying the isotopies themselves.
Thus, the auxiliary extension in \cref{operationdiag} can be split into a pair of extensions
$$\xymatrix{
\catMultIso_U\ar@{-->}[ddrr]^-{m_{i_1,i_2}}\cr
\partial\text{-}\catMultIso_U\ar@{-->}[drr]_-{\partial m_{i_1,i_2}}\ar@{>->}[u]^-{\simeq}\cr
\catIso_U⊔_{\catEmptyIso_U}\catIso_U\ar@{>->}[u]^-{\simeq}\ar[rr]_-{[i_1,i_2]_*} && T,\cr
}$$
where $\partial m_{i_1,i_2}$ provides the coherence data for multiplying the boundary of the pair of isotopies,
and the map $m_{i_1,i_2}$ supplies the coherence data for multiplying the pair of isotopies, compatible with the boundary coherence data.
Finally, by restricting along the map $\mu:\catIso_U→\catMultIso_U$, defined as in \cref{companion.monoidal.product} with $\RR^0$ replaced by $\RR^1$,
we obtain a product of isotopies
$$i_1\otimes i_2=m_{i_1,i_2}\mu:\catIso_U\to T,$$
which is the output data \cref{output}.
\end{example}

\begin{example}
\label{companion.monoidal.product.virtual.isomorphisms}
Assume \cref{target.and.family}.
A construction analogous to \cref{companion.monoidal.product.isotopies}
computes the monoidal product of a pair of virtual isomorphisms $(i_1:x_1→y_1,i_2:x_2→y_2)$.
Additionally, we can also prescribe a choice of the monoidal products $x_1⊗x_2$ and $y_1⊗y_2$,
by feeding these choices into the construction of the map $∂m_{i_1,i_2}$,
which we then lift to a map $m_{i_1,i_2}$.
\end{example}

We now explain how to convert an isotopy in $T$ to a 1-morphism.
In double category theory, the 1-morphism~$f$ constructed in \cref{companion.convert.isotopy.morphism} is known as the \emph{conjoint} of the isotopy~$i$.
Likewise, \cref{companion.convert.virtual.isomorphism.morphism} constructs the conjoint of a virtual isomorphism~$i$
and \cref{companion.convert.virtual.isomorphism.opposite.morphism} constructs the \emph{companion} of~$i$.

\begin{example}
\label{companion.convert.isotopy.morphism}
Assume \cref{target.and.family}.
To convert an isotopy connecting two objects into a 1-morphism,
consider the diagram of pushouts
$$\xymatrix{
\catIso_U\ar[r] & 0\text{-}\catMorIso_U \ar@{>->}[r]^-{\simeq} & \Pi\text{-}\catMorIso_U \cr
\catObj_U \ar[r]^-{0}\ar@{>->}[u]_{t}^{\simeq} & \catMor_U\ar@{>->}[r]^-{\simeq}\ar@{>->}[u]_{\simeq} & 1\text{-}\catMorIso_U \ar@{>->}[u]_{\simeq} \cr
& \catObj_U \ar[u]_-{1}\ar@{>->}[r]^-{t}_{\simeq}& \catIso_U, \ar[u]\cr
}$$
where the maps $0,1:\catObj_U\to \catMor_U$ are induced by the maps $[0]\to [1]$ that pick out the 0th and 1st vertex, respectively,
and the map $t:\catObj_U\to \catIso_U$ is induced by the inclusion $\RR^0\cong \{t\}\subset \RR^1$.
The universal property of the pushout produces an acyclic cofibration $\Pi\text{-}\catMorIso_U\to \catMorIso_U$.

Suppose we have an isotopy
$$i:\catIso_U→T,$$
such that restricting along the two maps $s,t:\catObj_U\to \catIso_U$ (as defined in \cref{companion.monoidal.product.isotopies})
yields objects $x$ and $y$, respectively.
Hence, $i$ is an isotopy connecting $x$ and $y$.
To convert this isotopy to a 1-morphism, we first add degenerate data
by precomposing with a codegeneracy map $\Pi\text{-}\catMorIso_U\to \catIso_U$, constructed as follows.
We have a codegeneracy map $c:\catMor_U\to \catIso_U$ induced by the map $[1]\to [0]$ and the inclusion $\RR^0\cong \{t\}\subset \RR^1$.
The identity map $\id:\RR^1\to \RR^1$ and the constant map $\RR^1\to \RR^1$, $x\mapsto t$
induce maps $\id:\catIso_U\to \catIso_U$ and $t:\catIso_U\to \catIso_U$, respectively.
The pair $(c,\id)$ induces a map $0\text{-}\catMorIso_U\to \catIso_U$ and the pair $(c,t)$ induces a map $1\text{-}\catMorIso_U\to \catIso_U$.
Hence, we obtain a further induced map out of the pushout
$$\Pi\text{-}\catMorIso_U\to \catIso_U.$$
Pulling back along this map is equivalent to adding the data of a degenerate 1-morphism on the codomain of the isotopy,
together with a degenerate isotopy on the codomain of the degenerate 1-morphism.
We pull back $i$ along this map to obtain a map $\tilde\imath:\Pi\text{-}\catMorIso_U\to T$, which is the input data \cref{input} for the construction.
We then extend along the acyclic cofibration:
$$\xymatrix{
\catMorIso_{U}\ar[rd]^-{ℓ_{i}} \cr
\Pi\text{-}\catMorIso_U\ar[r]_-{\tilde\imath}\ar@{>->}[u]_{\simeq} & T\cr
}$$
to obtain a map $ℓ_i$.
Finally, we restrict $ℓ_i$ along the inclusion
$\catMor_U\to \catMorIso_U$
induced by the inclusion $\RR^0≅\{s\}⊂\RR^1$,
resulting in a 1-morphism $f:\catMor_U\to T$, which is the output \cref{output}.
An illustration of this procedure is given in \cref{convert.isotopy}.

\begin{figure}[ht]
$$\vcenter{\xymatrix{
y \cr
x\ar[u]^-{i} \cr
}}\quad \to \quad \vcenter{\xymatrix{
y\ar@{=}[r] & y \cr
x\ar[u]^-{i} & y\ar@{=}[u] \cr
}}\quad \to \quad \vcenter{\xymatrix{
y\ar@{=}[r] & y \cr
x\ar[u]^-{i}\ar[r]_{f} & y\ar@{=}[u] \cr
}}\quad \to \quad \vcenter{\xymatrix{
& \cr
x\ar[r]_{f} & y \cr
}}$$
\caption{Converting an isotopy to a 1-morphism, as described in \cref{companion.convert.isotopy.morphism}.
Vertical arrows and equalities represent isotopies in $T$ and the horizontal arrows and equalities represent 1-morphisms.
Going from left to right, we first add degenerate data to the isotopy $i$ by pulling back along the map $\Pi\text{-}\catMorIso_U\to \catIso_U$,
then we lift this data to the map $ℓ_i:\catMorIso_U\to T$, which selects a cell in $T$.
Finally, we extract the 1-morphism at the source of the isotopy.}
\label{convert.isotopy}
\end{figure}
\end{example}

\begin{example}
\label{companion.convert.virtual.isomorphism.morphism}
Assume \cref{target.and.family}.
To convert a virtual isomorphism into a 1-morphism, we use the same construction as in \cref{companion.convert.isotopy.morphism}, replacing $\catIso_U$ by $\catVirt_U$.
The illustration in \cref{convert.isotopy}
continues to work if we replace vertical arrows by virtual isomorphisms.
\end{example}

\begin{example}
\label{companion.convert.virtual.isomorphism.opposite.morphism}
Assume \cref{target.and.family}.
To convert a virtual isomorphism into a 1-morphism in the opposite direction,
we use the construction of \cref{companion.convert.virtual.isomorphism.morphism},
exchanging the roles of $0$ and $1$ in the first factor of~$Δ$.

In \cref{convert.isotopy}, this corresponds to interchanging the vertical equality and the virtual isomorphism $i$
so that $i$ appears on the right side of the middle two squares.
\end{example}

\begin{example}
\label{companion.compose.morphisms}
Assume \cref{target.and.family}.
To compose two 1-morphisms, consider the pushout
$$\xymatrix{
\catObj_U\ar[r]^-{0}\ar[d]^-{1} & \catMor_U\ar[d] \cr
\catMor_U\ar[r] & \wedge\text{-}\catMor_U.\cr
}$$
To simplify notation, we set $\Delta\text{-}\catMor_U≔\Yo{(U,⟨1⟩,[2],[0],…,[0]),\RR^0,Δ^0}$.
We have an acyclic cofibration $\wedge \text{-}\catMor_U\to \Delta\text{-}\catMor_U$,
where the two summands are included using the coface maps $d^0,d^2:[1]→[2]$ (for the first factor of~$Δ$)
that miss the 0th and 2nd vertex, respectively.
A composable pair of morphisms is a map $(f,g):\wedge \text{-}\catMor_U\to T$.
We first extend a composable pair to a map
$$\xymatrix{
\Delta\text{-}\catMor_U\ar[dr]^-{ℓ_{f,g}} \cr
\wedge\text{-}\catMor_U\ar@{>->}[u]_{\simeq}\ar[r]_{(f,g)} & T \cr
}$$
and then restrict along the inclusion
$$\catMor_U→\Delta\text{-}\catMor_U$$
induced by the coface map $d^1:[1]→[2]$ that misses the 1st vertex.
This yields a composition $g\circ f:\catMor_U\to T$.
\end{example}

\section{Geometric sites and fibered geometric sites}

In this section, we introduce the notion of a \emph{geometric site} and a \emph{fibered geometric site}.
These notions allow us to generalize smooth families of bordisms to more structured types of parametrized bordisms, such as superfamilies of bordisms or holomorphic families.
This generality allows us to define the \emph{geometric} bordism category.
In the special case where parametrizing families are smooth manifolds, we obtain the \emph{smooth} bordism category.

\subsection{Geometric sites}

References on admissibility structures and geometric sites include
Joyal–Moerdijk \cite{JoyalMoerdijk}, Dubuc \cite{Dubuc.Etale}, Carchedi \cite{Carchedi}, Lurie \cite[Chapter 20]{Lurie.SAG}, Clough \cite[Section~2]{Clough}.

\begin{definition}
\label{admissibility.structure}
An \emph{admissibility structure} on a category~$\catC$
is a subcategory $\catC^\ad⊂\catC$
that contains all identity morphisms of~$\catC$ and satisfies the following properties.
\begin{conditions}[(1)]
\item\label{admissible.retracts} Admissible morphisms are closed under retracts in the category of morphisms in~$\catC$.
\item\label{admissible.composites} If morphisms $g$ and $g∘f$ are admissible, then so is~$f$.
\item\label{admissible.pullbacks} Base changes of admissible morphisms along arbitrary morphisms in~$\catC$ exist.
Furthermore, such base changes are themselves admissible morphisms.
\end{conditions}
We have the following definitions.

\label{geometric.site}
A \emph{geometric category} is a category equipped with an admissibility structure.
A \emph{geometric site} is a site equipped with an admissibility structure such that every covering family
consists of admissible morphisms.
\end{definition}

We have a similar definition in the $\smset$-enriched case, given by imposing the relevant structure on the underlying category,
obtained by evaluation at $\RR^0∈\cart$.

\begin{definition}
\label{enriched.admissibility.structure}
An \emph{$\smset$-enriched admissibility structure} on a $\smset$-enriched category~$\catC$
is a $\smset$-enriched subcategory $\catC^\ad⊂\catC$ such that for every $L∈\cart$, the inclusion $\catC^\ad_L⊂\catC_L$ (\cref{enriched.family.notation})
satisfies \cref{admissible.retracts,admissible.composites} of \cref{admissibility.structure}
and for $L=\RR^0$ also \cref{admissible.pullbacks}.

\label{enriched.geometric.site}
A \emph{$\smset$-enriched geometric category} is a $\smset$-enriched category equipped with a $\smset$-enriched admissibility structure.
A \emph{$\smset$-enriched geometric site} is a $\smset$-enriched site (\cref{def.site})
equipped with a $\smset$-enriched admissibility structure such that every covering family consists of admissible morphisms.
\end{definition}

\begin{remark}
For a $\smset$-enriched admissibility structure (\cref{enriched.admissibility.structure}),
\cref{admissible.pullbacks} of \cref{admissibility.structure} is often violated when $L≠\RR^0$.
Consider the setting of \cref{def.fraksmFEmb}:
$\fraksmFEmb_d$ is the $\smset$-enriched site of smooth families of $d$-dimensional smooth manifolds,
with mapping objects encoding smooth families of fiberwise smooth open embeddings.
Consider objects $p_1,p_2,q∈\fraksmFEmb_d$ and $L$-families of open embeddings $f_i:p_i→q$ in $\fraksmFEmb_d$.
Suppose that for different points $l∈L$ the pullbacks of $(f_1)_l$ and $(f_2)_l$ in $\smFEmb_d$ (\cref{def.smFEmb}) are not isomorphic as objects in $\smFEmb_d$.
Then there is no pullback of $f_1$ and $f_2$ in $\fraksmFEmb_d$.
\end{remark}

\begin{definition}
\label{eucl.geometric}
The sites $\man$ and $\cart$ (\cref{def.man.cart}) are turned into geometric sites (\cref{geometric.site}) by taking open embeddings as admissible morphisms.
\end{definition}

We now recall some elementary facts about Grothendieck fibrations, which will be needed in this section.
Recall that a \emph{Grothendieck fibration} $f:\catC\to \catD$ is a functor such that for all $y\in \catC$ and morphisms $b:z\to f(y)\in \catD$, there is an $f$-cartesian morphism $a:x\to y$ such that $f(a)=b$.
A choice of an $f$-cartesian lifting $a:x\to y$ for each $y$ and $b$ is called a \emph{cleavage}.
A cleavage is called a \emph{splitting cleavage} if it is compatible with compositions and identity maps.
A Grothendieck fibration admitting a splitting cleavage is called \emph{split}.
A Grothendieck fibration is called \emph{discrete} if the lifting $a$ is unique (and therefore necessarily $f$-cartesian), hence there is a canonical splitting cleavage.

A Grothendieck fibration $f$ with a cleavage corresponds to a pseudofunctor $\catD^{\op}\to \smallcat$, which sends $d\in \catD\mapsto f^{-1}(d)$.
The structure maps for the pseudofunctor are defined using $f$-cartesian liftings and depend on a choice of cleavage.
If $f$ admits a splitting cleavage, then the corresponding functor is strict.
If the fibration is discrete, the corresponding functor takes values in discrete categories.
The inverse construction is the Grothendieck construction
$$\smallint:\Fun(\catD^{\op},\smallcat)\to \smallcat_{/\catD}.$$
Given a functor $\catD^{\op}\to \smallcat$ that takes values in discrete small categories, i.e., in sets, applying the functor $\smallint$ yields a discrete fibration.
The domain of the corresponding fibration is just the category of elements of the functor.

\begin{definition}
\label{enriched.grothendieck.fibration}
Suppose $\catC$ and $\catD$ are $\smset$-enriched categories.
A \emph{$\smset$-enriched Grothendieck fibration} is a $\smset$-enriched functor $f:\catC→\catD$
such that for every $L∈\cart$ the functor $f_L:\catC_L→\catD_L$ (\cref{enriched.family.notation}) is a Grothendieck fibration
and for every morphism $L→L'$ the structure map $f_{L'}→f_L$ is a morphism of Grothendieck fibrations, i.e., preserves cartesian arrows.
If $\catC$ or $\catD$ is an ordinary category, we promote it to a $\smset$-enriched category
using the inclusion $\set→\smset$.

\label{enriched.cleavage}
A \emph{$\smset$-enriched cleavage} of~$f$ is given by a choice of cleavage for every~$f_L$ that is natural in $L∈\cart$.
Such a cleavage is \emph{splitting} if for every $L∈\cart$ the cleavage of $f_L$ is splitting.
\end{definition}

\begin{definition}
\label{admissibility.fibration}
Let $\catC$ be a category and let $\catD$ be a geometric category (\cref{geometric.site}).
A functor $f:\catC\to \catD$ is called an \emph{admissibility fibration} if the following condition holds.
\begin{itemize}
\item Given an object $y∈\catC$ and an \emph{admissible} map $b:z→f y$ there is an $f$-cartesian arrow $a:x→y$ such that $f a=b$.
\end{itemize}

\label{enriched.admissibility.fibration}
Similarly, if $\catC$ is a $\smset$-enriched category and $\catD$ is a $\smset$-enriched geometric category (\cref{enriched.geometric.site}),
a $\smset$-enriched functor $f:\catC\to \catD$ is called a \emph{$\smset$-enriched admissibility fibration}
if the functor $f_{\RR^0}:\catC_{\RR^0}→\catD_{\RR^0}$ is an admissibility fibration. 
\end{definition}

\begin{remark}
Suppose $f:\catC\to \catD$ is an admissibility fibration (\cref{admissibility.fibration}).
By \cref{admissible.composites} of \cref{admissibility.structure},
a morphism in $f^{-1}\catD^\ad$ is cartesian for the Grothendieck fibration $f^{-1}\catD^\ad→\catD^\ad$
if and only if it is cartesian for the Grothendieck fibration $f:\catC\to \catD$.
\end{remark}

\begin{remark}
For a $\smset$-enriched admissibility fibration (\cref{enriched.admissibility.fibration}) $f:\catC→\catD$,
the functor $f_L:\catC_L→\catD_L$ is often not an admissibility fibration when $L≠\RR^0$.
Consider the setting of \cref{def.frakcFEmb}:
$\frakcFEmb_d$ is the $\smset$-enriched site of holomorphic families of $d/2$-dimensional holomorphic manifolds,
with mapping objects encoding smooth families of fiberwise holomorphic embeddings.
The functor $\red:\frakcFEmb_d→\fraksmFEmb_d$ forgets the holomorphic structure.
Consider objects $p,q∈\frakcFEmb_d$ and an $L$-family of fiberwise open embeddings $f:\red p→\red q$ in $(\fraksmFEmb_d)_L$.
Suppose that for different points $l∈L$ the pullbacks of the fiberwise holomorphic structure on~$q$ along the maps $f_l$ are not biholomorphic.
Then there is no cartesian lifting of~$f$ in $(\frakcFEmb_d)_L$.
\end{remark}

\begin{proposition}
\label{lift.site}
Suppose $\catC$ is a category, $\catD$ is a geometric site (\cref{geometric.site}),
and $f:\catC→\catD$ is an admissibility fibration (\cref{admissibility.fibration}).
Then $\catC$ admits the structure of a geometric site where:
\begin{parts}[(a)]
\item\label{lift.admissible}
a morphism $a:x\to y$ in~$\catC$ is admissible if and only if it is $f$-cartesian and $fa$ is admissible in $\catD$;
\item\label{lift.covering}
a family of morphisms $\{a_i:x_i\to x\}_{i\in I}$ is a covering family if and only if $a_i$ is $f$-cartesian for all $i\in I$ and $\{fa_i:fx_i\to fx\}$ is a covering family in $\catD$.
\end{parts}

Suppose $\catC$ is a $\smset$-enriched category, $\catD$ is a $\smset$-enriched geometric site (\cref{enriched.geometric.site}),
and $f:\catC→\catD$ is a $\smset$-enriched admissibility fibration (\cref{enriched.admissibility.fibration}).
Then $\catC$ admits the structure of a $\smset$-enriched geometric site where:
\begin{parts}[(a)]
\item\label{enriched.lift.admissible}
a morphism $a:x\to y$ in~$\catC_L$ is admissible if and only if it is $f_L$-cartesian and $f_L(a)$ is admissible in $\catD_L$;
\item\label{enriched.lift.covering}
a family of morphisms $\{a_i:x_i\to x\}_{i\in I}$ in $\catC_{\RR^0}$
is a covering family if and only if $a_i$ is $f_{\RR^0}$-cartesian for all $i\in I$ and $\{fa_i:fx_i\to fx\}$ is a covering family in $\catD$.
\end{parts}
\end{proposition}

\begin{proof}
Below we give a proof for the unenriched case,
which also proves the $\smset$-enriched statement by taking $f$ to be the functor $f_L:\catC_L→\catD_L$
for every $L∈\cart$ for \cref{admissible.retracts,admissible.composites} in \cref{admissibility.structure}
and for $L=\RR^0$ for \cref{admissible.pullbacks}.

We begin by proving that \cref{lift.admissible} indeed defines an admissible structure.
Define $\catC^\ad$ to be the replete subcategory whose morphisms are $f$-cartesian arrows whose image under $f$ is admissible in $\catD$.
We must prove that $\catC^\ad$ satisfies \cref{admissible.retracts,admissible.composites,admissible.pullbacks} in \cref{admissibility.structure}.
To prove \cref{admissible.retracts}, let $a$ be a retract of a cartesian arrow $b$ such that $fb$ is admissible.
Since cartesian arrows are closed under retracts, $a$ is cartesian.
Since $fa$ is a retract of $fb$, $fa$ is admissible.
Hence, $a$ is admissible.
For \cref{admissible.composites}, let $a:x\to y$, $b:y\to z$ be morphisms in $\catC$ such that $b$ and $ba$ are cartesian and $f(b)$ and $f(ba)$ are admissible in~$\catD$.
By Johnstone \cite[Lemma B1.3.3(i)]{Johnstone}, $a$ is cartesian.
Since $fb$ and $f(ba)$ are admissible in $\catD$, so is $fa$ by \cref{admissible.composites} of \cref{admissibility.structure}.
Thus, $a$ is admissible because $a$ is $f$-cartesian and $fa$ is admissible in $\catD$.
For \cref{admissible.pullbacks}, we proceed as follows.

Suppose $a:y→x$ is an $f$-cartesian arrow in $\catC$ such that $fa$ is an admissible map in $\catD$.
Suppose $b:w→x$ is a morphism in $\catC$.
We want to show that the base change of~$a$ along~$b$ exists and is an admissible map.
First, we construct a candidate for the base change square.
Construct the pullback diagram of $fa$ and $fb$ on the right, with the resulting object denoted by~$p$ and morphisms denoted by $φ$ and~$γ$:
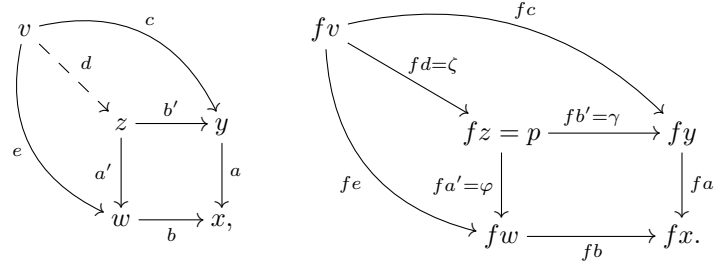
$$\xymatrix{
v \ar@/^1.5pc/[drr]^c \ar@/_1.5pc/[rdd]_e \ar@{-->}[rd]^d\cr
& z \ar[r]^{b'} \ar[d]_{a'} & y \ar[d]^a           \cr
& w \ar[r]_b                & x,\cr
}\qquad\xymatrix@C=4em{
fv \ar@/^1.8pc/[drr]^{fc} \ar@/_1.8pc/[rdd]_{f e} \ar[dr]^{f d=ζ}\cr
& fz=p \ar[r]^{fb'=γ} \ar[d]_{fa'=φ} & f y \ar[d]^{fa} \cr
& fw \ar[r]_{fb}                  & fx.\cr
}$$
Now lift the admissible map $φ:p→fw$ in $\catD$ to an $f$-cartesian arrow~$a'$ in $\catC$.
The arrows $γ$ and $ba'$ satisfy $f(ba')=(fa) γ$, so by the universal property of the $f$-cartesian arrow~$a$, there is a unique morphism $b':z→y$ such that $ab'=ba'$ and $γ=f b'$.
This gives the commutative square on the left, with both $a'$ and $a$ cartesian arrows.
It remains to prove that this square is a pullback square.

Suppose $v$ is an object and $e:v→w$ and $c:v→y$ are morphisms in $\catC$ such that $be=ac$.
We want to construct $d:v→z$ such that $a'd=e$ and $b'd=c$.
The arrows $fe$ and $fc$ satisfy $fb\circ fe = fa \circ fc$, so by the universal property there is a unique arrow $ζ:fv→fz$ making the above diagram commute.
In particular, $(fa')\zeta=fe$ so by the universal property of the $f$-cartesian arrow~$a'$, there is a unique morphism $d:v\to z$ such that $a'd=e$ and $ζ=fd$.
Since $ac=be=ba'd=ab'd$, the universal property for the cartesian arrow $a$ implies that $b'd=c$.

Now suppose $d':v→z$ is a morphism such that $a'd'=e$ and $b'd'=c$.
By the universal property of the pullback $p$, we must have $fd=fd'$.
But then we must have $d=d'$, by the universal property for the cartesian arrow $a'$.
This completes the proof of \cref{lift.admissible}.

To prove that \cref{lift.covering} defines a coverage, we argue as follows.
Suppose $x$ is an object, $b:w→x$ is a morphism, and $\{a_i:x_i→x\}_{i∈I}$ is a covering family in $\catC$.
By the above, every admissible map $a_i$ admits a base change $a'_i:w_i→w$ along~$b$.
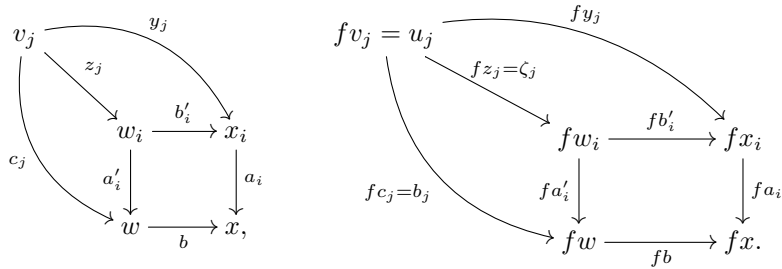
$$\xymatrix{
v_j \ar@/^1.5pc/[drr]^{y_j} \ar@/_1.5pc/[rdd]_{c_j} \ar[rd]^{z_j}\cr
& w_i \ar[r]^{b'_i} \ar[d]_{a'_i} & x_i \ar[d]^{a_i}           \cr
& w \ar[r]_b                & x,                   \cr
}\qquad\xymatrix@C=4em{
f v_j=u_j \ar@/^1.8pc/[drr]^{f y_j} \ar@/_1.8pc/[rdd]_{f c_j = b_j} \ar[dr]^{f z_j=ζ_j}\cr
& f w_i \ar[r]^{f b'_i} \ar[d]_{f a'_i} & f x_i \ar[d]^{f a_i} \cr
& f w \ar[r]_{f b}                  & f x.\cr
}$$
By definition of a coverage, the family $\{f a'_i:fw_i\to fw\}_{i∈I}$ can be refined by a covering family $\{b_j: u_j→f w\}_{j∈J}$,
which comes equipped with a map of sets $h:J→I$ and maps (necessarily admissible) $\{ζ_j:u_j→f w_{h(j)}\}_{j∈J}$
such that for every $j∈J$ we have $(f a'_{h(j)}) ζ_j = b_j$,
as depicted in the diagram above, assuming $i=h(j)$.
Lift the admissible maps $b_j$ to $f$-cartesian arrows $c_j:v_j\to w$, which form a covering family in $\catC$.
We have $(f a'_{h(j)}) ζ_j = f c_j$, so by the universal property of the $f$-cartesian arrow~$a'_{h(j)}$,
there is a unique morphism $z_j:v_j→w_{h(j)}$ such that $a'_{h(j)} z_j=c_j$ and $f z_j=ζ_j$.
Set $y_j=b'_{h(j)} z_j$, where $b'_i:w_i\to x_i$ is the projection map out of the pullback.
The covering family $\{c_j:v_j→w\}_{j∈J}$ together with the function $h:J→I$ and maps $\{y_j:v_j→x_{h(j)}\}_{j∈J}$
satisfies for every $j∈J$ the relation $b c_j=a_{h(j)}y_j$, which verifies the defining property of coverages.
\end{proof}

\begin{remark}
\label{def.open.embedding}
If $f:\catC→\catD$ is an admissibility fibration (\cref{admissibility.fibration}),
we refer to admissible maps in $\catC$ (i.e., $f$-cartesian arrows with an admissible $f$-image) using the same term as admissible maps in $\catD$.
The same convention is used for \emph{admissible subobjects}, i.e., subobjects defined by an admissible map.
For example, if $\catD=\cart$ or $\catD=\man$, this yields a notion of an \emph{open embedding} and \emph{open subobject} in $\catC$.
If $f$ is a $\smset$-enriched admissibility fibration (\cref{enriched.admissibility.fibration}) and $L∈\cart$,
we refer to admissible maps in $\catC_L$ as \emph{$L$-families of admissible maps} in~$\catC$.
For example, if $\catD=\cart$ or $\catD=\man$, this yields the notion of an \emph{$L$-family of open embeddings} in $\catC$.

If $r$ is an admissible map in $\catC$ and $f(r)$ is a monomorphism in $\catD$, then $r$ is a monomorphism in $\catC$.
Indeed, if $h_1$ and $h_2$ are parallel morphisms in~$\catC$ such that $r h_1=r h_2$,
then $f(r)f(h_1)=f(r)f(h_2)$, hence $f(h_1)=f(h_2)$ because $f(r)$ is a monomorphism.
By the universal property of the $f$-cartesian arrow~$r$, we obtain $h_1=h_2$, so $r$ is a monomorphism.

Likewise, if $f:\catC→\catD$ is a $\smset$-enriched admissibility fibration (\cref{enriched.admissibility.fibration})
and $r$ is an $L$-family of admissible maps in~$\catC$ such that $f(r)$ is a monomorphism in $\catD_L$, then $r$ is a monomorphism in~$\catC_L$.
\end{remark}

\begin{proposition}
\label{admissible.subobject}
If $f:\catC→\catD$ is an admissibility fibration (\cref{admissibility.fibration}) and every admissible map in $\catD$ is a monomorphism, then $f$ establishes
a bijection from the poset of admissible subobjects (\cref{def.open.embedding}) of $c∈\catC$
to the poset of admissible subobjects of $d=f(c)∈\catD$.
\end{proposition}

\begin{proof}
If $r:c_1→c$ is an admissible monomorphism in $\catC$, then $f(r)$ is an admissible monomorphism in~$\catD$.
Indeed, $f(r)$ is admissible by \cref{admissibility.fibration}, therefore $f(r)$ is a monomorphism by assumption.
Conversely, if $s:d_1→d$ is an admissible monomorphism in $\catD$,
then every cartesian lift of~$s$ is an admissible map $r:c_1→c$,
which is a monomorphism by \cref{def.open.embedding}.
If $s_1:d_1→d$ and $s_2:d_2→d$ are admissible monomorphisms in $\catD$
such that $s_1$ factors through $s_2$
and $r_1:c_1→c$ and $r_2:c_2→c$ are $f$-cartesian lifts of $s_1$ and $s_2$, respectively,
then the universal property of the $f$-cartesian arrow~$r_2$
shows that $r_1$ factors through~$r_2$.
Thus, these constructions are order-preserving and are mutually inverse to each other.
\end{proof}

\begin{definition}
\label{def.stcart}
\label{structured.cartesian.site}
Recall the geometric site $\cart$ from \cref{def.cart}.
A \emph{structured cartesian site} is a pair $(\stcart,\red)$,
where $\stcart$ is a small category and
$$\red: \stcart→\cart$$
is an admissibility fibration (\cref{admissibility.fibration}).
We equip $\stcart$ with the structure of a geometric site given by \cref{lift.site}.
We refer to the functor~$\red$ as the \emph{reduction functor}
and to admissible morphisms in~$\stcart$ as \emph{open embeddings} (\cref{def.open.embedding}).
\end{definition}

\begin{remark}
By \cref{def.stcart} and \cref{admissible.subobject}, the reduction functor~$\red$ establishes a bijective order-preserving correspondence
between open subobjects of $X∈\stcart$ and open subobjects of $\red X∈\cart$.
\end{remark}

\begin{example}
The site $\cart$ (\cref{def.cart}) with the identity reduction functor $\red=\id:\cart→\cart$ is a structured cartesian site (\cref{structured.cartesian.site}).
\end{example}

\begin{example}
\label{trivial.geometric.site}
The inclusion $\trivialsite→\cart$ of the full subcategory~$\trivialsite$ of~$\cart$ on terminal objects (i.e., cartesian spaces diffeomorphic to $\RR^0$) into $\cart$
is a structured cartesian site (\cref{def.stcart}).
Constructions with this site yield bordism categories without geometric families.
\end{example}

Many more examples can be constructed for manifolds with locally defined structures using \cref{lift.site}.

\begin{example}
\label{def.holomorphic.cart}
The category $\cman$ of holomorphic manifolds and holomorphic maps, together with the forgetful functor $\red:\cman→\man$ (\cref{def.man})
satisfies the conditions of \cref{lift.site} and therefore gives rise to a geometric site $\cman$.
Indeed, given an object $U\in \ccart$ and an open embedding $i:V\to \red(U)$,
a cartesian lift $i^*U\to U$ is given by pulling back the complex structure on $U$ to $V$ along~$i$.

The full subcategory on objects whose underlying smooth manifold is an object of $\cart$ (\cref{def.cart})
yields a structured cartesian site (\cref{structured.cartesian.site}) $\ccart$ together with the reduction functor $\red:\ccart→\cart$.
(Such holomorphic manifolds need not be biholomorphic with $\CC^n$, as witnessed by the unit disk in~$\CC$.)
\end{example}

We now define the geometric site of smooth supermanifolds.
There are two variants of smooth supermanifolds: real and complex.
There is a complexification functor from real to complex smooth supermanifolds,
which is faithful, but not full or essentially surjective.
It restricts to an equivalence on the full subcategories of reduced supermanifolds,
both of which are equivalent to the category of ordinary smooth manifolds.
For details, see Deligne–Morgan \cite[Section 4.8]{DeligneMorgan} and the $n$Lab \cite{nLab.csmanifold}.
Below, we define both real and complex smooth supermanifolds, as these variants are needed in \cref{1.1.Euclidean,def.super.euclidean}.

\begin{definition}
\label{def.sman}
Denote by $k$ the field $\RR$ or $\CC$.
The geometric site $\sman_k$ (\cref{geometric.site}) is defined as follows.
Objects are real/complex smooth supermanifolds~$M$, defined as pairs $(M_0,M_1)$,
where $M_0$ is a smooth manifold and $M_1$ is a sheaf on $M_0$ valued in $\CSAlg_k$,
the category of commutative $k$-superalgebras, i.e., $\ZZ/2$-graded commutative $k$-algebras.
We require that $M_0$ admits an open cover $\{U_i\}_{i∈I}$ such that for every~$i$
the restriction $M_1|_{U_i}$ is isomorphic to $U_i⨯k^{0|q}$,
i.e., the sheaf $V↦\sm(V,k)⊗_k \Lambda_k k^q$, for some $q≥0$.
Here $V⊂U_i$ is an open subset
and $\Lambda_k$ denotes the $\ZZ/2$-graded exterior $k$-algebra of a $k$-module.
Morphisms $(M_0,M_1)→(N_0,N_1)$ are pairs $(f_0,f_1)$, where $f_0:M_0→N_0$ is a smooth map
and $f_1:N_1→M_1$ is an $f_0$-map of sheaves (Stacks Project \cite[Tag 008J]{Stacks}).

The reduction functor
$$\red: \sman_k→\man, \qquad (M_0,M_1)↦M_0$$
is turned into an admissibility fibration (\cref{admissibility.fibration}) as follows.
Suppose $(N_0,N_1)∈\sman_k$.
The $\red$-cartesian lift of an open embedding $ι:M_0→N_0$ is the map $$\red^*ι:(M_0,ι^*N_1)→(N_0,N_1),$$
where $ι^*N_1$ is the restriction of $N_1$ to~$M_0$ along the map~$ι$.
We equip $\sman_k$ with the structure of a geometric site given by \cref{lift.site}.
\end{definition}

Following \cref{inclusion.manifold},
the following definition is designed so that for a fixed smooth real/complex supermanifold~$N$,
the class of smooth supermanifolds~$M$ that admit a closed inclusion into~$N$ is a set.
Furthermore, the full subcategory of the slice category $\sman_k/N$ on included submanifolds is a poset,
in particular, every isomorphism is identity
and the map $f:M→N$ can be recovered from $M$ and $N$.

\begin{definition}
\label{inclusion.supermanifold}
A \emph{closed inclusion of supermanifolds} is a morphism $f=(f_0,f_1):(M_0,M_1)→(N_0,N_1)$
such that $f_0:M_0→N_0$ is a closed inclusion of manifolds (\cref{inclusion.manifold})
and $f_1:N_1→M_1$ is a quotient $f_0$-map of sheaves,
meaning for every degree~$k$, open subset $U⊂M_0$, and $x∈M_1(U)_k$,
the element $x$ is equal to the set of sections $y∈N_1(V)_k$ such that $f_0^{-1}V⊃U$ and $f_1(y)|_U=x$.
\end{definition}

\begin{definition}
\label{def.supercart}
The structured cartesian site $\scart_k$ (\cref{structured.cartesian.site})
is defined as the full subcategory of $\sman_k$ (\cref{def.sman}) (with the induced coverage and reduction functor)
on objects~$S$ such that $S$ is isomorphic to $\RR^p⨯k^{0|q}$ for some $p,q≥0$
and $S$ admits a closed inclusion of supermanifolds (\cref{inclusion.supermanifold}) $S→\RR^m⨯k^{0|n}$ for some (necessarily unique) $m,n≥0$.
\end{definition}

\begin{definition}
\label{def.infinitesimal}
The Cahiers site~$\cahiers$ (Dubuc \cite{Dubuc.Cahiers}; see also Moerdijk–Reyes \cite[Appendix~2]{MoerdijkReyes}) is defined as follows.
Objects are pairs $(M,W)$,
where $M∈\cart$ (\cref{def.cart})
and $W$ is a \emph{Weil algebra} (Weil \cite[Section~1]{Weil}, Dubuc \cite[Définition~1.4]{Dubuc.Cahiers}, Kolář–Michor–Slovák \cite[Definition~35.2]{KMS}):
a finite-dimensional commutative real algebra~$W$ that admits a unique homomorphism~$ε$ of real algebras $W→\RR$.
Morphisms are pairs $$(f,g):(M_0,W_0)→(M_1,W_1),\qquad f:M_0→M_1, \quad g:\sm(M_1)⊗W_1→\sm(M_0)⊗W_0,$$
where $f$ is a morphism in $\cart$ and $g$ is a homomorphism of real algebras whose reduction is $$f^*=\sm(f):\sm(M_1)→\sm(M_0).$$
We have a reduction functor $$\red:\cahiers→\cart, \qquad (M,W)↦M, \quad (f,g)↦f.$$
The $\red$-cartesian lift of an open embedding $f:M_0→M_1$ is constructed by taking $W_0=W_1$ and $g=f^*⊗\id_{W_0}$.
The pair $(\cahiers,\red)$ satisfies the conditions of \cref{lift.site} and therefore produces a site $\cahiers$ satisfying \cref{def.stcart}.
\end{definition}

\subsection{Fibered geometric sites}

In this section, we define the notion of a fibered geometric site.
Roughly speaking, a fibered geometric site encodes the data of a structured cartesian family of structured manifolds.
For example, in the case where the geometric structure is complex, we get a holomorphic family of holomorphic manifolds.
Looking ahead, one can regard such a family as providing the data of a holomorphic family of bordisms, equipped with a holomorphic structure.

\begin{definition}
\label{def.smFEmb}
Let $d≥0$.
Recall the geometric sites $\cart$ and $\man$ from \cref{def.man.cart,eucl.geometric}.
We define a category $\smFEmb_d$ and a functor
$$\base:\smFEmb_d→\cart$$
as follows.
\begin{itemize}
\item
The objects of $\smFEmb_d$ are morphisms $p:M\to U$ in $\man$
such that $U\in\cart$ and the map~$p$ is a submersion with $d$-dimensional fibers.
We require that $p$ is the restriction of a projection $\RR^m⨯U→U$, for some $m≥0$,
to a closed included manifold (\cref{inclusion.manifold}) $M⊂\RR^m⨯U$.
\item
Morphisms $(p:M\to U)\to (q:N\to V)$ are defined as pairs of morphisms $(f:M\to N,g:U\to V)$ in $\man$ such that the diagram
$$\xymatrix{
M\ar[r]^-{f}\ar[d]_-{p} & N\ar[d]^-{q}\\
U\ar[r]^-{g} & V\\
}$$
commutes and such that $f$ is a fiberwise open embedding over~$g$, i.e., the induced map $M\to g^*N$ is an open embedding.
\end{itemize}
The \emph{base space functor} $$\base:\smFEmb_d→\cart$$
sends $p:M→U$ to~$U$.
\end{definition}

\begin{remark}
\label{flatfibration}
\cref{def.smFEmb} requires that the map~$p$ is the restriction of a projection of the form $\RR^m\times U\to U$.
Continuing \cref{man.small}, this singles out a small subcategory of the large category of submersions.
This requirement is not essential, but it implies that the functor $\base:\smFEmb_d\to \cart$ is a \emph{split} Grothendieck fibration.
This allows us to make the bordism category strictly functorial with respect to maps $U\to V$ in $\stcart$.
We could drop this requirement, but then the bordism category would only be a pseudofunctor with respect to these maps.
\end{remark}

As indicated in \cref{flatfibration}, we will show that the base space functor $\base:\smFEmb_d\to \cart$ is a split Grothendieck fibration.

\begin{proposition}
\label{smoothgrothfib}
The category $\smFEmb_d$ (\cref{def.smFEmb}) admits the structure of a geometric site (\cref{geometric.site}) defined as follows.
\begin{conditions}[(1)]
\item
An \emph{admissible map} in $\smFEmb_d$ is a morphism $(f,g)$ such that $g$ is an open embedding.
This forces $f$ to be an open embedding.
We refer to such maps $(f,g)$ as \emph{open embeddings} in $\smFEmb_d$.
\item\label{covering.family.smFEmb}
A \emph{covering family} in $\smFEmb_d$ is a family of open embeddings
$$\left\{\vcenter{\xymatrix{
M_{\alpha}\ar[r]^-{i_{\alpha}}\ar[d] & M\ar[d]\\
U_{\alpha}\ar[r]_{j_{\alpha}} & U\\
}}\right\}$$
such that the family $\{i_α\}$ is a covering family of~$M$ in $\man$.
We do not require that $\{j_α\}$ is a covering family of~$U$ in $\cart$.
\end{conditions}
\label{base.functor.split.smFEmb}
Moreover, the functor $\base:\smFEmb_d\to \cart$ is a split Grothendieck fibration.
\end{proposition}

\begin{proof}
The properties of an admissibility structure follow from the facts that open embeddings form an admissibility structure on $\cart$,
and retracts, compositions, and base changes in $\smFEmb_d$ are computed on the level of domains and codomains.

To see that \cref{covering.family.smFEmb} defines a coverage,
let $$(f,g):(p:M\to U)\to (q:N\to V)$$ be a morphism in $\smFEmb_d$ and let $\{(i_{\alpha},j_{\alpha}):q_{\alpha}\to q\}$ be a covering family on~$q$.
For any element $(i_{\alpha},j_{\alpha})$, we can form the pullback of $(i_{\alpha},j_{\alpha})$ along $(f,g):p\to q$:
$$\xymatrix{
M_{\alpha}≔M\times_{f,N,i_{\alpha}}N_{\alpha}\ar[d]^-{p_{\alpha}}\ar[r] & M\ar[d]^-{p}\cr
U_{\alpha}≔U\times_{g,V,j_{\alpha}}V_{\alpha}\ar[r] & U.\cr
}$$
Since submersions are stable under pullback, the map $p_{\alpha}:M_{\alpha}\to U_{\alpha}$ is a submersion with $d$-dimensional fibers.

Refine the open cover $\{U_α\}_{α∈A}$ of $⋃_{α∈A}U_α$ by a good open cover $\{U^\beta\}_{β∈B}$
with an indexing map $h:B→A$ and let $M^\beta$ be the preimage of $U^\beta$ under the submersion $p_{h(β)}$.
Then $M^\beta\to U^\beta$ is a submersion with $d$-dimensional fibers and $U^\beta\in \cart$, for each~$\beta$.
Since $\{i_{\alpha}:N_{\alpha}\to N\}$ is a cover of $N$, it follows that $\{M_α→M\}_{α∈A}$ is a covering family of~$M$
and we have a covering family
$$\left\{\vcenter{\xymatrix{
M^\beta\ar[r]\ar[d]^-{p^\beta} & M\ar[d]^-{p}\cr
U^\beta\ar[r] & U\cr
}}\right\}.$$
For every~$β$, the map $p^\beta\to p \to q$ factors through $$(i_{h(β)},j_{h(β)}):q_{h(β)}\to q$$ by construction.
This verifies the defining property of a coverage.

The splitting of the Grothendieck fibration~$\base$ is defined as follows.
Suppose $f:U→V$ is a morphism in $\cart$ and $q:N→V$ is an object in $\smFEmb_d$.
Recall that the map~$q$ is the restriction of the projection $Q:\RR^m⨯V→V$ for some $m≥0$ to the included submanifold $N⊂\RR^m⨯V$.
Consider the following pair of cartesian squares in $\set$:
$$\xymatrix@C=4em{
\RR^m⨯U \ar[r]^{\RR^m⨯f} \ar[d]_-{p} & \RR^m⨯V \ar[d]^-Q\cr
U \ar[r]_f & V,\cr
}\qquad\xymatrix@C=4em{
f^*N \ar[r]^{q^*f} \ar[d]_{f^*q} & N \ar[d]^q\cr
U \ar[r]_f & V,\cr
}$$
where the objects of the right square are included in the corresponding objects of the left square, and $p$ is the projection.
Endow $f^*N$ with the unique smooth structure making the right square cartesian in $\man$.
The morphism $(q^*f,f):f^*q→q$ in $\smFEmb_d$, obtained from the resulting right square,
is the cartesian arrow required by the splitting of~$\base$ for the pair $(f,q)$.
\end{proof}

Next we generalize the site $\smFEmb_d$ to allow for more types of geometric structures, as in the case of geometric sites.
We make the following definition.

\begin{definition}
\label{fibered.geometric.site}
Fix $d≥0$ and recall the site $\smFEmb_d$ from \cref{def.smFEmb}.
A \emph{fibered geometric site} is a triple $(\base, \red_t,\red_b)$,
where:
\begin{conditions}[(1)]
\item\label{base.split.fibration} $\base: \FEmb_d→\stcart$ is a split Grothendieck fibration;
\item\label{stcart.structured.cartesian} $(\stcart,\red_b)$ is a structured cartesian site (\cref{def.stcart});
\item\label{red.admissibility} $\red_t:\FEmb_d\to \smFEmb_d$ is an admissibility fibration (\cref{admissibility.fibration});
\end{conditions}
such that the following conditions hold.
\begin{conditions}[resume*]
\item
\label{fibered.geometric.site.diagram}
\label{reduction.preserves.splitting}
The pair $(\red_t,\red_b)$ is a morphism of split Grothendieck fibrations, meaning the functor $\red_t$ preserves the splitting cleavage of~$\base$
and the diagram
$$\xymatrix{
\FEmb_d \ar[r]^{\red_t} \ar[d]_\base & \smFEmb_d \ar[d]^\base\cr
\stcart \ar[r]_{\red_b} & \cart
}$$
commutes (strictly), where $\base: \smFEmb_d\to \cart$ is the split Grothendieck fibration of \cref{smoothgrothfib}.
\item
\label{base.creates.admissible}
The functor $\base:\FEmb_d→\stcart$ preserves and reflects admissible arrows.
\end{conditions}
We equip $\FEmb_d$ with the structure of a geometric site via the admissibility fibration $\red_t$, using \cref{lift.site}.

When the context is clear, we will omit the subscripts $t$ and $b$ on the functor $\red$.
We refer to the functor~$\red$ as the \emph{reduction functor}
and to admissible morphisms in~$\FEmb_d$ as \emph{open embeddings}.
Thus, an open embedding in $\FEmb_d$ is an $\red$-lift of an open embedding in $\smFEmb_d$.

We write $q:N→V$ to indicate that $q$ is an object in the category~$\FEmb_d$ and $\base q=V$.
The object $q$ gives rise to a map of manifolds $\red q:\red N→\red V$,
even though $N$ on its own is a purely syntactic construction that does not denote an object in any category.
\end{definition}

\begin{proposition}
\label{factor.red.cartesian.base.cartesian}
Let $(\base,\red_t,\red_b)$ be a fibered geometric site (\cref{fibered.geometric.site}).
Every morphism $f:p\to q\in \FEmb_d$ uniquely factors as a $\red_t$-cartesian arrow $g:p\to p'$ such that $\base g=\id$ (hence an open embedding),
followed by a split $\base$-cartesian arrow $h:p'\to q$.
\end{proposition}

\begin{proof}
Let $h:p'\to q$ be the $\base$-cartesian lift of $\base f$ for $q$.
By the universal property of~$h$, there is a unique morphism $g:p\to p'$ such that $hg=f$ and $\base g=\id$.
But since $\base g=\id$, \cref{base.creates.admissible} of \cref{fibered.geometric.site} implies that $g$ is $\red_t$-cartesian.
The uniqueness follows from the universal property of the $\base$-cartesian lift $h$.
\end{proof}

\begin{proposition}
\label{reflect.admissible}
Suppose all conditions of \cref{fibered.geometric.site} are satisfied except \cref{base.creates.admissible}.
Assume that the functor $\base$ preserves admissible arrows
and every arrow $f$ in $\FEmb_d$ such that $\base f=\id$ is $\red_t$-cartesian.
Then $\base$ reflects admissible arrows, hence \cref{fibered.geometric.site} is satisfied.
\end{proposition}

\begin{proof}
Suppose $f$ is a morphism in $\FEmb_d$ such that $\base f$ is admissible in $\stcart$.
We have to show that $f$ is admissible in $\FEmb_d$, i.e., $\red_t f$ is admissible in $\smFEmb_d$ and $f$ is $\red_t$-cartesian.
Observe that $\red_t f$ is admissible in $\smFEmb_d$ if and only if $\base\red_t f$ is admissible in $\cart$.
We have $\base\red_t f=\red_b\base f$, and the latter morphism is admissible because $\base f$ is admissible.

To show that $f$ is $\red_t$-cartesian,
use the proof of \cref{factor.red.cartesian.base.cartesian}
to factor $f$ as $ζ∘φ$, where $φ:p→q$ is a split $\base$-cartesian arrow and $ζ$ is a morphism in $\FEmb_d$ such that $\base ζ=\id$.
The morphism $ζ$ is $\red_t$-cartesian because $\base ζ=\id$.

It remains to show that $φ$ is $\red_t$-cartesian.
Consider the corresponding universal problem: given $ψ:r→q$ and $χ_{\red}:\red_t r→\red_t p$
such that $\red_t φ ∘ χ_{\red}=\red_t ψ$,
we have to construct $χ:r→p$ such that $φ∘χ=ψ$ and $\red_t χ=χ_{\red}$.
Construct the map $χ_{\base}:\base r→\base p$
using the universal property of the $\red_t$-cartesian arrow $\base φ:\base p→\base q$
applied to the maps $\base ψ:\base r→\base q$ and $\base χ_{\red}:\red_b\base r→\red_b\base p$ (using $\red_b\base=\base\red_t$).
Take $χ:r→p$ to be the map induced by the universal property of the $\base$-cartesian arrow $φ$
applied to the maps $ψ:r→q$ and $χ_{\base}:\base r→\base p$.

$$\def\z{\hskip-1em }\xymatrix{
r \ar[dr]_\chi \ar@/^1pc/[drr]^\psi &&&\z\base r \ar[dr]_{\chi_{\base}} \ar@/^1pc/[drr]^{\base\psi} &&&\z\red_t r \ar[dr]_{\chi_{\red}} \ar@/^1pc/[drr]^{\red_t\psi} &&&\z\red_b\base r \ar[dr]_{\base\chi_{\red}} \ar@/^1pc/[drr]^{\red_b\base\psi} \cr
               & p \ar[r]_\varphi & q &     & \base p \ar[r]_{\base\varphi}                 & \base q &                   & \red_t p \ar[r]_{\red_t\varphi} & \red_t q &                  & \red_b\base p \ar[r]_{\red_b\base\varphi} & \red_b\base q\cr
}$$

By construction, $φ∘χ=ψ$.
In particular, $\red_t φ∘\red_t χ=\red_t ψ$.
By assumption, $\red_t φ ∘ χ_{\red}=\red_t ψ$.
Hence, by the universal property of the $\base$-cartesian arrow $\red_t φ$, we have $χ_{\red}=\red_t χ$, as desired.
\end{proof}

\begin{remark}
\label{admissible.is.monomorphism}
Every open embedding in $\FEmb_d$ is a monomorphism
by \cref{def.open.embedding} and the fact that open embeddings in $\smFEmb_d$ are monomorphisms.
\end{remark}

\begin{example}
\label{standard.fibered.geometric.site}
The triple $(\smFEmb_d,\id,\base)$, where $\smFEmb_d$ and $\base$ are as in \cref{def.smFEmb} is a fibered geometric site (\cref{fibered.geometric.site}).
\end{example}

\begin{example}
\label{trivial.fibered.geometric.site}
Working with the structured cartesian site $\red:\trivialsite→\cart$ from \cref{trivial.geometric.site},
the triple $(\smEmb_d,ι,\base)$ is a fibered geometric site (\cref{fibered.geometric.site}).
Here $\smEmb_d$ is the full subcategory of $\smFEmb_d$ on objects $p:T→U$ with $U∈\trivialsite⊂\cart$,
$ι:\smEmb_d→\smFEmb_d$ is the inclusion functor,
and $\base:\smEmb_d→\trivialsite$ is the restriction of $\base$ from \cref{def.smFEmb}.
Constructions with this fibered geometric site yield bordism categories without geometric families.
\end{example}

\begin{example}
\label{def.fibered.holomorphic}
Recall the holomorphic geometric site $\ccart$ of \cref{def.holomorphic.cart} and the reduction functor $$\red_b:\ccart\to \cart,$$ which forgets the complex structure
and is an admissibility fibration (\cref{admissibility.fibration}).

Let $\cFEmb_d$ be the following category.
Objects are holomorphic submersions $p:M\to U$ with fibers of complex dimension $d$ such that the underlying smooth map $p:M\to U$ is an object in $\smFEmb_{2d}$.
Morphisms are commutative squares of holomorphic maps
$$\xymatrix{
M\ar[r]^-{f}\ar[d]_-{p} & N\ar[d]^-{q}\\
U\ar[r]^-{g} & V\\
}$$
whose underlying square of smooth maps is a morphism in $\smFEmb_{2d}$.
Let $\red_t:\cFEmb_d\to \smFEmb_{2d}$ be the functor that forgets the complex structure.
This functor is an admissibility fibration, with cartesian lifts given by pulling back the complex structure along an open embedding.
The base functor $\base:\cFEmb_d\to \ccart$ sends $(p:M\to U)\mapsto U$.
This functor is a split Grothendieck fibration, since
the construction of \cref{base.functor.split.smFEmb}
provides a splitting for $\smFEmb_{2d}→\cart$,
which we use to construct the underlying smooth manifold of the lift,
and then equip it with the pullback of the fiberwise complex structure.

By construction, the pair $(\red_t,\red_b)$ is a morphism of split Grothendieck fibrations
and every $\red_t$-lift of an admissible arrow in $\smFEmb_{2d}$ is $\red_t$-cartesian.
By \cref{reflect.admissible}, the functor $\base$ reflects admissible arrows,
which proves that $(\base,\red_t,\red_b)$ is a fibered geometric site (\cref{fibered.geometric.site}).
\end{example}

\begin{example}
\label{def.fibered.supercart}
Recall the geometric site $\scart_k$ of supercartesian spaces and the reduction functor $$\red:\scart_k→\cart$$ (\cref{def.supercart}).
Fix $d≥0$ and $e≥0$ and consider the following category $\SFEmb_{d|e}$.
An object is a morphism $p:M→U$ in $\sman_k$ (\cref{def.sman})
that is a submersion (meaning the $\ZZ/2$-graded tangent map is surjective in both degrees)
with fibers of super dimension $d|e$ (meaning the superdimension of the kernel of the tangent map is $d|e$)
such that $U∈\scart_k$.
We require that $M$ admits a closed inclusion (\cref{inclusion.supermanifold})
$e:M→\RR^m⨯k^{0|n}⨯U$ of smooth supermanifolds (\cref{def.sman}) for some (necessarily unique if $M≠∅$) $m,n≥0$,
and the inclusion must commute with the projection to~$U$.
Morphisms in $\SFEmb_{d|e}$ are commutative squares in $\sman_k$
$$\xymatrix{
M\ar[r]^-{f}\ar[d]_-{p} & N\ar[d]^-{q}\\
U\ar[r]^-{g} & V\\
}$$
such that
the map $M→g^*N$ induced by~$f$ is an open embedding in $\sman_k$ (\cref{def.supercart}).
We do not require any compatibility with the map~$e$.

Let $\red_t:\SFEmb_{d|e}\to \smFEmb_d$ be the functor that applies $\red$ to the submersion.
The functor~$\red_t$ is an admissibility fibration,
with cartesian lifts constructed separately for the source and target of the submersion,
and restricting the inclusion~$e$ along the resulting map of total spaces, which again yields a closed inclusion of supermanifolds.

The base functor $\base:\SFEmb_{d|e}\to \scart_k$ sends $(p:M\to U)\mapsto U$.
This functor is a split Grothendieck fibration, with the splitting constructed as follows (compare the proof of \cref{smoothgrothfib}).
Given a morphism $g:U→V$ in $\scart$ and an object $(q:N→V)∈\SFEmb_{d|e}$,
we take the lifting to be $p:∅→U$ if $N=∅$.
Otherwise, there is a unique closed inclusion $h:N→\RR^m⨯k^{0|n}⨯U$ such that the map~$q$ is the restriction of the projection map $Q:\RR^m⨯k^{0|n}⨯U→U$.
Consider the following pair of cartesian squares in $\sman$:
$$\xymatrix@C=6em{
\RR^m⨯k^{0|n}⨯U \ar[r]^{\RR^m⨯k^{0|n}⨯f} \ar[d]_-{p} & \RR^m⨯k^{0|n}⨯V \ar[d]^-Q\cr
U \ar[r]_f & V,\cr
}\qquad\xymatrix@C=4em{
f^*N \ar[r]^{q^*f} \ar[d]_{f^*q} & N \ar[d]^q\cr
U \ar[r]_f & V,\cr
}$$
where the objects of the right square are included in the corresponding objects of the left square, and $p$ is the projection.
Take $f^*N$ to be the unique smooth supermanifold included in $\RR^m⨯k^{0|n}⨯U$ making the right square cartesian in $\sman$.
The morphism $(q^*f,f):f^*q→q$ in $\SFEmb_d$, obtained from the resulting right square,
is the cartesian arrow required by the splitting of~$\base$ for the pair $(f,q)$.

Since the reduction of a closed inclusion of supermanifolds is a closed inclusion of manifolds (\cref{inclusion.manifold}),
the pair $(\red_t,\red_b)$ is a morphism of split Grothendieck fibrations.
By construction, the functor $\base:\SFEmb_d→\stcart$ preserves admissible arrows.
By \cref{reflect.admissible}, the functor~$\base$ reflects admissible arrows
and $(\base,\red_t,\red_b)$ is a fibered geometric site (\cref{fibered.geometric.site}).
\end{example}

Next, we define a $\smset$-enriched analog of the geometric site $\smFEmb_d$ (\cref{def.smFEmb}).

\begin{definition}
\label{def.fraksmFEmb}
Fix $d≥0$ and recall the categories $\smset$ (\cref{def.smset}) and $\smFEmb_d$ (\cref{def.smFEmb}).
We define a $\smset$-enriched geometric site (\cref{enriched.geometric.site}) $\fraksmFEmb_d$
and a split $\smset$-enriched Grothendieck fibration $\base:\fraksmFEmb_d→\cart$
as follows.
The objects are the same as in $\smFEmb_d$.
Given two objects $p:M→U$ and $q:N→V$,
the corresponding enriched hom-object $\fraksmFEmb_d(p,q)$ is a smooth set specified as follows.
\begin{itemize}
\item
The $L$-points (with $L\in \cart$) are given by
morphisms $(f,g)$ in $\smFEmb_d$ of the form
$$\xymatrix@C+2em{
M⨯L \ar[d]_{p⨯\id_L} \ar[r]^f & N⨯L \ar[d]^{q⨯\id_L}\cr
U⨯L \ar[r]_{g=h⨯\id_L} & V⨯L,\cr
}$$
for a (unique) map~$h:U→V$.
\item
The presheaf structure map for a morphism $l:L'→L$ sends the above square to
$$\xymatrix@C+2em{
M⨯L' \ar[d]_{p⨯\id_{L'}} \ar[r]^{f'} & N⨯L' \ar[d]^{q⨯\id_{L'}}\cr
U⨯L' \ar[r]_{h⨯\id_{L'}} & V⨯L',\cr
}$$
where $$f'=(π_1∘f∘(\id_M⨯l),π_2).$$
\end{itemize}
The $\smset$-enriched Grothendieck fibration $$\base:\fraksmFEmb_d→\cart$$
(\cref{enriched.grothendieck.fibration})
maps $p:M→U$ to $U$ and an $L$-family $(f,g)$ of morphisms to the map $h:U→V$ defined above.

We equip the category $\fraksmFEmb_d$ with the same coverage as $\smFEmb_d$.
The $\smset$-enriched admissible structure (\cref{enriched.admissibility.structure})
is defined by declaring $L$-families of admissible maps to be morphisms $(f,g)$ for which $h$ is an open embedding.
This forces $f$ and $g$ to be open embeddings.
The $\smset$-enriched splitting cleavage (\cref{enriched.cleavage}) over $\cart$
is defined for $L$-families by applying the splitting cleavage of $\smFEmb_d$ to the given morphism in $\cart$, then pulling back along the unique map $L→\RR^0$.
\end{definition}

We now define the notion of a fibered geometric site in the $\smset$-enriched setting, building upon \cref{fibered.geometric.site}.

\begin{definition}
\label{fibered.geometric.site.isotopy}
Fix $d≥0$ and recall the site $\fraksmFEmb_d$ from \cref{def.fraksmFEmb}.
A \emph{fibered geometric site with isotopies} is a triple $(\base, \red_t,\red_b)$,
where:
\begin{conditions}[(1)]
\item $\base: \frakFEmb_d→\stcart$ is a split $\smset$-enriched Grothendieck fibration (\cref{enriched.grothendieck.fibration});
\item $(\stcart,\red_b)$ is a structured cartesian site (\cref{structured.cartesian.site});
\item $\red_t:\frakFEmb_d\to \fraksmFEmb_d$ is a $\smset$-enriched admissibility fibration (\cref{enriched.admissibility.fibration});
\end{conditions}
such that the following conditions hold.
\begin{conditions}[resume*]
\item\label{fibered.geometric.site.diagram.isotopy}
The pair $(\red_t,\red_b)$ is a morphism of split $\smset$-enriched Grothendieck fibrations.
That is, the diagram of $\smset$-enriched categories
$$\xymatrix{
\frakFEmb_d \ar[r]^{\red_t} \ar[d]_\base & \fraksmFEmb_d \ar[d]^\base\cr
\stcart \ar[r]_{\red_b} & \cart
}$$
commutes (strictly), where $\base :\fraksmFEmb_d\to \cart$ is the split $\smset$-enriched Grothendieck fibration of \cref{def.fraksmFEmb}.
Moreover, the functor $\red_t$ preserves the splitting cleavage, where the cleavage of $\base:\fraksmFEmb_d\to \cart$ is given in \cref{def.fraksmFEmb}.
\item
\label{base.creates.admissible.isotopy}
For every $L∈\cart$, the functor $\base_L:(\frakFEmb_d)_L→\stcart$ preserves and reflects admissible arrows.
\end{conditions}

We equip $\frakFEmb_d$ with the structure of a $\smset$-enriched geometric site via the $\smset$-enriched admissibility fibration $\red_t$, using \cref{lift.site}.

When the context is clear, we will omit the subscripts $t$ and $b$ on the functor $\red$.
We refer to the functor~$\red$ as the \emph{reduction functor}.
For $L\in \cart$, we refer to admissible morphisms in $(\frakFEmb_d)_L$ as \emph{$L$-families of open embeddings}.
\end{definition}

\begin{remark}
\label{strip.enrichment}
Recall the notation $\cD_L$ for a $\smset$-enriched category~$\cD$ and $L∈\cart$ (\cref{enriched.family.notation}).
Suppose $(\base, \red_t,\red_b)$ is a fibered geometric site with isotopies (\cref{fibered.geometric.site.isotopy}).
Then the triple $(\base, (\red_t)_{\RR^0}, \red_b)$ is a fibered geometric site (\cref{fibered.geometric.site}).
We refer to it as the \emph{underlying fibered geometric site}.
\end{remark}

\begin{remark}
\label{enriched.splitting}
Suppose $\base: \frakFEmb_d→\stcart$ is a $\smset$-enriched functor
and $\base_{\RR^0}:(\frakFEmb_d)_{\RR^0}→\stcart$ is a split Grothendieck fibration.
Then $\base$ admits at most one $\smset$-enriched splitting (\cref{enriched.cleavage})
that turns $\base$ into a split $\smset$-enriched Grothendieck fibration.
Indeed, given an object $(q:N→V)∈\frakFEmb_d$ and a morphism $f:U→V$ in $\stcart$,
the naturality condition on an enriched cleavage forces
the lift of~$f$ to a cartesian $L$-family of morphisms in $(\frakFEmb_d)_L$ with codomain~$q$
to be given by the value of the structure map of $L→\RR^0$ on the lift of~$f$
to a cartesian $\RR^0$-family of morphisms in $(\frakFEmb_d)_{\RR^0}$ with codomain~$q$.
By construction, the resulting family of splitting cleavages of $\base_L:(\frakFEmb_d)_L→\stcart$
is natural in~$L$.
\end{remark}

\begin{proposition}
\label{factor.red.cartesian.base.cartesian.enriched}
Let $(\base,\red_t,\red_b)$ be a fibered geometric site with isotopies (\cref{fibered.geometric.site.isotopy}).
Given $L∈\cart$, every morphism $f:p\to q$ in $(\frakFEmb_d)_L$ uniquely factors as a $\red_t$-cartesian arrow $g:p\to p'$ such that $\base g=\id$ (hence an open embedding),
followed by a split $\base$-cartesian arrow $h:p'\to q$.
\end{proposition}

\begin{proof}
We work in the category $(\frakFEmb_d)_L$.
Set $ϖ:L→\RR^0$ to the unique morphism in $\cart$ and let $ϖ^*h:p'\to q$ be the trivial $L$-family of $\base$-cartesian lifts of $\base f$ for~$q$.
By the universal property of the $\base$-cartesian arrow $ϖ^*h$, there is a unique morphism $g:p\to p'$ such that $\base g=\id$ and $ϖ^*h ∘ g=f$.
But since $\base g=\id$, \cref{base.creates.admissible.isotopy} of \cref{fibered.geometric.site.isotopy} implies that $g$ is $\red_t$-cartesian.
\end{proof}

\begin{proposition}
\label{reflect.admissible.enriched}
Suppose all conditions of \cref{fibered.geometric.site.isotopy} are satisfied except \cref{base.creates.admissible.isotopy}.
Assume that the functor $\base$ preserves admissible arrows
and every arrow $f$ in $(\frakFEmb_d)_L$ such that $\base f=\id$ is $(\red_t)_L$-cartesian.
Then $\base$ reflects admissible arrows, hence \cref{fibered.geometric.site} is satisfied.
\end{proposition}

\begin{proof}
The proof of \cref{reflect.admissible} continues to work for $L$-families, since it makes no use of $\red_t$-cartesian liftings or pullbacks of admissible maps.
\end{proof}

\begin{remark}
\label{admissible.is.monomorphism.enriched}
Every $L$-family of open embeddings in $(\frakFEmb_d)_L$ is a monomorphism
by \cref{def.open.embedding} and the fact that $L$-families of open embeddings in $(\fraksmFEmb_d)_L$ are monomorphisms.
\end{remark}

\begin{example}
\label{standard.fibered.geometric.site.isotopy}
The triple $(\fraksmFEmb_d,\id,\base)$, where $\fraksmFEmb_d$ and $\base$ are as in \cref{def.fraksmFEmb} is a fibered geometric site with isotopies (\cref{fibered.geometric.site.isotopy}).
\end{example}

\begin{example}
\label{trivial.fibered.geometric.site.isotopy}
Continuing \cref{trivial.fibered.geometric.site},
consider the structured cartesian site $\red:\trivialsite→\cart$ from \cref{trivial.geometric.site}.
We construct a fibered geometric site with isotopies (\cref{fibered.geometric.site.isotopy}) $(\fraksmEmb_d,ι,\base)$ as follows.
Take $\fraksmEmb_d$ to be the full subcategory of $\fraksmFEmb_d$ on objects $p:T→U$ with $U∈\trivialsite⊂\cart$,
$ι:\fraksmEmb_d→\fraksmFEmb_d$ to be the inclusion functor,
and $\base:\fraksmEmb_d→\trivialsite$ to be the restriction of $\base$ from \cref{def.fraksmFEmb}.
Constructions with this fibered geometric site with isotopies yield bordism categories with isotopies, but without geometric families.
\end{example}

\begin{example}
\label{def.frakcFEmb}
Recall the holomorphic geometric site $\ccart$ of \cref{def.holomorphic.cart},
the holomorphic fibered geometric site $\cFEmb_d$ of \cref{def.fibered.holomorphic},
and the fibered geometric site with isotopies $\fraksmFEmb_d$ of \cref{def.fraksmFEmb,standard.fibered.geometric.site.isotopy}.

We define the \emph{holomorphic fibered geometric site with isotopies} $\frakcFEmb_d$ as follows.
Objects are the same as in $\cFEmb_d$.
Given objects $p:M→U$ and $q:N→V$, the set $\frakcFEmb_d(p,q)_L$
has as its elements morphisms in $\fraksmFEmb_{2d}(p,q)_L$ of the form
$$\xymatrix{
M⨯L\ar[r]^-{f}\ar[d]_-{p} & N⨯L\ar[d]^-{q}\\
U⨯L\ar[r]^-{h⨯\id_L} & V⨯L\\
}$$
such that for every $l∈L$ the restriction along $\{l\}⊂L$ yields a morphism in $\cFEmb_d$, i.e., a fiberwise holomorphic embedding.

Let $\red_t:\frakcFEmb_d→\fraksmFEmb_{2d}$ be the functor that forgets the complex structure.
This enriched functor is a $\smset$-enriched admissibility fibration (\cref{enriched.admissibility.fibration}),
since the underlying functor is an admissibility fibration (\cref{admissibility.fibration})
by \cref{def.fibered.holomorphic}.

The base functor $\base:\frakcFEmb_d\to \ccart$ sends $(p:M\to U)\mapsto U$.
This functor is a split $\smset$-enriched Grothendieck fibration,
since the construction of \cref{def.fibered.holomorphic} provides a splitting for $(\frakcFEmb_d)_{\RR^0}→\cart$,
and given $L∈\cart$, applying the structure map associated to the unique map $L→\RR^0$ produces a splitting for the Grothendieck fibration $(\frakcFEmb_d)_L→\cart$.
By construction, the pair $(\red_t,\red_b)$ is a morphism of split $\smset$-enriched Grothendieck fibrations
and every $\red_t$-lift of an admissible arrow in $\fraksmFEmb_{2d}$ is $\red_t$-cartesian.
By \cref{reflect.admissible.enriched}, the functor $\base$ reflects admissible arrows,
which proves that $(\base,\red_t,\red_b)$ is a fibered geometric site with isotopies (\cref{fibered.geometric.site.isotopy}).
\end{example}

\begin{example}
\label{def.fibered.supercart.isotopy}
Recall the geometric site $\scart$ of supercartesian spaces,
the reduction functor $$\red=\red_b:\scart→\cart$$ (\cref{def.supercart}),
the fibered geometric site $\SFEmb_{d|e}$ (\cref{def.fibered.supercart}),
and the associated functors $$\base:\SFEmb_{d|e}→\scart, \qquad \red_t:\SFEmb_{d|e}→\smFEmb_d.$$
Fix $d≥0$ and $e≥0$ and consider the following fibered geometric site with isotopies $\frakSFEmb_{d|e}$.
Objects are the same as in $\SFEmb_{d|e}$.
Given objects $(p:M→U),(q:N→V)∈\frakSFEmb_{d|e}$ and $L∈\cart$, the set of $L$-points $\frakSFEmb_{d|e}(p,q)_L$
is the set of commutative squares in $\sman$:
$$\xymatrix{
M⨯L\ar[r]^-{f}\ar[d]_-{p} & N⨯L\ar[d]^-{q}\\
U⨯L\ar[r]^-{g} & V⨯L\\
}$$
such that the underlying square of smooth maps is a morphism in $\fraksmFEmb_d(\red_t p,\red_t q)_L$ (\cref{def.fraksmFEmb})
and for every $l∈L$ the restriction along $\{l\}⊂L$ yields a morphism in $\SFEmb_{d|e}$, i.e., a fiberwise open embedding of smooth supermanifolds.
Let $\red_t:\frakSFEmb_{d|e}\to \fraksmFEmb_d$ be the functor that applies $\red$ to the submersion.

This enriched functor is a $\smset$-enriched admissibility fibration (\cref{enriched.admissibility.fibration}),
since the underlying functor is an admissibility fibration (\cref{admissibility.fibration})
by \cref{def.fibered.supercart}.
The base functor $\base:\frakSFEmb_{d|e}\to \scart$ sends $(p:M\to U)\mapsto U$.
This functor is a split $\smset$-enriched Grothendieck fibration,
since the construction of \cref{def.fibered.supercart} provides a splitting for $(\frakSFEmb_{d|e})_{\RR^0}→\scart$,
and given $L∈\cart$, applying the structure map associated to the map $L→\RR^0$ produces splittings for the Grothendieck fibration $(\frakSFEmb_{d|e})_L→\scart$.
By construction, the pair $(\red_t,\red_b)$ is a morphism of split $\smset$-enriched Grothendieck fibrations
and every $\red_t$-lift of an admissible arrow in $\fraksmFEmb_d$ is $\red_t$-cartesian.
By \cref{reflect.admissible.enriched}, the functor $\base$ reflects admissible arrows,
which proves that $(\base,\red_t,\red_b)$ is a fibered geometric site with isotopies (\cref{fibered.geometric.site.isotopy}).
\end{example}

\subsection{A relative Grothendieck construction for fibered geometric sites}
\label{relative.grothendieck.sec}

In this section, we introduce a relative Grothendieck construction for a fibered geometric site (\cref{fibered.geometric.site}).
The construction transfers the data of a presheaf of sets
$$F:\FEmb_d^{\op}\to \set$$
to a presheaf of categories
$$\relgro F:\stcart^\op\to \smallcat$$
on the base of the fibration $\base:\FEmb_d\to \stcart$.
This construction will be used in the definition of the bordism category, where we will transfer the data of cut grids (\cref{cutgrid}) on objects $p:M\to U\in \FEmb_d$
to the category of families of bordisms parametrized by~$U$.

\begin{definition}
\label{relative.grothendieck}
Suppose that
$$\base:\FEmb_d\to \stcart \quad \text{and}\quad \base:\frakFEmb_d\to \stcart$$
is a split Grothendieck fibration and split $\smset$-enriched Grothendieck fibration, respectively (for example, see \cref{fibered.geometric.site,fibered.geometric.site.isotopy}).
We define the following \emph{relative Grothendieck constructions}.
\begin{itemize}
\item Let $F∈\PSh(\FEmb_d,\set)$, let $\el(F)→\FEmb_d$ be the corresponding discrete fibration, with a canonical splitting cleavage.
The Grothendieck fibration $\base:\FEmb_d→\stcart$ is equipped with a splitting cleavage by definition.
Thus, the composite $\el(F)→\stcart$ is a Grothendieck fibration with a splitting cleavage.
Let $\relgro F:\stcart^{\op}\to \smallcat$ denote the corresponding strict functor.
This construction defines a functor
$$\relgro:\PSh(\FEmb_d,\set)→\PSh(\stcart,\smallcat).$$
The functor $\relgro$ is itself functorial in the split Grothendieck fibration~$\base$.

\item
The $\smset$-enriched version
$$\relgro:\PSh(\frakFEmb_d,\smset)→\PSh(\stcart,\smallcat^{\cart^\op})$$
is defined by setting $(\relgro F)_L=\relgrobare_{\base_L}F[L]$, where
$$F[L]:(\frakFEmb_d)_L^\op → \set, \qquad p↦F(p)_L, \qquad φ∈\frakFEmb_d(p,q)_L ↦ F(φ):F(q)_L→F(p)_L.$$
The functoriality in $L∈\cart$ is induced by the functoriality of the unenriched $\relgro$ with respect to $\base=\base_L$ and $F=F[L]$.
\end{itemize}
\end{definition}

Next, we provide an explicit description of the presheaf of categories $\relgro F$.
Applying the following proposition to the globular monoidal cut grid functor (\cref{globular.monoidal.cut.grid})
yields the explicit descriptions of the bordism categories in \cref{explicit.noncompact.bords,explicit.noncompact.isot}.

\begin{proposition}
\label{explicitgroth}
Assume the context of \cref{relative.grothendieck}.
Given $F∈\PSh(\FEmb_d,\set)$, the presheaf
$\relgro F∈\PSh(\stcart,\smallcat)$ can be described as follows.
Let $U\in \stcart$.
Then $(\relgro F)(U)$ is the Grothendieck construction of the functor $F|_{\base^{-1}(U)}$:
\begin{itemize}
\item
objects are pairs $(p:M→U,x∈F(p))$, where $p\in \FEmb_d$;
\item morphisms $(p,x)→(q,y)$
are maps $φ:p→q$ such that $\base φ=\id_U$ and $F(φ)(y)=x$.
\end{itemize}
 Let $f:U→V$ in $\stcart$, the corresponding functor
$$(\relgro F)(f):(\relgro F)(V)→(\relgro F)(U)$$
can be described as follows.
\begin{itemize}
\item On objects, $(\relgro F)(f)$ sends $(p,x)$ to $(p',x')$, where $\base^*f:p'→p$ is the $\base$-cartesian lift of~$f$ provided by the splitting cleavage of $\base:\FEmb_d→\stcart$
and $x'=F(\base^*f)(x)$.
\item
On morphisms, $(\relgro F)(f)$ sends $φ:(p,x)→(q,y)$ to $φ':(p',x')→(q',y')$,
where $φ'$ is induced by the universal property of the $\base$-cartesian arrow $\base^*f:q'\to q$.
\end{itemize}
An analogous description holds in the $\smset$-enriched case, using $F[L]$ instead of~$F$.
\end{proposition}

\begin{proof}
This is a simple unwinding of \cref{relative.grothendieck}.
\end{proof}

\section{Geometric bordism categories}
\label{bordcts}

In this section, we give a precise definition of the geometric bordism categories as geometric symmetric monoidal $(\infty,d)$-categories.
Specifically, for every field stack $\gs\in \Struct_d$ (\cref{geometric.structure}),
we define a corresponding geometric symmetric monoidal $(\infty,d)$-category $\Bord_{d}^\gs$ (\cref{bordstr}) that encodes bordisms equipped with fields given by~$\gs$.
The construction is manifestly functorial in $\gs$, which proves \cref{axiom.functorial}.
In \cref{axioms.section}, we will show that $\Bord_d^\gs$ satisfies the additional \cref{axiom.cocontinuous,axiom.embedded} in \cref{axioms}.

In parallel, we also define a geometric symmetric monoidal $(∞,d)$-category $\frakBord_{d}^\gs$ with isotopies (\cref{enrichedbordstr})
that encodes bordisms with isotopies and uses enriched field stacks in $\frakStruct_{d}$ (\cref{geometric.structure.isotopy}).
This proves \cref{axiom.functorial} for $\frakBord_d$.
In \cref{axioms.section}, we will prove that this version of the bordism category satisfies \cref{frakaxiom.cocontinuous,frakaxiom.embedded} in \cref{frakaxioms}.

We begin by defining the notion of a field stack.
In the construction of the bordism category, field stacks will be used to encode fields on bordisms.

\subsection{Field stacks}
\label{geostr}

As pointed out in the introduction, our definition of field stacks generalizes the traditional notion of tangential structures.
Roughly, the passage from the traditional approach to our approach is as follows.
Starting with a tangential structure $\xi:Y\to \tdeloop\GL(d)$ and a $d$-dimensional smooth manifold $M$
(to be thought of as the ambient manifold of a bordism), we can form the space of sections of the left map in the homotopy pullback
$$\xymatrix{
Y\times_{\tdeloop\GL(d)}M\ar[r]\ar[d] & Y\ar[d]^-{\xi}\\
M\ar[r]^-{\tau} & \tdeloop\GL(d),
}$$
where $\tau:M\to \tdeloop\GL(d)$ is the classifying map for the tangent bundle of~$M$.
If we vary the manifold $M$, we obtain a sheaf on the site whose objects are $d$-dimensional manifolds and morphisms are open embeddings.
Moreover, we will need fields to vary in families, parametrized over cartesian spaces.
We resolve these issues by working with sheaves on the site $\smFEmb_d$ of submersions (with $d$-dimensional fibers), with fiberwise open embeddings between them,
and its geometric refinement $\FEmb_d$.

Throughout this section, we fix an arbitrary natural number $d$.

\begin{definition}
\label{geometric.structure}
Fix a fibered geometric site $\FEmb_d$ (\cref{fibered.geometric.site}).
A \emph{$d$-dimensional field stack} is a simplicial presheaf on $\FEmb_d$.
We equip the category of $d$-dimensional field stacks
$$\Struct_d≔\sPSh(\FEmb_d)_{\inj,\Cech}$$
with the cartesian model structure of \cref{model.structure.sheaves}.
\end{definition}

\begin{remark}
Sheaves on the site $\smFEmb_d$ (\cref{def.smFEmb}) were considered (in an equivalent reformulation) by Nijenhuis \cite{Nijenhuis},
who used them to define natural bundles and natural mappings in differential geometry.
The latter are closely related to the geometric cobordism hypothesis, as shown in Grady–Pavlov \cite{GradyPavlov.GCH}.
The case of nonfiberwise field stacks (using the Quillen equivalent model category
of simplicial sheaves on the site of manifolds and etale maps) was considered by Freed–Teleman \cite[Appendix~A]{FreedTeleman}.
The special case of field stacks valued in groupoids instead of simplicial sets is considered by Ludewig–Stoffel \cite{LudewigStoffel},
where sheaves on the site of submersions with fiberwise etale maps are considered.
By the above remarks, our field stacks are a natural generalization of these structures.
Ayala–Francis–Tanaka \cite{AyalaFrancisTanaka} investigate
the enriched (\cref{geometric.structure.isotopy}) and unenriched variants of presheaves on $\smEmb_d$ and $\fraksmEmb_d$
(without geometric families).
\end{remark}

\begin{definition}
\label{geometric.structure.isotopy}
Fix a fibered geometric site with isotopies $\frakFEmb_d$ (\cref{fibered.geometric.site.isotopy}).
A \emph{$d$-dimensional isotopy field stack} is an object in the $\smsset$-enriched model category
$$\frakStruct_{d}≔\PSh(\frakFEmb_d,\smsset)_{\inj,\Cech},$$
given by the left Bousfield localization of the injective model structure on $\PSh(\frakFEmb_d,\smsset)$ with respect to the morphisms \cref{cech}.
Here $\PSh$ denotes the category of $\smsset$-enriched presheaves,
where we promoted the $\smset$-enrichment of $\frakFEmb_d$ to a $\smsset$-enrichment, by taking the constant smooth simplicial set.
The model structure exists and is a left proper combinatorial $\smsset$-enriched cartesian model category
by \cref{enriched.left.Bousfield.localization.exists,enriched.left.Bousfield.localization.criterion},
with the model structure on $\smsset$ defined in \cref{smsset.injective}.
\end{definition}

Our definition of a field stack is extremely versatile and captures all significant types of fields we can think of,
including metrics, gauge fields, topological structures, tangential structures, smooth maps to a target manifold, spinors, superstructures, and Euclidean structures.

In the following examples we work in the setting of \cref{standard.fibered.geometric.site},
taking $\stcart=\cart$, $\FEmb_d=\smFEmb_d$.

\begin{example}
Let $X$ be a smooth manifold of any dimension.
We can regard $X$ as a field stack via the sheaf that sends
$$(M\to U)\mapsto \sm(M,X),$$
for a submersion $M\to U$.
This is clearly a sheaf on $\smFEmb_d$ since the total space functor $\smFEmb_d\to \man$, which maps $M\to U$ to $M$, sends covering families to covering families.
This sheaf is not representable.
\end{example}

\begin{example}
The previous example can be generalized easily to all simplicial presheaves on the site $\man$ (or rather, a small subsite of $\man$ that is equivalent to~$\man$).
Indeed, the total space functor $T:\smFEmb_d\to \man$ induces a restriction functor
$$T^*:\sPSh(\man)_{\Cech}\to \Struct_d,$$
which manifestly preserves the homotopy descent property.
\end{example}

\begin{example}
\label{riemmet}
An example of a field stack that does not come from a simplicial presheaf on smooth manifolds and smooth maps
is given by the presheaf of fiberwise Riemannian metrics.
Let
$$\FRiem:\smFEmb_d^{\op}\to \set$$
be the presheaf that sends a submersion $p:M\to U$ with $d$-dimensional fibers to the set of metrics on the fiberwise tangent bundle $\ker \T p$ over~$M$.
A morphism $(a,b):(p:M\to U)\to (q:N\to V)$ is sent to the function
$$(a,b)^*:\FRiem(q)\to \FRiem(p)$$
that sends a metric $g$ on the fiberwise tangent bundle $\ker \T q$ to the pullback metric $$(a^*g)_{x}(v,w)=g_{a(x)}((T_x a)(v),(T_x a)(w)),\qquad x\in M.$$
This is a well-defined metric on $\ker Tq$ since $a$ is a fiberwise open embedding.

We can also consider Riemannian metrics with restrictions on sectional curvature (e.g., positive, negative, nonpositive, nonnegative),
since these properties are preserved by pullbacks along open embeddings.
For example, we define the subobject of positive sectional curvature metrics $\FRiem_{{\rm sc}>0}\subset \FRiem$
as the functor that sends $p:M\to U$ to the subset of metrics on the fiberwise tangent bundle such that for all $u\in U$,
the metric $g\vert_{p^{-1}(u)}$ on $T M_u\cong \ker(Tp\vert_{M_u})\to M_u=p^{-1}(u)$ is a metric of positive sectional curvature.
Such metrics can again be pulled back by fiberwise open embeddings and the property of having positive sectional curvature is preserved under such pullbacks.
\end{example}

So far, the examples of field stacks that we have considered have been sheaves of sets (i.e., they do not have higher morphisms).
We now consider the example where the fields are principal bundles with connection, which incorporates gauge transformations as higher morphisms.

\begin{example}
\label{gbunexample1}
Let $G$ be a Lie group.
Consider the field stack $G\Bunconn$ that sends
an object $(p:M→U)∈\smFEmb_d$ to the groupoid $G\Bunconn(p)$ defined as follows.
Objects are pairs $(φ:p→q,π:P→N)$,
where $(q:N→V)∈\smFEmb_d$, $φ=(a:M→N,b:U→V)$ is a morphism in $\smFEmb_d$,
and $π:P→N$ is a principal $G$-bundle over total space of~$q$
with a fiberwise connection along~$q$.
Morphisms $$(φ:p→q,π:P→N)→(φ':p→q',π:P'→N')$$
are given by smooth $U$-families $\{τ_u:π_{b(u)}→π_{b'(u)}\}_{u∈U}$ of connection-preserving morphisms of principal $G$-bundles with connection.
Composition of morphisms is defined in the obvious way.
Given a morphism $ψ=(e,f):p'→p$ in $\smFEmb_d$,
we set $$G\Bunconn(ψ)(φ,π)=(φψ,π), \qquad G\Bunconn(ψ)(\{τ_u\}_{u∈U})=\{τ_{f(u')}\}_{u'∈U'},$$
which yields a presheaf of groupoids $G\Bunconn$ on $\smFEmb_d$.

Consider the presheaf of groupoids $\deloop_\nabla G$ on $\smFEmb_d$ that sends an object $(p:M→U)∈\smFEmb_d$ to the groupoid whose objects are
trivial principal $G$-bundles $G⨯M→M$ with a fiberwise connection (i.e., a vertical 1-form on the submersion $p:M→U$ with values in $\mathfrak{g}$).
Morphisms are connection-preserving isomorphisms, specified as smooth maps $M→G$.
The structure maps for $\smFEmb_d$ are defined by pulling back vertical 1-forms and smooth $G$-valued functions.
Since the bundles are trivial, this construction yields a strict presheaf of groupoids.
The canonical inclusion $\deloop_∇ G→G\Bunconn$ is a weak equivalence in $\sPSh(\smFEmb_d)$,
i.e., a stalkwise weak equivalence.

We also have the following explicit model for a fibrant replacement $\frep\deloop_∇ G$ of $\deloop_∇ G$:
$$(\frep\deloop_∇ G)(p)=\colim_{\{p_α\}_{α∈A}}\map(\check{C}p,\deloop_∇ G),$$
where the colimit is taken over the category of open covers of~$p$ and their refinements
and $\check{C}p$ denotes the Čech nerve of the open cover $\{p_α\}_{α∈A}$.
Unfolding this formula, we see that a vertex in $(\frep\deloop_∇ G)(p:M→U)$
is given by the following data:
\begin{enumerate}
\item an open cover $\{p_{\alpha}:M_{\alpha}\to U_{\alpha}\}$ of $p:M\to U$;
\item for each $\alpha$, a vertical connection 1-form $\mathcal{A}\in \Omega^1_v(M_{\alpha};\mathfrak{g})$;
\item for each $\alpha,\beta$, a map
$$g_{\alpha\beta}:M_{\alpha\beta}=M_α∩M_β\to G$$
such that
$$g_{\alpha\beta}\cdot \mathcal{A}_{\alpha}=g^{-1}_{\alpha\beta}\mathcal{A}_{\beta}g_{\alpha\beta}+g_{\alpha\beta}^{-1}dg_{\alpha\beta}$$
and such that $g_{\alpha\beta}$ satisfies the Čech cocycle condition.
\end{enumerate}
This is precisely the Čech cocycle data for a principal $G$-bundle with fiberwise connection on~$M$.
\end{example}

The next example shows that our bordism category is capable of handling the supersymmetric Euclidean field theories of Stolz–Teichner \cite{StolzTeichner.SUSY}.
Here we will need the full generality of the fibered geometric site $\base:\FEmb_d\to \stcart$.

\begin{example}
\label{def.super.euclidean}
Recall the structured cartesian site $\scart_\CC$ of complex smooth supercartesian spaces (\cref{def.supercart})
and the fibered geometric site $\SFEmb_{2|1}$ of $2|1$-dimensional complex smooth supermanifolds (\cref{def.fibered.supercart}).
The reduction functor $\red$ is the usual reduction functor for supermanifolds.

Following Deligne–Freed \cite[Section~1.1]{DeligneFreed} and Stolz–Teichner \cite[Section~4.2 and (4.10)]{StolzTeichner.SUSY},
we consider the super Euclidean space $\EE^{2|1}$, defined as follows.
Let $\Delta=\CC$ be the representation of $\Spin(2)\cong \lgU(1)$ on $\CC$ given by $\theta\cdot z≔\bar \theta z$.
Let $\EE^{2|1}$ be the super Lie group whose underlying cs manifold is $\RR^{2|1}=\RR^2\times \Pi \Delta$
and whose multiplication is given in coordinates by
$$\EE^{2|1}\times \EE^{2|1}\to \EE^{2|1},
\quad (\tau_1,\bar \tau_1,\theta_1),(\tau_2,\bar \tau_2,\theta_2)\mapsto (\tau_1+\tau_2,\bar \tau_1+\bar \tau_2+\theta_1\theta_2, \theta_1+\theta_2),$$
where we are using complex-valued coordinate functions, as in Stolz–Teichner \cite[Remark~4.11]{StolzTeichner.SUSY} (see \cref{def.supercart} for additional references).
The Lie group $\Spin(2)$ acts on $\EE^{2|1}=\RR^{2}\times \Pi\Delta$ by the twofold cover $\Spin(2)\to \SO(2)$ and the spinor representation on $\Delta$.
We can therefore form the semidirect product $\EE^{2|1}\rtimes \Spin(2)$, which is the \emph{$2|1$-super Euclidean group}.
The semidirect product acts on $\RR^{2|1}$ via open embeddings: the first factor acts via the multiplication in $\EE^{2|1}$
and the second factor acts as above.

Now we consider the homotopy quotient
$$\Eucl_{2|1}≔\Yo{\RR^{2|1}→\RR^0}\hq(\Yo{\EE^{2|1}\rtimes \Spin(2)}∘\base^\op)\in \Struct_2,$$
where $(\RR^{2|1}→\RR^0)∈\SFEmb_{2|1}$,
$\EE^{2|1}\rtimes \Spin(2)$ is the group object in $\sman$ constructed above,
promoted to a presheaf on $\scart_\CC$ via the restricted Yoneda embedding~$\Yo{}$
and converted to a presheaf on $\SFEmb_{2|1}$ by composing with $\base^\op$.

Given $(p:M→U)∈\SFEmb_{2|1}$, the simplicial set $\Eucl_{2|1}(p)$ can be explicitly described
as the nerve of the action groupoid
of the action of the group $\sm(U,\EE^{2|1}\rtimes \Spin(2))$
on the set $\SFEmb_{2|1}(p,\RR^{2|1}→\RR^0)$
that sends $$(g:U→\EE^{2|1}\rtimes \Spin(2),φ=(e,b):p→(\RR^{2|1}→\RR^0)) ↦ ((g∘p)⋅e,b).$$

Proceeding similarly to \cref{gbunexample1},
a derived map $(p:M\to U)\to \Eucl_{2|1}$ can be identified with the following data.
\begin{enumerate}
\item An open cover $\{p_{\alpha}:M_{\alpha}\to U_{\alpha}\}$ of $p:M\to U$.
\item For each $\alpha$, a map of supermanifolds $e_{\alpha}:M_{\alpha}\to \RR^{2|1}$
such that induced map $(e,p):M_{\alpha}\to \RR^{2|1}\times U_{\alpha}$ is an open embedding.
\item For each $\alpha,\beta$, a map
$$g_{\alpha\beta}:U_{\alpha\beta}=U_α∩U_β\to \EE^{2|1}\rtimes \Spin(2)$$
such that $g_{\alpha\beta}\cdot e_{\alpha}=e_{\beta}$ and such that $g_{\alpha\beta}$ satisfies the Čech cocycle condition.
\end{enumerate}
Hence, a derived map (which is an object in the corresponding bordism category)
is precisely a $U$-family of $(\EE^{2|1}\rtimes \Spin(2),\RR^{2|1})$ supermanifolds in the sense of Stolz–Teichner \cite[Definition~2.26]{StolzTeichner.SUSY}.
\end{example}

\begin{definition}
\label{isotopification}
Suppose $\frakFEmb_d$ is a fibered geometric site with isotopies (\cref{fibered.geometric.site.isotopy})
and $\FEmb_d$ is its underlying fibered geometric site (\cref{strip.enrichment}).
We have a forgetful functor
$$\frakStruct_d → \Struct_d, \qquad \gs↦\gs_{\RR^0}, \qquad \gs_{\RR^0}(p)=\gs(p), \qquad \gs_{\RR^0}(f:p→q)=\gs(f)_{\RR^0}:\gs(q)_{\RR^0}→\gs(p)_{\RR^0}.$$

We define the \emph{isotopification functor} $$\frakI_d:\Struct_d→\frakStruct_d$$ as the left adjoint functor of $\gs↦\gs_{\RR^0}$.
\end{definition}

\begin{remark}
\label{isotopification.description}
The adjunction $$\frakI_d⊣(-)_{\RR^0}$$
in \cref{isotopification}
is a Quillen adjunction of the Čech-local injective model structures.
This can be seen as follows.
Additional details for the case $\FEmb_d=\smFEmb_d$ are available in Kenig–Pavlov \cite[Remark 4.1.15]{KenigPavlov}.

The isotopification functor~$\frakI_d$ can be constructed as $\frakI_d = ι_! λ^∞$,
where $$λ^∞:\Struct_d=\sPSh(\FEmb_d)→\smsPSh(\FEmb_d)$$
applies the inclusion $\sset→\smsset$ objectwise and
$$ι_!:\smsPSh(\FEmb_d)→\smsPSh(\frakFEmb_d)=\frakStruct_d$$
is the $\smsset$-enriched left Kan extension along the $\smset$-enriched functor
$$\FEmb_d\to \frakFEmb_d, \quad p\mapsto p, \quad \FEmb_d(p,q)_L=\FEmb_d(p,q)\lto3{c}\frakFEmb_d(p,q)_L,$$
where the map $c$ regards a morphism as a constant $L$-family of morphisms by pulling back along the unique map $L→\RR^0$.

Explicitly, given
$\gs∈\Struct_d$,
$p:M→U$ in $\FEmb_d$,
and $L∈\cart$, the simplicial set $(\frakI_d \gs)(p:M→U)_L$
can be described as $∐_{[φ]}\gs(q)$, where $φ:p→q$ runs over a set of representatives of isomorphism classes of objects in the following discrete groupoid.
Objects are elements $φ∈\frakFEmb_d(p,q)_L$, where $q∈\FEmb_d$ is arbitrary,
such that $\base φ=\id$ and composing $\red_t φ:\red_t p⨯\id_L→\red_t q⨯\id_L$ with the projection to $\red_t q$ yields a surjective map $\red_t p⨯\id_L→\red_t q$.
Morphisms $φ_1→φ_2$ are morphisms $h∈\FEmb_d(q_1,q_2)$ such that $ϖ^*h ∘ φ_1=φ_2$, where $ϖ^*$ pulls back along the unique map $ϖ:L→\RR^0$.

This explicit description immediately implies that the left adjoint functor $\frakI_d$
preserves injective cofibrations and injective acyclic cofibrations
and therefore is a left Quillen functor for the injective model structures.
Furthermore, $\frakI_d$ sends Čech nerves of covering families in $\FEmb_d$ to Čech nerves of the corresponding covering families in $\frakFEmb_d$.
Therefore, it descends to a left Quillen functor between the Čech-local injective model structures.
\end{remark}

\begin{example}
\label{friemex}
Applying the isotopification functor of \cref{isotopification} to the field stack $\FRiem ∈ \sPSh(\smFEmb_d)$ (\cref{riemmet}),
we get a field stack with isotopies $$\frakFRiem = \frakI_d \FRiem \in \frakStruct_d = \smsPSh(\fraksmFEmb_d).$$
By \cref{isotopification.description},
given $p∈\fraksmFEmb_d$ and $L∈\cart$, the $L$-points of the smooth set $\frakFRiem(p:M\to U)_L$ are given by
equivalence classes of pairs $(g,φ)$, where $g$ is a fiberwise Riemannian metric on some $(q:N\to U)∈\smFEmb_d$,
$φ∈\fraksmFEmb_d(M→U,N→U)_L$ is a surjective fiberwise open embedding $φ:M\times L\to N$ covering the projection map $U\times L\to U$,
and the equivalence relation is given by $(ψ^*g,φ)\sim(g,ψφ)$, where $ψ:(N→U)→(P→U)$ is an isomorphism in $\smFEmb_d$ covering the identity map on~$U$.
This smooth set is different from the smooth set of $U⨯L$-parametrized families of fiberwise Riemannian metrics on $M⨯L$.
In particular, the $U⨯L$-families of metrics in $\frakFRiem$ are ``$d$-thin''
in that the parametrizing map to the moduli stack must factor through a fixed family~$q$ of $d$-dimensional manifolds.

In the most elementary case, when $d=1$ and $(p:M→U)=(\RR⨯U→U)$ is the projection map,
we can also take $(q:N→U)=(\RR⨯U→U)$ to be the projection map.
A fiberwise Riemannian metric on~$N$ yields a fiberwise open embedding $σ:N→\RR$ induced by the fiberwise distance function.
The map~$σ$ is unique up to fiberwise translations by $τ:U→\RR$ and (in the unoriented case) reflections.
By composing $φ∈\fraksmFEmb(p,q)_L$ with $σ$, we get an $L$-family of fiberwise open embeddings of $p$ into~$\RR$.
Furthermore, the smooth set of such open embeddings is a disjoint union of two smooth sets (corresponding to the two orientations of~$\RR$),
both of which are contractible.
The action of translations and reflections on this smooth set is free.
Therefore, the ordinary quotient computes the homotopy quotient
and we have a weak equivalence of smooth simplicial sets
$$\frakFRiem(\RR\times U\to U)\simeq \tdeloop \sm(U, \RR),$$
which is natural in $U$.
Here $\tdeloop \sm(U, \RR)$ is the simplicial set given by the delooping of the additive abelian group of smooth real-valued functions.
(With a fiberwise orientation, we get a further factor of $\tdeloop\ZZ/2$.)
For details on this calculation, see \gcref{shapemetrics}.
A supersymmetric analog of this construction is considered in \cref{1.1.Euclidean}.
See also Kenig–Pavlov \cite{KenigPavlov} for the case of 1-dimensional Riemannian sigma-models.
\end{example}

\begin{example}
\label{1.1.Euclidean}
We work with the site $\sman_\RR$ of real smooth supermanifolds (\cref{def.sman}).
Consider the super Lie group $\EE^{1|1}$ with underlying real smooth supermanifold $\RR^{1|1}=\RR\times \RR^{0|1}$.
For $S\in \sman_\RR$, the set of $S$-points of $\EE^{1|1}$ is $\sm(S,\RR)_0⊕\sm(S,\RR)_1$ 
and the corresponding group structure is given by
$$(s_1,η_1)(s_2,η_2)=(s_1+s_2+η_1η_2,η_1+η_2),\qquad -(s,η)=(-s,-η), \qquad 0 = (0,0),$$
where $s_i∈\sm(S,\RR)_0$ and $η_i∈\sm(S,\RR)_1$.
We have an extension $\hat\EE^{1|1}≔\EE^{1|1}\rtimes\{1,-1\}$ that is classified by the homomorphism $\{1,-1\}→\Aut(\EE^{1|1})$
that sends $-1$ to the map
$$(s,η)↦(s,-η).$$
We have an action~$μ$ of $\hat\EE^{1|1}$ on $\RR^{1|1}$
induced by the actions of $\EE^{1|1}$ and $\{1,-1\}$ by translations and reflections in the odd coordinate, respectively.

Following Stolz–Teichner \cite[Section~3.2]{StolzTeichner.Elliptic} and Hohnhold \cite[Section 3.2]{Hohnhold},
we define the presheaf $$\Eucl_{1|1}∈\Struct_{1|1}=\sPSh(\SFEmb_{1|1})$$ of fiberwise oriented $1|1$-Euclidean metrics as the homotopy quotient
$$\Eucl_{1|1}=\Yo{\RR^{1|1}→\RR^0}\hq(\Yo{\hat\EE^{1|1}}∘\base^\op),$$
where the group object $\hat\EE^{1|1}$ is promoted to a presheaf on $\scart$ via the restricted Yoneda embedding
and converted to a presheaf on $\SFEmb_{1|1}$ by precomposing with~$\base^\op$.
Explicitly, $\Eucl_{1|1}$ sends an object $p:M→U$ to the groupoid
whose objects are fiberwise open embeddings $ρ:M→\RR^{1|1}$
and morphisms $ρ_1→ρ_2$ are given by maps $τ:U→\hat\EE^{1|1}$
such that $ρ_2$ equals the composition $$M\lto{11}{(τ∘p,ρ_1)}\hat\EE^{1|1}⨯\RR^{1|1}\lto9{μ}\RR^{1|1}.$$
Since the action is free, the canonical map $\Eucl_{1|1}→π_0 \Eucl_{1|1}$ is an objectwise weak equivalence in the model category $\sPSh(\SFEmb_{1|1})$,
where $π_0$ denotes the presheaf of sets given by taking sets of connected components objectwise.

The presheaf~$\Eucl_{1|1}$ is not a sheaf: for example, its value on a circle is the empty simplicial set.
The associated sheaf of~$\Eucl_{1|1}$ yields $1|1$-Euclidean structures in the conventional sense.
Since the map from~$\Eucl_{1|1}$ to its associated sheaf is a weak equivalence in the Čech-local model structure,
we stick to the simpler variant~$\Eucl_{1|1}$.

Following \cref{friemex}, we compute the isotopification $\frakI_{1|1}\Eucl_{1|1}$,
where $$\frakI_{1|1}: \Struct_{1|1} → \frakStruct_{1|1}=\smsPSh(\frakSFEmb_{1|1})$$
is the isotopification functor of \cref{isotopification} for the fibered geometric site with isotopies $\frakSFEmb_{1|1}$ (\cref{def.fibered.supercart.isotopy}).
\cref{isotopification.description} tells us that
for every $p∈\frakSFEmb_{1|1}$ and $L∈\cart$, the set $\frakI_{1|1}\Eucl_{1|1}(p)_L$
is the set of equivalence classes of pairs $(g,φ)$.
Here $g$ is a $1|1$-Euclidean structure on some $q∈\SFEmb_{1|1}$,
$φ∈\frakSFEmb_d(p,q)_L$ is an $L$-family of maps such that $\base φ=\id_{\base p}$
and the map $\red(π∘φ)$ is surjective on the total spaces,
where $π:q⨯\id_L→q$ is the projection map.
The equivalence relation is $(ψ^*g,φ)\sim(g,ψφ)$, where $ψ:q→q'$ is an isomorphism in $\SFEmb_{1|1}$ such that $\base ψ=\id$.

Taking $(p:M→U)≅(\RR^{1|1}⨯U→U)$ to be isomorphic to the projection map,
we can also take $(q:N→U)=(\RR^{1|1}⨯U→U)$ to be the projection map.
Again following \cref{friemex},
we present $\frakI_{1|1}\Eucl_{1|1}(p)$ as the quotient
of the smooth set of fiberwise embeddings $p→\RR^{1|1}$
by the free action of the group $\sm(U,\hat\EE^{1|1})$:
$$\frakI_{1|1}\Eucl_{1|1}(p)≅\frakSFEmb_{1|1}(p,\RR^{1|1})\hq\sm(U,\hat\EE^{1|1}).$$
Since the action is free, the ordinary quotient computes the homotopy quotient.

Consider the $\smset$-enriched category $\cD=\fraksmFEmb_1/\tdeloop\RR^⨯$,
whose objects are pairs $(p,Ψ)$, where $p∈\fraksmFEmb_1$ and $Ψ$ is a real smooth line bundle over the total space of~$p$
and morphisms $(p_1,Ψ_1)→(p_2,Ψ_2)$ are morphisms $φ:p_1→p_2$ in $\fraksmFEmb_1$ together with a morphism of line bundles $Ψ_1→φ^*Ψ_2$.
There is an equivalence of $\smset$-enriched categories
$$χ:\cD=\fraksmFEmb_1/\tdeloop\RR^⨯→\frakSFEmb_{1|1}$$
that sends $(p,Ψ)↦ΠΨ=ΛΨ^*$.

We have a chain of weak equivalences of smooth sets
$$\frakSFEmb_{1|1}(p,\RR^{1|1})≅\cD(χ^{-1}p,χ^{-1}\RR^{1|1})≃\fraksmFEmb_1(\red p,\red\RR^{1|1})⨯\Or(χ^{-1}p)≃\Or(\red p)⨯\Or(χ^{-1}p),$$
where $\Or(\red p)$ is the two-element set of fiberwise orientations of~$\red p$ and $\Or(χ^{-1}p)$ is the two-element set of orientations
of the line bundle component of $χ^{-1}p$.
The action of $\sm(U,\hat\EE^{1|1})$ on $\Or(\red p)⨯\Or(χ^{-1}p)$ is given by the composite of
$$\sm(U,\hat\EE^{1|1})⨯\Or(\red p)⨯\Or(χ^{-1}p)\lto{11}{(\sm(U,ψ),\id)}\sm(U,\{1,-1\})⨯\Or(\red p)⨯\Or(χ^{-1}p)\lto3{α}\Or(\red p)⨯\Or(χ^{-1}p),$$
where $ψ:\hat\EE^{1|1}→\{1,-1\}$ is the projection homomorphism
and $α$ makes the first factor act on the third factor,
using the fact that every map $U→\{1,-1\}$ is constant.
Therefore, we have a weak equivalence of smooth simplicial sets
$$\frakI_{1|1}\Eucl_{1|1}(p)≃(\Or(\red p)⨯\Or(χ^{-1}p))\hq\sm(U,\hat\EE^{1|1})≃\Or(\red p)⨯\tdeloop \sm(U, \EE^{1|1}),$$
which is natural in~$p$.

Thus, we have an objectwise weak equivalence $\frakI_{1|1}\Eucl_{1|1}→\gs$ in $\frakStruct_{1|1}$, where
$$\gs:\frakSFEmb_{1|1}^\op→\smsset, \qquad p↦\Or(\red p)⨯\tdeloop \sm(U, \EE^{1|1}), \qquad (φ:p→q) ↦ \Or(\red φ)⨯\tdeloop\sm(\base φ,\EE^{1|1}).$$
\end{example}

\subsection{Cuts, cut tuples, and cut grids}

In this section, we define the notions of a cut, cut tuple, and cut grid on an object $(p:M\to U)\in \smFEmb_d$,
which are responsible for implementing the structure of a symmetric monoidal $d$-category for bordisms.
The notion of a cut (\cref{cut}) is inspired by Stolz–Teichner \cite[Definition~2.21]{StolzTeichner.SUSY}.
The notion of a cut grid (\cref{cutgrid}) roughly resembles the constructions of Lurie \cite[Definition~2.2.9]{Lurie.TFT}
and Calaque–Scheimbauer \cite[Definition~5.1]{CalaqueScheimbauer}.
The notion of a globular cut grid is inspired by Henriques \cite[\S2.2]{Henriques}, who talks about \emph{cobordisms with thin parts}.

We begin with the notion of a cut on an object $p:M\to U\in \smFEmb_d$.

\begin{definition}
\label{cut}
A \emph{cut} of an object $p:M\to U$ in $\smFEmb_d$ is a triple $(C_{<},C_=,C_>)$ of subsets of $M$
such that there is a smooth map $h:M\to \RR$ (called the \emph{height function})
whose fiberwise-regular values form an open neighborhood of~$0$ and we have $h^{-1}(-\infty,0)=C_{<}$, $h^{-1}(0)=C_=$, and $h^{-1}(0,\infty)=C_>$.
Here $t∈\RR$ is a \emph{fiberwise-regular value} if for all $m∈M$ such that $h(m)=t$, the map $T h|_{\ker T p}$ is surjective.
We set
$$C_\le=C_<\cup C_=, \qquad C_{\ge}=C_>\cup C_=.$$
We equip the set of cuts with a natural ordering $\le$, with $C\le C'$ if and only if $C_<\subset C'_<$.
With this ordering, the set of cuts is a partially ordered set.

We define a functor
\begin{equation}
\Cut_{\po}:\smFEmb_d^\op\to \poset
\end{equation}
as follows.
\begin{enumerate}
\item An object $p:M\to U$ is sent to the poset of cuts $\Cut(p:M\to U)$ on~$p$.
\item A morphism $(f,g):(M\to U)\to (N\to V)$ is sent to the order-preserving function
$$\Cut(f,g):\Cut(N\to V)\to \Cut(M\to U), \quad \Cut(f,g)(C)=(f,g)^*C=(f^{-1}(C_{<}),f^{-1}(C_{=}),f^{-1}(C_{>})).$$
That this is a well defined cut follows from the fact that $f$ is a fiberwise open embedding.
In particular, if $h$ is a height function for the cut $C$, then $h\circ f$ is a height function for the cut $(f,g)^*C$.
\end{enumerate}
Composing the functor $\Cut_{\po}$ with the nerve functor $\mathcal{N}:\poset\to \sset$ yields a functor
\begin{equation}
\Cut:\smFEmb_d^\op\to \sset.
\end{equation}
We call an $m$-simplex in $\Cut(p:M\to U)$ a \emph{cut} $[m]$-\emph{tuple} $C$ for $p:M\to U$.
Thus, a cut $[m]$-tuple is a collection of cuts $C_j=(C_{<j},C_{=j},C_{>j})$ of $p:M\to U$ indexed by vertices $j\in [m]$ such that
$$C_0\leq C_1\leq \cdots \leq C_m.$$
\end{definition}

We introduce notation for open and closed regions between cuts, as well as cut tuples obtained by removing some of the cuts.

\begin{notation}
\label{notation.cut}
Let $p:M\to U$ be an object in $\smFEmb_d$.
Let $C$ be a cut $[m]$-tuple.
Given $j≤j'$, we set
$$C_{(j,j')}≔C_{>j}\cap C_{<j'}, \qquad C_{[j,j']}≔C_{\ge j}\cap C_{\le j'}.$$
We also denote by
$C_{⟨j,j'⟩}$
the cut $[j'-j]$-tuple obtained from~$C$ by removing the cuts $C_k$ with $k<j$ or $k>j'$.
\end{notation}

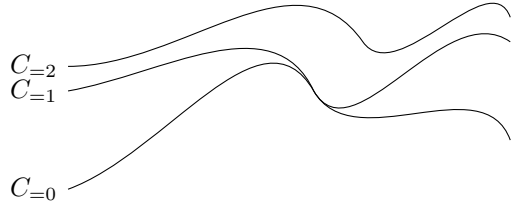
\begin{figure}[ht]
\begin{center}
\begin{tikzpicture}[scale=.65]
\draw (0,0) to [out=20,in=120] (5,2) to [out=-60, in=110] (9,1);
\draw (0,2) to [out=10,in=115] (5,2) to [out=300, in=145] (9,3);
\draw (0,2.5) to [out=0, in=125] (6,3) to [out=305, in=110] (9,3.5);
\node at (0-.7,0) {$C_{=0}$};
\node at (0-.7,2) {$C_{=1}$};
\node at (0-.7,2.5) {$C_{=2}$};
\end{tikzpicture}
\end{center}
\caption{A cut $[2]$-tuple $C=\{(C_{<j},C_{=j},C_{>j})\}_{j\in [2]}$ on $\RR^2→\RR^0$.
The cuts $C_{=0}$ and $C_{=1}$ intersect and we have $C_{\leq 0}\subset C_{\leq 1}$.
Cuts are not allowed to intersect transversally, as this would violate the ordering in \cref{cut}.}
\label{cut.figure}
\end{figure}

Next, we will extend the functor $\Cut$ further to a functor
$$\tCut:(\Delta^{\times d})^\op\times \smFEmb_d^\op\to \set,$$
where the subscript notation will become apparent in a moment.
Morally, $\tCut$ sends a pair consisting of a multisimplex~${\bf m}$ and $p:M→U$ in $\smFEmb_d$ to the set of collections of cut $[m_i]$-tuples, where $i\in \{1,\ldots, d\}$.
However, we do not want arbitrary collections of cut tuples.
Instead, we only take those cut tuples that intersect transversally in directions indexed by different elements of $\{1,\ldots,d\}$.
This motivates the following definition.

\begin{definition}
\label{cutgrid}
Fix $d\geq 0$.
We define the \emph{cut grid functor}
$$\tCut:(\Delta^{\times d})^\op\times \smFEmb_d^\op\to \set$$
as follows.
Let $p:M\to U$ be an object in $\smFEmb_d$ and let ${\bf m}=([m_1],\ldots,[m_d])\in \Delta^{\times d}$ be a multisimplex.
The set $\tCut({\bf m},p)$ has elements:
\begin{itemize}
\item for each $1\leq i\leq d$, a cut $[m_i]$-tuple $C^i$ on $p:M\to U$,
\end{itemize}
which satisfy the following transversality property.
\begin{enumerate}
\item[$\pitchfork$.]
For every subset $S\subset \{1,\ldots,d\}$ and $j:S→\ZZ$ such that for every $i∈S$ we have $0\leq j_i\leq m_i$,
there is a smooth map $h_j:M\to \RR^S$ such that for every $i\in S$, the map $$\pi_i\circ h_j:M\to\RR,$$
where $\pi_i:\RR^S\to \RR$ is the $i$th projection, yields the $j_i$-th cut $C_{j_i}^i$ in the cut tuple $C^i$, as in \cref{cut}.
We require that the fiberwise-regular values of $h_j$ form an open neighborhood of $0$ in $\RR^S$.
\end{enumerate}
The structure maps for the $i$th factor of $\Delta$ in $\Delta^{\times d}$ are given by applying a simplicial map to the $i$th cut tuple~$C^i$.
We observe that the transversality property is still satisfied after applying a simplicial map, since simplicial maps simply remove cuts or duplicate cuts,
i.e., the function $h_j$ in the above definition can be left unchanged.
The structure maps for the factor $\smFEmb_d$ are given by applying the corresponding structure maps to every cut tuple.
We call the elements of $\tCut({\bf m},p)$ \emph{cut ${\bf m}$-grids}.
\end{definition}

The condition $\pitchfork$ in \cref{cutgrid} guarantees that cuts in \emph{different} simplicial directions intersect transversally.
This is forced by the requirement that the functions~$h_j$ have the origin as a regular value and by the compatibility of $h_j$ with the cuts.
We remind the reader that cuts in the \emph{same} simplicial direction are allowed to intersect and even overlap,
but are not allowed to intersect transversally (see \cref{cut.figure}).
The next example illustrates this point.

\begin{example}
\label{transcut}
We provide an example of an object $p:M\to \RR^0\in \smFEmb_2$, equipped with cut tuples in each direction that \emph{do not} assemble to form a cut grid.
Take $M=\RR^2$ and consider the following two cut $[0]$-tuples (i.e., cuts) $C^1$ and $C^2$.
The set $C^1_{=}$ is given by the $y$-axis in $\RR^2$, $C_{<}^1=\{(x,y)\in \RR^2 \mid x<0\}$, $C_{>}^1=\{(x,y)\in \RR^2 \mid x>0\}$.
The set $C^2_{=}$ is given by the parabola $x=y^2$, $C^2_{<}=\{(x,y)\in \RR^2 \mid x<y^2\}$, $C^2_{>}=\{(x,y)\in \RR^2 \mid x>y^2\}$.
In this case, the two cuts $C^1_{=}$ and $C^2_{=}$ are tangent at the origin (see \cref{figure3}).

These cut tuples do not form a cut grid for the following reason.
The condition $\pitchfork$ in \cref{cutgrid} implies that for $S=\{1,2\}$ there is a smooth function $h_j:\RR^2\to \RR^2$,
corresponding to the function $j:\{1,2\}\to \ZZ$ defined by $j(1)=0$, $j(2)=0$, such that $h_j$ has zero as a regular value
and $(\pi_1\circ h_j)^{-1}(0)=C^1_{=}$, $(\pi_2\circ h_j)^{-1}(0)=C^2_{=}$.

The standard basis vector $e_2=(0,1)$ is tangent to both curves $C^1_{=}$ and $C^2_{=}$.
Therefore, since $\pi_1\circ h_j$ vanishes on~$C^1_{=}$ and $\pi_2\circ h_j$ vanishes on $C^2_{=}$,
we must have $d\pi_1dh_j(0,0)e_2=d\pi_2dh_j(0,0)e_2=0$.
Thus, $dh_j(0,0)e_2=0$, which implies that $(0,0)$ is not a regular value of $h_j$.

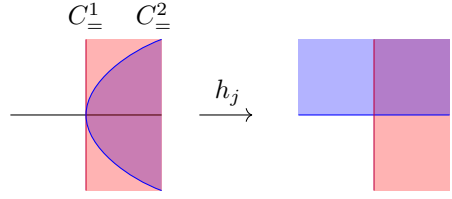
\begin{figure}[ht]
\begin{center}
\begin{tikzpicture}
\fill[blue, opacity=.3] (2,1) to[out=-70-90,in=180-90, looseness=.8] (1,0) to[out=-90, in=250-90, looseness=.8] (2,-1) -- (2,1);
\draw (0,0) -- (2,0);
\draw[purple] (1,-1) -- (1,1);
\fill[red, opacity=.3] (1,1) -- (1,-1) -- (2,-1) -- (2,1) -- (1,1);
\draw[blue] (2,1) to[out=-70-90,in=180-90, looseness=.8] (1,0) to[out=-90, in=250-90, looseness=.8] (2,-1);
\draw[->] (2.5,0) -- (3.2,0);
\node at (2.88,.3) {$h_j$};
\begin{scope}[xshift=1.5in]
\fill[red, opacity=.3] (1,1) -- (1,-1) -- (2,-1) -- (2,1) -- (1,1);
\fill[blue, opacity=.3] (0,1) -- (0,0) -- (2,0) -- (2,1) -- (0,1);
\draw[blue] (0,0) -- (2,0);
\draw[purple] (1,-1) -- (1,1);
\end{scope}
\node at (1.9,1.3) {$C^2_{= }$};
\node at (1, 1.3) {$C^1_{= }$};
\end{tikzpicture}
\end{center}
\caption{Illustrating \cref{transcut}, the map $h_j$ sends the blue region $C_{>}^2$ into the blue region $y>0$ and red region $C_{>}^1$ into red region $x>0$.
It sends the blue curve $C_{=}^2$ into the blue line $y=0$ and the red line $C^1_{=}$ into the red line $x=0$.
The two curves $C^2_{=}$ and $C^1_{=}$ are tangent at the origin.}
\label{figure3}
\end{figure}
\end{example}

We introduce notation for open and closed regions between cells in cut grids.
\begin{notation}
\label{cut.grid.core}
Recall \cref{notation.cut}.
For a cut ${\bf m}$-grid $C$, a subset $S\subset \{1,\ldots,d\}$ and $j,j':S→\ZZ$ satisfying $0\leq j_i\leq j'_i\leq m_i$ for all $i∈S$,
we use the following notation.
\begin{itemize}
\item
We define $$C_{[j,{j'}]}≔\bigcap_{i\in S} C^i_{[j_i,j'_i]}\subset M,\qquad C_{(j,{j'})}≔\bigcap_{i\in S} C^i_{(j_i,j'_i)}\subset M.$$
These subsets are the regions between corresponding cuts in various directions.
\item We define $C_{⟨j,{j'}⟩}$ to be the cut ${\bf m}'$-grid,
whose $i$th cut tuple is $C^i$ if $i∉S$
and $C^i_{⟨j_i,j'_i⟩}$ if $i∈S$,
with ${\bf m}'$ defined accordingly.
\end{itemize}
\end{notation}

\begin{remark}
\label{proper.maps}
Recall the following definitions and theorems from the Stacks Project \cite{Stacks}.
\begin{itemize}
\item (Tag 005O.)
A continuous map $f:X→Y$ of topological spaces is \emph{proper}
if it is separated (meaning the induced map $X→X⨯_Y X$ is closed)
and universally closed (meaning for every continuous $Z→Y$, the induced map $Z⨯_Y X→Z$ is closed).
\item (Tag 005R.)
A continuous map~$f$ is universally closed if and only if $f$ is closed and every fiber of $f$ is a compact topological space.
\item (Tag 0CY3.)
A continuous map with a Hausdorff domain is separated.
\item (Tag 0CY4.)
Separated maps are closed under base change.
Therefore, proper maps are closed under base change.
\item If two maps are separated, universally closed, or proper,
then their composition has the same property.
The inclusion of a closed subset is a proper map.
Restricting a proper map to a closed subset of its domain again yields a proper map.
\end{itemize}
\end{remark}

\begin{definition}
\label{compactgrid}
Let $C$ be a cut ${\bf m}$-grid for $p:M\to U$.
We say that $C$ is \emph{compact} if for $S=\{1,2,\ldots,d\}$, the restriction of $p$ to $C_{[j,j']}$
is proper, for all $j,j':S\to \ZZ$.
\end{definition}

The following examples demonstrate that to avoid pathological situations, we must require the restriction of~$p$ to be closed and have compact fibers.

\begin{example}
Take $(p:M→U)=(\RR^1→\RR^0)$.
Consider the cut $([1])$-grid~$C$, whose only cut tuple is $(C^1_0,C^1_1)$,
with $$C^1_0=((-∞,0),\{0\},(0,∞)), \qquad C^1_1=(\RR^1,∅,∅).$$
Take $j=0$ and $j'=1$.
The restriction of~$p$ to $C_{[j,j']}$
is the map $q:[0,∞)→\RR^0$.
The fiber of~$q$ at the only point of~$\RR^0$ is not compact, therefore the map~$q$ is not proper.
Thus, the cut grid~$C$ is not compact.
\end{example}

\begin{example}
\label{noncompact.grid.compact.fibers}
Take $(p:M→U)=(\RR^2→\RR^1)$ to be the projection map to the first coordinate~$x$.
Consider the cut $([1])$-grid~$C$, whose only cut tuple is $(C^1_0,C^1_1)$,
with
$$C^1_{<i}=\{(x,y)\mid xy<i+1\lor x<0\},\qquad C^1_{=i}=\{(x,y)\mid xy=i+1\land x>0\}, \qquad C^1_{>i}=\{(x,y)\mid xy>i+1\land x>0\}.$$
Take $j=0$ and $j'=1$.
The restriction of~$p$ to $C_{[j,j']}$
is a map~$q$ with compact fibers,
namely, the fiber at $x∈\RR$ is empty if $x≤0$ and $\{x\}⨯[1/x,2/x]$ if $x>0$.
However, the map~$q$ is not a closed map, since its image is $(0,∞)⊂\RR^1=U$, which is not a closed subset of~$U$.
Therefore, $q$ is not a proper map and the cut grid~$C$ is not compact.
\begin{figure}[ht]
\begin{tikzpicture}[domain=.4:4]
\draw[->] (-1.2,0) -- (4.2,0) node[right] {$x$};
\draw[->] (0,-1.2) -- (0,2.2) node[above] {$y$};
\draw[thick] plot (\x,1/\x);
\draw[thick, domain=.7:4] plot (\x,2/\x);
\node at (.5,1) {$C^1_{=0}$};
\node at (2,1.5) {$C^1_{=1}$};
\end{tikzpicture}
\caption{\cref{noncompact.grid.compact.fibers}: A noncompact cut grid with compact fibers.}
\end{figure}
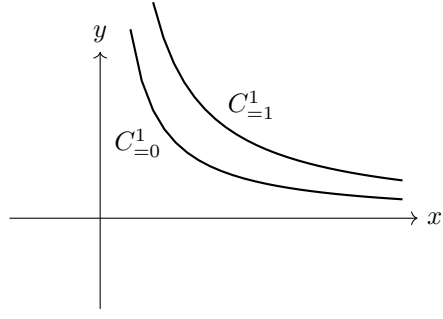
\end{example}

\begin{figure}[ht]
\begin{tikzpicture}[scale=.60]
\draw (0,0) to [out=20,in=120] (5,2) to [out=-60, in=110] (9,1);
\draw (0,2) to [out=10,in=115] (5,2) to [out=300, in=145] (9,3);
\draw (0,2.5) to [out=0, in=125] (6,3) to [out=305, in=110] (9,3.5);
\node at (0-.7,0) {$C^1_{=0}$};
\node at (0-.7,2) {$C^1_{=1}$};
\node at (0-.7,2.8) {$C^1_{=2}$};
\draw (1,6) to [out=270, in=110] (2,3) to [out=290, in=110] (1,-1);
\draw (1.5,6) to [out=-45, in=110] (2,3) to [out=290, in=110] (1,-1);
\draw (3,6) to [out=-45, in=110] (5,3) to [out=290, in=110] (1,-1);
\draw (4,6) to [out=-90, in=110] (6,3) to [out=290, in=90] (4,-1);
\draw (5,6) to [out=200, in =180] (9,4);
\node at (.6,6.5+.3) {$C_{=0}^2$};
\node at (1.6,6.5+.3) {$C_{=1}^2$};
\node at (3,6.5+.3) {$C_{=2}^2$};
\node at (4,6.5+.3) {$C_{=3}^2$};
\node at (5.2,6.5+.3) {$C_{=4}^2$};
\end{tikzpicture}
\caption{A cut $([2],[4])$-grid on $\RR^2$.}
\end{figure}
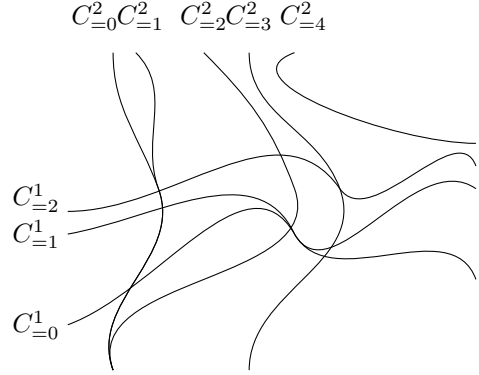

The notion of a cut ${\bf m}$-grid will allow us to define a higher categorical structure on bordisms which is the $d$-fold generalization of a double category.
A \emph{double category} is a categorical structure that has objects, horizontal 1-morphisms, vertical 1-morphisms, and 2-cells,
whose boundary is a square with two horizontal 1-morphisms and two vertical 1-morphisms.
Having multiple notions of 1-morphisms is an undesirable feature of a bordism category, since it destroys the properties
and constructions pertaining to the notion of dualizability in the setting of the cobordism hypothesis.
To eliminate this issue, we can pass from double categories to bicategories.
A \emph{bicategory} is a (pseudo) double category in which vertical morphisms are identities, a property also known as \emph{globularity}.
To get the correct generalization of a bicategory, we will need a \emph{globular} version of the above cut ${\bf m}$-grids.
Our cut ${\bf m}$-grids are general enough that we can simply extract those cut grids that satisfy the globular condition.

\begin{figure}
$$
\xymatrix{
\bullet \ar[r]^-{f}="1" \ar[d]_{a} & \bullet\ar[d]_{b}\ar[r]^-{g}="3" & \bullet \ar[d]^-{c}
\\
\bullet\ar[r]_{h}="2" & \bullet \ar[r]_{i}="4" & \bullet
\ar@{}"1";"2"|(0.25){\ }="5"
\ar@{}"1";"2"|(0.75){\ }="6"
\ar@{}"3";"4"|(0.25){\ }="7"
\ar@{}"3";"4"|(0.75){\ }="8"
\ar@{=>}_{u}"5";"6"
\ar@{=>}^-v "7";"8"
} \qquad \vcenter{\vspace{1cm}
\xymatrix{
\bullet \ar@/^2pc/[r]^-{f}="1"\ar@/_2pc/[r]_{h}="2" & \bullet \ar@/^2pc/[r]^-{g}="3"\ar@/_2pc/[r]_-{i}="4" & \bullet
\ar@{}"1";"2"|(0.25){\ }="5"
\ar@{}"1";"2"|(0.75){\ }="6"
\ar@{}"3";"4"|(0.25){\ }="7"
\ar@{}"3";"4"|(0.75){\ }="8"
\ar@{=>}_{u}"5";"6"
\ar@{=>}^-v "7";"8"
}}
$$
\caption{The left picture illustrates a composable pair of 2-cells in a double category.
The 1-morphisms on the boundary of the two 2-cells $u$ and $v$ can be either vertical or horizontal with no additional constraints.
In contrast, the right picture illustrates a composable pair of 2-morphisms in a bicategory.
In this case, the vertical 1-morphisms on the boundary of the two 2-cells are forced to be identities, while the horizontal 1-morphisms are allowed to be arbitrary.}
\end{figure}
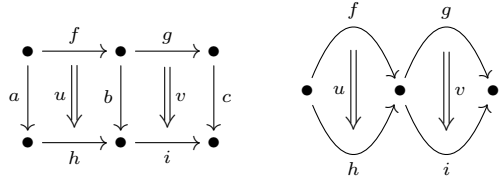

\begin{notation}
\label{vijcut}
Fix $d\geq 0$.
Let $p:M\to U$ be an object in $\smFEmb_d$ and let ${\bf m}=([m_1],\ldots,[m_d])\in \Delta^{\times d}$ be a multisimplex.
Given $i∈\{1,…,d\}$ and $j∈[m_i]$, we define a multisimplex ${\bf m}^i$ and a morphism $V^i_j:{\bf m}^i→{\bf m}$ as follows:
$${\bf m}^i_k=\begin{cases}[m_k],&k≠i;\cr[0],&k=i;\cr\end{cases}\qquad\qquad (V^i_j)_k=\begin{cases}\id_{[m_k]},&k≠i;\cr[0]≅\{j\}\into[m_i],&k=i.\cr\end{cases}$$
For each such $i$ and $j$, we denote the corresponding presheaf structure map by
$$v^i_j=\tCut(V^i_j,p):\tCut({\bf m},p:M\to U)\to \tCut({\bf m}^i,p:M\to U).$$
\end{notation}

\begin{definition}
\label{globgrid}
Assume \cref{vijcut}.
We let
 $$\tCutglob({\bf m},p:M\to U)\subset \tCut({\bf m},p:M\to U)$$
be the subset comprising cut ${\bf m}$-grids $C$
such that for all $i∈\{1,…,d\}$ and $j∈[m_i]$
the cut grid
$$D=v^i_jC\in \tCut({\bf m}^i,p:M\to U)$$
(\cref{vijcut}) satisfies the following property.
\begin{itemize}
\item
The closed subset $D_{[0,{\bf m}^i]}⊂M$ (\cref{cut.grid.core}) admits an open neighborhood $N⊂M$
such that the restriction of~$D$ to $p\vert_{N}:N\to U$
is a simplicial degeneration of a cut grid in
$$\tCut(([m_1],\ldots,[m_{i-1}],[0],\ldots,[0]), p\vert_{N}:N\to U),$$
in the simplicial directions $i+1,i+2,\ldots,d$.
\end{itemize}
The above property is preserved under pullback of cut grids, so that $\tCutglob$ defines a functor on $(\Delta^{\times d})^\op\times \smFEmb_d^\op$.
We call an element of $\tCutglob({\bf m},p:M\to U)$ a \emph{globular cut ${\bf m}$-grid}.
If a globular cut ${\bf m}$-grid is compact, we call it a \emph{compact globular cut ${\bf m}$-grid}.
\end{definition}

The next example illustrates cut ${\bf m}$-grids and globular cut ${\bf m}$-grids in the case $d=2$.

\begin{example}
\label{globular.cut.grid.example}
Let $d=2$, $U=\RR^0$, $p:\RR^2\to \RR^0$, and ${\bf m}=([0],[1])$.
We define a cut ${\bf m}$-grid $C$ on~$p$ as follows.
Fix $\epsilon>0$ and let $\psi:\RR\to[0,∞)$ be a smooth function satisfying $\psi(x)=0$ for all $x\in (-\epsilon,\epsilon)$
and $\psi(x)=1$ for all $x\in \RR\setminus (-2\epsilon,2\epsilon)$.
Take $C^1_0=(x<0, x=0, x>0)$, $C^2_0=(y<0,y=0,y>0)$, and $C^2_1=(y<\psi, y=\psi,y>\psi)$.
An image of the cut grid is depicted below.
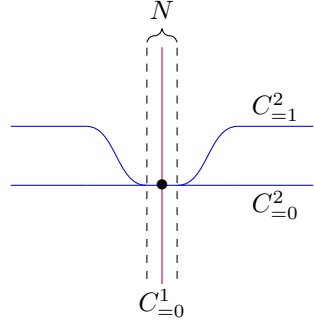
\begin{center}
\begin{tikzpicture}[yscale=1.3]
\draw [decorate, decoration={brace, amplitude=5pt, raise=2pt}] (-0.2,1.2) -- (0.2,1.2) node [black, midway, yshift=0.5cm] {$N$};
\node at (0,1.3) (t) {};
\node at (0,-1.3) (s) {};
\node at (0,-1.4) (t') {$C_{=0}^1$};
\draw[purple] (t) -- (s);
\node at (-1,-.2) (q) {};
\node at (1,-.2) (p) {};
\node at (1.5,-.4) (q') {$C_{=0}^2$};
\draw[blue] (q.center) -- (p.center);
\node at (-1,.4) (q) {};
\node at (1,.4) (p) {};
\node at (1.5,.6) (q') {$C_{=1}^2$};
\node at (0,-.2) (int') { $\bullet$};
\draw[blue] (q.center) to [out=0,in=180] (int');
\draw[blue] (int') to [out=0,in=180] (p.center);
\node at (-2,.4) (1) {};
\node at (-1,.4) (2) {};
\node at (-2,-.2) (3) {};
\node at (-1,-.2) (4) {};
\node at (1,.4) (5) {};
\node at (2,.4) (6) {};
\node at (1,-.2) (7) {};
\node at (2,-.2) (8) {};
\draw[blue] (1.center) -- (2.center);
\draw[blue] (3.center) -- (4.center);
\draw[blue] (5.center) -- (6.center);
\draw[blue] (7.center) -- (8.center);
\draw[dashed] (0.2,1.2) -- (0.2,-1.2);
\draw[dashed] (-0.2,1.2) -- (-0.2,-1.2);
\end{tikzpicture}
\end{center}
We claim that this cut grid is globular in the sense of \cref{globgrid}.
The nonvacuous globularity conditions use~$i$ such that $1≤i<d$.
In our case, $i=1$.
Since $[m_1]=[0]$, i.e., there is only one cut in the first direction, we have
$j=0$,
${\bf m}^i_j={\bf m}=([0],[1])$,
$v^i_j=\id$,
and $D=C$.

The globularity condition says that there is an open neighborhood $N\subset \RR^2$ of the closed set $D_{[0,{\bf m}^i_j]}=\{(0,0)\}$
such that restricting the cut grid to $N$ yields a cut grid that is degenerate in the second direction.
Take $N=(-\epsilon,\epsilon)⨯\RR$, as depicted by the region bounded by the dashed vertical lines above.
Then restricting $D$ to $N$ yields a degenerate cut grid in the second direction, since $\psi=0$ on the interval $(-\epsilon,\epsilon)$.

Another example of a globular cut grid on $\RR^2$, taking ${\bf m}=([1],[1])$, is depicted below.
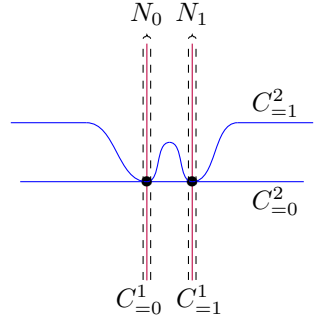
\begin{center}
\begin{tikzpicture}[yscale=1.3]
\node at (-.2,1.3) (t) {};
\node at (-.2,-1.3) (s) {};
\node at (-.3,-1.4) (t') {$C_{=0}^1$};
\draw[purple] (t) -- (s);
\node at (.4,1.3) (tt) {};
\node at (.4,-1.3) (ss) {};
\node at (.5,-1.4) (tt') {$C_{=1}^1$};
\draw[purple] (tt) -- (ss);
\node at (-2,-.2) (q) {};
\node at (2,-.2) (p) {};
\node at (1.5,-.4) (q') {$C_{=0}^2$};
\draw[blue] (q) --(p);
\node at (-1,.4) (qq) {};
\node at (1,.4) (pp) {};
\node at (1.5,.6) {$C_{=1}^2$};
\node at (-.2,-.2) (int) {$\bullet$};
\node at (.4,-.2) (int') {$\bullet$};
\node at (.1,.2) (i) {};
\draw[blue] (qq.center) to [out=0,in=180] (int.center);
\draw[blue] (int.center) to [out=0,in=180] (i.center);
\draw[blue] (i.center) to [out=0,in=180] (int'.center);
\draw[blue] (int'.center) to [out=0,in=180] (pp.center);
\node at (-2,.4) (3) {};
\node at (-1,.4) (4) {};
\node at (1,.4) (7) {};
\node at (2,.4) (8) {};
\draw[blue] (3.center) -- (4.center);
\draw[blue] (7.center) -- (8.center);
\draw[dashed] (-.25,-1.2) -- (-.25,1.2);
\draw[dashed] (-.15,-1.2) -- (-.15,1.2);
\draw[dashed] (.45,-1.2) -- (.45,1.2);
\draw[dashed] (.35,-1.2) -- (.35,1.2);
\draw [decorate, decoration={brace, amplitude=1.5pt, raise=2pt}] (-.25,1.2) -- (-.15,1.2) node [black, midway, yshift=0.4cm] {$N_0$};
\draw [decorate, decoration={brace, amplitude=1.5pt, raise=2pt}] (.35,1.2) -- (.45,1.2) node [black, midway, yshift=0.4cm] {$N_1$};
\end{tikzpicture}
\end{center}
In this case, the structure map $v^1_j$ removes $C^1_j$, where $j∈\{0,1\}$.
The resulting cut grid satisfies the condition of \cref{globgrid}, with the neighborhoods $N_j$ depicted by the regions bounded by the dashed vertical lines.
\end{example}

\begin{example}
Set $d=2$.
The following images depict cut $([1],[0])$- and $([1],[1])$-grids on a $2$-manifold given by the gray sheet
(as an object in $\smFEmb_2$, we take the base of the submersion to be~$\RR^0$).
The image on the left is a cut $([1],[0])$-grid and the image in the center is a cut $([1],[1])$-grid.
The image on the right depicts a globular cut $([1],[1])$-grid.
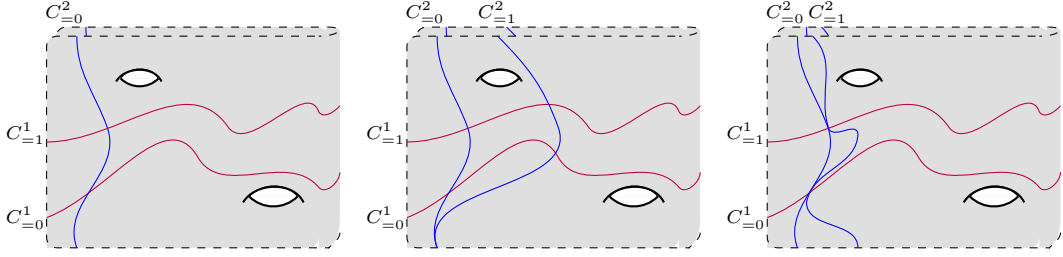
\begin{center}
\begin{tikzpicture}[scale=.40]
\fill[fill=gray!25, rounded corners] (0+.3,-1+.3) -- (0+.3,6+.3) -- (9+.3,6+.3) -- (9+.3,-1+.3) -- (0+.3,-1+.3);
\draw[ blue] (1+.3,6+.3) to [out=270, in=110] (2+.3,3+.3);
\draw[dashed, rounded corners] (0+.3,-1+.3) -- (0+.3,6+.3) -- (9+.3,6+.3) ;
\fill[fill=gray!25, rounded corners] (0,-1) -- (0,6) -- (9,6) -- (9,-1) -- (0,-1);
\draw[dashed, rounded corners] (9.7,-.5) to[out=270,in=0, looseness=.8] (9,-1) -- (0,-1) -- (0,6) -- (9,6);
\fill[fill=gray!25, dashed] (9,6) -- (9,-1) to[out=0,in=270, looseness=1.2] (9.7,.5) -- (9.7,6.2);
\filldraw[fill=gray!25, dashed] (9+.3,6+.3) to [out=0, in=0, looseness=4] (9,6);
\draw [purple] (0,0) to [out=20,in=120] (5,2) to [out=-60, in=130] (9,1) to [out=310, in=-90] (9.7,1.5);
\draw [purple] (0,2.5) to [out=0, in=125] (6,3) to [out=305, in=110] (9,3.5) to [out=300, in=-130] (9.7,3.7);
\node at (0-.7,0) {$\scriptstyle C^1_{=0}$};
\node at (0-.7,2.8) {$\scriptstyle C^1_{=1}$};
\draw[ blue] (1,6) to [out=270, in=110] (2,3) to [out=290, in=110] (1,-1);
\node at (.6,6.5+.3) {$\scriptstyle C_{=0}^2$};
\draw[fill=white] (2.65-.2,4.6) to [out=57, in=118] (3.9-.2,4.6);
\draw[thick] (2.5-.2,4.5) to [out=60, in=115] (4-.2,4.5);
\draw[thick, fill=white] (2.65-.2,4.6) to [out=-45, in=-135] (3.9-.2,4.6);
\draw[fill=white] (6.5+.2,.5+.2) to [out=46, in=130] (8.5-.2,.5+.2);
\draw[thick] (6.5,.5) to [out=60, in=115] (8.5,.5);
\draw[thick, fill=white] (6.5+.2,.5+.2) to [out=-45, in=-135] (8.5-.2,.5+.2);
\end{tikzpicture}
\begin{tikzpicture}[scale=.40]
\fill[fill=gray!25, rounded corners] (0+.3,-1+.3) -- (0+.3,6+.3) -- (9+.3,6+.3) -- (9+.3,-1+.3) -- (0+.3,-1+.3);
\draw[ blue] (1+.3,6+.3) to [out=270, in=110] (2+.3,3+.3);
\draw[ blue] (3+.3,6+.3) to [out=-45, in=110] (5+.3,3+.3);
\draw[dashed, rounded corners] (0+.3,-1+.3) -- (0+.3,6+.3) -- (9+.3,6+.3) ;
\fill[fill=gray!25, rounded corners] (0,-1) -- (0,6) -- (9,6) -- (9,-1) -- (0,-1);
\draw[dashed, rounded corners] (9.7,-.5) to[out=270,in=0, looseness=.8] (9,-1) -- (0,-1) -- (0,6) -- (9,6);
\fill[fill=gray!25, dashed] (9,6) -- (9,-1) to[out=0,in=270, looseness=1.2] (9.7,.5) -- (9.7,6.2);
\filldraw[fill=gray!25, dashed] (9+.3,6+.3) to [out=0, in=0, looseness=4] (9,6);
\draw [purple] (0,0) to [out=20,in=120] (5,2) to [out=-60, in=130] (9,1) to [out=310, in=-90] (9.7,1.5);
\draw [purple] (0,2.5) to [out=0, in=125] (6,3) to [out=305, in=110] (9,3.5) to [out=300, in=-130] (9.7,3.7);
\node at (0-.7,0) {$\scriptstyle C^1_{=0}$};
\node at (0-.7,2.8) {$ \scriptstyle C^1_{=1}$};
\draw[ blue] (1,6) to [out=270, in=110] (2,3) to [out=290, in=110] (1,-1);
\draw[ blue] (3,6) to [out=-45, in=110] (5,3) to [out=290, in=110] (1,-1);
\node at (.6,6.5+.3) {$\scriptstyle C_{=0}^2$};
\node at (3,6.5+.3) {$ \scriptstyle C_{=1}^2$};
\draw[fill=white] (2.65-.2,4.6) to [out=57, in=118] (3.9-.2,4.6);
\draw[thick] (2.5-.2,4.5) to [out=60, in=115] (4-.2,4.5);
\draw[thick, fill=white] (2.65-.2,4.6) to [out=-45, in=-135] (3.9-.2,4.6);
\draw[fill=white] (6.5+.2,.5+.2) to [out=46, in=130] (8.5-.2,.5+.2);
\draw[thick] (6.5,.5) to [out=60, in=115] (8.5,.5);
\draw[thick, fill=white] (6.5+.2,.5+.2) to [out=-45, in=-135] (8.5-.2,.5+.2);
\end{tikzpicture}
\begin{tikzpicture}[scale=.40]
\fill[fill=gray!25, rounded corners] (0+.3,-1+.3) -- (0+.3,6+.3) -- (9+.3,6+.3) -- (9+.3,-1+.3) -- (0+.3,-1+.3);
\draw[ blue] (1+.3,6+.3) to [out=270, in=110] (2+.3,3+.3);
\draw[ blue] (1.5+.3,6+.3) to [out=-45, in=110] (2+.3,3+.3);
\draw[dashed, rounded corners] (0+.3,-1+.3) -- (0+.3,6+.3) -- (9+.3,6+.3) ;
\fill[fill=gray!25, rounded corners] (0,-1) -- (0,6) -- (9,6) -- (9,-1) -- (0,-1);
\draw[dashed, rounded corners] (9.7,-.5) to[out=270,in=0, looseness=.8] (9,-1) -- (0,-1) -- (0,6) -- (9,6);
\fill[fill=gray!25, dashed] (9,6) -- (9,-1) to[out=0,in=270, looseness=1.2] (9.7,.5) -- (9.7,6.2);
\filldraw[fill=gray!25, dashed] (9+.3,6+.3) to [out=0, in=0, looseness=4] (9,6);
\draw [purple] (0,0) to [out=20,in=120] (5,2) to [out=-60, in=130] (9,1) to [out=310, in=-90] (9.7,1.5);
\draw [purple] (0,2.5) to [out=0, in=125] (6,3) to [out=305, in=110] (9,3.5) to [out=300, in=-130] (9.7,3.7);
\node at (0-.7,0) {$\scriptstyle C^1_{=0}$};
\node at (0-.7,2.8) {$\scriptstyle C^1_{=1}$};
\draw[ blue] (1,6) to [out=270, in=110] (2,3) to [out=290, in=110] (1,-1);
\draw[ blue] (1.5,6) to [out=-45, in=110] (2.05,2.9) to [out=-45, in =110] (3,2.8) to [out=-90, in=80] (1.3,.5) to [out=-90,in=90] (3,-1);
\node at (.6,6.5+.3) {$\scriptstyle C_{=0}^2$};
\node at (2,6.5+.3) {$\scriptstyle C_{=1}^2$};
\draw[fill=white] (2.65-.2,4.6) to [out=57, in=118] (3.9-.2,4.6);
\draw[thick] (2.5-.2,4.5) to [out=60, in=115] (4-.2,4.5);
\draw[thick, fill=white] (2.65-.2,4.6) to [out=-45, in=-135] (3.9-.2,4.6);
\draw[fill=white] (6.5+.2,.5+.2) to [out=46, in=130] (8.5-.2,.5+.2);
\draw[thick] (6.5,.5) to [out=60, in=115] (8.5,.5);
\draw[thick, fill=white] (6.5+.2,.5+.2) to [out=-45, in=-135] (8.5-.2,.5+.2);
\end{tikzpicture}
\end{center}
\end{example}

Finally, we treat the factor~$Γ$ in a manner analogous to~$Δ^{⨯d}$.
Observe that the monoidal structure on bordisms is given by the disjoint union,
and disjoint unions of bordisms can be interpreted as compositions in a direction where all sources and targets are empty.
In our setting, a cut $[m]$-tuple $C=(C_0,…,C_m)$ on $p:M→U$ such that for every~$i$ we have $C_{=i}=∅$
partitions $M$ into disjoint open subsets (\cref{cut.grid.core}):
$$M=C_{<0}⊔C_{(0,1)}⊔⋯⊔C_{(m-1,m)}⊔C_{>m}.$$
We think of such a cut $[m]$-tuple as a collection of $m$ bordisms $C_{(0,1)}$, …, $C_{(m-1,m)}$,
together with two “trash bins” $C_{<0}$ and $C_{>m}$, which are discarded once we pass to a sufficiently small neighborhood of $C_{[0,m]}$.
Such a disjoint partition can also be encoded by a smooth map $M→\{-∞,1,…,m,∞\}$,
which describes almost the same data as the map $C^⊗:M→⟨ℓ⟩$ in \cref{monoidal.cut.grid} (taking $ℓ=m$),
with two important changes.
\begin{itemize}
\item
The totally ordered set of intervals
$$\{(-∞,0),(0,1),…,(m-1,m),(m,∞)\}≅\{-∞,1,…,m,∞\}$$ indexing the disjoint bordisms
is replaced by the (unordered) pointed finite set $$⟨ℓ⟩=\{*,1,…,ℓ\}.$$
\item Instead of two trash bins $C_{<0}$ and $C_{>m}$ indexed by the elements $(-∞,0)$ and $(0,∞)$,
we have a single trash bin indexed by the element $*∈⟨ℓ⟩$.
\end{itemize}

\begin{definition}
\label{monoidal.cut.grid}
Fix $d\geq 0$.
We define the \emph{monoidal cut grid functor}
$$\mtCut:Γ^\op⨯(\Delta^{\times d})^\op\times \smFEmb_d^\op\to \set,$$
as follows.
Let $p:M\to U$ be an object in $\smFEmb_d$, $⟨ℓ⟩$ be an object in~$Γ$, and ${\bf m}=([m_1],\ldots,[m_d])\in \Delta^{\times d}$ be a multisimplex.
The set $\mtCut({\bf m},p)$ has as elements ordered pairs:
\begin{itemize}
\item a smooth map $C^⊗:M→⟨ℓ⟩$;
\item a cut ${\bf m}$-grid $(C^1,…,C^d)∈\tCut({\bf m},p)$ (\cref{cutgrid}).
\end{itemize}
The structure maps for the factor $Γ$ leave $C^i$ unchanged and compose $C^⊗:M→⟨ℓ⟩$ with the given map of pointed finite sets $⟨ℓ⟩→⟨ℓ'⟩$.
The structure maps for the factor $Δ^{⨯d}$ leave $C^⊗$ unchanged and for $(C^1,…,C^d)$ are given in \cref{cutgrid}.
The structure maps for the factor $\smFEmb_d$ are given by composing $C^⊗:M→⟨ℓ⟩$ with the given map $M'→M$ of total spaces
and applying the corresponding structure map to the cut grid~$(C^1,…,C^d)$.
We call the elements of $\tCut(⟨ℓ⟩,{\bf m},p)$ \emph{monoidal cut $(⟨ℓ⟩,{\bf m})$-grids}.
\end{definition}

We introduce notation similar to \cref{cut.grid.core}, but incorporating the monoidal structure.

\begin{notation}
Recall \cref{notation.cut} and \cref{cut.grid.core}.
For a monoidal cut $(⟨ℓ⟩,{\bf m})$-grid $C$, a subset $S\subset \{1,\ldots,d\}$ and $j,j':S→\ZZ$ satisfying $0\leq j_i\leq j'_i\leq m_i$ for all $i∈S$,
we use the following notation.
\begin{itemize}
\item
\label{monoidal.cut.grid.core}
\label{nonembedded.core}
We define $$C_{⊗,[j,{j'}]}≔C_{[j,{j'}]}∩(C^⊗)^{-1}(⟨ℓ⟩∖\{*\})\subset M.$$
These subsets are the regions between corresponding cuts in various directions, with the trash bin discarded.
In the case where $S=\{1,\ldots,d\}$, $j=0$ and $j'(i)=m_i$, we set
$$\ncore(p,C)= C_{⊗,[j,j']}$$
and call it the \emph{core} of the monoidal cut grid~$C$.

\item We define $C_{⟨j,{j'}⟩}$ to be the monoidal cut $(⟨ℓ⟩,{\bf m}')$-grid
with the same map $C^⊗$ and the cut grid $C_{⟨j,{j'}⟩}$, with ${\bf m}'$ defined accordingly.
\end{itemize}
\end{notation}

\begin{definition}
\label{compactmonoidalgrid}
Let $C$ be a monoidal cut $(⟨ℓ⟩,{\bf m})$-grid for $p:M\to U$ (\cref{monoidal.cut.grid}).
We say that $C$ is \emph{compact} if for $S=\{1,2,\ldots,d\}$ and all $j,j':S\to \ZZ$, the restriction of $p$ to $C_{⊗,[j,j']}$ (\cref{monoidal.cut.grid.core})
is a proper map of topological spaces (\cref{proper.maps}).

\label{globular.monoidal.cut.grid}
Let $$\mtCutglob(⟨ℓ⟩,{\bf m},p:M\to U)\subset \mtCut(⟨ℓ⟩,{\bf m},p:M\to U)$$
be the subset comprising pairs $(C^⊗,C)$, where the cut grid~$C$ restricted to $(C^⊗)^{-1}(⟨ℓ⟩∖\{*\})$ is globular (\cref{globgrid}).
The above property is preserved under pullback of cut grids, so that $\mtCutglob$ defines a functor
$$\mtCutglob:Γ^\op⨯(\Delta^{\times d})^\op\times \smFEmb_d^\op→\set.$$
We call an element of $\mtCutglob(⟨ℓ⟩,{\bf m},p:M\to U)$ a \emph{globular monoidal cut $(⟨ℓ⟩,{\bf m})$-grid}.
If a globular monoidal cut $(⟨ℓ⟩,{\bf m})$-grid is compact, we call it a \emph{compact globular monoidal cut $(⟨ℓ⟩,{\bf m})$-grid}.
\end{definition}

In order to define the isotopy version of the bordism category, we will need to allow for smooth families of monoidal cut grids.
The smooth family version is easily obtained from the above as follows.

\begin{definition}
\label{monoidal.cut.grid.smooth}
\label{globular.monoidal.cut.grid.smooth}
Recall the categories $\smset$ (\cref{def.smset}) and $\fraksmFEmb_d$ (\cref{def.fraksmFEmb}) and the functor $\mtCut$ (\cref{monoidal.cut.grid}).
We define $\smset$-enriched functors
$$\frakmtCut:Γ^\op⨯(\Delta^{\times d})^\op\times \fraksmFEmb_d^\op\to \smset,
\qquad\frakmtCutglob⊂\frakmtCut$$
as follows.
\begin{itemize}
\item
On objects, we set $$\frakmtCut(⟨ℓ⟩,{\bf m},p:M→U)_L=\mtCut(⟨ℓ⟩,{\bf m},p⨯\id_L:M⨯L→U⨯L),$$
where $L∈\cart$ is an object in the site used to define smooth sets (\cref{def.smset}).
\item The structure maps for $Γ⨯Δ^{⨯d}$ are induced by the structure maps of $\mtCut$.
\item
The enriched structure maps for $\fraksmFEmb_d$ are defined
for a fixed $(⟨ℓ⟩,{\bf m})∈Γ⨯Δ^{⨯d}$ (omitted from notation below) as follows.
For $X,Y\in \smset$, we let $[X,Y]$ denote the internal hom.
Given objects $p_1:M_1→U_1$ and $p_2:M_2→U_2$, we construct the morphism of smooth sets
$$(\frakmtCut)_{p_1,p_2}:\fraksmFEmb_d(p_1,p_2)→[\frakmtCut(p_2),\frakmtCut(p_1)]$$
by passing to the adjoint map
$$\fraksmFEmb_d(p_1,p_2)\times \frakmtCut(p_2)→\frakmtCut(p_1)$$
and defining its $L$-points to be the map of sets
$$\fraksmFEmb_d(p_1,p_2)_L\times \frakmtCut(p_2)_L→\frakmtCut(p_1)_L,\qquad (φ,α)↦φ^*α.$$
The subfunctor $\frakmtCutglob⊂\frakmtCut$ is constructed using $\mtCutglob$ instead of $\mtCut$.
\end{itemize}
\end{definition}

\subsection{Categories of bordisms}
\label{categories.of.bordisms}

In this section, we define two variants of the geometric symmetric monoidal $(\infty,d)$-category of bordisms:
one which includes isotopies for fields (\cref{bord.isotopy}) and one without (\cref{bord}).
In both cases, roughly speaking, a composable tuple of $k$-morphisms with $k\leq d$ is an object $p:M\to U$ of $\FEmb_d$,
together with a compact globular monoidal cut grid
(\cref{globular.monoidal.cut.grid})
that is degenerate in the last $(d-k)$-directions,
along with a vertex in $\gs(p:M\to U)$, where $\gs∈\Struct_d$ is a field stack (\cref{geometric.structure}).
In the case of the isotopy bordism category, we have isotopies of fields (\cref{geometric.structure.isotopy}) and compact monoidal cut grids
(\cref{globular.monoidal.cut.grid.smooth}) as $k$-morphisms for $k>d$.
See \cref{bord.explicit,bord.isotopy.explicit} for more details.

We begin with a version of the bordism category that allows for noncompact bordisms.
This version will be used as an intermediate step and we will pass to the subobject on compact bordisms.

\begin{definition}
\label{noncompact.bord}
Fix $d\geq 0$ and a fibered geometric site $\FEmb_d$ (\cref{fibered.geometric.site}), including the site $\stcart$ and the reduction functor $\red=\red_t:\FEmb_d→\smFEmb_d$.
Recall $\Struct_d$ from \cref{geometric.structure} and fix $\gs\in \Struct_d$.
Recall also the globular monoidal cut grid functor $\mtCutglob$ from \cref{globular.monoidal.cut.grid}.
Consider the functor
\begin{equation}\label{grothcutfunct}Γ^\op⨯(Δ^{⨯d})^\op→\PSh(\FEmb_d,\sset), \qquad (⟨ℓ⟩,{\bf m})↦((p:M→U)↦\mtCutglob(⟨ℓ⟩,{\bf m},\red p)⨯\gs(p)).\end{equation}
Apply degreewise the relative Grothendieck construction functor~$\relgro$ (\cref{relative.grothendieck}) to obtain a functor
$$Γ^\op⨯(Δ^{⨯d})^\op\to \PSh(\FEmb_d,\sset)\lto7{(\relgro)^{Δ^\op}}\PSh(\stcart,\smallcat^{Δ^\op}).$$
Moving the category $\stcart$ to the left side using currying yields a functor
\begin{equation}\label{nccatbord}\ncBord^\gs_d:\stcart^\op⨯Γ^\op⨯(Δ^{⨯d})^\op\to \smallcat^{Δ^\op},\end{equation}
which we call the \emph{geometric symmetric monoidal $(\infty,d)$-category of noncompact bordisms with $\gs$-structure}.
The construction $\gs↦\ncBord_d^\gs$ is $\sset$-enriched functorial in~$\gs$.
\end{definition}

The previous definition has the advantage that it is manifestly functorial in $\Gamma$, $\Delta^{\times d}$, $\stcart$, $\gs$, and $\mtCutglob$,
however it is not clear what the construction has to do with bordisms.
This will be explained after the unwinding of the definition of the bordism category in \cref{bord.explicit} (see also \cref{examplebords} for many examples).

In the next proposition, we give an explicit description of the value of the functor $\ncBord_d^{\gs}$ on a triple $(U,\langle \ell\rangle, {\bf m})\in \stcart\times \Gamma\times \Delta^{\times d}$.

\begin{proposition}
\label{explicit.noncompact.bords}
The functor $\ncBord^{\gs}_d$ in \cref{nccatbord} has the following explicit description.
Let $(U,\langle \ell\rangle,{\bf m})\in \stcart\times \Gamma\times \Delta^{\times d}$ and $[k]\in \Delta$.
The category of $k$-simplices $\ncBord^\gs_d(U,\langle \ell\rangle,{\bf m})_k$ is given by the following.
\begin{itemize}
\item Objects are triples $(p:M\to U,C,\sigma)$, where $C$ is a globular monoidal cut $(\langle \ell\rangle,{\bf m})$-grid on $\red p$ and $\sigma\in \gs(p:M\to U)_k$ is a $k$-simplex.
\item Morphisms $(p,C,\sigma)\to (q,C',\sigma')$ are maps $\varphi:p\to q\in \FEmb_d$ such that $\base \varphi=\id_U$, $\varphi^*C'=C$ and $\varphi^*\sigma'=\sigma$,
where the functor $\base:\FEmb_d→\stcart$ is given by the data of a fibered geometric site (\cref{fibered.geometric.site}).
We call a morphism a \emph{cut-respecting embedding}.
\end{itemize}
\end{proposition}
\begin{proof}
This is a direct application of \cref{explicitgroth}, taking $F$ to be the functor obtained by evaluating \cref{grothcutfunct}
at a fixed object $(\langle \ell\rangle,{\bf m})∈Γ⨯Δ^{⨯d}$ and $[k]\in \Delta$.
\end{proof}
Taking $\FEmb_d=\smFEmb_d$ and $\gs=\munit$,
for every $(U,⟨ℓ⟩,{\bf m})∈\stcart⨯Γ⨯Δ^{⨯d}$, the category $$\cC=\ncBord_d^\gs(U,⟨ℓ⟩,{\bf m})$$ has an initial object $(p:M→U,C,σ)$, where $M=∅$.
Thus, the nerve of $\cC$ is weakly contractible.
The geometric meaning of this is that morphisms in $\cC$ need not preserve the core of the monoidal cut grid.
This is addressed in the next proposition, where we construct a subobject $\CatBord_d$ of $\ncBord_d$ which will be used to construct the bordism category.

\begin{proposition}
\label{compactbords}
Recall \cref{noncompact.bord} and let $\ncBord^\gs_d$ be the functor defined by \cref{nccatbord}.
Recall the explicit description of $\ncBord_d^\gs$ in \cref{explicit.noncompact.bords}.
For all $(U,⟨ℓ⟩,{\bf m})∈\stcart⨯Γ⨯Δ^{⨯d}$ and $[k]\in \Delta$, there is a subcategory
$$\CatBord_d^\gs(U,⟨ℓ⟩,{\bf m})_k⊂\ncBord_d^{\gs}(U,⟨ℓ⟩,{\bf m})_k$$
defined by the following conditions.
\begin{conditions}[(c1)]
\item
\label{bordism.compact}
On objects $(p,C,\sigma)$, we require the globular monoidal cut grid~$C$ to be compact (\cref{compactmonoidalgrid}).
\item
\label{bordism.core}
On morphisms $\varphi$, we require the image of $\red\varphi$ to contain $\ncore(\red p',C')$ (\cref{monoidal.cut.grid.core}).
\end{conditions}
Moreover, $\CatBord_d^{\gs}$ defines a subpresheaf of $\ncBord^\gs_d$,
which we call the \emph{geometric symmetric monoidal $(\infty,d)$-category of compact bordisms with $\gs$-structure}.
The construction $\gs↦\CatBord_d^\gs$ is $\sset$-enriched functorial in~$\gs$.
\end{proposition}

\begin{proof}
We first prove that $\CatBord_d^\gs(U,⟨ℓ⟩,{\bf m})_k$ is a well defined category.
By \cref{explicit.noncompact.bords},
a morphism $\varphi:(p,C,\sigma)\to (p',C',\sigma')$ is a map $\varphi:p\to p'$ such that $\base φ=\id_U$, $φ^*C'=C$, and $\varphi^*\sigma'=\sigma$.
If $\varphi$ satisfies \cref{bordism.core}, then $\image\red\varphi⊃\ncore(\red p',C')$.
Therefore, the map~$\red φ$ restricts to a bijection
\begin{equation}\label{condition2image}\ncore(\red p,C)=\ncore(\red p,φ^*C')\mathrel{\ltoarr2^{≅}}\ncore(\red p',C')∩\image \red φ=\ncore(\red p',C').\end{equation}
Clearly identity morphisms satisfy \cref{bordism.core}.
Suppose we have morphisms
$$φ:(p,C,σ)→(p',C',σ'), \qquad \varphi':(p',C',\sigma')\to (p'',C'',\sigma'')$$
satisfying \cref{bordism.core}.
Then the composition $φ'φ:p→p''$ satisfies
$$\image\red(φ'φ)=\red φ'(\image\red φ)⊃\red φ'(\ncore(\red p',C'))=\ncore(\red p'',C'')$$
(with the last equality following from \cref{condition2image}), hence it satisfies \cref{bordism.core}.
This proves we have a well defined category.

It remains to show that $\CatBord_d^{\gs}$ is a well defined subpresheaf.
Hence, we must show that applying a structure map in $\stcart\times \Gamma\times \Delta^{\times d}$ preserves \cref{bordism.compact,bordism.core}.
It suffices to show that \cref{bordism.compact,bordism.core} are preserved by structure maps of the form $f\times \id_c$,
where $f$ is a morphism in one of the factors and $c$ is an object in the remaining two factors.

Let $f$ be a morphism in~$\Gamma$.
Given an object $(p,C,σ)$, we want to show that that the globular monoidal cut grid $f^*C$ is compact.
Denote by~$M$ the domain of~$\red p$.
The structure map corresponding to~$f$ repartitions the set of connected components of $M$ and moves some connected components to the trash bin.
Thus for each $j,j'$, the corresponding subset $(f^*C)_{⊗,[j,{j'}]}$ (\cref{monoidal.cut.grid.core})
is given by removing some connected components of $C_{⊗,[j,{j'}]}$.
Hence the restriction of $\red p$ to $(f^*C)_{\otimes,[j,j']}$ is still proper (\cref{proper.maps}), so \cref{bordism.compact} is preserved.
\cref{bordism.core} is preserved since $\ncore(\red p',f^*C')\subset \ncore(\red p',C')$ and the embedding~$φ$ does not change under the structure map of~$f$.

Let $f$ be a morphism in~$\Delta^{\times d}$.
We can assume that $f$ is a face or degeneracy map in a single simplicial direction.
Then $f$ either repeats a cut or omits a cut.
In both cases, $(f^*C)_{\otimes,[j,j']}\subset C_{\otimes,[j,j']}$ is a closed subset for all $j,j'$.
Hence the restriction of $p$ to $(f^*C)_{\otimes,[j,j']}$ is proper (\cref{proper.maps}) and \cref{bordism.compact} is preserved.
Again \cref{bordism.core} is preserved since $\ncore(\red p',f^*C')\subset \ncore(\red p',C')$ and the embedding~$φ$ does not change under the structure map of~$f$.

Let $f$ be a morphism in~$\stcart$.
Using \cref{explicitgroth} applied to the functor obtained by evaluating \cref{grothcutfunct} at a fixed object $(\langle \ell\rangle,{\bf m})\in \Gamma\times \Delta^{⨯d}$,
we see that the structure map of~$f$ pulls back the cut grid $C$ by the cartesian arrow $\red \base^*f: \red p'\to \red p$
where $\base^*f$ is the cartesian arrow in the splitting cleavage of the Grothendieck fibration~$\base$.
Since $\red$ preserves the splitting cleavage by \cref{fibered.geometric.site},
it follows that $\red \base^*f=\base^* \red f:\red p'\to \red p$ is the cartesian lift of $\red f$.

To verify \cref{bordism.compact}, we must show that the restriction of $\red p'$ to $(f^*C)_{\otimes,[j,j']}$ is proper for all $j,j'$.
First, by definition of the cartesian liftings for the base space functor $\base:\smFEmb_d\to \cart$ (\cref{smoothgrothfib}),
the map $\red\base^*f:\red p'\to \red p$ is the natural projection in the pullback diagram in the category $\man$:
$$\xymatrix@C=4em{
\red U'\times_{\red U} \red M\ar[r]^-{\red\base^*f}\ar[d]^-{\red p'} & \red M\ar[d]^-{\red p}\cr
\red U'\ar[r]^-{\red f} & \red U.\cr
}$$
Since the cut grid~$f^*C$ on $\red p'$ is obtained by pullback along $\red\base^*f$, it follows that for each $j,j'$, we have a pullback diagram of topological spaces:
$$\xymatrix@C=3em{
(f^*C)_{\otimes,[j,j']}\ar[r]^-{\red\base^*f}\ar[d]^-{\red p'} & C_{\otimes,[j,j']}\ar[d]^-{\red p}\cr
\red U'\ar[r]^-{\red f} & \red U.\cr
}$$
By assumption, $\red p$ is proper.
Since proper maps are closed under base change (\cref{proper.maps}), it follows that $\red p':(f^*C)_{\otimes,[j,j']}\to \red U'$ is also proper.

To verify \cref{bordism.core}, suppose a map $φ:p→q$ is a morphism $(p,C,σ)→(q,C',σ')$.
We have $$\image(\red f^*φ)=(\red\base^*f)^{-1}(\image(\red φ))⊃(\red\base^*f)^{-1}(\ncore(\red p',C'))=\ncore(\red f^*p',f^*C').\qedhere$$
\end{proof}

We provide an example of an object in both the noncompact (\cref{explicit.noncompact.bords}) and compact (\cref{compactbords}) bordism categories in \cref{figure2}.
We also provide an example of a morphism in the compact bordism category (\cref{compactbords}) in \cref{morphismcatbords}.

\begin{figure}[ht]
\begin{center}
\begin{tikzpicture}[scale=.60]
\filldraw[fill=gray!25, rounded corners, dashed] (0+.3,-1+.3) -- (0+.3,6+.3) -- (9+.3,6+.3) -- (9+.3,-1+.3) -- (0+.3,-1+.3);
\draw[purple] (5+.3,2+.3) to [out=-60, in=110] (9+.3,1+.3);
\draw[purple] (5+.3,2+.3) to [out=300, in=145] (9+.3,3+.3);
\draw[purple] (6+.3,3+.3) to [out=305, in=110] (9+.3,3.5+.3);
\draw[blue] (1+.3,6+.3) to [out=270, in=110] (2+.3,3+.3);
\draw[blue] (1.5+.3,6+.3) to [out=-45, in=110] (2+.3,3+.3);
\draw[blue] (3+.3,6+.3) to [out=-45, in=110] (5+.3,3+.3);
\draw[blue] (4+.3,6+.3) to [out=-90, in=110] (6,3);
\draw[orange] (5+.3,6+.3) to [out=200, in =180] (9+.3,4+.3);
\filldraw[fill=gray!25, rounded corners, dashed] (0,-1) -- (0,6) -- (9,6) -- (9,-1) -- (0,-1);
\draw[purple] (0,0) to [out=20,in=120] (5,2) to [out=-60, in=110] (9,1);
\draw[purple] (0,2) to [out=10,in=115] (5,2) to [out=300, in=145] (9,3);
\draw[purple] (0,2.5) to [out=0, in=125] (6,3) to [out=305, in=110] (9,3.5);
\node at (0-.7,0) {$C^1_{=0}$};
\node at (0-.7,2) {$C^1_{=1}$};
\node at (0-.7,2.8) {$C^1_{=2}$};
\draw[blue] (1,6) to [out=270, in=110] (2,3) to [out=290, in=110] (1,-1);
\draw[blue] (1.5,6) to [out=-45, in=110] (2,3) to [out=290, in=110] (1,-1);
\draw[blue] (3,6) to [out=-45, in=110] (5,3) to [out=290, in=110] (1,-1);
\draw[blue] (4,6) to [out=-90, in=110] (6,3) to [out=290, in=90] (4,-1);
\draw[orange] (5,6) to [out=200, in =180] (9,4);
\node at (.6,6.5+.3) {$C_{=0}^2$};
\node at (1.6,6.5+.3) {$C_{=1}^2$};
\node at (3,6.5+.3) {$C_{=2}^2$};
\node at (4.1,6.5+.3) {$C_{=3}^2$};
\node at (5.2,6.5+.3) {$C^2_{= 4}$};
\draw[fill=white] (2.65-.2,4.6) to [out=57, in=118] (3.9-.2,4.6);
\draw[thick] (2.5-.2,4.5) to [out=60, in=115] (4-.2,4.5);
\draw[thick, fill=white] (2.65-.2,4.6) to [out=-45, in=-135] (3.9-.2,4.6);
\draw[fill=white] (6.5+.2,.5+.2) to [out=46, in=130] (8.5-.2,.5+.2);
\draw[thick] (6.5,.5) to [out=60, in=115] (8.5,.5);
\draw[thick, fill=white] (6.5+.2,.5+.2) to [out=-45, in=-135] (8.5-.2,.5+.2);
\end{tikzpicture}
\begin{tikzpicture}[scale=.60]
\fill[fill=gray!25, rounded corners] (0+.3,-1+.3) -- (0+.3,6+.3) -- (9+.3,6+.3) -- (9+.3,-1+.3) -- (0+.3,-1+.3);
\draw[blue] (1+.3,6+.3) to [out=270, in=110] (2+.3,3+.3);
\draw[blue] (1.5+.3,6+.3) to [out=-45, in=110] (2+.3,3+.3);
\draw[blue] (3+.3,6+.3) to [out=-45, in=110] (5+.3,3+.3);
\draw[blue] (4+.3,6+.3) to [out=-90, in=110] (6,3);
\draw[blue] (5+.3,6+.3) to [out=200, in =180] (9+.3,4+.3);
\draw[dashed, rounded corners] (0+.3,-1+.3) -- (0+.3,6+.3) -- (9+.3,6+.3) ;
\fill[fill=gray!25, rounded corners] (0,-1) -- (0,6) -- (9,6) -- (9,-1) -- (0,-1);
\draw[dashed, rounded corners] (9.7,-.5) to[out=270,in=0, looseness=.8] (9,-1) -- (0,-1) -- (0,6) -- (9,6);
\fill[fill=gray!25, dashed] (9,6) -- (9,-1) to[out=0,in=270, looseness=1.2] (9.7,.5) -- (9.7,6.2);
\filldraw[fill=gray!25, dashed] (9+.3,6+.3) to [out=0, in=0, looseness=4] (9,6);
\draw[purple] (0,0) to [out=20,in=120] (5,2) to [out=-60, in=130] (9,1) to [out=310, in=-90] (9.7,1.5);
\draw[purple] (0,2) to [out=10,in=115] (5,2) to [out=300, in=145] (9,3) to [out=325, in=-110] (9.7,3.3);
\draw[purple] (0,2.5) to [out=0, in=125] (6,3) to [out=305, in=110] (9,3.5) to [out=300, in=-130] (9.7,3.7);
\node at (0-.7,0) {$D^1_{=0}$};
\node at (0-.7,2) {$D^1_{=1}$};
\node at (0-.7,2.8) {$D^1_{=2}$};
\draw[blue] (1,6) to [out=270, in=110] (2,3) to [out=290, in=110] (1,-1);
\draw[blue] (1.5,6) to [out=-45, in=110] (2,3) to [out=290, in=110] (1,-1);
\draw[blue] (3,6) to [out=-45, in=110] (5,3) to [out=290, in=110] (1,-1);
\draw[blue] (4,6) to [out=-90, in=110] (6,3) to [out=290, in=90] (4,-1);
\draw[blue] (5,6) to [out=200, in =180] (9,4) to [out=0, in=-130] (9.7,4.2);
\node at (.6,6.5+.3) {$D_{=0}^2$};
\node at (1.6,6.5+.3) {$D_{=1}^2$};
\node at (3,6.5+.3) {$D_{=2}^2$};
\node at (4.1,6.5+.3) {$D_{=3}^2$};
\node at (5.2,6.5+.3) {$D^2_{= 4}$};
\draw[fill=white] (2.65-.2,4.6) to [out=57, in=118] (3.9-.2,4.6);
\draw[thick] (2.5-.2,4.5) to [out=60, in=115] (4-.2,4.5);
\draw[thick, fill=white] (2.65-.2,4.6) to [out=-45, in=-135] (3.9-.2,4.6);
\draw[fill=white] (6.5+.2,.5+.2) to [out=46, in=130] (8.5-.2,.5+.2);
\draw[thick] (6.5,.5) to [out=60, in=115] (8.5,.5);
\draw[thick, fill=white] (6.5+.2,.5+.2) to [out=-45, in=-135] (8.5-.2,.5+.2);
\end{tikzpicture}
\end{center}
\caption{Let $\stcart=\cart$.
Let $d=2$, $\gs=\ast$, $U=\RR^0$, $\ell=1$, $k=0$, and ${\bf m}=([2],[4])$.
The image on the left depicts an object $(p,C,\sigma=\ast)\in \ncBord^\gs_d(U,\langle \ell\rangle,{\bf m})_k$.
The image on the right depicts an object $(q,D,\tau=\ast)\in \CatBord^\gs_d(U,\langle \ell\rangle,{\bf m})_k$.
The gray region is the ambient 2-dimensional smooth manifold, which is the total space of the submersion $p:M\to \RR^0$.
On the left, the ambient manifold is two parallel sheets connected by two “tunnels”.
On the right, the ambient manifold is one folded sheet connected by two tunnels.
The cuts $C_{=j}^i$ and $D_{=j}^i$ are part of the data of cut grids $C=(C^1,C^2)$ and $D=(D^1,D^2)$ on $p$ and $q$, respectively.
On the left, the orange cut $C^2_{4}$ makes the core noncompact.
On the right, the core is compact since the two sheets are joined together.
}
\label{figure2}
\end{figure}

\begin{figure}[ht]
\begin{tikzpicture}[scale=.40]
\draw[dashed, rounded corners] (5+.3,-1+.3) -- (0+.3,-1+.3) -- (0+.3,4+.3) -- (5+.3,4+.3) -- (5+.3,-1+.3);
\fill[fill=gray!25, rounded corners] (5+.3,-1+.3) -- (0+.3,-1+.3) -- (0+.3,4+.3) -- (5+.3,4+.3) -- (5+.3,-1+.3);
\draw[ blue] (1.5+.3,4+.3) to [out=-70, in=110] (2+.3,3+.3);
\draw[ blue] (2+.3,4+.3) to [out=-90, in=110] (2.05+.3,2.9+.3);
\draw [purple] (0+.3,0+.3) to [out=20,in=120] (5+.3,2+.3);
\draw [purple] (0,2.5+.3) to [out=0, in=125] (5+.3,3+.3);

\fill[fill=gray!25, rounded corners] (5,-1) -- (0,-1) -- (0,4) -- (5,4) -- (5,-1);
\draw[dashed, rounded corners] (5,-1) -- (0,-1) -- (0,4) -- (5,4) -- (5,-1);
\draw [purple] (0,0) to [out=20,in=120] (5,2);
\draw [purple] (0,2.5) to [out=0, in=125] (5,3);
\node at (0-.7,0) {$\scriptstyle C^1_{=0}$};
\node at (0-.7,2.8) {$\scriptstyle C^1_{=1}$};
\draw[ blue] (1.5,4) to [out=-70, in=110] (2,3) to [out=290, in=110] (1,-1);
\draw[ blue] (2,4) to [out=-90, in=110] (2.05,2.9) to [out=-45, in =110] (3,2.8) to [out=-90, in=80] (1.3,.5) to [out=-90,in=90] (3,-1);
\node at (.6,4.5+.3) {$\scriptstyle C_{=0}^2$};
\node at (2,4.5+.3) {$\scriptstyle C_{=1}^2$};
\draw[->] (6,2)--(8,2);

\end{tikzpicture}
\begin{tikzpicture}[scale=.40]
\fill[fill=gray!25, rounded corners] (0+.3,-1+.3) -- (0+.3,6+.3) -- (9+.3,6+.3) -- (9+.3,-1+.3) -- (0+.3,-1+.3);
\draw[ blue] (1+.3,6+.3) to [out=270, in=110] (2+.3,3+.3);
\draw[ blue] (1.5+.3,6+.3) to [out=-45, in=110] (2+.3,3+.3);
\draw[dashed, rounded corners] (0+.3,-1+.3) -- (0+.3,6+.3) -- (9+.3,6+.3);
\fill[fill=gray!25, rounded corners] (0,-1) -- (0,6) -- (9,6) -- (9,-1) -- (0,-1);
\draw[dashed, rounded corners] (9.7,-.5) to[out=270,in=0, looseness=.8] (9,-1) -- (0,-1) -- (0,6) -- (9,6);
\fill[fill=gray!25, dashed] (9,6) -- (9,-1) to[out=0,in=270, looseness=1.2] (9.7,.5) -- (9.7,6.2);
\filldraw[fill=gray!25, dashed] (9+.3,6+.3) to [out=0, in=0, looseness=4] (9,6);
\draw [purple] (0,0) to [out=20,in=120] (5,2) to [out=-60, in=130] (9,1) to [out=310, in=-90] (9.7,1.5);
\draw [purple] (0,2.5) to [out=0, in=125] (6,3) to [out=305, in=110] (9,3.5) to [out=300, in=-130] (9.7,3.7);
\node at (0-.7,0) {$\scriptstyle C^1_{=0}$};
\node at (0-.7,2.8) {$\scriptstyle C^1_{=1}$};
\draw[ blue] (1,6) to [out=270, in=110] (2,3) to [out=290, in=110] (1,-1);
\draw[ blue] (1.5,6) to [out=-45, in=110] (2.05,2.9) to [out=-45, in =110] (3,2.8) to [out=-90, in=80] (1.3,.5) to [out=-90,in=90] (3,-1);
\node at (.6,6.5+.3) {$\scriptstyle C_{=0}^2$};
\node at (2,6.5+.3) {$\scriptstyle C_{=1}^2$};
\draw[fill=white] (2.65-.2,4.6) to [out=57, in=118] (3.9-.2,4.6);
\draw[thick] (2.5-.2,4.5) to [out=60, in=115] (4-.2,4.5);
\draw[thick, fill=white] (2.65-.2,4.6) to [out=-45, in=-135] (3.9-.2,4.6);
\draw[fill=white] (6.5+.2,.5+.2) to [out=46, in=130] (8.5-.2,.5+.2);
\draw[thick] (6.5,.5) to [out=60, in=115] (8.5,.5);
\draw[thick, fill=white] (6.5+.2,.5+.2) to [out=-45, in=-135] (8.5-.2,.5+.2);
\draw[dashed, rounded corners] (5+.2,-1+.2) -- (0+.2,-1+.2) -- (0+.2,4+.2) -- (5+.2,4+.2) -- (5+.2,-1+.2);
\end{tikzpicture}
\caption{Let $\stcart=\cart$.
Let $d=2$, $\gs=\ast$, $U=\RR^0$, $\ell=1$, $k=0$, and ${\bf m}=([1],[1])$.
The image depicts a morphism in $\CatBord_d^{\gs}(U,\langle \ell\rangle, {\bf m})_0$, given by the inclusion of an open subset of the bordism on the right.
Since the open subset contains the core, this morphism satisfies condition \cref{bordism.core} in \cref{compactbords}.}\label{morphismcatbords}
\end{figure}

We are now ready to introduce the geometric extended bordism category.

\begin{definition}
\label{bord}
\label{bordstr}
Assume the context of \cref{noncompact.bord}.
Composing the subfunctor $\CatBord_d^\gs$ in \cref{compactbords} with the degreewise nerve functor followed by the diagonal functor yields a functor
\begin{equation}\label{nerveofcatbord.unenriched}\Bord^\gs_d:\stcart^\op⨯Γ^\op⨯(Δ^{⨯d})^\op\lto{8}{\CatBord_d^\gs}\smallcat^{Δ^\op}\lto5{\nerve^{Δ^\op}}\sset^{Δ^\op}\lto5{\diag}\sset,\end{equation}
hence an object $\Bord^\gs_d∈\smcat{d}$ (\cref{globular.model.structure}, taking $\site=\stcart$).
We call $\Bord^\gs_d$ the \emph{geometric symmetric monoidal $(\infty,d)$-category of bordisms with $\gs$-structure}.

In the simplicial set $\Bord_d^\gs(U,⟨ℓ⟩,{\bf m})$,
vertices are known as \emph{bordisms} and 1-simplices are known as \emph{virtual isomorphisms of bordisms} (\cref{cats.defs,virtual.isomorphism}).

Since the construction $\gs↦\CatBord_d^\gs$ is $\sset$-enriched functorial in~$\gs$ (\cref{compactbords}),
the construction \cref{nerveofcatbord.unenriched} defines an $\sset$-enriched functor
$$\Bord_d:\Struct_d\to \smcat{d}, \qquad \gs\mapsto \Bord_d^{\gs}.$$
\end{definition}

We now unwind \cref{bord}.

\begin{remark}
\label{bord.explicit}
Assume the context of \cref{bord}.
Given $\gs∈\Struct_d$, $(U,⟨ℓ⟩,{\bf m})∈\stcart⨯Γ⨯Δ^{⨯d}$, and $[k]∈Δ$,
the set $\Bord_d^\gs(U,⟨ℓ⟩,{\bf m})_k$ has elements given by a composable chain of morphisms
in the category $\CatBord_d^\gs(U,⟨ℓ⟩,{\bf m})_k$:
$$(p_0,C_0,σ_0)\lto3{φ_1}(p_1,C_1,σ_1)\lto3{φ_2}⋯\lto3{φ_k}(p_k,C_k,σ_k).$$
Explicitly,
\begin{enumerate}
\item each $p_i:M_i\to U_i$ is an object of $\FEmb_d$ (\cref{fibered.geometric.site});
\item each $C_i\in \mtCutglob(⟨ℓ⟩,{\bf m},\red p_i)$ is a compact globular monoidal cut grid on $p_i:M_i\to U_i$ (\cref{globular.monoidal.cut.grid});
\item each $\sigma_i\in \gs(p_i:M_i\to U_i)_k$ is a $k$-simplex in the field stack~$\gs$ (\cref{geometric.structure});
\item the morphisms $φ_i:p_{i-1}→p_i$ are maps in $\FEmb_d$ such that $\base(φ_i)=\id_U$, $(\red φ_i)^*C_i=C_{i-1}$, and $φ^*σ_i=σ_{i-1}$,
with $\image\red φ_i⊃\ncore(\red p_i,C_i)$.
\end{enumerate}

The structure maps corresponding to morphisms in $\stcart⨯Γ⨯Δ^{⨯d}$ are inherited from the structure maps of $\CatBord_d^\gs$ (\cref{compactbords}).
\end{remark}

In the case where $d=2$, $\stcart=\cart$, $\gs=\ast$, $U=\RR^0$, $\ell=1$, ${\bf m}=([2],[4])$, and $k=0$,
an example of an element in $\Bord_d^{\gs}(U,\langle\ell\rangle,{\bf m})_k$ is given by the right image in \cref{figure2}.
Changing $U$, $\gs$, $\ell$, or ${\bf m}$ can be visualized by adding additional data to an image of this form (which we call a bordism) as follows.
\begin{itemize}
\item Passing from $U=\RR^0$ to arbitrary $U$ amounts to taking $U$-parametrized families of bordisms.
One can imagine smoothly varying the right image in \cref{figure2}, along with its cut grid, with respect to the parameter space $U$.
\item Passing from $\gs=\ast$ to arbitrary $\gs$ amounts to adding the data of a field $\sigma\in \gs(p)_0$ on the bordism,
such as a Riemannian metric or a principal $G$-bundle with connection.
\item Passing from $\ell=1$ to arbitrary $\ell$ amounts to labeling the connected components of the bordism with elements of $\langle \ell\rangle$.
In \cref{figure2}, one can imagine adding more connected components to the image on the right in \cref{figure2}
and labeling each connected component with an element of $\langle \ell\rangle$.
Some components may in particular be labeled by the trash bin $\ast\in \langle \ell\rangle$.
\item Passing from ${\bf m}=([2],[4])$ to arbitrary multisimplices amounts to changing the number of cuts in the cut grid.
One can imagine adding or removing cuts in the cut grid in \cref{figure2}, according to the multisimplex.
\end{itemize}

We now provide a conceptual explanation of the $k$-simplices in \cref{bord.explicit}.
For simplicity, we take $\gs=\ast$.
In this case, a $k$-simplex in \cref{bord.explicit}
is just a $k$-simplex in the nerve of the category $\CatBord_d^{\gs}(U,\langle \ell\rangle, {\bf m})$ of \cref{compactbords} (hence a composable chain of morphisms).
By definition, a morphism must map the cuts in the cut grid on the domain to corresponding cuts in the cut grid of the codomain.
In \cref{figure2}, one can image a map between two bordisms of the form depicted on the right, where the map must send cuts to corresponding cuts.
Working in a single fiber over the parametrizing space $U$, such a map is an open embedding on the ambient manifold.
By \cref{bordism.core} in \cref{compactbords}, the image must contain the core,
so on a sufficiently small neighborhood of the core (in both domain and codomain) this map is a diffeomorphism.
Thus the category $\CatBord_d^{\gs}(U,\langle \ell\rangle, {\bf m})$ encodes the groupoid of diffeomorphisms between bordisms.

\begin{remark}
\label{bord.isotopy.explicit}
Substituting $T=\Bord_d^{\gs}$ in \cref{cats.defs} (see \cref{nonfibrant.target})
yields explicit descriptions of $U$-families of $k$-morphisms in the bordism category for $0\leq k\leq d$:
a $U$-family of $k$-morphisms is given by an element of the set $$\Bord_d^{\gs}(U,\langle 1\rangle,[1],\ldots,[1],[0],\ldots,[0])_0,$$ where $[1]$ is repeated $k$~times.
For example, a circle can be interpreted as a 1-morphism in the 2-dimensional unoriented bordism category, by taking $d=2$, $k=1$, $U=\RR^0$, $\gs=\ast$.
The ambient manifold is the cylinder $C_z=\{(x,y,z) \mid x^2+y^2=1\}$.
The cut tuple in the first direction is given by the pair of cuts
$$C^1_0=(\emptyset,\emptyset, C_z), \quad C^1_1=(C_z,\emptyset,\emptyset)$$
The cut tuple in the second direction is given by
$$C^2_0=(z<0, z=0, z>0).$$
A picture of this 1-morphism is given in \cref{circle.1mor}.
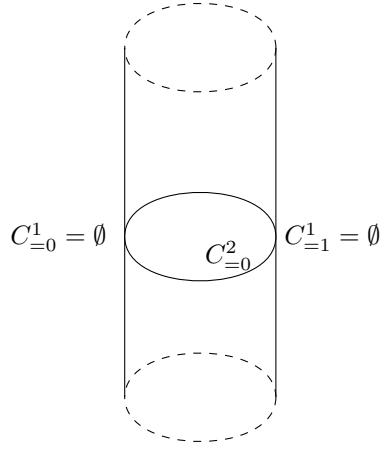
\begin{figure}
\input{1-morph.tikz}
\caption{A circle regarded as a 1-morphism in the unoriented 2-dimensional bordism category.}
\label{circle.1mor}
\end{figure}
\end{remark}

The version of the bordism category described in \cref{bord.explicit} \emph{does not} encode the homotopy type of the space of diffeomorphisms of bordisms
like Galatius–Madsen–Tillmann–Weiss \cite{GMTW} and Lurie \cite{Lurie.TFT}.
The bordism category with isotopies, which we now construct, does encode the homotopy type.

\begin{definition}
\label{fraknccatbord}
Fix $d\geq 0$ and a fibered geometric site with isotopies $\frakFEmb_d$ (\cref{fibered.geometric.site.isotopy}),
including the site $\stcart$ and the reduction functor $\red=\red_t:\frakFEmb_d→\fraksmFEmb_d$.
Recall $\frakStruct_d$ from \cref{geometric.structure.isotopy} and fix $\gs\in \frakStruct_d$.
Recall also the globular monoidal cut grid functor with isotopies $\frakmtCutglob$ from \cref{globular.monoidal.cut.grid.smooth}.
Consider the functor
\begin{equation}\label{isot.grothcutfunct}Γ^\op⨯(Δ^{⨯d})^\op→\PSh(\frakFEmb_d,\smsset), \qquad (⟨ℓ⟩,{\bf m})↦((p:M→U)↦\frakmtCutglob(⟨ℓ⟩,{\bf m},\red p)⨯\gs(p)).\end{equation}
Apply objectwise the $\smset$-enriched relative Grothendieck construction functor~$\relgro$ (\cref{relative.grothendieck}) to obtain a functor
$$Γ^\op⨯(Δ^{⨯d})^\op\to \PSh(\frakFEmb_d,\smsset)\lto{14}{(\relgro)^{Δ^\op}}\PSh(\stcart,\smallcat^{\cart^\op\times Δ^\op}).$$
Moving the category $\stcart$ to the left side using currying yields a functor
\begin{equation}\label{isot.nccatbord}\frakncBord^\gs_d:\stcart^\op⨯Γ^\op⨯(Δ^{⨯d})^\op\to \smallcat^{\cart^\op⨯Δ^\op},\end{equation}
which we call the \emph{geometric symmetric monoidal $(\infty,d)$-category of noncompact bordisms with isotopies}.
The construction $\gs↦\frakncBord_d^\gs$ is $\smsset$-enriched functorial in~$\gs$.
\end{definition}

The following proposition is the analog of \cref{explicit.noncompact.bords} in the isotopy case.

\begin{proposition}
\label{explicit.noncompact.isot}
The functor $\frakncBord^\gs_d$ in \cref{isot.nccatbord} has the following explicit description.
Let $(U,\langle \ell\rangle,{\bf m})\in \stcart\times \Gamma\times \Delta^{\times d}$, $L\in \cart$, and $[k]\in \Delta$.
The category $\frakncBord^\gs_d(U,\langle \ell\rangle,{\bf m})_{L,k}$ is given by the following.
\begin{itemize}
\item Objects are triples $(p:M\to U,C,\sigma)$, where $p∈\frakFEmb_d$,
$C∈\frakmtCutglob(⟨ℓ⟩,{\bf m},\red p)_L$ is an $L$-family of globular monoidal cut $(\langle \ell\rangle,{\bf m})$-grids on $\red p$,
and $\sigma\in \gs(p:M\to U)_{L,k}$.
\item Morphisms $(p,C,\sigma)\to (p',C',\sigma')$ are maps $\varphi:p\to p'$ in $(\frakFEmb_d)_L$
such that $\base \varphi=\id_U$, $(\red\varphi)^*C'=C$, and $\varphi^*\sigma'=\sigma$.
We call a morphism an \emph{$L$-family of cut-respecting embeddings}.
\end{itemize}
\end{proposition}

\begin{proof}
This is an application of \cref{explicitgroth},
replacing $\base:\FEmb_d→\stcart$ with $\base_L:(\frakFEmb_d)_L→\stcart$
and taking $F$ to be the functor obtained by evaluating \cref{isot.grothcutfunct}
at fixed objects $(\langle \ell\rangle,{\bf m})∈Γ⨯Δ^{⨯d}$ and $(L,[k])\in \cart\times \Delta$.
\end{proof}

The next proposition is the analog of \cref{compactbords} in the isotopy case.

\begin{proposition}
\label{compactbords.isotopy}
Recall \cref{fraknccatbord} and the functor~$\frakncBord_d^\gs$.
Recall the explicit description of $\frakncBord_d^\gs$ in \cref{explicit.noncompact.isot}.
For all $(U,⟨ℓ⟩,{\bf m})∈\stcart⨯Γ⨯Δ^{⨯d}$ and $(L,[k])\in \cart\times \Delta$, there is a subcategory
$$\frakCatBord_d^\gs(U,⟨ℓ⟩,{\bf m})_{L,k}⊂\frakncBord_d^\gs(U,⟨ℓ⟩,{\bf m})_{L,k}$$
defined by the following conditions.
\begin{enumerate}[(c1)]
\item
\label{bordism.compact.isotopy}
On objects $(p,C,\sigma)$, we require the $L$-family of globular monoidal cut grids~$C$ (\cref{globular.monoidal.cut.grid.smooth}) to be compact (\cref{compactmonoidalgrid}).
\item
\label{bordism.core.isotopy}
On morphisms $\varphi$, we require the image of $\red\varphi$ to contain $\ncore(\red p',C')$ (\cref{monoidal.cut.grid.core}).
\end{enumerate}
Moreover, $\frakCatBord_d^\gs$ defines a subpresheaf of $\frakncBord_d^\gs$.
The construction $\gs↦\frakCatBord_d^\gs$ is $\smsset$-enriched functorial in~$\gs$.
\end{proposition}

\begin{proof}
The proof is adapted from the proof of \cref{compactbords} by evaluating at a pair $(L,k)$ instead of just $k$.
\end{proof}

The following definition is the isotopy analog of \cref{bord}.

\begin{definition}
\label{bord.isotopy}
\label{frakbord}
\label{enrichedbordstr}
Assume the context of \cref{fraknccatbord}.
Composing the subfunctor $\frakCatBord_d^\gs$ in \cref{compactbords.isotopy} with the objectwise nerve functor followed by the objectwise diagonal functor yields a functor
$$\frakBord^\gs_d:\stcart^\op⨯Γ^\op⨯(Δ^{⨯d})^\op
\lto{7}{\frakCatBord_d^\gs}\smallcat^{\cart^\op\times Δ^\op}
\lto{12}{\nerve^{\cart^\op\times Δ^\op}}\sset^{\cart^\op\times Δ^\op}
\lto{9}{\diag^{\cart^\op}}\smsset,$$
hence an object $\frakBord^\gs_d∈\fraksmcat{d}$ (\cref{globular.model.structure.smooth}).
We call $\frakBord^\gs_d$ the \emph{geometric symmetric monoidal $(\infty,d)$-category of bordisms with $\gs$-structure and isotopies}.

In the simplicial set $\Bord_d^\gs(U,⟨ℓ⟩,{\bf m})_L$,
vertices are known as \emph{$L$-families of bordisms}
and 1-simplices are known as \emph{$L$-families of virtual isomorphisms of bordisms} (\cref{cats.defs,virtual.isomorphism}).
If $L=\RR^1$, we get \emph{isotopies of bordisms} and \emph{isotopies of virtual isomorphisms of bordisms}.

Since the construction $\gs↦\frakCatBord_d^\gs$ is $\smsset$-enriched functorial in~$\gs$ (\cref{compactbords.isotopy}),
the construction of $\frakBord_d^\gs$ defines an $\smsset$-enriched functor
$$\frakBord_d:\frakStruct_d\to \fraksmcat{d}, \qquad \gs\mapsto \frakBord_d^{\gs}.$$
\end{definition}

\begin{remark}
\label{frakbordl0}
Recall the isotopification functor $$\frakI_d:\Struct_d→\frakStruct_d$$
and its right adjoint $\gs↦\gs_{\RR^0}$ from \cref{isotopification}.

Given $\gs∈\frakStruct_d$, consider the bordism category
$$\frakBord_d^\gs∈\fraksmcat{d}=\smsPSh(\site\times\Gamma\times\Delta^{\times d})_{\glob}.$$
By evaluating the values of this presheaf of smooth simplicial sets at $\RR^0∈\cart$ we get
$$\Bord_d^{\gs_{\RR^0}}∈\smcat{d}=\sPSh(\site\times\Gamma\times\Delta^{\times d})_{\glob}.$$
This can be seen by comparing the explicit descriptions of these objects in
\cref{compactbords,compactbords.isotopy}.

Conversely, given $\gs∈\Struct_d$, we can consider the bordism category
$$\frakBord_d^{\frakI_d\gs}∈\fraksmcat{d}.$$
Many practical examples of bordism categories arise in this fashion.
\end{remark}

\begin{definition}
\label{def.FFTspace}
Fix $d≥0$.
A \emph{functorial field theory} with geometric structure $\gs∈\Struct_d$ valued in $\vcat∈\smcat{d}$ is a point in the derived mapping space
$$\FFT_{d,\vcat}^\gs = \dmap(\Bord_d^\gs,\vcat).$$
A \emph{functorial field theory with isotopies} with geometric structure $\gs∈\frakStruct_d$ valued in $\vcat∈\fraksmcat{d}$ is a point in the derived mapping space
$$\frakFFT_{d,\vcat}^\gs = \dmap(\frakBord_d^\gs,\vcat).$$
\end{definition}

In particular, if $\vcat$ is fibrant, then a functorial field theory is simply a morphism $\Bord_d^\gs→\vcat$, consistent with the traditional definition
of a functorial field theory.

\begin{remark}
\label{bordism.not.fibrant}
As explained in \cref{not.fibrant}, it is neither necessary nor beneficial for symmetric monoidal $(∞,d)$-categories to be fibrant in some model structure.
A functorial field theory (\cref{def.fft}) is a point in a derived mapping object $\dmap(\Bord_d^\gs,\vcat)$,
which can be presented as a morphism of the form $\Bord_d^\gs→\vcat$ if $\Bord_d^\gs$ is cofibrant and $\vcat$ is fibrant.
All objects in the model structure for geometric symmetric monoidal $(∞,d)$-categories (\cref{globular.model.structure}, with or without isotopies) are cofibrant.

Although our bordism categories $\Bord_d$ and $\frakBord_{d}$ are not fibrant objects in $\smcat{d}$ or $\fraksmcat{d}$, respectively,
they are local with respect to some of the morphisms in \cref{globular.model.structure}.
Since these locality properties are not useful for the development of functorial field theories, we omit their proofs.
\begin{itemize}
\item For a field stack $\gs$ that satisfies homotopy descent in $\Struct_{d}$ (respectively $\frakStruct_{d}$), the category $\Bord_d^\gs$ (respectively $\frakBord_d^\gs$) is local with respect to the Čech morphisms \cref{homotopy.descent}, Segal $\Gamma$-maps \cref{monoidal1,monoidal2}, Segal $\Delta$-maps \cref{segal1}.
\item If the field stack $\gs$ does not satisfy homotopy descent, then all three of the above conditions may be violated.
This is the case when $\gs$ is representable, corresponding to the case of embedded bordism categories
(\cref{geometrically.framed.bordisms}).
\item For any field stack $\gs$, the object $\Bord_d^\gs$ is local with respect to the Segal completion maps \cref{segal2}.
If $d\geq 5$, and $\gs=\ast$, then $\frakBord_d^\gs$ is not local with respect to the completion maps, due to the existence of nontrivial $h$-cobordisms.
\item Both $\Bord_d$ and $\frakBord_d$ are local with respect to the globular maps in \cref{globular.maps}.
\end{itemize}
If desired, one could make the following simple adjustments to the definitions of bordism categories,
which produce weakly equivalent bordism categories that satisfy all locality conditions of \cref{globular.model.structure}.
\begin{itemize}
\item In the case of $\frakBord_d$, add h-cobordisms to the smooth simplicial set of objects,
as in Lurie \cite[Definition~2.2.10]{Lurie.TFT} and Calaque–Scheimbauer \cite[Definition~5.24]{CalaqueScheimbauer}.
\item Replace the field stack~$\gs$ by its associated ∞-sheaf $\frep\gs$ in all constructions.
\end{itemize}
We remark that the remaining fibrancy condition, that of \emph{injective} fibrancy, is typically false unless $\gs=∅$.
\end{remark}

\begin{remark}
Working with nonfibrant models for the bordism category gives rise to significant simplifications in the proof of locality
of functorial field theories in Grady–Pavlov \cite{GradyPavlov.Loc},
which uses representable field stacks (\cref{geometrically.framed.bordisms}).
These stacks do not satisfy descent and give rise to nonfibrant bordism categories.
Nonfibrancy should not be seen as a disadvantage in this case,
but rather as an advantage that allows us to work with a much smaller model for the bordism category in our proofs.
\end{remark}

\begin{remark}
\label{cut.grid.functor.parameter}
The constructions of \cref{noncompact.bord,compactbords,bord} are functorial in $\mtCutglob$.
More precisely, set $$\catC=\PSh(Γ⨯Δ^{⨯d}⨯\smFEmb_d,\set)/\mtCut.$$
Then substituting an arbitrary object of~$\catC$ instead of $\mtCutglob$ in \cref{bord}
produces an object of $\smcat{d}$ for which \cref{compactbords} continues to hold,
given that $\ncore(\red p,C)$ is defined for an arbitrary monoidal cut grid~$C$ in $\mtCut$.
We define the \emph{uple bordism category} $\Bord_{d,\uple}^\gs$ (see also \cref{uple.globular})
by taking the identity map on $\mtCut$ as the object of~$\catC$.

Likewise, the constructions of \cref{fraknccatbord,compactbords.isotopy,bord.isotopy} are functorial in $\frakmtCutglob$
as an object in $$\catC=\PSh(Γ⨯Δ^{⨯d}⨯\fraksmFEmb_d,\smset)/\frakmtCut.$$
We define the \emph{uple bordism category with isotopies} $\frakBord_{d,\uple}^\gs$ (see also \cref{uple.globular})
by taking the identity map on $\frakmtCut$ as the object of~$\catC$.
\end{remark}

\begin{remark}
\label{cut.grid.equivalence}
An objectwise isotopy equivalence of objects in the category $$\catC=\PSh(Γ⨯Δ^{⨯d}⨯\fraksmFEmb_d,\smset)/\frakmtCut$$
of \cref{cut.grid.functor.parameter} induces an objectwise weak equivalence
of the resulting bordism categories in $\fraksmcat{d}$.
\end{remark}

\cref{cut.grid.equivalence} can be used to establish equivalences of various models of bordism categories.
For example, it can be used to define other versions of globular monoidal cut grid functors,
e.g., the ones where the strict globularity condition is relaxed to a \emph{cylindrical} globularity condition (meaning the condition
of being simplicially degenerate is relaxed to the condition of being isotopic to a simplicially degenerate element),
and show them to be weakly equivalent to the (strictly) globular bordism category with isotopies (\cref{bord.isotopy}).
We can also add a specific choice of height functions to a cut grid, as is commonly done in the literature (Lurie \cite[Definition~2.2.9]{Lurie.TFT},
Calaque–Scheimbauer \cite[Definition~5.1]{CalaqueScheimbauer},
Schommer-Pries \cite[Definition~5.8]{SchommerPries}).

\begin{conjecture}
\label{comparison.conjecture}
Taking $\stcart=1$ and $\frakFEmb_d=\fraksmEmb_d$
(\cref{trivial.fibered.geometric.site.isotopy})
yields a symmetric monoidal $(∞,d)$-category
$\frakBord_d^\gs$
that is weakly equivalent to the symmetric monoidal $(∞,d)$-categories of bordisms
of Calaque–Scheimbauer \cite[Definition~9.10]{CalaqueScheimbauer} and Schommer-Pries \cite[Definition~5.8]{SchommerPries}
with the tangential structure associated to $\gs∈\frakStruct_d$ as constructed in \scref{fiberwiseshape}.
\end{conjecture}

\subsection{Examples of bordism categories}
\label{examplebords}

In this section, we provide some examples of bordisms in both $\Bord_d$ and $\frakBord_d$ for various types of fields.
In each example, we describe an element of the set obtained from \cref{bord,bord.isotopy} by evaluating at the following parameters:
\begin{enumerate}[series=params]
\item $d\geq 0$;
\item $\base:\FEmb_d\to \stcart$ (\cref{fibered.geometric.site});
\item $\gs\in \Struct_d$ (\cref{geometric.structure});
\item $U\in \stcart$ (\cref{structured.cartesian.site});
\item $\langle \ell\rangle\in \Gamma$;
\item ${\bf m}\in \Delta^{\times d}$;
\item $[k]\in \Delta$.
\end{enumerate}
In the case of the isotopy bordism category, we evaluate at
\begin{enumerate}[resume*=params]
\setcounter{enumi}{1}\item $\base:\frakFEmb_d\to \stcart$ (\cref{fibered.geometric.site.isotopy});
\setcounter{enumi}{2}\item $\gs\in \frakStruct_d$ (\cref{geometric.structure.isotopy});
\setcounter{enumi}{7}\item $L\in \cart$.
\end{enumerate}
If we do not set any of the above parameters in an example, we implicitly keep these parameters free.
In the case of the functor~$\base$, omitting the description means we take the standard smooth geometric structure,
as described in \cref{standard.fibered.geometric.site,standard.fibered.geometric.site.isotopy}.

\begin{notation}
Following the conventions of \cref{cats.defs}, we introduce the following notation in the case where $T$ is $\Bord_d^\gs$.
\begin{itemize}
\item We call an element of $\Bord^{\gs}_d(U,\langle 1\rangle,([1],\hdots,[1],[0],\hdots,[0]))_0$, with $[1]$ repeated $n$ times, a \emph{$U$-family of $n$-bordisms}.
\item We call an element of $\Bord^{\gs}_d(U,\langle 1\rangle,{\bf m})_1$ a \emph{$U$-family of virtual isomorphisms}.
We will depict a virtual isomorphism by an arrow $f:x\to y$, where $d_0f=y$ and $d_1f=x$.
\end{itemize}
In the isotopy case, we have the following notation.
\begin{itemize}
\item We call an element of $\frakBord^{\gs}_d(U,\langle 1\rangle,{\bf m})_\RR$ an \emph{$U$-family of isotopies}.
\end{itemize}
In the special case where $U=\RR^0$, we drop “$U$-family of” in all of the above.
\end{notation}

We begin with the simple example of a point in the bordism category.

\begin{example}
\label{thepoint}
Set $d=2$, $\gs=\ast$, $U=\RR^0$, $\ell=1$, ${\bf m}=([0],[0])$, and $k=0$.
An element in $\Bord_d^\gs(U,⟨ℓ⟩,{\bf m})_k$
is a pair $(p,C)$, where $p:M\to \RR^0$ and $C$ is a compact globular monoidal cut grid on~$p$.
We have a canonical $0$-bordism, which we call the \emph{point}, given by taking $M=\RR^2$ and the monoidal cut grid given by the coordinate cross
$$C^1_0=(y<0,y=0,y>0),\qquad C^2_0=(x<0,x=0,x>0),\qquad C^⊗=1∈⟨1⟩,$$
as depicted in the following figure:
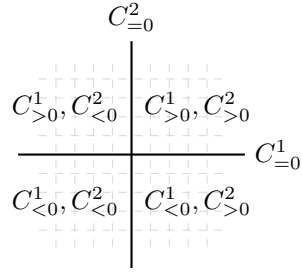
\begin{center}
\begin{tikzpicture}[scale=.25]
\draw[help lines, color=gray!30, dashed] (-4.9,-4.9) grid (4.9,4.9);
\draw[thick] (-6,0)--(6,0) node[right]{$C^1_{=0}$};
\draw[thick] (0,-6)--(0,6) node[above]{$C^2_{=0}$};

\node at (-3.5,2.5) {$C^1_{>0}, C^2_{<0}$};
\node at (3.5,2.5) {$C^1_{>0}, C^2_{>0}$};
\node at (-3.5,-2.5) {$C^1_{<0}, C^2_{<0}$};
\node at (3.5,-2.5) {$C^1_{<0}, C^2_{>0}$};

\end{tikzpicture}
\end{center}
Observe that we have 3 other points with the same sets $C_{=0}^1$ and $C_{=0}^2$, obtained by reversing the normal orientation of a cut,
i.e., reversing $C_{>0}$ and $C_{<0}$.
All these points are related by virtual isomorphisms in the bordism category, given by the obvious change of basis.
\end{example}

\begin{example}
\label{globulargrid}
Set $d=2$, $\gs=\ast$, $U=\RR^0$, $\ell=1$, ${\bf m}=([2],[2])$, and $k=0$.
A composable chain of $2$-bordisms can be constructed as follows.
Take $M=\RR^2$ and $p:M\to \RR^0$ the projection.
We construct a globular monoidal cut ${\bf m}$-grid as follows.
Choose a smooth bump function $\phi\geq 0$ with support in the interval $(-1,1)$ such that $\phi(0)=2$.
Consider the two translates $\psi_1=\phi(x+1)$ and $\psi_2=\phi(x-1)$ and the function $\psi=\psi_1+\psi_2$.
Then we have a globular monoidal cut grid with $C^⊗=1∈⟨1⟩$ and the following cuts:
$$C^1_0=(x<-2,x=-2,x>-2),\qquad C^1_1=(x<0,x=0,x>0),\qquad C^1_2=(x<2,x=2,x>2),$$
$$C^2_0=(y<-\psi,y=-\psi,y>-\psi),\qquad C^2_1=(y<0,y=0,y>0),\qquad C^2_2=(y<\psi,y=\psi,y>\psi).
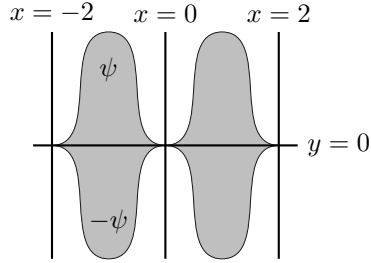
$$
\begin{center}
\begin{tikzpicture}[scale=.25]
\draw (-6,0) to[out=0,in=180] (-3,6);
\draw (-3,6) to[out=0,in=180] (0,0);

\draw (0,0) to[out=0,in=180] (3,6);
\draw (3,6) to[out=0,in=180] (6,0);

\draw (-6,0) to[out=0,in=180] (-3,-6);
\draw (-3,-6) to[out=0,in=180] (0,0);

\draw (0,0) to[out=0,in=180] (3,-6);
\draw (3,-6) to[out=0,in=180] (6,0);

\fill[fill=gray, opacity=.5] (-6,0) to[out=0,in=180] (-3,6) to[out=0,in=180] (0,0) to[out=0,in=180] (3,6) to[out=0,in=180] (6,0);
\fill[fill=gray, opacity=.5] (-6,0) to[out=0,in=180] (-3,-6) to[out=0,in=180] (0,0) to[out=0,in=180] (3,-6) to[out=0,in=180] (6,0);

\draw[thick] (-7,0)--(7,0) node[right]{$y=0$};
\draw[thick] (0,-6)--(0,6) node[above]{$x=0$};
\draw[thick] (-6,-6)--(-6,6) node[above]{$x=-2$};
\draw[thick] (6,-6)--(6,6) node[above]{$x=2$};

\node at (-3,4) {$\psi$};
\node at (-3,-4) {$-\psi$};
\end{tikzpicture}
\end{center}

The core of this bordism is the light gray region bounded below by $-\psi$, above by $\psi$, on the left by $x=-2$ and on the right by $x=2$.
Removing the cut $C^1_1$ using the first face map in the first simplicial direction yields the composition of the two bordisms,
the first one being given by the second face in the first simplicial direction, bounded by the pair of vertical lines $x=-2$ and $x=0$
and
the second one being given by the zeroth face in the first simplicial direction, bounded by the pair of vertical lines $x=0$ and $x=2$.
The composition is bounded by the vertical lines $x=-2$ and $x=2$.
\end{example}

The next example shows that in the uple case, manifolds with corners give examples of morphisms in the uple bordism category (see \cref{uple.globular}).
This example can be modified to obtain globular bordisms by deforming the cuts to a globular cut grid, similar to the grid constructed in \cref{globulargrid}.
Such a construction is conceptually clear, although somewhat tedious to write down.

\begin{example}
Set $d=2$, $\gs=\ast$, $U=\RR^0$, $\ell=1$, ${\bf m}=([2],[2])$, and $k=0$.
We provide a simple example of a composable chain of 2-bordisms in the uple bordism category of \cref{cut.grid.functor.parameter}.

An object in the bordism category is thus a pair $(p:M→\RR^0,C)$,
where $M$ is a smooth 2-manifold and $C$ is a compact monoidal (uple) cut grid on~$p$.
We choose $M$ to be the torus $M=S^1\times S^1$.
We now construct an (uple) cut grid on $M$.
Let $e:M\into \RR^3$ be the embedding of the torus given by $$e(\theta,\phi)=((\cos\theta+3)\cos\phi,(\cos\theta+3)\sin\phi,\sin\theta).$$
Take the (uple) monoidal cut grid with $C^⊗=1∈⟨1⟩$ and the following cuts:
$$C^1_0=e^{-1}(x<-3,x=-3,x>-3),\qquad C^1_1=e^{-1}(x<0,x=0,x>0),\qquad C^1_2=e^{-1}(x<3,x=3,x>3),$$
$$C^2_0=e^{-1}(y<-3,y=-3,y>-3),\qquad C^2_1=e^{-1}(y<0,y=0,y>0),\qquad C^2_2=e^{-1}(y<3,y=3,y>3).$$
A sketch of the resulting bordism is given below.
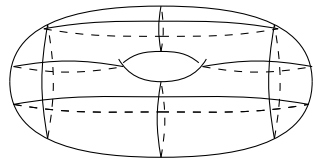
\begin{center}
\begin{tikzpicture}
\draw (0,0) to[out=90, in=180] (2,1);
\draw (2,1) to[out=0, in=90] (4,0);
\draw (0,0) to[out=-90, in=180] (2,-1);
\draw (2,-1) to[out=0, in=-90] (4,0);
\draw (3.5,-.75) to[out=90+10, in=-90-10] (3.5,.75);
\draw[dashed] (3.5,-.75) to[out=90-10, in=-90+10] (3.5,.75);
\draw (.5,-.75) to[out=90+10, in=-90-10] (.5,.75);
\draw[dashed] (.5,-.75) to[out=90-10, in=-90+10] (.5,.75);
\draw (1.5,.25) to[out=40, in=180] (2,.4);
\draw (2,.4) to[out=0, in=140] (2.5,.25);
\draw (1.44,.3) to[out=-60, in=180] (2,0);
\draw (2,0) to[out=0, in=-90-30] (2.6,.3);
\draw (2,-1) to[out=100,in=-90-10] (2,0);
\draw[dashed] (2,-1) to[out=90-10,in=-90+10] (2,0);
\draw (2,.4) to[out=100,in=-90-10] (2,1);
\draw[dashed] (2,.4) to[out=90-10,in=-90+10] (2,1);
\draw (.05,-.3) to[out=10,in=180] (2,-.2);
\draw (2,-.2) to[out=0,in=170] (4-.05,-.3);
\draw[dashed] (.05,-.3) to[out=-10,in=180] (2,-.4);
\draw[dashed] (2,-.4) to[out=0,in=190] (4-.05,-.3);
\draw (.45,1-.3) to[out=10,in=180] (2,1-.2);
\draw (2,1-.2) to[out=0,in=170] (4-.45,1-.3);
\draw[dashed] (.45,1-.3) to[out=-10,in=180] (2,1-.4);
\draw[dashed] (2,1-.4) to[out=0,in=190] (4-.45,1-.3);
\draw (.05,.2) to[out=10,in=170] (1.5,.2);
\draw[dashed] (.05,.2) to[out=-10,in=-170] (1.5,.2);
\draw[dashed] (.05,-.3) to[out=-10,in=180] (2,-.4);
\draw[dashed] (2,-.4) to[out=0,in=190] (4-.05,-.3);
\draw (2.55,.2) to[out=10,in=170] (4,.2);
\draw[dashed] (2.55,.2) to[out=-10,in=-170] (4,.2);
\end{tikzpicture}
\end{center}
In this case, the torus is the ambient manifold.
The region lying between the planes $x=-3$, $x=3$, $y=-3$, and $y=3$ is a manifold with corners.
This manifold with corners is the core of the bordism.
\end{example}

The following example of a bordism category is used to define sigma models, encoding fields given by a smooth map to a fixed target manifold~$X$.

\begin{example}
Let $X$ be a smooth manifold, viewed as a sheaf on $\smFEmb_d$ via
$$\gs(p:M\to U)=C^{\infty}(M,X).$$
Fix $U\in \cart, \langle \ell \rangle\in \Gamma, {\bf m}\in \Delta^{\times d}$.
A bordism in $\Bord_d^\gs(U,\langle \ell \rangle, {\bf m})_0$ is given by a triple $(p:M\to U,C,f)$,
where $C$ is a compact globular monoidal cut $(\langle \ell\rangle,{\bf m})$-grid on $p$ and $f:M\to X$ is a smooth map.
A virtual isomorphism $\varphi\in \Bord_d^\gs(U,\langle \ell \rangle, {\bf m})_1$
such that $d_1\varphi=(p,C,f)$ and $d_0\varphi=(p',C',f')$ is a morphism $\varphi=(\alpha,\beta):p\to p'\in \smFEmb_d$
such that $\varphi^*C'=C$, $\ncore(p',C')\subset \image α$, and $f'\circ \alpha=f$,
where $\alpha:M\to M'$ is the map on total spaces of $p$ and $p'$, respectively.
\end{example}

\begin{example}
\label{metricsregb}
(See \cref{metricsregb.figure}.)
Set $d=1$, $U=\RR^0$, $\ell=1$, ${\bf m}=[1]$, and $k=0$.
Let $\gs$ be the sheaf $\FRiem\in \Struct_1$ of fiberwise Riemannian metrics (\cref{riemmet}).

A bordism in $\Bord_d^\gs(U,⟨ℓ⟩,{\bf m})_k$ is a triple $(M,C,g)$,
where $M$ is a 1-dimensional manifold, $C$ is a compact globular monoidal cut $(\langle 1\rangle,[1])$-grid on $p:M\to \RR^0$, and $g$ is a Riemannian metric on $M$.
In the case where $\ncore(p,C)$ is connected, the metric gives $\ncore(p,C)$ a Riemannian length $t\in \RR_{≥0}$.

A virtual isomorphism of such bordisms is an element $\varphi\in \Bord_1^{\FRiem}(\RR^0,⟨1⟩,[1])_1$,
given by an isometric embedding $\varphi:(M,g)\to (M',g')$ such that $\varphi^*C'=C$ and $\image\varphi⊃\ncore(p,C')$.
\end{example}

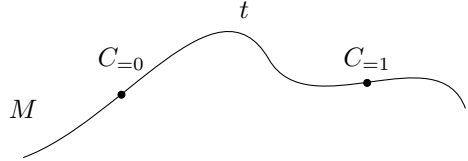
\begin{figure}[ht]
\begin{center}
\begin{tikzpicture}[scale=.65]
\draw (0,0) to [out=20,in=120] (5,2) to [out=-60, in=110] (9,1);
\filldraw (2,1.28) circle (2pt);
\filldraw (7,1.52) circle (2pt);
\node at (2,2) {$C_{=0}$};
\node at (7,2) {$C_{=1}$};
\node at (4.5,3) {$t$};
\node at (0,1) {$M$};
\end{tikzpicture}
\end{center}
\caption[]{A vertex in $\Bord_1^{\FRiem}(\RR^0,⟨1⟩,[1])$.
The 1-manifold $M$ is equipped with a Riemannian metric $g$.
The core is the 1-manifold with boundary lying between the cuts $C_{=0}$ and $C_{=1}$.
The Riemannian length of the core is $t\in \RR_{≥0}$.}
\label{metricsregb.figure}
\end{figure}

\begin{example}
Set $d=2$, $U=\RR^0$, $\ell=1$, ${\bf m}=([1],[1])$.
Consider the field stack $G\Bunconn$ from \cref{gbunexample1}, where $G$ is a Lie group.
An object in the corresponding bordism category is thus a quadruple $(C,φ:(M→\RR^0)→(N→V),P,\nabla)$,
where $M$ is a 2-dimensional manifold,
$C$ is a compact globular monoidal $(\langle 1\rangle,([1],[1]))$-cut grid on~$M$,
$φ=(a:M→N,b:\RR^0→V)$ is a morphism in $\smFEmb_d$,
and $(P,\nabla)$ is a principal $G$-bundle over~$N$ equipped with a fiberwise connection $\nabla$.

A virtual isomorphism is given by a pair $(ψ,g)$,
where $ψ$ is an open embedding $ψ:M \to M'$ such that $ψ^*C'=C$,
$\image ψ ⊃ \ncore(p,C')$,
and $g:φ^*(P,\nabla)\to ψ^*φ'^*(P',\nabla')$ is a connection-preserving isomorphism.
\end{example}

\begin{example}
Recall (\cref{1.1.Euclidean}) the field stack $$\Eucl_{1|1}∈\frakStruct_{1|1}=\smsPSh(\frakSFEmb_{1|1})$$
of fiberwise oriented $1|1$-Euclidean metrics
on the fibered geometric site $\frakSFEmb_{1|1}$ of real smooth $1|1$-\hskip0pt supermanifolds (\cref{def.fibered.supercart}).
Consider the bordism category $$\cB=\frakBord_1^{\frakI_{1|1} \Eucl_{1|1}},$$
where $\frakI_{1|1}$ is the isotopification functor (\cref{isotopification}).
We examine the subobject~$Ψ$ of the moduli stack of objects $$\cB(-,⟨1⟩,[0])∈\smsPSh(\scart)$$
on those bordisms whose core is a single point.

Define a smooth category as a presheaf of categories on $\cart$.
Smooth categories embed into smooth simplicial sets via the objectwise nerve functor,
and therefore inherit the relevant notions such as weak equivalences in $\smsset$.

Pick $U∈\scart$ and evaluate~$Ψ$ on~$U$
to obtain a smooth simplicial set $$Ψ(U)⊂\cB(U,⟨1⟩,[0]),$$
which is the objectwise nerve of a smooth category~$\cC$.
Fix $L∈\cart$.
Objects in~$\cC_L$ are triples $(p,C,σ)$, where $(p:M→U)∈\SFEmb_{1|1}$, $C$ is an $L$-family of compact globular monoidal cut $(⟨1⟩,[0])$-grids on~$\red p$,
and $σ∈\frakI_{1|1} \Eucl_{1|1}(p)_L$ is
given by an $L$-family of fiberwise open embeddings of~$p$ into~$\RR^{1|1}$.
Morphisms in~$\cC_L$ of the form $(p_1,C_1,σ_1)→(p_2,C_2,σ_2)$ are pairs $(φ,τ)$, where $φ:p_1→p_2$ is a morphism in $(\frakSFEmb_{1|1})_L$
such that $$\base φ=\id_U, \qquad (\red φ)^*C_2=C_1, \qquad \image(\red φ)⊃\ncore(\red p_2,C_2)$$
and the map $τ:U→\hat\EE^{1|1}$ satisfies $σ_2∘φ=τ⋅σ_1$.

Consider the inclusion functor $ι:Ψ_0→Ψ$ such that $Ψ_0(U)_L⊂Ψ(U)_L$ is the full subcategory on objects $(q,C_q,σ_q)$ for which
$q:\RR^{1|1}⨯U→U$ is the projection map and $σ_q=ϖ^*\id_q∈\frakI_{1|1} \Eucl_{1|1}(q)_L$, where $ϖ$ is the unique map $L→\RR^0$ in $\cart$.
For every object $(p,C,σ)∈Ψ$
the comma category $(p,C,σ)/ι$ has as its objects morphisms $$(τ⋅σ,τ):(p,C,σ)→(q,C_q,σ_q),$$
where $τ:U→\hat\EE^{1|1}$ satisfies $(\red τ⋅σ)^*C_q=C$.
In particular, setting $τ=0$ and $C_q=(\red σ)(C)$ (with the induced normal orientation of the cut) shows that the category $(p,C,σ)/ι$ is nonempty.
Morphisms in $(p,C,σ)/ι$ of the form $(τ⋅σ,τ)→(τ'⋅σ',τ')$ are pairs $$(φ,τ_0),\qquad (φ∘(τ⋅σ),τ_0⋅τ)=(τ'⋅σ',τ'),$$
so $$τ_0=τ'⋅τ^{-1}, \qquad φ=(τ'⋅σ')∘(τ⋅σ)^{-1}=(τ'∘τ^{-1})⋅(σ'⋅σ^{-1}).$$
Thus, $(p,C,σ)/ι$ is a contractible groupoid and $ι$ is a final functor.

The smooth simplicial set $Ψ_0(U)$ is the homotopy quotient of the smooth subset~$Ψ_1$ of $\frakmtCutglob(U,⟨1⟩,[0])_L(q)$
comprising cut grids with a singleton core,
by the action of the discrete group $\sm(U,\hat\EE^{1|1})$ that acts via reduced fiberwise translations.
The smooth set~$Ψ_1$ is isomorphic to two copies of the smooth set of sections of $\red q$ (corresponding to the two normal orientations of the cut),
and therefore is weakly equivalent to $\ZZ/2$ in the model category $\smsset$.
Thus, $Ψ≃Ψ_0$ is weakly equivalent to the presheaf $$U↦\ZZ/2⨯\tdeloop \sm(U, \hat\EE^{1|1}).$$

Analogously, if we drop fiberwise orientations from $\Eucl_{1|1}$, then $Ψ$ is weakly equivalent to $U↦\tdeloop \sm(U, \hat\EE^{1|1})$.
\end{example}

\begin{example}
\label{2.1.Euclidean.bordisms}
Using the $2|1$-Euclidean field stack $\Eucl_{2|1}$ of \cref{def.super.euclidean},
we define the (extended) $2|1$-Euclidean bordism category as
$$2|1\text{-}{\sf EBord}≔\Bord_2^{\Eucl_{2|1}}.$$
This provides a fully extended variant of the Stolz–Teichner bordism category \cite[Definitions 2.46 and 4.4]{StolzTeichner.SUSY}.
\end{example}

\begin{example}
\label{unoriented}
Taking $\gs=\ast$ to be the terminal object in $\Struct_d$, we get the \emph{unoriented bordism category} $\Bord^\ast_d$.
We make the following observations about the moduli space of connected 0-bordisms in the unoriented bordism category.

Recall that the simplicial set $\Bord^\ast_d(\RR^0,\langle 1\rangle,{\bf 0})$
is the nerve of the category $\CatBord^\ast_d(\RR^0,\langle 1\rangle,{\bf 0})$ (\cref{compactbords}).
We concentrate our attention on the disjoint summand~$\cA$ of this simplicial set
that comprises all 0-bordisms with connected core, i.e., 0-bordisms whose core is a single point.
The simplicial set~$\cA$ is weakly equivalent to the nerve of the category whose objects
are cut ${\bf 0}$-grids on $p:\RR^d→\RR^0$
and morphisms are cut grid-preserving smooth embeddings $\RR^d→\RR^d$.
This category is connected.
By a variant of Segal's theorem, the simplicial set~$\cA$
is weakly equivalent to the delooping of the group
of germs of diffeomorphisms $\RR^d→\RR^d$ that preserve the basepoint~0,
the coordinate hyperplanes $x_i=0$,
and the coordinate half-spaces $x_i>0$.

Denote by $\lgO(d)$ the orthogonal group (not a Lie group).
There is an injective map of sets
$$\lgO(d)\cong V_d(\RR^d)\into \Bord^{\ast}_d(\RR^0,\langle 1\rangle,{\bf 0})_0,$$
obtained by sending an orthonormal framing $\{e_i\}_{i=1}^d$ of $\RR^d$
to the compact globular monoidal cut $(\langle 1\rangle,{\bf 0})$-grid on $p:\RR^d→\RR^0$ with $C_0^i=(x_i<0, x_i=0, x_i>0)$,
where $x_i$ is the $i$th coordinate function of the framing.
This function extends to a map of simplicial sets
\begin{equation}\label{eodmap}\ast\simeq \tdelooptotal\lgO(d)=\lgO(d)\hq\lgO(d)\to \Bord^\ast_d(\RR^0,\langle 1\rangle,{\bf 0}),\end{equation}
by taking the nerve of the functor
$$\lgO(d)\rtimes \lgO(d)\to \CatBord^{\ast}_d(\RR^d,\langle 1\rangle,{\bf 0})$$
(with $\CatBord$ as in \cref{compactbords}),
defined on objects as above,
and on morphisms by sending $Q:A\to Q A$ to the linear map $Q:\RR^d\to \RR^d$, which preserves the cut grids corresponding to $A$ and $QA$.
The map \eqref{eodmap} is far from being a weak equivalence, even if we restrict to the summand of connected $0$-bordisms, as shown in the previous paragraph.
\end{example}

We now move on to examples of isotopy bordism categories.
Here we must specify the parameter $L\in \cart$, which encodes an $L$-family of isotopies of fields.

\begin{example}
\label{unoriented.isotopy}
Continuing \cref{unoriented},
taking $\gs=\ast$ to be the terminal object in $\frakStruct_d$, we get the \emph{unoriented isotopy bordism category}
$\frakBord^{\ast}_d$.
We make the following observations about the moduli space of connected 0-bordisms in the unoriented isotopy bordism category.
Let $\lgO(d)$ be the smooth set $L↦\sm(L,\lgO(d))$.
There is a monomorphism of smooth sets
$$\lgO(d)\into \frakBord^*_d(\RR^0,\langle 1\rangle,{\bf 0})_0,$$
which on $L$-points is obtained by sending an $L$-family of orthonormal framings $\{e_i\}_{i=1}^d$ of $\RR^d$
to the $L$-family of compact globular monoidal cut $(\langle 1\rangle,{\bf 0})$-grids
on $p:\RR^d→\RR^0$, with $C_0^i=(x_i<0, x_i=0, x_i>0)$,
where $x_i$ is the $i$th coordinate function of the framing.
This function extends to a map of smooth simplicial sets
$$\ast \simeq \delooptotal \lgO(d)=\lgO(d)\hq\lgO(d)\to \frakBord^{\ast}_d(\RR^0,\langle 1\rangle,{\bf 0}),\eqlabel{enriched.map.points}$$
by taking the nerve of the internal functor
$$\lgO(d)\rtimes \lgO(d)\to \frakCatBord^{\ast}_d(\RR^d,\langle 1\rangle,{\bf 0})$$
(with $\frakCatBord$ as in \cref{compactbords.isotopy}),
defined on objects as above, and on morphisms by sending an $L$-family of orthogonal matrices $Q:A\to Q\cdot A$ to the $L$-family of linear maps $Q:\RR^d\to \RR^d$,
which preserves the $L$-families of cut grids corresponding to $A$ and $QA$.

In this case, the map \cref{enriched.map.points} is a weak equivalence onto the connected component of connected 0-bordisms,
in the model structure of \cref{smsset.injective}.
Hence, the simplicial presheaf~$\cA$ of connected 0-bordisms is contractible.
This is established by reproducing the arguments of \cref{unoriented} in the isotopy context,
establishing a weak equivalence from the moduli stack~$\cA$ of connected 0-bordisms
to the delooping of the sheaf of groups
of germs of $L$-families of diffeomorphisms $\RR^d→\RR^d$ that preserve the basepoint~0,
the coordinate hyperplanes $x_i=0$, and the coordinate half-spaces $x_i>0$.
The presence of the parameter~$L$ allows for a classical argument with a “zoom in homotopy”,
which establishes the contractibility of~$\cA$.
\end{example}

\begin{example}
\label{1driem}
Recall (\cref{metricsregb}) the description of a vertex and 1-simplex in $\Bord_1^{\FRiem}(\RR^0,⟨1⟩,[1])$.
Set $d=1$, $\gs=\frakFRiem$ (\cref{friemex}), $U=\RR^0$, $\ell=1$, ${\bf m}=[1]$, $k=0$, $L=\RR$.
An object in $\frakBord_1^\frakFRiem(U,⟨ℓ⟩,{\bf m})_{L,k}$ is a triple $(M, C, g)$,
where $M$ is a 1-dimensional manifold, $C$ is an $\RR$-family of compact globular monoidal cut $(\langle 1\rangle,[1])$-grids on $M$,
and $g\in \Gamma(M\times \RR,{\rm Sym}^2(T^*M))$ is an $\RR$-family of Riemannian metrics on $M$
obtained by pulling back a Riemannian metric $g'$ on some 1-manifold $N$ along an isotopy $M\times\RR\to N$ of open embeddings $M→N$.

As an example, fix a 1-dimensional manifold $M$ and a compact globular monoidal cut $(\langle 1\rangle,[1])$-grid on $M$.
Consider the $\RR$-family of compact globular monoidal cut grids $C$ obtained by keeping the cut $C_1$ fixed
and moving the cut $C_0$ in \cref{metricsregb.isotopy.figure} in the direction of $C_1$,
with the two cuts being equal at $t=1$.
This deformation defines an isotopy that appears to collapse the Riemannian interval.
However, this is not the case, since the same isotopy also moves the source 0-bordism along the same interval.
Thus, the isotopy of 1-bordisms constructed above translates the data of a 1-bordism to the data of an isotopy of 0-bordisms given by the sources of 1-bordisms,
similar to \cref{companion.convert.isotopy.morphism}, but reversing the direction of the (invertible) process.
The object $(M,C,g)$ can be depicted by the following diagram from \cref{convert.isotopy}:
$$\xymatrix{
y\ar@{=}[r] & y \cr
x\ar[u]^-{i}\ar[r]_{f} & y,\ar@{=}[u] \cr
}$$
where vertical arrows and equalities represent isotopies and the horizontal arrows and equalities represent 1-morphisms.
See \cref{metricsregb.isotopy.figure} for an illustration.

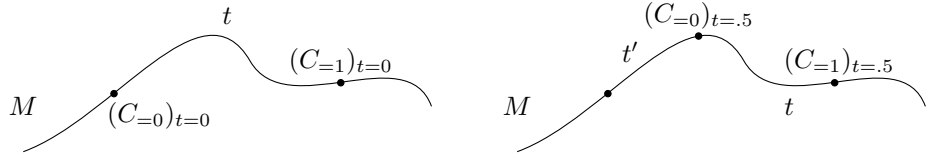
\begin{figure}[ht]
\begin{center}
\begin{tikzpicture}[scale=.6]
\draw (0,0) to [out=20,in=120] (5,2) to [out=-60, in=110] (9,1);
\filldraw (2,1.28) circle (2pt);
\filldraw (7,1.52) circle (2pt);
\node at (3,.8) {$(C_{=0})_{t=0}$};
\node at (7,2) {$(C_{=1})_{t=0}$};
\node at (4.5,3) {$t$};
\node at (0,1) {$M$};
\end{tikzpicture}
\qquad
\begin{tikzpicture}[scale=.6]
\draw (0,0) to [out=20,in=120] (5,2) to [out=-60, in=110] (9,1);
\filldraw (2,1.28) circle (2pt);
\filldraw (7,1.52) circle (2pt);
\filldraw (4,2.55) circle (2pt);
\node at (7.1,2) {$(C_{=1})_{t=.5}$};
\node at (4,3) {$(C_{=0})_{t=.5}$};
\node at (6,1) {$t$};
\node at (2.5,2.2) {$t'$};
\node at (0,1) {$M$};
\end{tikzpicture}
\end{center}
\caption{A depiction of the bordism $(M,C,g)$ in \cref{1driem}.
The left and right images illustrate the $\RR$-family of $(\langle 1\rangle,[1])$-grids $C_t$ evaluated at $t=0$ and $t=.5$, respectively.
At $t=1$, the cuts $C_{=0}$ and $C_{=1}$ coincide.
The Riemannian length of the core of each individual cut tuple changes from $t$ to $t'$ as $t$ changes from $t=0$ to $t=.5$.
However, the length of the \emph{entire $\RR$-family} is recorded throughout the deformation.
The bottom and top horizontal maps in the diagram in \cref{1driem}
correspond to $t=0$ and $t=1$, respectively.
The left map $i:x→y$ is the isotopy moving the vertex~0 from $C_0$ to~$C_1$.
The right map corresponds to $C_1$, which stays in place.
}
\label{metricsregb.isotopy.figure}
\end{figure}
\end{example}

\subsection{Examples of embedded bordisms}
\label{geometrically.framed.bordisms}

In this section, we give examples of bordisms in the case where the field stack is a representable object.
In this case, the field stack does not satisfy homotopy descent.
Hence, the bordism category is not local with respect to the Segal maps.
Nevertheless, this example is essential for our formulation of the axioms (\cref{axioms,frakaxioms})
and our proof of locality and the geometric cobordism hypothesis in \cite{GradyPavlov.Loc,GradyPavlov.GCH}.

We begin with the nonisotopy case.

\begin{example}
\label{one.dimensional.bordisms}
Set $d=1$, $U=\RR^0$, $\ell=1$, ${\bf m}=[0]$, $k=0$.
Consider the representable presheaf $\Yo{p}$ on the object $(p:\RR^1\to \RR^0)\in \smFEmb_d$.
A $0$-bordism in $\Bord_d^{\Yo{p}}(U,⟨ℓ⟩,{\bf m})_k$ is given by a triple $(M,C,i)$,
where $M$ is a smooth 1-manifold,
$C$ is a compact globular monoidal cut $(\langle 1\rangle,[0])$-grid on $M$,
and $i:M\to \RR$ is an open embedding.
We call a 0-bordism $(M,C,i)$ a \emph{point} if the core is connected.
Observe that the normal orientation of the cut $C$ induces an orientation on a sufficiently small open neighborhood of the core.
We call a point \emph{positive} if $i$ is orientation-preserving (with respect to the orientation induced by $C$ and the standard orientation of $\RR$).
We call a point \emph{negative} if it is not positive.

For example, we have a positive point $+_s$ given by taking $M=\RR$, $i=\id$, and the cut grid given by a single cut $C=(C_{<0},C_{=0},C_{>0})=(x<s, x=s,x>s)$.
Similarly, we have the \emph{negative point} $-_s$, given by again taking $M=\RR$ and $i=\id$, but setting $C=(C_{<0},C_{=0},C_{>0})=(x>s,x=s,x<s)$.
Taking $i=-\id$ instead of $i=\id$ yields another positive point $+_{\bar s}$ and another negative point $-_{\bar s}$.
In the notation $±_s$, $±_{\bar s}$, the subscript is $C_{=0}=s$, with a bar on top if $i=-\id$.
We illustrate points by a line segment (depicting~$M$) with an arrow pointing in the direction from $C_{<0}$ to $C_{>0}$, as in \cref{posandnegpointsfig}.
The cases $i=\id$ and $i=-\id$ are indicated by the red and blue color, respectively.
\begin{figure}[ht]
\begin{center}
\tikzset{->-/.style={decoration={markings,mark=at position #1 with {\arrow[scale=1.2]{>}}},postaction={decorate}}}
\tikzset{-<-/.style={decoration={markings,mark=at position #1 with {\arrow[scale=-1.2]{>}}},postaction={decorate}}}
\begin{tikzpicture}

\draw[-<-=.15,thick, red] (-5.5,-1) -- (-4.5,-1);
\node at (-5,-1) {$\bullet$};
\node at (-5,-1+.4) {$-_0$};

\node at (-4,-1) {$≔$};

\begin{scope}[xshift=-4em]
\draw (-2,0) -- (2,0);
\node at (0,0) {$\bullet$};
\node at (0,-.4) {$C_{=0}$};
\draw (0,-.2) -- (0,.2);
\node at (-1.5,-.4) {$C_{>0}$};
\node at (1.5,-.4) {$C_{<0}$};
\draw[->] (0,-.7)--(0,-1.7);
\draw (-2,-2) -- (2,-2);
\node at (.7,-1) {$i=\id$};
\node at (1.5,-1.7) {$\RR$};
\node at (1.5,.3) {$M=\RR$};
\end{scope}
\end{tikzpicture}
\qquad
\qquad
\begin{tikzpicture}

\draw[->-=.3,thick, blue] (-5.5,-1) -- (-4.5,-1);
\node at (-5,-1) {$\bullet$};
\node at (-5,-1+.4) {$-_{\bar0}$};

\node at (-4,-1) {$≔$};

\begin{scope}[xshift=-4em]
\draw (-2,0) -- (2,0);
\node at (0,0) {$\bullet$};
\node at (0,-.4) {$C_{=0}$};
\draw (0,-.2) -- (0,.2);
\node at (-1.5,-.4) {$C_{<0}$};
\node at (1.5,-.4) {$C_{>0}$};
\draw[->] (0,-.7)--(0,-1.7);
\draw (-2,-2) -- (2,-2);
\node at (.7,-1) {$i=-\id$};
\node at (1.5,-1.7) {$\RR$};
\node at (1.5,.3) {$M=\RR$};
\end{scope}
\end{tikzpicture}
\\
\vspace{1cm}
\begin{tikzpicture}

\draw[->-=.3,thick, red] (-5.5,-1) -- (-4.5,-1);
\node at (-5,-1) {$\bullet$};
\node at (-5,-1+.4) {$+_0$};

\node at (-4,-1) {$≔$};

\begin{scope}[xshift=-4em]
\draw (-2,0) -- (2,0);
\node at (0,0) {$\bullet$};
\node at (0,-.4) {$C_{=0}$};
\draw (0,-.2) -- (0,.2);
\node at (-1.5,-.4) {$C_{<0}$};
\node at (1.5,-.4) {$C_{>0}$};
\draw[->] (0,-.7)--(0,-1.7);
\draw (-2,-2) -- (2,-2);
\node at (.7,-1) {$i=\id$};
\node at (1.5,-1.7) {$\RR$};
\node at (1.5,.3) {$M=\RR$};
\end{scope}
\end{tikzpicture}
\qquad
\qquad
\begin{tikzpicture}

\draw[-<-=.15,thick, blue] (-5.5,-1) -- (-4.5,-1);
\node at (-5,-1) {$\bullet$};
\node at (-5,-1+.4) {$+_{\bar0}$};

\node at (-4,-1) {$≔$};

\begin{scope}[xshift=-4em]
\draw (-2,0) -- (2,0);
\node at (0,0) {$\bullet$};
\node at (0,-.4) {$C_{=0}$};
\draw (0,-.2) -- (0,.2);
\node at (-1.5,-.4) {$C_{>0}$};
\node at (1.5,-.4) {$C_{<0}$};
\draw[->] (0,-.7)--(0,-1.7);
\draw (-2,-2) -- (2,-2);
\node at (.7,-1) {$i=-\id$};
\node at (1.5,-1.7) {$\RR$};
\node at (1.5,.3) {$M=\RR$};
\end{scope}
\end{tikzpicture}
\end{center}
\caption{Positive and negative points in the embedded bordism category.
The red line segments represent bordisms with the embedding $i=\id_\RR$.
The blue line segments represent bordisms with the embedding $i=-\id_\RR$.
The arrow indicates the normal orientation of the cut~$C$ in~$M$.
Bordisms labeled~$+$ are virtually isomorphic (corresponding to $k=1$), by the cut-respecting embedding $\varphi=-\id$.
Similarly, bordism labeled~$-$ are isomorphic.}
\label{posandnegpointsfig}
\end{figure}
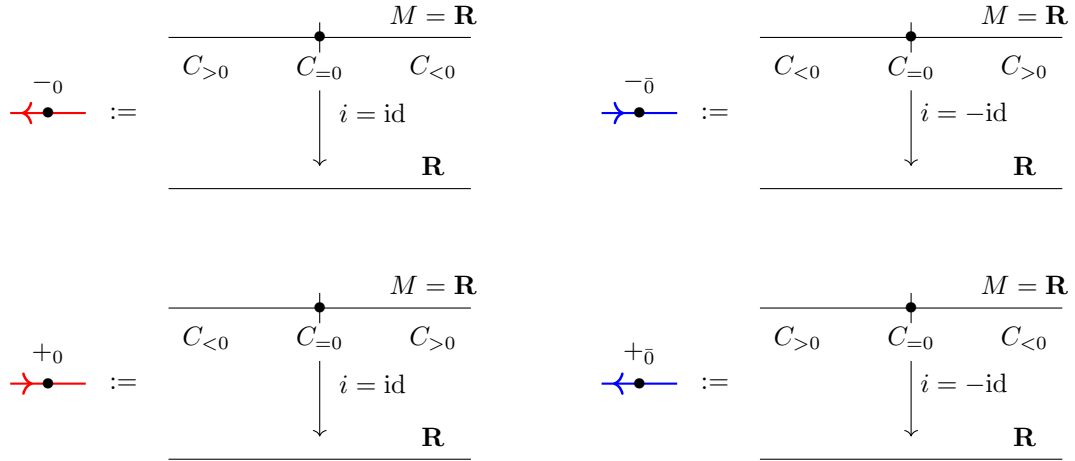

The positive point~$+_0$ is not isomorphic to the negative point~$-_0$.
Indeed, a virtual isomorphism $\varphi:(M,C,i)\to (M',C',i')$ in the bordism category is given by an open embedding $\varphi:M\into M'$
such that $i'\circ \varphi=i$, $\varphi^*C'=C$, and $\image\varphi⊃\ncore(p,C')=C'_{=0}$,
where the second condition implies the third for the case of connected 0-bordisms.
Suppose we have a virtual isomorphism $\varphi:+_0\to -_0$.
Then since $\varphi$ has to commute with $i=i'=\id$, we must have $\varphi=\id$.
But then $\varphi^*C'=C'≠C$, since the normal orientation of~$C'$ is opposite that of~$C$, contradicting the fact that $\varphi$ is a virtual isomorphism.

Since we are in the nonisotopy case, the value of $i(s)$ is an isomorphism invariant, so $+_s\not\cong +_t$ for $s\neq t$.
In fact, the moduli stack of points is weakly equivalent to the disjoint union of two copies
of the representable presheaf of~$\RR$ (on the site $\cart$),
indicating $i(s)$ and whether the point is positive or negative.
\end{example}

\begin{example}
\label{elbows}
Set $d=1$, $U=\RR^0$, $\ell=1$, ${\bf m}=[1]$, $k=0$.
Fix real numbers $s<t$.
We have a 1-bordism $$\epsilon_{s,t}: +_s⊔-_t→∅,$$ which we call a \emph{right elbow}, constructed as follows.
The domain of~$\epsilon_{s,t}$ is the 0-bordism $+_s⊔-_t$ given by $(M=\RR,C,i=\id)$,
whose only cut~$C$ is given by
$$C=((-\infty,s)\cup (t,\infty),\{s,t\}, (s,t)).$$
(See the construction of the map~$\tilde x$ in \cref{frcircle} for an explanation of the notation $+_s⊔-_t$.)

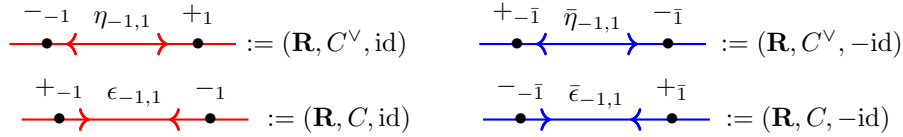
\begin{figure}[ht]
\begin{center}
\tikzset{->-/.style={decoration={markings,mark=at position #1 with {\arrow[scale=1.2]{>}}},postaction={decorate}}}
\tikzset{-<-/.style={decoration={markings,mark=at position #1 with {\arrow[scale=-1.2]{>}}},postaction={decorate}}}
\begin{tikzpicture}
\draw[-<-=.25, ->-=.7, thick, red] (-2.5,-3)--(.5,-3);
\node at (-2,-3) {$\bullet$};
\node at (0,-3) {$\bullet$};
\node at (-2, -2.6) {$-_{-1}$};
\node at (0,-2.6) {$+_1$};
\node at (-1,-2.7) {$\eta_{-1,1}$};
\node at (1.7,-3) {$≔ (\RR,C^{\vee},\id)$};
\end{tikzpicture}
\qquad
\begin{tikzpicture}
\draw[-<-=.25, ->-=.7, thick, blue] (-2.5,-3)--(.5,-3);
\node at (-2,-3) {$\bullet$};
\node at (0,-3) {$\bullet$};
\node at (-2, -2.6) {$+_{-\bar1}$};
\node at (0,-2.6) {$-_{\bar1}$};
\node at (-1,-2.7) {$\bar\eta_{-1,1}$};
\node at (1.9,-3) {$≔ (\RR,C^{\vee},-\id)$};
\end{tikzpicture}

\begin{tikzpicture}
\draw[->-=.3, -<-=.65, thick, red] (-2.5,-3)--(.5,-3);
\node at (-2,-3) {$\bullet$};
\node at (0,-3) {$\bullet$};
\node at (-2, -2.6) {$+_{-1}$};
\node at (0,-2.6) {$-_1$};
\node at (-1,-2.7) {$\epsilon_{-1,1}$};
\node at (1.7,-3) {$≔ (\RR,C,\id)$};
\end{tikzpicture}
\qquad
\begin{tikzpicture}
\draw[->-=.3, -<-=.65, thick, blue] (-2.5,-3)--(.5,-3);
\node at (-2,-3) {$\bullet$};
\node at (0,-3) {$\bullet$};
\node at (-2, -2.6) {$-_{-\bar1}$};
\node at (0,-2.6) {$+_{\bar1}$};
\node at (-1,-2.7) {$\bar\epsilon_{-1,1}$};
\node at (1.8,-3) {$≔ (\RR,C,-\id)$};
\end{tikzpicture}
\end{center}
\caption{Elbows in the geometrically framed bordism category.
The two elbows labeled $\epsilon$ are isomorphic by the map $-\id$.
Similarly, the two bordisms labeled $\eta$ are isomorphic.}
\label{1bordisms.framed}
\end{figure}

Consider the 1-bordism $\epsilon_{s,t}$ given by the triple $(M=\RR,C,i=\id)$,
where $C$ is the compact globular monoidal cut $(\langle 1\rangle,[1])$-grid given by the two cuts
\begin{equation}\label{epsgrid}((-\infty,s)\cup (t,\infty),\{s,t\}, (s,t))\leq (\RR,\emptyset,\emptyset).\end{equation}

We can also encode the \emph{left elbow}~$\eta_{s,t}:\emptyset\to -_s\sqcup +_t$ by the triple $(M=\RR^d,C^{\vee},i=\id)$, where the cut grid $C^{\vee}$ is given by
\begin{equation}\label{etagrid}(\emptyset,\emptyset,\RR)\leq ((s,t),\{s,t\},(-\infty,s)\cup (t,\infty)).\end{equation}
Taking $i=-\id$, we get two additional elbows $\bar\epsilon_{s,t}$ and $\bar\eta_{s,t}$.
We depict several of these right and left elbows in \cref{1bordisms.framed}.
We have virtual isomorphisms $\epsilon_{s,t}→\bar\epsilon_{-t,-s}$
and $\eta_{s,t}→\bar\eta_{-t,-s}$, given by $φ=-\id$.

Since we are in the nonisotopy case, the value of $(i(s),i(t))$ is an isomorphism invariant, so $\epsilon_{s,t}\not\cong \epsilon_{s',t'}$ for $(s,t)≠(s',t')$.
In fact, the moduli stack of right elbows is weakly equivalent to the representable presheaf of~$\{(a,b)∈\RR^2\mid a<b\}$ (on the site $\cart$), where $(a,b)=(i(s),i(t))$.
\end{example}

\begin{example}
\label{d1bords}
The 1-bordisms $\eta$ and $\epsilon$ described in \cref{elbows} can almost be identified with the unit and counit for a duality between the positive and negative point,
but not quite.
The failure of the duality is a consequence of the fact that cylinders in the bordism category are not invertible.
For example, we can form the composition that corresponds to one of the triangle identities
by composing the 1-bordisms $\epsilon_{-1,0} \sqcup \id_{+_{1}}$ and $\id_{+_{-1}}\sqcup \eta_{0,1}$:
\begin{center}
\begin{tikzpicture}
\tikzset{->-/.style={decoration={markings,mark=at position #1 with {\arrow[scale=1.2]{>}}},postaction={decorate}}}
\tikzset{-<-/.style={decoration={markings,mark=at position #1 with {\arrow[scale=-1.2]{>}}},postaction={decorate}}}

\draw[->-=.4, -<-=.8, thick, red] (-4.5,-3)--(-2,-3);
\draw[-<-=.15, ->-=.7, thick, red] (-2,-3)--(.5,-3);

\node at (-4,-3) {$\bullet$};
\node at (-2,-3) {$\bullet$};
\node at (0,-3) {$\bullet$};
\node at (-3,-3.3) {$\epsilon_{-1,0}$};
\node at (-1,-3.3) {$\eta_{0,1}$};

\node at (-4,-2.6) {$+_{-1}$};
\node at (-2, -2.6) {$-_0$};
\node at (0,-2.6) {$+_1$};
\end{tikzpicture}
\end{center}
The composition is the interval
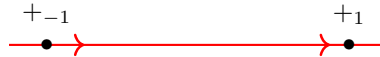
\begin{center}
\begin{tikzpicture}
\tikzset{->-/.style={decoration={markings,mark=at position #1 with {\arrow[scale=1.2]{>}}},postaction={decorate}}}
\tikzset{-<-/.style={decoration={markings,mark=at position #1 with {\arrow[scale=-1.2]{>}}},postaction={decorate}}}

\draw[->-=.4, thick, red] (-4.5,-3)--(-2,-3);
\draw[ ->-=.7, thick, red] (-2,-3)--(.5,-3);

\node at (-4,-3) {$\bullet$};
\node at (0,-3) {$\bullet$};

\node at (-4,-2.6) {$+_{-1}$};
\node at (0,-2.6) {$+_1$};
\end{tikzpicture}
\end{center}
which is not isomorphic to the identity.
Moreover, the 0-bordism $+_{-1}$ is not isomorphic to the 0-bordism $+_1$.
They do become isotopic in the isotopy bordism category via an isotopy that translates from~$-1$ to~$1$.
Likewise, the other triangle identity also needs such isotopies.

It is precisely this failure that prevents the nonisotopic bordism category $\Bord_d$ from satisfying the geometric cobordism hypothesis.
However, in the isotopy case, there is an isotopy that contracts the interval to the identity 1-bordism,
which implements duality for the 0-bordism~$+_1$.
We explore the isotopy analog in \cref{d1bordsisot}.
\end{example}

\begin{example}
\label{isotopyofpoints}
Set $d=1$, $U=\RR^0$, $\ell=1$, ${\bf m}=[0]$, $k=0$, $L=\RR$.
Consider the representable $\smset$-enriched presheaf $\Yo{p}$ on the object $(p:\RR^1\to\RR^0)\in\fraksmFEmb_d$.
An isotopy of 0-bordisms in the resulting bordism category is given by a triple $(M,C,i)$,
where $M$ is a smooth 1-manifold,
$C$ is an $\RR$-family of compact globular monoidal cut $(\langle 1\rangle,[0])$-grids on $M$,
and $i:M\times \RR\to \RR$ is an $\RR$-family of open embeddings.

We call an isotopy of 0-bordisms an \emph{isotopy of points} if the 0-bordism $(M,C_l,i_l)$ is a point for all $l\in \RR$.
We call an isotopy of points \emph{positive} if $(M,C_l,i_l)$ is positive for all $l\in \RR$ (\cref{one.dimensional.bordisms}).
Similarly, we call an isotopy of points \emph{negative} if it is not positive.
Any isotopy of points must be either positive or negative, since the space of orientations of a point is a $\ZZ/2$-torsor and $\RR$ is connected.

Fix $s,t∈\RR$ and consider the positive points $+_s$ and $+_t$ in \cref{one.dimensional.bordisms} as objects in
$\frakBord_d^{\Yo{p}}(U,⟨ℓ⟩,{\bf m})_{\RR^0,k}$.
Consider the isotopy of points $\rho^+_{s,t}=(M=\RR,C^+,M⨯L→\RR)$, where
$$C^+_l=(x<s(1-l)+tl, x=s(1-l)+tl, x>s(1-l)+tl)$$
and the map $M⨯L=\RR⨯\RR→\RR$ is given by the projection onto~$M$.
Then $\rho^+_{s,t}$ is an isotopy of points connecting $+_s$ and $+_t$.
Hence $+_s\cong +_t$ (contrast with the nonisotopy case in \cref{one.dimensional.bordisms}).
If $s≤t$, this isotopy can be converted to a 1-morphism in the bordism category via \cref{companion.convert.isotopy.morphism},
namely, an interval going from~$s$ to~$t$.
If $s>t$, by \cref{nonfibrant.target} we get a 1-morphism in the fibrant replacement of the bordism category.
In the case of negative points, we have a similar isotopy, given by $\rho_{s,t}^-=(\RR,C^-,M⨯L→\RR)$.
\end{example}

\begin{example}
\label{d1bordsisot}
The bordisms depicted in \cref{1bordisms.framed} exhibit $+_{-1}$ as the dual of $-_0$.
(Duality for $U$-families of objects in a smooth symmetric monoidal $(∞,d)$-category~$T$
can be defined as duality in the symmetric monoidal 1-category given by evaluating the fibrant replacement of~$T$ at $U∈\stcart$ and taking the homotopy 2-category.)

The presence of isotopies allows us to take $\epsilon_{-1,0}$ and $\eta_{0,1}$ as the counit and unit for the 0-bordism $+_{-1}$.
The first triangle identity is given by composing the two 1-bordisms $\epsilon_{-1,0}\sqcup \id_{+_1}$ and $\id_{+_{-1}}\sqcup \eta_{0,1}$,
which again yields the interval depicted in \cref{d1bords}.
Now however, the interval is isotopic to the identity 1-bordism on $+_{-1}$ via the isotopy of cut grids
$$C_l=\bigl((x<-1,x=-1,x>-1)\leq (x<1-2l, x=1-2l,x>1-2l)\bigr), \qquad 0\leq l\leq 1.$$
For $l=0$, this gives the cut grid on the interval in \cref{d1bords}.
For $l=1$, the cut grid is the simplicial degeneration on the cut $(x<-1,x=-1,x>1)$, implementing the identity 1-bordism on~$+_{-1}$.
For the other triangle identity, replace $\epsilon_{-1,0}$ by the isotopic bordism $\epsilon_{1,2}$ and then form a similar composition.
\end{example}

The next example shows how to compute the trace of a point in the embedded bordism category without explicitly identifying a fibrant replacement.
Although the example shows that it is possible to do this,
a more geometrically appealing calculation of the trace can be obtained by observing that
the embedded bordism category is weakly equivalent to the framed bordism category, as explained in \cref{replacewframed}.

\begin{example}
\label{frcircle}
We reexamine \cref{d1bords}
in the setting of the isotopy bordism category $\frakBord_d^{\RR→\RR^0}$
and its fibrant replacement (\cref{nonfibrant.target}).
We compute the trace of $\id_{+_1}$ by composing the 1-bordism~$\eta_{-1,1}$,
the appropriate braiding 1-morphism from the target of~$\eta_{-1,1}$ to the source of~$\epsilon_{-1,1}$, and the 1-bordism $\epsilon_{-1,1}$,
as depicted in \cref{traceofpoint}.
Set $d=1$, $U=\RR^0$, $\ell=1$.
We take ${\bf m}=[0]$ for objects and ${\bf m}=[1]$ for morphisms.
We take $L=\RR^0$ for all objects and morphisms except for the isotopy used to construct~$β$, for which we take $L=\RR^1$.
The embedding $i$ is always taken to be the restriction of the identity map $\RR→\RR$ to $M⊂\RR$.

\begin{figure}[ht]
\begin{center}
\begin{tikzpicture}
\tikzset{
->-/.style={decoration={markings,mark=at position #1 with {\arrow[scale=1.2]{>}}},postaction={decorate}},
-<-/.style={decoration={markings,mark=at position #1 with {\arrow[scale=-1.2]{>}}},postaction={decorate}}
}

\draw[-<-=.25, ->-=.7, thick, red] (3,-2.5)--(3,.5);
\node at (3,-2) {$\bullet$};
\node at (3,0) {$\bullet$};
\node at (2.6,-2) {$-_{-1}$};
\node at (2.6,0) {$+_1$};
\node at (3,1) {$x$};

\draw[-<-=.6, thick, red] (5,-2.5)--(5,-1.2);
\draw[->-=.4, thick, red] (5,-.8)--(5,.5);
\node at (5,-2) {$\bullet$};
\node at (5,0) {$\bullet$};
\node at (4.6,-2) {$-_{-1}$};
\node at (4.6,0) {$+_1$};
\node at (5,1) {$x'$};

\draw[-<-=.6, thick, red] (7,-2.5)--(7,-1.2);
\draw[->-=.4, thick, red] (7,-.8)--(7,.5);
\node at (7,-1) {$\otimes$};
\node at (7,-2) {$\bullet$};
\node at (7,0) {$\bullet$};
\node at (6.6,-2) {$-_{-1}$};
\node at (6.6,0) {$+_{1}$};
\node at (7,1) {$z_1\otimes z_2$};

\draw[->-=.75, thick, red] (9,-2.5)--(9,-1.2);
\draw[-<-=.25, thick, red] (9,-.8)--(9,.5);
\node at (9,-1) {$\otimes$};
\node at (9,-2) {$\bullet$};
\node at (9,0) {$\bullet$};
\node at (8.6,-2) {$+_{1}$};
\node at (8.6,0) {$-_{-1}$};
\node at (9,1) {$z_2\otimes z_1$};

\draw[->-=.75, thick, red] (11,-2.5)--(11,-1.2);
\draw[-<-=.25, thick, red] (11,-.8)--(11,.5);
\node at (11,-1) {$\otimes$};
\node at (11,-2) {$\bullet$};
\node at (11,0) {$\bullet$};
\node at (10.6,-2) {$+_{-1}$};
\node at (10.6,0) {$-_1$};
\node at (11,1) {$z_2'\otimes z'_1$};

\draw[->-=.75, thick, red] (13,-2.5)--(13,-1.2);
\draw[-<-=.25, thick, red] (13,-.8)--(13,.5);
\node at (13,-2) {$\bullet$};
\node at (13,0) {$\bullet$};
\node at (12.6,-2) {$+_{-1}$};
\node at (12.6,0) {$-_1$};
\node at (13,1) {$y'$};

\draw[->-=.32, -<-=.65, thick, red] (15,-2.5)--(15,.5);
\node at (15,-2) {$\bullet$};
\node at (15,0) {$\bullet$};
\node at (14.6,-2) {$+_{-1}$};
\node at (14.6,0) {$-_1$};
\node at (15,1) {$y$};

\draw[->] (1.5,-1)--(2.5,-1); \node at (1,-1) {$\emptyset$}; \node at (2,-.5) {$\eta_{-1,1}$};
\draw[<-] (3.5,-1)--(4.5,-1); \node at (4,-.5) {$f'$};
\draw[<-] (5.5,-1)--(6.5,-1); \node at (6,-.5) {$f$};
\draw[->] (7.5,-.5) to[out=-30,in=150] (8.5,-1.5); \node at (8,0) {$b$};
\draw[->] (7.5,-1.5) to[out=30,in=-150] (8.5,-.5);
\draw[->] (9.5,-.5)--(10.5,-.5); \node at (10,0) {$β$};
\draw[->] (9.5,-1.5)--(10.5,-1.5);
\draw[->] (11.5,-1)--(12.5,-1); \node at (12,-.5) {$g$};
\draw[->] (13.5,-1)--(14.5,-1); \node at (14,-.5) {$g'$};
\draw[->] (15.5,-1)--(16.5,-1); \node at (17,-1) {$\emptyset$}; \node at (16,-.5) {$\epsilon_{-1,1}$};

\node at (5.4,.3) {$x'_2$};
\node at (5.4,-2.3) {$x'_1$};
\node at (13.4,.3) {$y'_1$};
\node at (13.4,-2.3) {$y'_2$};
\node at (7.4,.3) {$z_2$};
\node at (7.4,-2.3) {$z_1$};
\node at (9.4,.3) {$z_1$};
\node at (9.4,-2.3) {$z_2$};
\node at (11.4,.3) {$z'_1$};
\node at (11.4,-2.3) {$z'_2$};
\end{tikzpicture}
\end{center}
\caption{The trace of the point $\id_{+_1}$ in the fibrant replacement of $\frakBord_d^{\RR^1\to \RR^0}$.
The 1-bordisms $\eta_{-1,1}$ and $\epsilon_{-1,1}$ are described in \cref{elbows}.
The remaining morphisms are virtual isomorphisms or isotopies that are converted to invertible 1-bordisms using \cref{companion.convert.virtual.isomorphism.morphism} or \cref{companion.convert.isotopy.morphism}, respectively.
Inverting arrows pointing to the left and forming the composition computes the trace.}
\label{traceofpoint}
\end{figure}

First, we adjust the target of~$\eta_{-1,1}$ and the source of~$\epsilon_{-1,1}$ using virtual isomorphisms (\cref{cats.defs,virtual.isomorphism}) $f'$ and $g'$,
to ensure that the ambient manifold~$M$ splits as a disjoint union of two intervals.
The cut grid on $\eta_{-1,1}$ is given in \cref{etagrid}, with target cut $((-1,1),\{-1,1\},(-\infty,-1)\cup (1,\infty))$ on $M_0=\RR$.
We let $x$ be the 0-bordism given by taking $M_0=\RR$, equipped with the target cut of $\eta_{-1,1}$.
Setting $M_1=(-2,0)\sqcup (0,2)$, we have an embedding $f':M_1\into M_0$.
Pulling back the monoidal cut grid defining~$x$ along~$f'$ to $M_1$ yields a 0-bordism~$x'$.

Following the construction of \cref{companion.monoidal.product}, we write
$$x'=x'_1\otimes x'_2=m_{x'_1,x'_2}μ,$$
where $x'_1$ has the same data as $x'$, except the map $C^⊗:M_1\to \langle 1\rangle$ sends $(0,2)\mapsto \ast\in \langle 1\rangle$.
Similarly, $x'_2$ has the same data as $x'$, except the map $C^⊗:M_1\to \langle 1\rangle$ sends $(-2,0)\mapsto \ast\in \langle 1\rangle$.
We have $\ncore(x'_1)=\{-1\}$ and $\ncore(x'_2)=\{1\}$, since the construction of the core in \cref{nonembedded.core} discards the trash bin.
The multiplied pair $m_{x'_1,x'_2}$ in \cref{companion.monoidal.product} exhibiting the monoidal product is given by the same data as $x'$,
but with the map $C^⊗:M_1\to \langle 2\rangle$ sending $(-2,0)\mapsto 1\in \langle 2\rangle$ and $(0,2)\mapsto 2\in \langle 2\rangle$.

Next, we modify $x_1'$ and $x_2'$, by throwing away the trash bin using the map~$f$ as follows.
Consider the bordism~$z_1$ with ambient manifold given by the summand $(-2,0)$ in~$M_1$.
The cut grid is given by a single cut $((-1,0),\{-1\}, (-2,-1))$.
We have $\ncore(z_1)=\{-1\}$.
The inclusion $(-2,0)\into M_1$ gives a morphism $z_1\to x_1'$ in the bordism category.
Similarly, we have a bordism $z_2'$ with ambient manifold $(0,2)$ and cut grid given by $((0,1),\{1\},(1,2))$.
We have a morphism $z_2\to x_2'$ given by the inclusion.
By \cref{companion.monoidal.product}, we can form the monoidal product $z_1\otimes z_2$ by lifting the pair $(z_1,z_2)$ to a multiplied pair in the fibrant replacement.
By \cref{companion.monoidal.product.virtual.isomorphisms},
the two inclusions $z_1\to x_1'$ and $z_2\to x_2'$ induce a virtual isomorphism
$$f:z_1\otimes z_2\to x_1'\otimes x_2'$$
in the fibrant replacement.

We have an analogous construction of the virtual isomorphisms $g'$ and~$g$.
We let $y$ be the 0-bordism obtained by taking $M_2=\RR$ and monoidal cut grid given by the cut $((-\infty,-1)\cup(1,\infty),\{-1,1\},(-1,1))$,
i.e., the source of $\epsilon_{-1,1}$ in \cref{elbows}.
Let $y'$ be obtained from $y$ by pulling back the monoidal cut grid along the inclusion $g':M_1\into M_2$.
The obvious symmetric construction yields bordisms $y_1'$, $y_2'$.
We have morphisms $z'_1\to y_1'$ and $z'_2\to y_2'$ given by the inclusions of the corresponding summands, hence also a virtual isomorphism
$$g:z'_2\otimes z'_1\to y_2'\otimes y_1'=y'.$$

Next, let $\beta_1:z_1\to z'_1$ and $\beta_2:z_2\to z_2'$ be the isotopies $\beta_1(x,t)=2t+x$ and $\beta_2(x,t)=x-2t$, which translate the intervals $(-2,0)$ and $(0,2)$ to $(0,2)$ and $(-2,0)$, respectively.
We use the construction in \cref{companion.monoidal.product.isotopies} to form the monoidal product $\beta=\beta_2\otimes \beta_1$ of the pair of isotopies,
where the map $\partial m_{\beta_1,\beta_2}$ is determined by the data for the monoidal products $z_2\otimes z_1$ and $z_2'\otimes z_1'$.

Finally, we form the braiding virtual isomorphism $b:z_1\otimes z_2\to z_2\otimes z_1$, as described in \cref{braidingiso}.
We convert all five virtual isomorphisms above into 1-morphisms
using \cref{companion.convert.virtual.isomorphism.morphism,companion.convert.virtual.isomorphism.opposite.morphism}.
We convert the isotopy~$\beta$ to a 1-morphism using \cref{companion.convert.isotopy.morphism}.
Using \cref{companion.compose.morphisms},
we then form the composition
$$\tr(\id_{+_1})=\epsilon_{-1,1}\circ g'\circ g\circ \beta\circ b\circ f^{-1}\circ (f')^{-1}\circ \eta_{-1,1},$$
which is a 1-bordism in the fibrant replacement of the bordism category $\frakBord_1^{\RR^1\to \RR^0}$.
That this computes the trace follows from the fact that $\epsilon_{-1,1}$ and $\eta_{-1,1}$ are the counit and unit for a duality, as observed in \cref{d1bordsisot}.
\end{example}

\begin{example}
\label{replacewframed}
Let
$$\Fr_d\in \frakStruct_d$$
be the field stack with isotopies that sends a submersion $p:M\to U$ with $d$-dimensional fibers to the smooth set of framings of the vertical tangent bundle of~$p$.
That is,
$\Fr_d(p)_L$ is the set of isomorphisms of vector bundles
$$(r⨯\id_L:\T_pM⨯L→M⨯L) \quad \lto5{} \quad (π_{2,3}:\RR^d⨯M⨯L→M⨯L),$$
where $\T_pM=\ker \T p$ is the vertical tangent bundle of $p:M→U$
and $r:\T_p M→M$ is the projection map.

We have a canonical map
\begin{equation}\label{maptoframings}ζ:\Yo{\RR^d\to \RR^0}\to \Fr_d\end{equation}
induced by the enriched Yoneda lemma from the element of $\Fr_d(\RR^d→\RR^0)$
given by the canonical trivialization of the tangent bundle of~$\RR^d$.
The map~$ζ$ induces a map
$$\frakBord_d^ζ: \frakBord_d^{\Yo{\RR^d→\RR^0}}→\frakBord_d^{\Fr_d}$$
of bordism categories, whose domain is the $d$-dimensional embedded bordism category.

The comparison map~$ζ$ is a local weak equivalence by \scref{geometric.framings}.
Hence, the map $\frakBord_d^ζ$ is a local weak equivalence by \ecref{a1a2a3}.

Continuing \cref{frcircle},
which computes the trace of the identity morphism on a point in $\frakBord_d^{\Yo{\RR^d→\RR^0}}$,
we explain how to compute the trace of its image in $\frakBord_d^{\Fr_d}$.
These bordism categories are weakly equivalent,
but the construction itself is much shorter
because $\Fr_d$ is an ∞-sheaf, whereas the weakly equivalent field stack with isotopies $\Yo{\RR^d→\RR^0}$ is not.
This enables a direct construction of the braiding bordism~$b$ below.

Take $d=1$, $U=\RR^0$, $ℓ=1$, $[m]=[3]$, $L=\RR^0$, $M=[0,4]/\{0\sim4\}≅S^1$,
and the cut $[3]$-tuple~$C$ on $p:M→U$ uniquely determined by the following data:
$$C_{<0}=∅,\qquad C_{=0}=∅,\qquad C_{(0,1)}=(3,4),\qquad C_{=1}=\{0,3\},\qquad C_{(1,2)}=(0,1)∪(2,3),$$
$$C_{=2}=\{1,2\},\qquad C_{(2,3)}=(1,2),\qquad C_{=3}=∅,\qquad C_{>3}=∅.$$
Equip $M$ with the framing induced by the standard framing of the tangent bundle of $[0,4]$.
The resulting 1-bordism is depicted below.
It computes the trace of the 0-bordism $P_+=\frakBord_d^ζ(+_0)$ given by applying the map~$\frakBord_d^ζ$ to the 0-bordism~$+_0$ of \cref{one.dimensional.bordisms}.
The dual 0-bordism $P_-=\frakBord_d^ζ(-_0)$ is obtained by equipping $P_+$ with the opposite framing.
The left and right semicircles below implement the unit~$η$ and counit~$ϵ$ of~$P_+$, respectively,
whereas the middle part is the braiding~$b$ of~$P_+$ and its dual~$P_-$.
The curved arrows denote the framing on the circle.
The arrows near the points depict the normal orientations of cuts.

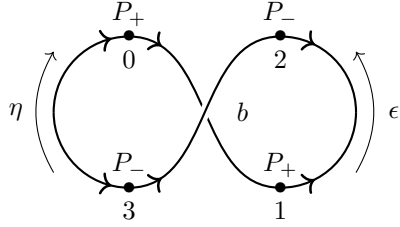
\begin{center}
\begin{tikzpicture}\label{framedcircle}
\tikzset{->-/.style={decoration={markings,mark=at position #1 with {\arrow[scale=1.2]{>}}},postaction={decorate}}}
\tikzset{-<-/.style={decoration={markings,mark=at position #1 with {\arrow[scale=-1.2]{>}}},postaction={decorate}}}
\tikzset{on top/.style={preaction={draw=white,-,line width=#1}},
on top/.default=4pt}

\node at (2,1.3) {$P_-$};
\node at (0,1.3) {$P_+$};
\node at (0,-.7) {$P_-$};
\node at (2,-.7) {$P_+$};

\draw[->-=.15, thick] (0,1) to [out=0,in=180] (2,-1);
\draw[->-=.15, thick, on top] (0,-1) to [out=0,in=180] (2,1);

\node at (2,-1) {$\bullet$};
\node at (2,1) {$\bullet$};
\node at (0,1) {$\bullet$};
\node at (0,-1) {$\bullet$};

\draw[-<-=.15, thick] (0,1) to[out=180, in=90] (-1,0);
\draw[->-=.85, thick] (-1,0) to[out=-90, in=180] (0,-1);

\draw[->-=.3, thick] (2,1) to[out=0, in=90] (3,0);
\draw[-<-=.7, thick] (3,0) to[out=-90, in=0] (2,-1);

\draw[->] (-1,-.8) to [out=120,in=240] (-1,.8);
\draw[->] (3,-.8) to [out=240+180,in=120+180] (3,.8);

\node at (3.5,0) {$\epsilon$};
\node at (1.5,0) {$b$};
\node at (-1.5,0) {$\eta$};

\node at (0,-1.3) {$3$};
\node at (0,.7) {$0$};
\node at (2,-1.3) {$1$};
\node at (2,.7) {$2$};
\end{tikzpicture}
\end{center}

With minor modifications, the same figure also describes the trace of the identity morphism on a point in the 2-dimensional framed bordism category $\frakBord_2^{\Fr_2}$.
Replace $M$ with a small open tubular neighborhood of the immersed circle in the figure
and equip it with the framing induced from the standard framing of~$\RR^2$.
Replace the sets used to define the cut tuple $C$ with their products with~$\RR$ (implementing the vertical coordinate).
The resulting bordisms implement the duality for the image of the 0-bordism $+_0$ in $\frakBord_2^{\Fr_2}$.
\end{example}

\section{Verification of the axioms}
\label{axioms.section}

In this section, we will prove that the bordism category without isotopies $\Bord_{d}$ (\cref{bord}) satisfies \cref{axiom.functorial,axiom.cocontinuous,axiom.embedded} (\cref{axioms})
and that the bordism category with isotopies $\frakBord_{d}$ (\cref{bord.isotopy}) satisfies \cref{frakaxiom.functorial,frakaxiom.cocontinuous,frakaxiom.embedded} (\cref{frakaxioms}).

\cref{axiom.functorial,frakaxiom.functorial} were established in \cref{bord,bord.isotopy}.
\cref{section.cocontinuous} proves \cref{axiom.cocontinuous,frakaxiom.cocontinuous}.
In \cref{section.embedded} we introduce
the embedded bordism categories $\EBord_d^q$ and $\frakEBord_d^q$ and use them to formulate and prove \cref{axiom.embedded,frakaxiom.embedded}.

\subsection{Proof of the second axiom}
\label{section.cocontinuous}

\begin{proposition}
\label{bord.cocontinuous}
\label{def.FFT}
Fix $d≥0$.
The $\sset$-enriched functor
$$\Bord_d:\Struct_d→\smcat{d}$$
(\cref{bord}) is an $\sset$-enriched left adjoint functor that preserves monomorphisms and objectwise weak equivalences.
In particular, $\Bord_d$ admits an $\sset$-enriched right adjoint functor
$$\FFT_d:\smcat{d}→\Struct_d,\qquad \vcat↦(q↦\map(\Bord_d^{\Yo{q}},\vcat)),$$
which sends $\vcat\in \smcat{d}$ to the field stack whose $q$-points ($q∈\FEmb_d$)
are given by the simplicial set of $\vcat$-valued functorial field theories
with fields on bordisms given by embeddings into~$q$.
\end{proposition}

\begin{proof}
Recall (\cref{bord}) that $\Bord_d$ is given by composing
$$\Struct_d⨯\stcart^\op⨯Γ^\op⨯(Δ^{⨯d})^\op\lto{8}{\CatBord_d}\smallcat^{Δ^\op}\lto5{\nerve^{Δ^\op}}\sset^{Δ^\op}$$
with the diagonal functor
$$\diag:\sset^{Δ^\op}→\sset,$$
and moving all factors except $\Struct_d$ to the right side using currying.

Since the diagonal functor is an $\sset$-enriched left adjoint functor that preserves monomorphisms and
sends objectwise weak equivalences in $\sset^{Δ^\op}$ to weak equivalences in $\sset$,
it remains to show that $\nerve^{Δ^\op}∘\CatBord_d$
is an $\sset$-enriched left adjoint functor
that preserves monomorphisms and objectwise weak equivalences.

All of these properties are verified objectwise for every $(U,⟨ℓ⟩,{\bf m})∈\stcart⨯Γ⨯Δ^{⨯d}$.
Accordingly, we work with the composite
$$\Struct_d\lto{15}{\CatBord_d(U,⟨ℓ⟩,{\bf m})}\smallcat^{Δ^\op}\lto5{\nerve^{Δ^\op}}\sset^{Δ^\op}.$$

Given $\gs∈\Struct_d$ and $[k],[l]∈Δ$,
recall (\cref{bord.explicit}) the explicit description of the set
of $k$-simplices in $\nerve(\CatBord_d^\gs(U,⟨ℓ⟩,{\bf m})_l)$:
$$(p_0,C_0,σ_0)\lto3{φ_1}(p_1,C_1,σ_1)\lto3{φ_2}⋯\lto3{φ_k}(p_k,C_k,σ_k),$$
where $σ_i∈\gs(p_i)_l$.
Since $σ_i=φ_{i+1}^*⋯φ_k^*(σ_k)$, the same data can be presented by dropping $σ_i$ from the tuple, while retaining the data of~$σ_k$:
$$α=\left( (p_0,C_0)\lto3{φ_1}(p_1,C_1)\lto3{φ_2}⋯\lto3{φ_k}(p_k,C_k) \right),\qquad σ_k∈\gs(p_k)_l.$$
Consequently, we have a natural (in $[k]$, $[l]$, and $\gs$) isomorphism
$$\nerve(\CatBord_d^\gs(U,⟨ℓ⟩,{\bf m})_l)_k ≅ ∐_{α∈\nerve(\CatBord_d^*(U,⟨ℓ⟩,{\bf m})_l)_k} \gs(p_k)_l,$$
where $\ast∈\Struct_d$ is the terminal field stack.

Fix $[k],[l]∈Δ^\op$.
Interpreting the right side as a functor in $\gs$ of the form $\Struct_d→\set$,
it preserves monomorphisms because monomorphisms are stable under disjoint unions
and it preserves small colimits because small colimits commute with small coproducts.
Thus, $\nerve^{Δ^\op}∘\CatBord_d(U,⟨ℓ⟩,{\bf m})$ preserves monomorphisms and small colimits.

Fix $[k]∈Δ^\op$.
Interpreting the right side as a functor in $\gs$ and $[l]$
of the form $\Struct_d→\sset$,
it preserves objectwise weak equivalences because weak equivalences of simplicial sets are stable under small coproducts
and it preserves tensorings because products of simplicial sets are computed levelwise
and the product functor $A⨯-$ (where $A$ is a fixed set) preserves disjoint unions.
Thus, $\nerve^{Δ^\op}∘\CatBord_d(U,⟨ℓ⟩,{\bf m})$ preserves objectwise weak equivalences and tensorings.
\end{proof}

The following isotopy variant of \cref{bord.cocontinuous} has an entirely analogous proof, working with $L$-families for some $L∈\cart$.

\begin{proposition}
\label{bord.cocontinuous.enriched}
\label{def.frakFFT}
The $\smsset$-enriched functor
$$\frakBord_d:\frakStruct_d→\fraksmcat{d}$$
is an $\smsset$-enriched left adjoint functor that preserves monomorphisms and objectwise weak equivalences.
In particular, $\frakBord_d$ admits an $\smsset$-enriched right adjoint functor
$$\frakFFT_d:\fraksmcat{d}→\frakStruct_d,\qquad \vcat↦(q↦\map_{\smsset}(\frakBord_d^{\Yo{q}},\vcat)).$$
\end{proposition}

\begin{proof}
Recall (\cref{bord.isotopy}) that $\frakBord_d$ is given by composing
$$\frakStruct_d⨯\stcart^\op⨯Γ^\op⨯(Δ^{⨯d})^\op\lto{8}{\frakCatBord_d}\smallcat^{\cart^\op⨯Δ^\op}\lto{13}{\nerve^{\cart^\op⨯Δ^\op}}\sset^{\cart^\op⨯Δ^\op}$$
with the diagonal functor
$$\diag^{\cart^\op}:\sset^{\cart^\op⨯Δ^\op}→\smsset,$$
and moving all factors except $\frakStruct_d$ to the right side using currying.

Since the diagonal functor is an $\smsset$-enriched left adjoint functor that preserves monomorphisms and
sends objectwise weak equivalences in $\smsset^{Δ^\op}$ to weak equivalences in $\smsset$,
it remains to show that $\nerve^{\cart^\op⨯Δ^\op}∘\frakCatBord_d$
is an $\smsset$-enriched left adjoint functor
that preserves monomorphisms and objectwise weak equivalences.

All of these properties are verified objectwise for every $(U,⟨ℓ⟩,{\bf m})∈\stcart⨯Γ⨯Δ^{⨯d}$.
Accordingly, we work with the composite
$$\frakStruct_d\lto{15}{\frakCatBord_d(U,⟨ℓ⟩,{\bf m})}\smallcat^{\cart^\op⨯Δ^\op}\lto{13}{\nerve^{\cart^\op⨯Δ^\op}}\smsset^{Δ^\op}.$$

Given $\gs∈\frakStruct_d$, $[k],[l]∈Δ$, and $L∈\cart$,
we have the following explicit description of the set
of $k$-simplices in $\nerve(\frakCatBord_d^\gs(U,⟨ℓ⟩,{\bf m})_{L,l})$
analogous to the unenriched case in \cref{bord.explicit},
using \cref{explicit.noncompact.isot}:
$$(p_0,C_0,σ_0)\lto3{φ_1}(p_1,C_1,σ_1)\lto3{φ_2}⋯\lto3{φ_k}(p_k,C_k,σ_k),$$
where $p_i∈\frakFEmb_d$, $C∈\frakmtCutglob(⟨ℓ⟩,{\bf m},\red p)_L$, $σ_i∈\gs(p_i)_{L,l}$,
and $φ_i:p_{i-1}→p_i$ are morphisms in $\frakFEmb_d(p_{i-1},p_i)_L$.
Since $σ_i=φ_{i+1}^*⋯φ_k^*(σ_k)$, the same data can be presented by dropping $σ_i$ from the tuple, while retaining the data of~$σ_k$:
$$α=\left( (p_0,C_0)\lto3{φ_1}(p_1,C_1)\lto3{φ_2}⋯\lto3{φ_k}(p_k,C_k) \right),\qquad σ_k∈\gs(p_k)_{L,l}.$$
Consequently, we have a natural (in $[k]$, $[l]$, $L$, and $\gs$) isomorphism
$$\nerve(\frakCatBord_d^\gs(U,⟨ℓ⟩,{\bf m})_{L,l})_k ≅ ∐_{α∈\nerve(\frakCatBord_d^*(U,⟨ℓ⟩,{\bf m})_{L,l})_k} \gs(p_k)_{L,l},$$
where $\ast∈\frakStruct_d$ is the terminal field stack with isotopies.

Fix $[k],[l]∈Δ^\op$ and $L∈\cart$.
Interpreting the right side as a functor in $\gs$ of the form $\frakStruct_d→\set$,
it preserves monomorphisms because monomorphisms are stable under disjoint unions
and it preserves small colimits because small colimits commute with small coproducts.
Thus, $\nerve^{\cart^\op⨯Δ^\op}∘\frakCatBord_d(U,⟨ℓ⟩,{\bf m})$ preserves monomorphisms and small colimits.

Fix $[k]∈Δ^\op$.
Interpreting the right side as a functor in $\gs$, $[l]$, and $L$ of the form $\frakStruct_d→\smsset$,
it preserves objectwise weak equivalences because weak equivalences of smooth simplicial sets are stable under small coproducts
and it preserves tensorings because products of smooth simplicial sets are computed levelwise (with respect to $L$ and $[l]$)
and the product functor $A⨯-$ (where $A$ is a fixed set) preserves disjoint unions.
Thus, $\nerve^{\cart^\op⨯Δ^\op}∘\frakCatBord_d(U,⟨ℓ⟩,{\bf m})$ preserves objectwise weak equivalences and tensorings.
\end{proof}

As an immediate corollary of \cref{bord.cocontinuous,bord.cocontinuous.enriched},
we have that $\Bord_d$ and $\frakBord_d$ satisfy \cref{axiom.cocontinuous,frakaxiom.cocontinuous}, respectively.
The following theorem is a precursor to a much stronger result in \ecref{a1a2a3},
which proves a similar statement for sheaves instead of presheaves,
meaning $\Struct_d$ and $\frakStruct_d$ are equipped with the Čech-local injective model structure of \cref{geometric.structure,geometric.structure.isotopy}
instead of the injective model structure.

\begin{theorem}
\label{proof.cocontinuous}
Fix $d\geq 0$.
Recall
the categories $\Struct_d$ (\cref{geometric.structure}), $\frakStruct_{d}$ (\cref{geometric.structure.isotopy}),
and the model categories $\smcat{d}$ and $\fraksmcat{d}$ (\cref{globular.model.structures}).
Denote by $\Struct_{d,\inj}$ and $\frakStruct_{d,\inj}$ the injective model structures (\cref{injective.model.structure})
on $\sPSh(\FEmb_d)$ and $\sm \sPSh(\frakFEmb_d)$, respectively.
The functor
$$\Bord_{d}:\Struct_{d,\inj}\to \smcat{d}, \quad \gs\mapsto \Bord_{d}^\gs$$
is an $\sset$-enriched left Quillen functor that preserves all weak equivalences.
In particular, it is homotopy cocontinuous.
Similarly, the functor
$$\frakBord_{d}:\frakStruct_{d,\inj}\to \fraksmcat{d}, \quad \gs\mapsto \frakBord_{d}^\gs$$
is an $\smsset$-enriched left Quillen functor that preserves all weak equivalences.
In particular, it is homotopy cocontinuous.
\end{theorem}

\begin{proof}
Recall (\cref{injective.model.structure}) that in an injective model structure
the class of cofibrations coincides with the class of monomorphisms.
In particular, all objects are cofibrant.

By \cref{bord.cocontinuous}, the functor $\Bord_d$ is an $\sset$-enriched left adjoint functor
that preserves monomorphisms and objectwise weak equivalences.
Since all objects in $\Struct_{d,\inj}$ are cofibrant and weak equivalences are objectwise weak equivalences,
the functor $\Bord_d$ is an $\sset$-enriched left Quillen functor.

By \cref{bord.cocontinuous.enriched}, the functor $\frakBord_d$ is an $\smsset$-enriched left adjoint functor
that preserves monomorphisms and objectwise weak equivalences.
Since all objects in $\frakStruct_{d,\inj}$ are cofibrant and weak equivalences are objectwise weak equivalences,
the functor $\frakBord_d$ is an $\smsset$-enriched left Quillen functor.
\end{proof}

\begin{remark}
Working in the setting of \cref{proof.cocontinuous}, we get the following formulas for $\gs∈\Struct_d$ or $\gs∈\frakStruct_d$, respectively,
where $⊗$ denotes weighted colimits:
$$\Bord_d^\gs≅\gs⊗_{\FEmb_d}\Bord_d^{\Yo{-}}≃\gs⊗_{\FEmb_d}\EBord_d≃\hocolim_{α:Δ^n⨯\Yo{p}→\gs}\EBord_d^p,$$
$$\frakBord_d^\gs≅\gs⊗_{\frakFEmb_d}\frakBord_d^{\Yo{-}}≃\gs⊗_{\frakFEmb_d}\frakEBord_d≃\hocolim_{α:Δ^n⨯\Yo{L}⨯\Yo{p}→\gs}\frakEBord_d^p.$$
As another practical consequence, for $\gs∈\Struct_d$ we get (using $\sset$-enriched cocontinuity of $\frakI_d$ and $\frakBord_d$)
$$\frakBord_d^{\frakI_d\gs}≅\frakBord_d^{\frakI_d(\gs⊗_{\FEmb_d}\Yo{-})}≅\gs⊗_{\FEmb_d}\frakBord_d^{\frakI_d\Yo{-}}≃\gs⊗_{\FEmb_d}\frakEBord_d≃\hocolim_{α:Δ^n⨯\Yo{p}→\gs}\frakEBord_d^p.$$
The main theorem of \cite{GradyPavlov.Loc} can now be interpreted as saying that the weighted colimits above
send Čech-local weak equivalences $\gs_1→\gs_2$ (and not just objectwise weak equivalences) to weak equivalences in $\smcat{d}$ or $\fraksmcat{d}$, respectively.
\end{remark}

\begin{remark}
Recall the moduli stacks of functorial field theories
$$\FFT_{d,\vcat}^\gs = \dmap(\Bord_d^\gs,\vcat),\qquad \frakFFT_{d,\vcat}^\gs = \dmap(\frakBord_d^\gs,\vcat)$$ from \cref{def.FFTspace}.
The adjunctions $\Bord_d⊣\FFT_d$ and $\frakBord_d⊣\frakFFT_d$ of \cref{def.FFT,def.frakFFT} yield
$$\FFT_{d,\vcat}^\gs ≅ \map(\gs,\rdf\FFT_d(\vcat)),\qquad \frakFFT_{d,\vcat}^\gs ≅ \map(\gs,\rdf\frakFFT_d(\vcat)),$$
where $\rdf\FFT_d$ and $\rdf\frakFFT_d$ denote the right derived functors of $\FFT_d$ and $\frakFFT_d$, respectively.
This reduces the problem of computing moduli stacks of functorial field theories
to the problem of computing field stacks $\rdf\FFT_d(\vcat)$ and $\rdf\frakFFT_d(\vcat)$.
The main theorem of \cite{GradyPavlov.Loc} enables a computation of these field stacks in terms of their values on objects of $\FEmb_d$ and $\frakFEmb_d$
whose reduction has the form $\RR^d⨯U→U$.
The main theorem of \cite{GradyPavlov.GCH} then computes the latter values as the invertible part of~$\vcat$.
\end{remark}

\subsection{Proof of the third axiom}
\label{section.embedded}

\cref{proof.cocontinuous} reduces the study of bordism categories $\Bord_d^\gs$ and $\frakBord_d^\gs$ with arbitrary field stacks~$\gs$ to the study of bordism categories
$\smash{\Bord_d^{\Yo{p:M→U}}}$ and $\smash{\frakBord_d^{\Yo{p:M→U}}}$ with representable field stacks $\Yo{p:M→U}$,
meaning bordisms are embedded into a fixed $d$-dimensional manifold~$M$, or a family thereof.
In this section, we simplify the description of the resulting bordism category further
and refer to it as the \emph{embedded bordism category}.
Variants of nonextended embedded bordism categories have appeared in the literature previously, see, for example,
Galatius–Madsen–Tillmann–Weiss \cite[Section~2]{GMTW}, Galatius–Randal-Williams \cite[Definition~3.8]{GalatiusRandalWilliams}.
The main advantage of this category is that it has no nontrivial virtual isomorphisms.
That is, the embedded bordism category is just a presheaf of sets (or smooth sets in the isotopy case),
as opposed to a simplicial presheaf (or a smooth simplicial presheaf in the isotopy case).
Another advantage is that the ambient manifolds of bordisms are open subsets of a $d$-dimensional manifold,
which can eventually be taken to be~$\RR^d$.
This simplifies the cutting and gluing constructions in the sequel papers \cite{GradyPavlov.Loc,GradyPavlov.GCH}.

A key observation (see the proof of \cref{embcat.v1.iso}) enabling the simplification of the description
of $\smash{\Bord_d^{\Yo{p:M→U}}}$ and $\smash{\frakBord_d^{\Yo{p:M→U}}}$
is that the nerve of the category $\smash{\CatBord_d^{\Yo{p:M→U}}(U,⟨ℓ⟩,{\bf m})}$
is homotopy 0-truncated, hence weakly equivalent to its set of connected components.
These connected components can be explicitly described as germs of the image of the core in~$M$, which motivates the definition of an embedded bordism (\cref{embcat.v1}).
Thus, the comparison map from $\smash{\frakBord_d^{\Yo{p:M→U}}}$ to the embedded bordism category is conceptually simple:
it sends a bordism equipped with a fiberwise embedding into $p$ to the germ of the image of its core under the embedding into $p$.

We begin by defining the embedded bordism category functors.
The following definition has the advantage of being concise and having all the details
by virtue of building upon the existing \cref{bord,bord.isotopy}.
However, it is not as simple as the isomorphic construction in \cref{embcat.v3}, which does not make use of
any constructions in \cref{categories.of.bordisms} and instead defines the embedded bordism category directly in terms of compact globular monoidal cut grids.

Fix $d≥0$ and a fibered geometric site $\FEmb_d$ (\cref{fibered.geometric.site})
or a fibered geometric site with isotopies $\frakFEmb_d$ (\cref{fibered.geometric.site.isotopy}).
Recall also the functors $\mtCutglob$ (\cref{globular.monoidal.cut.grid})
and $\frakmtCutglob$ (\cref{globular.monoidal.cut.grid.smooth}),
as well as the notation $\ncore(p,C)$ (\cref{nonembedded.core}).

\begin{definition}
\label{embcat.v1}
Fix $d≥0$.
Recall the functors $\Bord_d:\Struct_d→\smcat{d}$ (\cref{bordstr})
and $\frakBord_d:\frakStruct_d→\fraksmcat{d}$ (\cref{bord.isotopy}).
The functor
$$\EBord_d:\FEmb_d\to \smcat{d}$$
is defined as the composition
$$\FEmb_d\xrightarrow{\Yo{}}\Struct_d\xrightarrow{\Bord_d}\PSh(\stcart⨯Γ⨯Δ^{⨯d},\sset)\xrightarrow{\PSh(\stcart⨯Γ⨯Δ^{⨯d},\pi_0)}\PSh(\stcart⨯Γ⨯Δ^{⨯d},\set),$$
where the last functor is given by composing with $\pi_0:\sset\to \set$.
\label{embeddedbordcat.enriched}
Similarly, the functor
$$\frakEBord_d:\frakFEmb_d\to\fraksmcat{d}$$
is defined as the composition
$$\frakFEmb_d\xrightarrow{\Yo{}}\frakStruct_d\xrightarrow{\frakBord_d}\PSh(\stcart⨯Γ⨯Δ^{⨯d},\smsset)\xrightarrow{\PSh(\stcart⨯Γ⨯Δ^{⨯d},\smpi)}\PSh(\stcart⨯Γ⨯Δ^{⨯d},\smset),$$
where the last functor is given by composing with
$$\smpi:\smsset\to \smset, \qquad (\smpi X)_L=\pi_0(X_L),\qquad X∈\smsset,\qquad L\in \cart.$$

\label{emapconstr.unenriched}
\label{emapconstr.enriched}
We have quotient maps that are natural in~$q∈\FEmb_d$:
$$\ecm^q:\Bord_d^{\Yo{q}}\to \pi_0\Bord_d^{\Yo{q}}=\EBord_d^q, \qquad \frakecm^q:\frakBord_d^{\Yo{q}}\to \pi_0\frakBord_d^{\Yo{q}}=\frakEBord_d^q.$$
\end{definition}

We now prove that the comparison maps $\ecm^q$ and $\frakecm^q$ in \cref{emapconstr.unenriched} are weak equivalences.

\begin{definition}
\label{weakly.cofiltered}
A category~$\cC$ is \emph{weakly cofiltered} if every finite connected diagram $G:I→\cC$ admits a cone.
A presheaf~$\cD$ of categories on~$\cart$ (\cref{def.cart}) is \emph{locally weakly cofiltered}
if for every $L∈\cart$ and finite connected diagram $G:I→\cD(L)$
there is a covering family $\{f_i:U_i→L\}_{i∈I}$
such that for every $i∈I$ the diagram $\cD(f_i)∘G:I→\cD(U_i)$ admits a cone.
\end{definition}

\begin{proposition}
\label{weakly.cofiltered.contractible}
If $\cC∈\smallcat$ is weakly cofiltered (\cref{weakly.cofiltered}), then the canonical map $\nerve\cC→π_0\nerve\cC$ is a simplicial weak equivalence.
If $\cD∈\smallcat^{\cart^\op}$ is locally weakly cofiltered, then the canonical map $\nerve\cD→\tilde\pi_0\nerve\cD$ is a weak equivalence in $\smsset$.
Here $(\tilde\pi_0 F)(L)=π_0(F(L))$ is the presheaf of connected components of~$F$.
\end{proposition}

\begin{proof}
A weakly cofiltered category is a disjoint union of cofiltered categories.
Since the nerve of a cofiltered category is weakly contractible, the first claim follows.
For the second claim, observe that the map
$$\nerve\cD \to \tilde \pi_0(\nerve\cD)$$
is a stalkwise weak equivalence, equivalently,
a Čech-local weak equivalence (Jardine \cite[Section~4.1]{Jardine}, see also Pavlov \cite[Proposition~12.5]{Pavlov.Diffeo}),
and hence a weak equivalence in $\smsset$.
\end{proof}

\begin{proposition}
\label{embcat.v1.iso}
Fix $d≥0$.
For every $q∈\FEmb_d$, the presheaf $\CatBord_d^{\Yo{q}}$ (\cref{compactbords}) is valued in preorders
and the presheaf $\frakCatBord_d^{\Yo{q}}$ (\cref{compactbords.isotopy}) is valued in presheaves of preorders. 

Furthermore, the natural transformations
$$\ecm:\Bord_d∘\Yo{}\to \EBord_d, \qquad \frakecm:\frakBord_d∘\Yo{}\to \frakEBord_d$$ (\cref{emapconstr.unenriched})
of $\sset$- or $\smsset$-valued presheaves, respectively,
are objectwise weak equivalences.

Finally, the maps $\ecm$ and $\frakecm$ identify two bordisms $b_0$ and $b_1$
if and only if they are connected by a single span $b_0←b→b_2$ of bordisms,
locally with respect to an open cover of~$L$ in the case of $\frakecm$.
\end{proposition}

\begin{proof}
Fix $q∈\FEmb_d$ and $(U,⟨ℓ⟩,{\bf m})∈\stcart⨯Γ⨯Δ^{⨯d}$.
For the case of $\frakecm$, also fix $L\in \cart$.
Set $$\cC=\CatBord_d^{\Yo{q}}(U,⟨ℓ⟩,{\bf m})_0∈\smallcat \qquad{\rm or}\qquad \cD=\frakCatBord_d^{\Yo{q}}(U,⟨ℓ⟩,{\bf m})_{0}∈\smallcat^{\cart^\op},$$ respectively.
We start by showing that $\cC$ and $\cD(L)$ are preorders.

By \cref{compactbords,compactbords.isotopy},
morphisms in $\cC$ or $\cD(L)$ are given by morphisms~$φ$ in $\FEmb_d$ or $(\frakFEmb_d)_L$ with certain properties, including $\base φ=\id_U$.
Suppose $$\def\gap{\mskip 15mu } φ_1,φ_2:(p,C,σ)→(p',C',σ'), \gap φ_i:p→p',\gap \base φ_i=\id,\gap (\red φ_i)^*C'=C,\gap σ'∘φ_i=σ,\gap\image(\red φ_i)⊃\ncore(\red p',C')$$
are morphisms in~$\cC$.
Set $χ:p''→q$ to the $\base$-cartesian arrow that lifts $\base σ'$.
Use the universal property of~$χ$ to factor~$σ'$ through~$χ$ as a map $ψ:p'→p''$ such that $\base ψ=\id$ (so $ψ$ is an open embedding):
$$\xymatrix{
p \ar@<.5ex>[rr]^{\varphi_1} \ar@<-.5ex>[rr]_{\varphi_2} \ar[dr]_\sigma && p' \ar[dl]^{\sigma\smash{'}} \ar[d]^\psi\cr
& q & p''.\ar[l]^\chi\cr
}$$
Since $χψφ_1=σ=χψφ_2$ and $\base(ψφ_1)=\base(ψφ_2)=\id$, the universal property of~$χ$ implies that $ψφ_1=ψφ_2$.
To prove that $φ_1=φ_2$, it remains to show that $ψ$ is a monomorphism,
which follows from \cref{admissible.is.monomorphism,admissible.is.monomorphism.enriched}
and the fact that $ψ$ is an open embedding.

Next, we prove that $\cC$ admits pullbacks,
which implies that $\cC$ is weakly cofiltered,
and hence $\ecm$ is an objectwise weak equivalence by \cref{weakly.cofiltered.contractible}.
Every morphism $ω_1:(p_1,C_1,\sigma_1)\to (p,C,\sigma)$ in $\cC$ satisfies $\base ω_1=\id$, by definition.
By \cref{base.creates.admissible} of \cref{fibered.geometric.site}, the underlying map $ω_1:p_1\to p$ in~$\FEmb_d$ is admissible.
Since admissible maps are closed under base change (\cref{admissibility.structure}),
given any map $(p_2,C_2,\sigma_2)\to (p,C,\sigma)$ with an underlying map $ω_2:p_2\to p$ in $\FEmb_d$,
the fiber product
$$\xymatrix{
p_{12} \ar[r]^-{i_2}\ar[d]_-{i_1} & p_2 \ar[d]^-{ω_2}\cr
p_1\ar[r]_{ω_1} & p\cr
}$$
exists in $\FEmb_d$.
Take the cut grid on $p_{12}≔p_1⨯_p p_2$ to be $C_{12}=\red(ω_1 i_1)^*C$
and the map $σ_{12}:p_{12}\to q$ to be $σ_{12}=σ ω_1 i_1$.
By the commutativity of the above square, we have a commutative diagram in $\cC$:
$$\xymatrix{
(p_{12},C_{12},σ_{12}) \ar[r]^-{i_2}\ar[d]_-{i_1} & (p_2,C_2,\sigma_2) \ar[d]^-{ω_2}\cr
(p_1,C_1,\sigma_1) \ar[r]_{ω_1} & (p,C,\sigma).\cr
}$$
This last square is a pullback square: the cartesian property follows from the cartesian property of the preceding square for~$p_{12}$.
Thus, $\cC$ admits pullbacks.

Finally, we prove that $\cD$ locally admits cones for cospans,
which implies that $\cD$ is locally weakly cofiltered,
and hence $\frakecm$ is an objectwise weak equivalence by \cref{weakly.cofiltered.contractible}.
We first prove the claim when $\frakFEmb_d=\fraksmFEmb_d$ (\cref{fibered.geometric.site.isotopy}).
Let $\omega_1:(p_1,C_1,\sigma_1)\to (p,C,\sigma)$
and $\omega_2:(p_2,C_2,\sigma_2)\to (p,C,\sigma)$ be morphisms in $\cD(L)$.
By definition, $\omega_1$ and $\omega_2$ are morphisms in $(\fraksmFEmb_d)_L$,
i.e., morphisms of the form $\omega_1:p_1\times \id_L\to p\times \id_L$ and $\omega_2:p_2\times \id_L\to p\times \id_L$ in $\smFEmb_d$ (\cref{def.fraksmFEmb})
such that $\flat \omega_1=\flat \omega_2=\id_{U\times L}$.
Hence, $\omega_1$ and $\omega_2$ are open embeddings and we can form the pullback in $\smFEmb_d$, given by $\image(\omega_1)\cap \image(\omega_2)$.
This pullback need not be isomorphic in $\smFEmb_d$ to a product of the form $p_{12}⨯\id_L$
and so does not define an object in $\cD(L)$ in general.

To simplify notation below, we let $r=p\times\id_{L}$ and $r_i=p_i\times \id_{L}$, $i=1,2$.
Recall that $\ncore(r,C)\subset \image(\omega_1)\cap \image(\omega_2)$ and the restriction of $r$ to $\ncore(r,C)$ is proper.
Fix $l\in L$.
By properness, there is an open neighborhood $i_B:B\into L$ containing $l$ 
and an open subset $N\subset \image(l^*\omega_1)\cap \image(l^*\omega_2)$ containing $\ncore(l^*r,l^*C)\subset r^{-1}(U\times \{l\})$ such that 
$\ncore(i_B^*r,i_B^*C)\subset N\times B\subset \image(\omega_1)\cap \image(\omega_2)$.
Let $p_{12}$ denote the restriction of $p$ to $N$.  
Denote by $ι:p_{12}→p$ the inclusion map, which is a morphism in $\fraksmFEmb_d(p_{12},p)_{\RR^0}$.

The previous construction gives a commutative diagram 
$$\xymatrix{
(p_{12}, C_{12},\sigma_{12}) \ar[r]^-{i_2} \ar[d]_-{i_1} \ar[dr]^{\varpi^*\iota} & i_B^*(p_2,C_2,\sigma_2) \ar[d]^-{i_B^* ω_2}\cr
i_B^*(p_1,C_1,\sigma_1)\ar[r]_{i_B^* ω_1} & i_B^*(p,C,\sigma)\cr
}$$
in $\cD(B)$,
where
$\varpi:B→\RR^0$ is the terminal map
and $i_1$ and $i_2$ are induced by factoring the trivial $B$-family of inclusions $\varpi^*ι:p_{12}→p$ through the open embeddings $i_B^* ω_1$ and $i_B^* ω_2$.

For the general case of $\frakFEmb_d$, let $\omega_1$ and $\omega_2$ be morphisms in $\cD(L)$.
The previous construction gives an open cover $\{i_B:B\into L\}_{B∈I}$ and for every $B∈I$ a commutative square 
$$\xymatrix@C=4em{
p_{12} \ar[r]^-{i_2} \ar[d]_-{i_1} \ar[dr]^{\varpi^*\iota} & \red p_2 \ar[d]^-{i_B^* \red ω_2}\cr
\red p_1\ar[r]_{i_B^* \red ω_1} & \red p\cr
}$$
in $(\fraksmFEmb_d)_B$.
We lift this diagram to a commutative diagram in $\frakFEmb_d$ as follows. 
Since $\red_t:\frakFEmb_d\to \fraksmFEmb_d$ is a $\smset$-enriched admissibility fibration (\cref{enriched.admissibility.fibration}),
the morphism $ι:p_{12}\to \red p$ in $(\fraksmFEmb_d)_{\RR^0}$ can be lifted to an $\red_t$-cartesian morphism $ι':p'_{12}\to p$ in $(\frakFEmb_d)_{\RR^0}$.
By pulling back along the terminal map $ϖ:B→\RR^0$, we get the trivial $B$-family of morphisms $ϖ^* ι'$.
Now factor $ϖ^* ι'$ through the $\red$-cartesian arrows $ω_1$ and $ω_2$,
obtaining the $B$-families of morphisms $i'_1$ and $i'_2$
such that $\red i'_1=i_1$ and $\red i'_2=i_2$.
We get a commutative square in $(\frakFEmb_d)_B$:
$$\xymatrix@C=3em{
p'_{12} \ar[r]^-{i'_2} \ar[d]_-{i'_1} \ar[dr]^{ϖ^*ι'} & p_2 \ar[d]^-{i_B^* ω_2}\cr
p_1\ar[r]_{i_B^* ω_1} & p.\cr
}$$
Then we have a commutative square 
$$\xymatrix{
(p'_{12}, C_{12},\sigma_{12}) \ar[r]^-{i'_2} \ar[d]_-{i'_1} \ar[dr]^{ϖ^*ι'} & i_B^* (p_2,C_2,\sigma_2) \ar[d]^-{i_B^* ω_2}\cr
i_B^*(p_1,C_1,\sigma_1)\ar[r]_{i_B^* ω_1} & i_B^*(p,C,\sigma)\cr
}$$
in $\cD(B)$, which proves that $\cD$ locally admits cones for cospans.

To prove the final claim, suppose $b,b'∈\cC$ are identified by the quotient map $\cC→π_0\cC$.
This means that $b$ and $b'$ are connected by a chain of morphisms (in either direction) in the preorder~$\cC$.
Inserting identities if necessary, we get a chain of zigzags $b=b_0←b_1→b_2←b_3→b_4←⋯←b_{2k}=b'$.
Since $\cC$ has pullbacks, we can compose the chain of zigzags to a single zigzag, i.e., a span $b←b_1→b'$.

Likewise, if $b,b'∈\cD(L)$ are identified by the quotient map $\cD(L)→π_0\cD(L)$,
then they are connected by a chain of zigzags $b=b_0←b_1→b_2←b_3→b_4←⋯←b_{2k}=b'$.
As shown above, for every index $i∈(0,k)$ there is an open cover $\{U_j\}_{j∈J}$ of~$L$
such that for every $j∈J$ the restriction of the cospan $b_{2i-1}→b_{2i}←b_{2i+1}$ to $U_j$ admits a cone.
By passing to the common refinement of these open covers,
we get a single open cover $\{V_k\}_{k∈K}$ such that
for every $k∈K$ we have a single zigzag $b←b_1→b'$ connecting the restrictions of $b$ and $b'$ to~$V_k$.
\end{proof}

We now provide the following alternative description of the embedded bordism category,
which allows us to replace the data of a morphism~$σ$ in \cref{embcat.v1} by a subobject.

\begin{proposition}
\label{embcat.v3}
Recall \cref{embcat.v1}.
Fix $d≥0$ and $(U,⟨ℓ⟩,{\bf m})∈\stcart⨯Γ⨯Δ^{⨯d}$.

For every $q:N→V$ in $\FEmb_d$, the set $\EBord_d^q(U,⟨ℓ⟩,{\bf m})$
is isomorphic to the set~$\cD$ of equivalence classes of triples consisting of:
\begin{enumerate}
\item a morphism $f:U→V$ in $\stcart$;
\item\label{embcat.v3.subobject}
an open embedding $ρ:p→f^*q$ in $\FEmb_d$ such that $\base ρ=\id_U$;
\item a compact globular monoidal cut $(\langle \ell\rangle,{\bf m})$-grid $C∈\mtCutglob(⟨ℓ⟩,{\bf m},\red p)$.
\end{enumerate}
Two triples $(f_1,ρ_1,C_1)$ and $(f_2,ρ_2,C_2)$ are defined to be equivalent if
\begin{enumerate}[resume]
\item $f_1=f_2$;
\item there are open embeddings $τ_k:p→p_k$ in $\FEmb_d$ such that
\begin{enumerate}
\item $\base τ_1=\base τ_2=\id_U$;
\item $(\red τ_1)^*C_1=(\red τ_2)^*C_2=C$ is a compact globular monoidal cut grid;
\item $ρ_1 τ_1=ρ_2 τ_2$;
\item $\ncore(\red p_k,C_k)\subset (\red τ_k)(\red p)$ for $k∈\{1,2\}$.
\end{enumerate}
\end{enumerate}

For every $q:N→V$ in $\frakFEmb_d$ and $L∈\cart$, the set $\frakEBord_d^q(U,⟨ℓ⟩,{\bf m})_L$
is isomorphic to the set~$\cD$ of equivalence classes of triples consisting of:
\begin{enumerate}
\item a morphism $f:U→V$ in $\stcart$;
\item\label{embcat.v3.subobject.isotopy}
an $L$-family of open embeddings $ρ:p→f^*q$ in $(\frakFEmb_d)_L$ such that $\base ρ=\id_U$;
\item an $L$-family of compact globular monoidal cut grids $C$ in $\frakmtCutglob(⟨ℓ⟩,{\bf m},\red p)_L$.
\end{enumerate}
Two triples $(f_1,ρ_1,C_1)$ and $(f_2,ρ_2,C_2)$ are defined to be equivalent if
\begin{enumerate}[resume]
\item $f_1=f_2$;
\item there is an open cover $\{ι_i:U_i→L\}_{i∈I}$ of $L$ such that for every $i∈I$
there are open embeddings $τ_k:p→p_k$ in $(\frakFEmb_d)_{U_i}$
such that
\begin{enumerate}
\item $\base τ_1=\base τ_2=\id_U$;
\item $(\red τ_1)^*ι_i^*C_1=(\red τ_2)^*ι_i^*C_2=C$ is an $L$-family of compact globular monoidal cut grids;
\item $ι_i^*ρ_1 ∘ τ_1=ι_i^*ρ_2 ∘ τ_2$;
\item $\ncore(\red ι_i^*p_k,ι_i^*C_k)\subset (\red τ_k)(\red p)$ for $k∈\{1,2\}$.
\end{enumerate}
\end{enumerate}
\end{proposition}

\begin{proof}
Recall that the splitting cleavage of $\base$ provides for every morphism $f:U→V$ in $\stcart$ a $\base$-cartesian morphism $\base^*f:f^*q→q$.
In the $\smset$-enriched case, the $\smset$-enriched splitting cleavage of~$\base$ provides
for every morphism $f:U→V$ in $\stcart$ a $\base$-cartesian $L$-family of morphisms $\base^*f:f^*q→q$
(\cref{enriched.grothendieck.fibration}).

The proof given below is formulated for the unenriched case.
The same proof also works in the enriched case, by fixing some $L∈\cart$
and working with $L$-families. 

Set $$\cC=\CatBord_d^{\Yo{q}}(U,⟨ℓ⟩,{\bf m})_0.$$
An object in~$\cC$ is a triple $(p,C,σ)$,
where $p∈\FEmb_d$, $C$ is a compact globular monoidal cut $(\langle \ell\rangle,{\bf m})$-grid in $\mtCutglob(⟨ℓ⟩,{\bf m},\red p)$,
and $σ:p→q$ is a morphism in $\FEmb_d$.
A morphism $φ:(p,C,σ)→(p',C',σ')$ in~$\cC$ is
a morphism $φ:p→p'$ in $\FEmb_d$ such that $\base φ=\id$, $(\red φ)^*C'=C$, $σ'∘φ=σ$, and $\image(\red φ)⊃\ncore(\red p',C')$.
By \cref{embcat.v1.iso}, the category~$\cC$ is a preorder.

We construct a functor $κ:\cC→\cD$ as follows.
Send an object $(p,C,σ)$ to the equivalence class of the triple $(\base σ,ρ,C)$,
where $ρ:p→f^*q$ is a unique morphism such that $\base^*f∘ρ=σ$,
as provided by the universal property of the $\base$-cartesian arrow~$\base^*f$.
The morphism $ρ$ is an open embedding by \cref{factor.red.cartesian.base.cartesian}.
$$\vcenter{\xymatrix{
p \ar[r]^\rho \ar[dr]_\sigma & f^*q \ar[d]\cr
& q\cr
}}
\qquad
\lmapsto5{\base}
\qquad
\vcenter{\xymatrix{
U \ar[r]^{\id} \ar[dr]_f & U \ar[d]^f\cr
& V.\cr
}}$$

Given a morphism $φ:(p,C,σ)→(p',C',σ')$ in~$\cC$,
we claim that $$κ(p,C,σ)=(\base σ,ρ,C)\sim(\base σ',ρ',C')=κ(p',C',σ').$$
To establish the middle equivalence, we verify the conditions in its definition.
We have $\base φ=\id_U$, so $\base σ=\base σ'$.
Take $τ_1=\id_p$ and $τ_2=φ$.
Now $\base τ_1=\base τ_2=\id_U$.
The cut grid $(\red τ_1)^*C=C=(\red φ)^*C'=(\red τ_2)^*C'$ is a compact globular monoidal cut grid because $C$ is one.
To show that $ρ_1 τ_1=ρ_2 τ_2$, we rewrite it as $ρ = ρ' ∘ φ$.
To prove the latter equality, we compose with the $\base$-cartesian arrow $\base^*f$,
obtaining $\base^*f ∘ ρ = σ = σ' ∘ φ = \base^*f ∘ ρ' ∘ φ$, as desired.
The condition $\ncore(\red p_k,C_k)\subset (\red τ_k)(\red p)$ is tautological for $k=1$
and for $k=2$ it holds because $τ_2=φ$.

It remains to show that the induced map of sets $π_0 \cC→\cD$ is a bijection.
To show surjectivity, suppose $(f,ρ:p→f^*q,C)$ is a representative of an equivalence class in~$\cD$.
Then $(p,C,\base^*f∘ρ)$ is an object in~$\cC$.
Furthermore, $$κ(p,C,\base^*f∘ρ)=(\base(\base^*f∘ρ),ρ,C)=(f∘\id_U,ρ,C)=(f,ρ,C),$$
proving the surjectivity of~$κ$.

For injectivity, suppose $$κ(p_1,C_1,σ_1)=(\base σ_1,ρ_1,C_1)\sim(\base σ_2,ρ_2,C_2)=κ(p_2,C_2,σ_2).$$
Then $\base σ_1=\base σ_2$.
Denote by $τ_1:p→p_1$ and $τ_2:p→p_2$ the morphisms provided by the definition of equivalence.
Then $(p,C,σ)∈\cC$ is an object in~$\cC$,
where $C=(\red τ_1)^*C_1=(\red τ_2)^*C_2$
and $σ=σ_1 τ_1=\base^*f ρ_1 τ_1 = \base^* f ρ_2 τ_2 = σ_2 τ_2$.
Furthermore, we have morphisms in~$\cC$:
$$τ_1:(p,C,σ)→(p_1,C_1,σ_1), \qquad τ_2:(p,C,σ)→(p_2,C_2,σ_2).$$
Indeed, by assumptions on $τ_1$ and $τ_2$ we have
$\base τ_1=\base τ_2=\id_U$,
$C=(\red τ_1)^*C_1=(\red τ_2)^*C_2$,
$σ_1∘τ_1=σ=σ_2∘τ_2$,
and $\image(\red τ_k)⊃\ncore(\red p_k,C_k)$ for $k∈\{1,2\}$.
Thus, we have morphisms $(p_1,C_1,σ_1)←(p,C,σ)→(p_2,C_2,σ_2)$,
hence the equivalence classes of $(p_1,C_1,σ_1)$ and $(p_2,C_2,σ_2)$ are equal.
\end{proof}

\begin{remark}
\label{deck.transformations}
For $q\in \FEmb_d$, the category $\EBord^q_d$ does not satisfy the Segal condition in $\smcat{d}$,
since representable objects in $\sPSh(\FEmb_d)$ do not satisfy homotopy descent.
This issue can be fixed by taking the fibrant replacement $\et(q)$ of $q$.
In the case where $\FEmb_d=\smFEmb_d$, the sheaf $\et(p)$ sends a submersion $p:M\to U$
to the set of commutative squares $(f,g):p\to q$ in \cref{def.smFEmb}, but where $f$ is only required to be a fiberwise etale map.

Although the fibrancy properties of $\EBord_d^{\et(q)}$ are better,
this comes at a cost since now an object of the category $\cC=\CatBord_d^{\et(q)}(U,⟨ℓ⟩,{\bf m})_k$ (\cref{bord})
can have nontrivial automorphisms, given by deck transformations of bordisms.
For example, a disjoint union of two identical bordisms in
$\CatBord^{\et(p)}_d(U,⟨ℓ⟩,{\bf m})_k$ possesses such a deck transformation, which exchanges the two connected components via the map~$φ$ in \cref{bord}.
The main advantage of working with the nonfibrant presheaf $\EBord^q_d$ is that it is valued in sets, not categories.
\end{remark}

\begin{remark}
\label{reduction.to.enriched}
Continuing \cref{frakbordl0},
suppose $\frakFEmb_d$ is a fibered geometric site with isotopies (\cref{fibered.geometric.site.isotopy})
and $\FEmb_d$ is its underlying fibered geometric site (\cref{strip.enrichment}).
Suppose $q$ is an object in $\frakFEmb_d$ (hence also in $\FEmb_d$)
and consider the bordism category
$$\frakEBord_d^q∈\fraksmcat{d}=\smsPSh(\site\times\Gamma\times\Delta^{\times d})_{\glob}.$$
By evaluating the values of this presheaf of smooth simplicial sets at $\RR^0∈\cart$ we get
$$\EBord_d^q∈\smcat{d}=\sPSh(\site\times\Gamma\times\Delta^{\times d})_{\glob}.$$
This can be seen by comparing the explicit descriptions of these objects in
\cref{compactbords,compactbords.isotopy}, or in \cref{embcat.v3}.
\end{remark}

This completes the verification of \cref{axiom.embedded,frakaxiom.embedded}.
We are now ready to complete the proof of \cref{existencebord,existencebord2}.

\begin{theorem}
\label{unenriched.axioms.hold}
There exists a geometric symmetric monoidal $(\infty,d)$-category of bordisms (\cref{axioms}), unique up to a weakly contractible choice.
Furthermore, the bordism category $\Bord_d$ (\cref{bord}) satisfies \cref{axioms}.
\end{theorem}

\begin{proof}
\cref{axiom.functorial} was established in \cref{bord}.
By \cref{proof.cocontinuous}, \cref{axiom.cocontinuous} is satisfied.
By \cref{embcat.v1.iso}, the map $\ecm$ in \cref{emapconstr.enriched} is a weak equivalence.
This proves \cref{axiom.embedded}.

For the uniqueness part, observe that $\Struct_d$ is a homotopy cocompletion of $\FEmb_d$.
Therefore, the relative category of homotopy cocontinuous functors $\Struct_d→\smcat{d}$
equipped with a weak equivalence $F∘\Yo{}→\EBord_d$ is weakly contractible.
\end{proof}

\begin{theorem}
\label{enriched.axioms.hold}
There exists a geometric symmetric monoidal $(\infty,d)$-category with isotopies of bordisms (\cref{frakaxioms}), unique up to a weakly contractible choice.
Furthermore, the bordism category $\frakBord_d$ (\cref{bord.isotopy}) satisfies \cref{frakaxioms}.
\end{theorem}

\begin{proof}
\cref{frakaxiom.functorial} was established in \cref{bord.isotopy}.
By \cref{proof.cocontinuous}, \cref{frakaxiom.cocontinuous} is also satisfied.
By \cref{embcat.v1.iso}, the map $\frakecm$ in \cref{emapconstr.enriched} is a weak equivalence.
This proves \cref{frakaxiom.embedded}.

For the uniqueness part, observe that $\frakStruct_d$ is a homotopy cocompletion of $\frakFEmb_d$.
Therefore, the relative category of homotopy cocontinuous functors $F:\frakStruct_d→\fraksmcat{d}$
equipped with a weak equivalence $F∘\Yo{}→\frakEBord_d$ is weakly contractible.
\end{proof}

\begin{remark}
\label{uple.globular}
The multiple versions of the bordism category are defined by using the uple cut grid functors of \cref{monoidal.cut.grid,monoidal.cut.grid.smooth},
as explained in \cref{cut.grid.functor.parameter}.
We remark that the proof of the main theorem also works in the uple case.
In more detail, there is an obvious axiomatization of the uple category of bordisms
with the same \cref{axiom.functorial,axiom.cocontinuous} and with \cref{axiom.embedded} obtained by removing the word globular.
The proof of \cref{axiom.embedded} in this section carries through verbatim in the uple case.
\end{remark}

\end{document}

%% file: lectures.tikzstyles
\tikzstyle{none}=[]
\tikzstyle{new style 0}=[fill=white, draw=black, shape=circle]
\tikzstyle{new style 1}=[fill=black, draw=black, shape=circle]
\tikzstyle{new style 3}=[fill=black, draw=black, shape=circle, inner sep=0pt, minimum size=5pt]

\tikzstyle{new edge style 0}=[->]
\tikzstyle{new edge style 1}=[<-]
\tikzstyle{new edge style 2}=[-]
\tikzstyle{new edge style 3}=[-, dashed, fill=none]
\tikzstyle{new edge style 4}=[-, dashed, orange]
\tikzstyle{new edge style 5}=[-, orange]
\tikzstyle{new edge style 6}=[double, ->]
\tikzstyle{new edge style 7}=[triple, ->]
\tikzstyle{new edge style 8}=[-, thick, blue]
\tikzstyle{new edge style 9}=[-, fill={rgb,255: red,203; green,203; blue,203}]
\tikzstyle{new edge style 10}=[-, thick, red]
\tikzstyle{new edge style 11}=[-, thick]

%% file: 1-morph.tikz
\begin{tikzpicture}[scale=.5]
	\begin{pgfonlayer}{nodelayer}
		\node [style] (0) at (1, 16) {};
		\node [style] (1) at (1, 11) {};
		\node [style] (2) at (1, 6.75) {};
		\node [style] (4) at (-3, 11) {};
		\node [style] (6) at (1, 6.75) {};
		\node [style] (7) at (-3, 6.75) {};
		\node [style] (8) at (1, 16) {};
		\node [style] (9) at (-3, 16) {};
		\node [style] (10) at (-0.25, 10.5) {$C^2_{=0}$};
		\node [style] (12) at (-4.75, 11) {$C^1_{=0}=\emptyset$};
		\node [style] (13) at (2.5, 11) {$C^1_{=1}=\emptyset$};
	\end{pgfonlayer}
	\begin{pgfonlayer}{edgelayer}
		\draw (0.center) to (2.center);
		\draw [in=-90, out=-90] (1.center) to (4.center);
		\draw [in=90, out=90] (1.center) to (4.center);
		\draw [style=new edge style 3, in=-90, out=-90] (6.center) to (7.center);
		\draw [style=new edge style 3, in=90, out=90] (6.center) to (7.center);
		\draw [style=new edge style 3, in=-90, out=-90] (8.center) to (9.center);
		\draw [style=new edge style 3, in=90, out=90] (8.center) to (9.center);
		\draw (9.center) to (7.center);
	\end{pgfonlayer}
\end{tikzpicture}